\documentclass{lmcs}

\title{The internal languages of univalent categories}

\author{Niels van der Weide\lmcsorcid{0000-0003-1146-4161}}

\address{Department of Programming Languages and Compilers, E\"otv\"os Lor\'and University (ELTE), Budapest, Hungary}
\email{vanderweide@inf.elte.hu}


\usepackage{amsmath}
\usepackage{amsfonts}
\usepackage{amsthm}
\usepackage{hyperref}

\usepackage{mathtools}

\usepackage{tikz-cd}
\usepackage[all,cmtip]{xy}

\usepackage{xspace}
\usepackage{xifthen}

\usepackage[capitalize]{cleveref}
\usepackage{url}

\usepackage{dsfont}
\usepackage{graphicx}

\usepackage[shortcuts]{extdash}
\usepackage{turnstile}
\usepackage{xargs}
\usepackage{bbold}

\newcommand{\link}[1]{\textcolor{blue}{\underline{#1}}}

\newcommand{\UniMath}{\href{https://github.com/UniMath/UniMath}{\nolinkurl{UniMath}}\xspace}
\newcommand{\CoqLink}{\href{https://rocq-prover.org}{\nolinkurl{Rocq}}\xspace}

\newcommand{\coqdocbasebaseurl}{https://nmvdw.github.io/InternalLanguageUnivCats}

\newcommand{\coqdocbaseurl}{\coqdocbasebaseurl /UniMath.}
\newcommand{\urlhash}{\#}
\newcommand{\coqdocurl}[2]{\coqdocbaseurl #1.html\urlhash #2}
\newcommand{\nolinkcoqident}[1]{\nolinkurl{#1}}
\makeatletter
\newcommand{\coqident}{\begingroup\@makeother\#\@coqident}
\newcommand{\@coqident}[3][]{%
	\ifthenelse{\isempty{#2}}%
	{\nolinkcoqident{#3}}%
	{\ifthenelse{\isempty{#1}}%
		{\href{\coqdocurl{#2}{#3}}{\nolinkcoqident{#3}}}%
		{\href{\coqdocurl{#2}{#3}}{\nolinkcoqident{#1}}}}%
	\endgroup}

\makeatother

\newcommand{\conceptDef}[3]{\link{\href{\coqdocurl{#2}{#3}}{\textbf{#1}}}}
\newcommand{\constfont}[1]{\ensuremath{\mathsf{#1}}}
\newcommand{\cat}[1]{\ensuremath{\constfont{#1}}\xspace}

\newcommand{\defeq}{\coloneqq}
\newcommand{\trunc}[1]{\mathopen{}\left\Vert #1\right\Vert \mathclose{}}
\newcommand{\truncmap}[1]{| #1 |}

\newcommand{\idtoequiv}[0]{\cat{idtoequiv}}

\newcommand{\Pred}{\cat{P}}
\newcommand{\Qred}{\cat{Q}}

\newcommand{\UnivU}{\mathcal{U}}

\newcommand{\El}{\cat{El}}
\newcommand{\UnivUP}{\mathcal{U}_*}

\newcommand{\ww}{\overline{w}}
\newcommand{\xx}{\overline{x}}
\newcommand{\yy}{\overline{y}}
\newcommand{\zz}{\overline{z}}
\newcommand{\ff}{\overline{f}}
\renewcommand{\gg}{\overline{g}}
\newcommand{\ffone}{\overline{f_1}}
\newcommand{\ggone}{\overline{g_1}}
\newcommand{\fftwo}{\overline{f_2}}
\newcommand{\ggtwo}{\overline{g_2}}
\newcommand{\hh}{\overline{h}}
\newcommand{\mm}{\overline{m}}
\renewcommand{\l}{\cat{l}}
\renewcommand{\r}{\cat{r}}
\renewcommand{\ll}{\overline{\l}}
\newcommand{\rr}{\overline{\r}}
\newcommand{\FF}{\overline{\F}}
\newcommand{\GG}{\overline{\G}}
\newcommand{\HH}{\overline{\Hm}}
\newcommand{\tautau}{\overline{\tau}}
\newcommand{\thetatheta}{\overline{\theta}}
\newcommand{\etaeta}{\overline{\eta}}
\newcommand{\epsilonepsilon}{\overline{\varepsilon}}
\newcommand{\prodMor}[2]{\cat{p}_{#1, #2}}
\newcommand{\prodStable}[3]{#1^{\Pi}_{#2, #3}}
\newcommand{\functorProd}[1]{#1_{\Pi}}

\newcommand{\id}[1][]{\cat{id}_{#1}}
\newcommand{\comp}[2]{#1 \mathbin{\cdot} #2}
\newcommand{\idd}[1][]{\overline{\cat{id}}_{#1}}

\newcommand{\B}{\cat{B}}
\newcommand{\C}{\cat{C}}
\newcommand{\D}{\cat{D}}
\newcommand{\Set}{\cat{Set}}
\newcommand{\Fam}{\cat{Fam}}
\newcommand{\Op}[1]{#1^{\cat{op}}}
\newcommand{\iso}[2]{#1 \mathrel{\cong} #2}
\newcommand{\slice}[2]{#1 / #2}
\newcommand{\isodisp}[3]{#1 \mathrel{\cong_{#3}} #2}
\newcommand{\idtoiso}[2]{\cat{idtoiso}_{#1, #2}}
\newcommand{\relation}[2]{\overline{#1}_{#2}}

\newcommand{\F}{\cat{F}}
\newcommand{\G}{\cat{G}}
\newcommand{\Hm}{\cat{H}}
\newcommand{\LF}{\cat{L}}

\newcommand{\nt}{\tau}
\newcommand{\ntnt}{\overline{\tau}}

\newcommand{\dob}[2]{{#1}(#2)}
\newcommand{\dobP}[2]{{#1} \> #2}
\newcommand{\dmor}[3]{#1 \rightarrow_{#3} #2}
\newcommand{\dtwo}[3]{#1 \Rightarrow_{#3} #2}
\newcommand{\total}[1]{\int{#1}}
\newcommand{\proj}[1]{\pi_{#1}}
\newcommand{\dinvcell}[3]{#1 \cong_{#3} #2}
\newcommand{\dadjequiv}[3]{#1 \simeq_{#3} #2}
\newcommand{\fiber}[2]{#1[#2]}
\newcommand{\fiberfunctor}[2]{#1[#2]}
\newcommand{\functorfiber}[1]{#1^*}
\newcommand{\factmor}[1]{[#1]}
\newcommand{\reindex}[2]{#1^*(#2)}

\newcommand{\diagD}[0]{\overline{\Delta}}
\newcommand{\fstD}[0]{\overline{\pi_1}}
\newcommand{\sndD}[0]{\overline{\pi_2}}
\newcommand{\diagfstD}[0]{\overline{\delta_1}}
\newcommand{\diagsndD}[0]{\overline{\delta_2}}

\newcommand{\CwFFunc}[0]{\cat{T}}
\newcommand{\Ty}[0]{\cat{Ty}}
\newcommand{\Tm}[0]{\cat{Tm}}
\newcommand{\Tms}[2]{\Tm_{#1}(#2)}
\newcommand{\FTms}[1]{#1_{\Tm}}
\newcommand{\ECtx}[0]{\langle \rangle}
\newcommand{\ECtxMap}[1][]{\diamond_{#1}}
\newcommand{\CtxExt}[2]{#1 . #2}
\newcommand{\CtxPr}[1][]{\pi_{#1}}
\newcommand{\PairSub}[2]{\langle #1 , #2 \rangle}
\newcommand{\var}[0]{\cat{q}}
\newcommand{\substTy}[2]{#1^*(#2)}
\newcommand{\substTyMor}[1]{\overline{#1}}
\newcommand{\substTm}[2]{#1^*(#2)}
\newcommand{\substFunctor}[1]{#1_{\cat{sub}}}
\newcommand{\coeTm}[2]{#1 \mathbin{\uparrow} #2}
\newcommand{\compfunctor}[0]{\chi}
\newcommand{\ArrD}[1]{#1^{\downarrow}}
\newcommand{\Arr}[1]{#1^{\rightarrow}}
\newcommand{\cod}[0]{\cat{cod}}
\newcommand{\FibTerm}[1][]{\mathds{1}_{#1}}
\newcommand{\DemTy}[1]{\widehat{#1}}
\newcommand{\DemIso}[1]{\gamma_{#1}}
\newcommand{\DepSum}[1]{\Sigma_{#1}}
\newcommand{\DepSumType}[2]{\Sigma_{#1}#2}
\newcommand{\DepProd}[1]{\Pi_{#1}}

\newcommand{\ExtIdS}{=_{\cat{ext}}}
\newcommand{\ExtId}[2]{#1 =_{\cat{ext}} #2}

\newcommand{\Empty}{\mathds{O}}
\newcommand{\Term}{\mathbb{1}}
\newcommand{\FTerm}[1]{#1_{\ECtx}}
\newcommand{\Nat}{\mathbb{N}}
\newcommand{\NatC}{\mathbb{n}}

\newcommand{\Quot}{\cat{Quot}}

\newcommand{\DisjSum}{+}

\newcommand{\FNat}[1]{#1_{\Nat}}
\newcommand{\FOmega}[1]{#1_{\Omega}}

\newcommand{\twocell}{\Rightarrow}

\newcommand{\whiskerl}{\vartriangleleft}
\newcommand{\whiskerr}{\vartriangleright}

\newcommand{\lunitor}[0]{\lambda}
\newcommand{\linvunitor}[0]{\lambda^{-1}}
\newcommand{\runitor}[0]{\rho}

\newcommand{\lassociator}[0]{\alpha}
\newcommand{\rassociator}[0]{\alpha^{-1}}
\newcommand{\adjequiv}[2]{#1 \mathrel{\simeq} #2}
\newcommand{\biequiv}[2]{#1 \mathrel{\simeq} #2}
\newcommand{\invcell}[2]{#1 \mathrel{\cong} #2}
\newcommand{\adjunction}[4]{#1 \nsststile{#4}{#3} #2}
\newcommand{\identitor}[1]{#1_{\cat{i}}}
\newcommand{\compositor}[1]{#1_{\cat{c}}}

\newcommand{\UnivCat}{\cat{UnivCat}}
\newcommand{\FinLim}{\cat{FinLim}}
\newcommand{\LCCC}{\cat{LCCC}}
\newcommand{\LCCCD}{\cat{LCCC}_{\cat{D}}}
\newcommand{\LPropD}[1]{\cat{FinLim}^{#1}_{\cat{D}}}

\newcommand{\subbicat}[2]{\cat{Sub}(#1, #2)}
\newcommand{\fullsubbicat}[1]{\cat{FullSub}(#1)}
\newcommand{\subbicatD}[2]{\cat{DSub}(#1, #2)}
\newcommand{\fullsubbicatD}[1]{\cat{DFullSub}(#1)}
\newcommand{\sigmabicat}[2]{\Sigma_{#1} #2}
\newcommand{\idtoadjequiv}[2]{\cat{idtoiso}^{2, 0}_{#1, #2}}
\newcommand{\idtoinvcell}[2]{\cat{idtoiso}^{2, 1}_{#1, #2}}
\newcommand{\idtoadjequivdisp}[2]{\overline{\cat{idtoiso}}^{2, 0}_{#1, #2}}
\newcommand{\idtoinvcelldisp}[2]{\overline{\cat{idtoiso}}^{2, 1}_{#1, #2}}
\newcommand{\modif}[2]{\xymatrix@C=1pc{#1 \ar@3{->}[r] & #2}}

\newcommand{\TerminalObjectDisp}{\cat{T}}
\newcommand{\UnivCatT}{\cat{UnivCat}_{\TerminalObjectDisp}}
\newcommand{\UnivDispCat}{\cat{UnivDispCat}}
\newcommand{\UnivDispCatT}{\cat{UnivDispCat}_{\TerminalObjectDisp}}
\newcommand{\CleavingT}{\cat{Cleaving}_{\TerminalObjectDisp}}
\newcommand{\CompCat}{\cat{CompCat}_{\cat{lax}}}
\newcommand{\FullCompCat}{\cat{FullCompCat}}
\newcommand{\DFullCompCat}{\FullCompCat_{\cat{D}}}
\newcommand{\DProdType}{\cat{Prod}_{\cat{D}}}
\newcommand{\DEqualizer}{\cat{Eq}_{\cat{D}}}
\newcommand{\DUnit}{\cat{Unit}_{\cat{D}}}
\newcommand{\SigmaType}{\cat{Sigma}_{\cat{D}}}
\newcommand{\Democracy}{\cat{Dem}_{\cat{D}}}
\newcommand{\DFLCompCat}{\cat{DFLCompCat}}
\newcommand{\DFLCompCatD}{\cat{DFLCompCat}_{\cat{D}}}
\newcommand{\DFLCompCatPi}{\cat{DFLCompCat}^{\Pi}}
\newcommand{\DFLCompCatPiD}{\cat{DFLCompCat}^{\Pi}_{\cat{D}}}

\newcommand{\DFLCompCatPropD}[1]{\cat{DFLCompCat}^{#1}_{\cat{D}}}

\newcommand{\MonoidalDisp}{\cat{Mon}}
\newcommand{\UnivCatM}{\cat{UnivCat}_{\MonoidalDisp}}

\newcommand{\LocalProp}{\cat{P}}
\newcommand{\LocalPropCatT}[1]{#1_{\cat{Cat}}}
\newcommand{\LocalPropFunctorT}[1]{#1_{\cat{Fun}}}
\newcommand{\LocalPropCat}[2]{\LocalPropCatT{#1}(#2)}
\newcommand{\LocalPropFunctor}[2]{\LocalPropFunctorT{#1}(#2)}

\newcommand{\LocalAnd}[2]{#1 \land #2}
\newcommand{\StrictInitial}{\cat{Initial}^{\cat{strict}}}
\newcommand{\StableCoproduct}{\cat{Coproduct}^{\cat{stable}}}
\newcommand{\Extensive}{\cat{Extensive}}
\newcommand{\Regular}{\cat{Regular}}
\newcommand{\Exact}{\cat{Exact}}
\newcommand{\SubobjClass}{\cat{SubObjClass}}

\newcommand{\NNO}{\cat{NNO}^{\cat{P}}}

\newcommand{\Pretop}{\cat{PreTop}}
\newcommand{\APretop}{\cat{PreTop}_{\mathbb{N}}}

\newcommand{\PiPretop}{\cat{PreTop}_{\Pi}}
\newcommand{\Eltop}{\cat{ElemTop}}
\newcommand{\EltopNNO}{\cat{ElemTop}_{\mathbb{N}}}

\newcommand{\PretopP}{\cat{Is}{\Pretop}}
\newcommand{\APretopP}{\cat{Is}{\APretop}}

\newcommand{\EltopP}{\cat{Is}{\Eltop}}
\newcommand{\EltopNNOP}{\cat{Is}{\EltopNNO}}

\newcommand{\PretopCC}{\cat{CompCat}_{\Pretop}}
\newcommand{\APretopCC}{\cat{CompCat}_{\APretop}}

\newcommand{\PiPretopCC}{\cat{CompCat}_{\PiPretop}}
\newcommand{\EltopCC}{\cat{CompCat}_{\Eltop}}
\newcommand{\EltopNNOCC}{\cat{CompCat}_{\EltopNNO}}

\newcommand{\FinLimToCompCat}{\cat{H}}
\newcommand{\CompCatToFinLim}{\cat{U}}
\newcommand{\FinLimCompCatUnit}{\xi}
\newcommand{\FinLimCompCatCounit}{\zeta}

\newcommand{\FinLimToCompCatPi}{\overline{\cat{H}}_{\Pi}}
\newcommand{\CompCatToFinLimPi}{\overline{\cat{U}}_{\Pi}}
\newcommand{\FinLimCompCatUnitPi}{\overline{\xi}_{\Pi}}
\newcommand{\FinLimCompCatCounitPi}{\overline{\zeta}_{\Pi}}

\newcommand{\FinLimOb}{\FinLim^{\cat{ob}}}
\newcommand{\FinLimEl}{\FinLim^{\cat{el}}}
\newcommand{\FinLimUniv}{\FinLim^{\cat{U}}}
\newcommand{\CatUniv}{\cat{u}}
\newcommand{\FunctorUniv}[1]{#1_\cat{u}}
\newcommand{\CatEl}{\cat{el}}
\newcommand{\CatElM}{\cat{el}^{\cat{mor}}}
\newcommand{\CatElSub}{\cat{el}^{\cat{sub}}}
\newcommand{\CatElEq}{\cat{el}_{\cat{eq}}}
\newcommand{\FunctorEl}[1]{#1_{\cat{el}}}
\newcommand{\CatElIR}{\cat{El}} 

\newcommand{\UnivCatOb}{\UnivCat^{\cat{ob}}}
\newcommand{\UnivCatObTwo}{\UnivCat^{\cat{ob}_2}}

\newcommand{\CompCatUnivSub}[1]{\cat{u}_{#1}}
\newcommand{\DFLCompCatOb}{\DFLCompCat^{\cat{ob}}}
\newcommand{\DFLCompCatEl}{\DFLCompCat^{\cat{el}}}
\newcommand{\DFLCompCatUniv}{\DFLCompCat^{\cat{U}}}
\newcommand{\CompFunctorUniv}[1]{#1_\cat{u}}
\newcommand{\CatUnivC}[1][]{\overline{\CatUniv}_{#1}}
\newcommand{\CompFunctorUnivC}[1]{#1_{\CatUnivC}}
\newcommand{\CompCatElTy}{\cat{el}}
\newcommand{\CompCatElSub}{\cat{el}^{\cat{sub}}}
\newcommand{\CompCatElEq}{\cat{el}_{\cat{eq}}}
\newcommand{\CompCatFunctorEl}[1]{#1_{\cat{el}}}

\newcommand{\FinLimToCompCatUniv}{\overline{\cat{H}}_{\cat{U}}}
\newcommand{\CompCatToFinLimUniv}{\overline{\cat{U}}_{\cat{U}}}
\newcommand{\FinLimCompCatUnitUniv}{\overline{\xi}_{\cat{U}}}
\newcommand{\FinLimCompCatCounitUniv}{\overline{\zeta}_{\cat{U}}}

\newcommand{\MorToTm}[1]{\widetilde{#1}}
\newcommand{\TmToMor}[1]{\overline{#1}}

\newcommand{\codeU}{\cat{c}}
\newcommand{\isoU}{\cat{i}}
\newcommand{\NatCode}{\codeU^{\Nat}}
\newcommand{\NatCodeCtx}{\overline{\codeU^{\Nat}}}
\newcommand{\NatIso}{\isoU^{\Nat}}
\newcommand{\NatCodeTm}{\overline{\codeU^{\Nat}}}
\newcommand{\NatIsoTm}{\overline{\isoU^{\Nat}}}
\newcommand{\SubObjCode}{\codeU^{\Omega}}
\newcommand{\SubObjIso}{\isoU^{\Omega}}
\newcommand{\ResizeCode}{\codeU^{\cat{resize}}}
\newcommand{\ResizeIso}{\isoU^{\cat{resize}}}
\newcommand{\SigmaCode}{\codeU^{\Sigma}}
\newcommand{\SigmaIso}{\isoU^{\Sigma}}
\newcommand{\PiCode}{\codeU^{\Pi}}
\newcommand{\PiIso}{\isoU^{\Pi}}

\newcommand{\EltopNNOUniv}{\cat{ElemTop}_{\mathbb{N}, \CatUniv}}
\newcommand{\EltopNNOCCUniv}{\cat{DFLCompCat}_{\EltopNNO, \CatUniv}}
\newcommand{\FinLimToCompCatUnivTop}{\cat{H}_{\CatUniv}}
\newcommand{\CompCatToFinLimUnivTop}{\cat{U}_{\CatUniv}}
\newcommand{\FinLimCompCatUnitUnivTop}{\xi_{\CatUniv}}
\newcommand{\FinLimCompCatCounitUnivTop}{\zeta_{\CatUniv}}

\newcommand{\EltopNNOUnivNat}{\cat{U}_{\Nat}}
\newcommand{\EltopNNOUnivOmega}{\cat{U}_{\Omega}}
\newcommand{\EltopNNOUnivResize}{\cat{U}_{\cat{resize}}}
\newcommand{\EltopNNOUnivSigma}{\cat{U}_{\Sigma}}
\newcommand{\EltopNNOUnivProd}{\cat{U}_{\Pi}}

\newcommand{\EltopNNOTopUnivD}{\cat{ElemTop}_{\mathbb{N}, \cat{U}}^{\cat{D}}}
\newcommand{\EltopNNOCCTopUnivD}{\cat{DFLCompCat}_{\EltopNNO, \cat{U}}^{\cat{D}}}
\newcommand{\EltopNNOTopUniv}{\cat{ElemTop}_{\mathbb{N}, \cat{U}}}
\newcommand{\EltopNNOCCTopUniv}{\cat{DFLCompCat}_{\EltopNNO, \cat{U}}}

\newtheoremstyle{linkthm}
  {6pt}
  {6pt}
  {\itshape}
  {}
  {\bfseries}
  {{\bfseries .}}
  {5pt plus 1pt minus 1pt}%
  {\link{\href{#3}{\thmname{#1} \thmnumber{#2}}}}

\theoremstyle{linkthm}
\newtheorem{propL}[thm]{Proposition}

\newtheorem*{thrm*}{Theorem}

\crefname{prop}{Proposition}{Propositions}
\crefname{propL}{Proposition}{Propositions}
\crefname{lem}{Lemma}{Lemmata}
\crefname{thrm}{Theorem}{Theorems}
\crefname{cor}{Corollary}{Corollaries}

\theoremstyle{plain}
\newtheorem{problem}[thm]{Problem}

\crefname{problem}{Problem}{Problems}

\theoremstyle{definition}

\crefname{defi}{Definition}{Definitions}
\Crefname{defi}{Definition}{Definitions}
\crefname{rem}{Remark}{Remarks}
\Crefname{rem}{Remark}{Remarks}
\crefname{exa}{Example}{Examples}
\Crefname{exa}{Example}{Examples}
\crefname{axiom}{Axiom}{Axioms}
\Crefname{axiom}{Axiom}{Axioms}

\crefname{definition}{definition}{definitions}
\Crefname{definition}{Definition}{Definitions}
\crefname{proposition}{proposition}{propositions}
\Crefname{proposition}{Proposition}{Propositions}
\crefname{lemma}{lemma}{lemmas}
\Crefname{lemma}{Lemma}{Lemmas}
\crefname{corollary}{corollary}{corollaries}
\crefname{corollary}{Corollary}{Corollaries}
\crefname{theorem}{theorem}{theorems}
\Crefname{theorem}{Theorem}{Theorems}
\crefname{example}{example}{examples}
\Crefname{example}{Example}{Examples}

\newcounter{constructionCounter}
\renewcommand{\theconstructionCounter}{\arabic{section}.\arabic{propL}}
\newenvironment{construction}[2]
{\stepcounter{propL}\refstepcounter{constructionCounter}\noindent{\bfseries\link{\href{#1}{Construction \theconstructionCounter}}} {(for Problem~\ref{#2})}{\bfseries .}} 
  { \qed \par\vspace{5pt plus 1pt minus 1pt}%
   }

\crefname{constructionCounter}{Construction}{Constructions}
\Crefname{constructionCounter}{Construction}{Constructions}

\begin{document}

\begin{abstract}
Internal language theorems are fundamental in categorical logic,
since they express an equivalence between syntax and semantics.
One such theorem was proven by Clairambault and Dybjer,
who corrected the result originally by Seely.
More specifically,
they constructed a biequivalence
between the bicategory of locally Cartesian closed categories
and the bicategory of democratic categories with families with extensional identity types, $\Sigma$-types, and $\Pi$-types.
This theorem expresses that the internal language of locally Cartesian closed categories is extensional Martin-L\"of type theory with dependent sums and products.
In this paper, we study the theorem by Clairambault and Dybjer for univalent categories,
and we extend it to various classes of toposes,
among which are $\Pi$-pretoposes,
elementary toposes,
and elementary toposes with a universe.
The results in this paper have been formalized using the proof assistant Rocq and the UniMath library.

\end{abstract}

\maketitle

\tableofcontents

\section{Introduction}
\label{sec:introduction}
\textbf{Martin-L\"of type theory} serves as a foundation for (constructive) mathematics~\cite{martin-lof:1984a},
and it is used as the language of various proof assistants, such as Rocq~\cite{team:2024}, Lean~\cite{demoura:2015}, and Agda~\cite{agda}.
The basic notions of Martin-L\"of type theory are terms and types.
Terms represent the constructions in our system and every term has a type that specifies the construction.
For instance, we have function types, of which the terms are functions,
and we have product types, of which the terms are pairs.
One can interpret logic in type theory using the proposition-as-types interpretation,
where one views propositions as types and proofs as terms~\cite{howard:1980}.

An interesting feature of type theory is its treatment of equality.
More specifically, Martin-L\"of type theory comes with an internal and an external notion of equality,
namely propositional and definitional equality.
Following the propositions-as-types interpretation,
we represent propositional equality as a type,
which often is called the \textbf{identity type}.
Terms of that type are identifications of two terms.
Definitional equality on the other hand,
is not represented by a type,
but instead it is an equivalence relation on terms
such that one can mechanically verify whether two terms are definitionally equal.
From the point of view of proof assistants, one could view it as follows:
the user has to deal with propositional equality,
whereas the proof assistant deals with definitional equality.

We can roughly distinguish two flavors of Martin-L\"of type theory
depending on how we treat the identity type.
The first possibility is \textbf{extensional type theory}.
The distinguishing feature of extensional type theory is the reflection rule,
which says that these two notions of equality coincide.
More precisely, if we can prove that two terms are propositionally equal,
then these terms also are definitionally equal.
While this rule implies principles such as function extensionality and uniqueness of identity proofs,
it comes at a cost: type checking is not decidable for extensional type theory~\cite{hofmann:1997a}.

The other possibility is \textbf{intensional type theory},
which is Martin-L\"of type theory without the reflection rule.
By dropping this rule, we gain decidable type checking.
However, the cost is that we lose extensionality principles,
and thus the identity type becomes less specified.
For instance,
neither function extensionality nor uniqueness of identity proofs are provable in intensional type theory~\cite{boulier:2017,hofmann:1998}.

\paragraph*{Univalent Foundations.}
Univalent foundations~\cite{program:2013} is a version of intensional Martin-L\"of type theory
where identity types are more extensional compared to plain intensional type theory
and in which type checking is still decidable.
The distinguishing feature of univalent foundations is the \textbf{univalence axiom}.
This axiom allows us to identify types whenever we have an equivalence between them,
and thus two types have the same properties if they are equivalent.
Not only types, but also functions have a more extensional identity in this setting,
since function extensionality can be derived from the univalence axiom.

However, the univalence axiom contradicts uniqueness of identity proofs.
Uniqueness of identity proofs says that identity types are trivial,
in the sense that objects can only be identified in at most one way.
If we assume the univalence axiom,
types can be identified in multiple ways,
since there can be different equivalences between types.
For this reason, we view types differently in univalent foundations:
instead of viewing types as sets, we see \textbf{types as spaces},
terms as points, and identities as paths~\cite{program:2013}.
One can make this precise by giving a model of type theory using simplicial sets~\cite{kapulkin:2021}.

The power of the univalence axiom comes from the fact that identity types are more expressive.
More specifically, one can use the univalence axiom to prove \textbf{structure identity principles} for various mathematical structures.
Such principles say, for instance,
that identity of algebraic structures, like groups and rings, is the same as isomorphism~\cite{coquand:2013,ahrens:2021a}.
Structure identity principles thus allow us to treat isomorphic objects as equal,
and they imply that isomorphic structures have the same properties.
These principles capture the informal mathematical practice of using isomorphic objects interchangeably,
and they allow us to treat structures independently of their representation.
    
\paragraph*{Univalent Categories.}
Structure identity principles become especially interesting when considering \textbf{(higher) categorical structures}.
Categorical structures can be studied up to multiple notions of equivalence,
and we can incarnate such structures in multiple ways as types in univalent foundations.
More specifically,
every notion of equivalence of some categorical structure
gives rise to a notion of univalent structure
that satisfies a suitable structure identity principle~\cite{rasekh:2024}.

To see this phenomenon in action, let us consider the example of categories.
We can study categories up to two notions of equivalence.
The first possibility is given by \textbf{isomorphisms of categories}.
Isomorphisms give us a strict notion of equivalence of categories,
because they identify objects up to strict equality.
This way the only information carried by identity proofs of objects is the fact that they are equal.

The other possibility is given by \textbf{adjoint equivalences}\footnote{If we say adjoint equivalence, we mean an adjunction whose unit and counit are isomorphisms. Equivalences of categories are slightly weaker, since they do not necessarily satisfy the triangle equations of adjunctions. Every equivalence can be refined into an adjoint equivalence~\cite{maclane:2013}.}.
Adjoint equivalences are weaker in nature than isomorphisms,
because they identify objects up to isomorphism
rather than the stricter notion of equality.
By viewing objects up to isomorphism,
identity of objects inherently carries an isomorphism,
and, as a result, we can identify objects potentially in multiple ways.

The notion of \textbf{univalent category}~\cite{ahrens:2015} allows us
to study categories up to adjoint equivalence in univalent foundations.
Univalent categories are defined to be categories
in which identity of objects corresponds to isomorphism.
The structure identity principle for such categories says
that identity is the same as adjoint equivalence.
Note that this principle reflects the common practice in category theory of viewing equivalent categories as the same.
In addition, while univalent categories are quite rare in set-theoretic foundations,
they naturally occur in univalent foundations.
This is because
the aforementioned structure identity principles give us many instances of univalent categories,
such as the category of groups or rings.
Finally, a large part of category theory can constructively be developed in univalent foundations using univalent categories.

\paragraph*{Category Theory and Logic.}
Category theory is useful to logic and type theory,
because one can use \textbf{categorical semantics} to reason about the metatheory of logical systems
and to prove the independence of various axioms.
For instance,
one can use sheaf models to show that various analytical principles are independent from intuitionistic higher-order logic~\cite{fourman:1979a},
and one can construct a category in which the Church--Turing thesis holds~\cite{hyland:1982}.
Category theory can also be used to prove the consistency of cubical type theory~\cite{bezem:2013},
homotopy type theory~\cite{kapulkin:2021},
and the calculus of constructions~\cite{reus:1999a}.

Of particular importance to categorical logic are \textbf{internal language theorems}.
Such theorems express an equivalence between the collection of theories and models.
More specifically,
every theory gives rise to a model (i.e., completeness),
every model gives rise to a theory (i.e., the internal language),
and going back and forth is isomorphic to the identity.
These assignments also are functorial,
meaning that they also act on translations between theories and morphisms of models.
One internal language theorem was proven by Seely~\cite{seely:1984},
who constructed an equivalence between the categories of locally Cartesian closed categories
and of extensional Martin-L\"of type theories with $\Sigma$- and $\Pi$-types.
This theorem tells us that locally Cartesian closed categories give us models of type theory,
and that we can use type theory to reason internally to such a category.

However, Seely's proof suffers from a defect
caused by a divergence between category theory,
where one generally works up to isomorphism,
and type theory,
where one typically requires strict equalities of types.
In a locally Cartesian closed category, Seely interprets \textbf{substitution as pullback}.
However, since pullbacks are only determined up to isomorphism,
substitution rules only hold up to isomorphism in the model.
As a consequence, not every rule of type theory is validated in that model,
because substitution rules are required to hold up to strict equality.

Hofmann~\cite{hofmann:1994} rectified Seely's interpretation of type theory in locally Cartesian closed categories
using ideas from B\'enabou~\cite{benabou:1985}.
To understand these ideas, we phrase them in terms of fibrations and \textbf{split fibrations}.
We can see the general picture as follows:
fibrations give us a notion of context and dependent type
together with a substitution operation on types that is determined up to isomorphism.
Split fibrations are those fibrations for which the substitution laws actually hold up to equality.
We can soundly interpret Martin-L\"of type theory using split fibrations,
because substitution laws in type theory are required as equalities.
The main idea behind the rectification is thus to replace a fibration by an equivalent split fibration,
which is possible in general \cite[Corollary 5.2.5]{jacobs:2001}.

Since we replace fibrations by split fibrations that only are equivalent,
we lose the equivalence of categories that Seely originally stated.
This is because equivalences of categories require that going back and forth gives you something isomorphic to the input,
but, due to the replacement, we only get something that is adjoint equivalent, which is weaker.
To recover the internal language theorem, Clairambault and Dybjer~\cite{clairambault:2014} phrased it as a biequivalence between 2-categories instead.
As such, we can view Martin-L\"of type theory as the \textbf{internal language up to isomorphism} of locally Cartesian closed categories.

There are two parts to proving an internal language theorem,
and we demonstrate these parts by looking more closely at the internal language theorem by Clairambault and Dybjer~\cite{clairambault:2014} for categories with finite limits.
As their notion of model, they use democratic categories with families that support extensional identity types and $\Sigma$-types,
and they show how to \textbf{interpret extensional type theory with $\Sigma$-types in their notion of model}.
While it is direct to give such an interpretation for categories with families,
there also are categorical structures to model type theory for which an interpretation requires more work.
They proceed to prove their internal language theorem by constructing a biequivalence
between the bicategory of models
and the bicategory of categories with finite limits.
For extensional type theory with $\Sigma$-types and $\Pi$-types, they take the same approach.

In this paper,
we are concerned with internal language theorems for univalent categories in univalent foundations,
and we consider the first of the aforementioned parts.
We consider various classes of categories,
such as locally Cartesian closed categories and toposes,
and we build upon the
biequivalence by Clairambault and Dybjer~\cite{clairambault:2014}
and the soundness and completeness theorem by Maietti~\cite{maietti:2005}.
Since univalent categories are identified up to adjoint equivalence,
we conclude that Martin-L\"of type theory is the internal language of each of these classes of categories.
We leave the development of a suitable syntax and its interpretation in comprehension categories as future work.

\Cref{tab:equivalences} gives an overview of the internal language theorems that we consider.
Note that when we say topos, we mean elementary topos.
In this table the sum type $A \DisjSum B$ comes with a disjointness axiom,
and the type $\Omega$ contains all types with at most one inhabitant (i.e., a type of all propositions).
We do not require $\CatUniv$ in a category with a family of types to be closed under any type former.
However,
the universe $\CatUniv$ in an elementary topos with $\Nat$ and a universe is required to contain the natural numbers and subobject classifier
and to be closed under propositional resizing, $\Sigma$-types, and $\Pi$-types~\cite{streicher:2005}.

\begin{table}[]
\begin{center}
\begin{tabular}{c | c | c}
Category Theory                                & Type Theory                                                                                      & Biequivalence             \\
\hline
Finite Limits                                  & $\Term, \times, \ExtIdS, \Sigma$                                                                   & \Cref{sec:fin-lim}        \\
Locally Cartesian Closed                       & $\Term, \times, \ExtIdS, \Sigma, \Pi$                                                            & \Cref{sec:lccc}           \\
Pretopos                                       & $\Empty, \Term, \times, \ExtIdS, \Sigma, \DisjSum, \Quot$                                          & \Cref{sec:topos}          \\
Arithmetic Pretopos                            & $\Empty, \Term, \times, \ExtIdS, \Sigma, \DisjSum, \Quot, \Nat$                                    & \Cref{sec:topos}          \\
$\Pi$-Pretopos                               & $\Empty, \Term, \times, \ExtIdS, \Sigma, \Pi, \DisjSum, \Quot$                                   & \Cref{sec:topos}          \\
Topos                                          & $\Empty, \Term, \times, \ExtIdS, \Sigma, \Pi, \DisjSum, \Quot, \Omega$                           & \Cref{sec:topos}          \\
Topos with $\Nat$                              & $\Empty, \Term, \times, \ExtIdS, \Sigma, \Pi, \DisjSum, \Quot, \Omega, \Nat$                     & \Cref{sec:topos}          \\
Categories with a family of types              & $\Term, \times, \ExtIdS, \Sigma, \CatUniv$                                                         & \Cref{sec:universes}      \\ 
Toposes with $\Nat$ and a Universe             & $\Empty, \Term, \times, \ExtIdS, \Sigma, \Pi, \DisjSum, \Quot, \Omega, \Nat, \CatUniv$           & \Cref{sec:universe-types}
\end{tabular}
\end{center}
\caption{Various classes of categories and their internal language}
\label{tab:equivalences}
\end{table}

\paragraph*{Contributions.}
This paper has four contributions.
Our first contribution is an analogue of the internal language theorem by Clairambault and Dybjer~\cite[Theorem 6.1]{clairambault:2014}
for univalent categories (\Cref{constr:biequiv,constr:lccc-biequiv}).
More specifically,
we show that the bicategories of univalent categories with finite limits
and of univalent locally Cartesian closed categories
are biequivalent to bicategories of univalent comprehension categories that support suitable type formers.
This development gives us a new perspective on the semantics of dependent type theory,
since we do not have to use split fibrations to guarantee soundness.
In addition, various constructions and proofs get simplified when we use univalent categories.
The reason for these simplifications is twofold.
On the one hand,
univalent categories allow us to treat adjoint equivalences as identities,
which makes it immediate to prove that adjoint equivalences preserve various kinds of categorical structure,
like exponentials.
On the other hand,
the distinction between chosen and existing structure
evaporates for univalent categories.
If we have some categorical structure that is determined by some universal property,
then this structure is unique up to isomorphism,
and actually unique up to identity if we work with univalent categories.
As a consequence,
we do not need the axiom of choice to obtain,
for instance,
chosen limits from existing ones.

Our second contribution is an extension of the internal language theorems to various classes of toposes (\Cref{exa:internal-language-topos}).
Here we extend the method by Maietti~\cite[Proposition 2.13]{maietti:2005},
which was used to prove soundness and completeness,
to obtain biequivalences.
The additional examples that we consider, are pretoposes, arithmetic pretoposes, $\Pi$-pretoposes, elementary toposes, and elementary toposes with an NNO.
Our third contribution is another extension, namely to universe types (\Cref{sec:universes,sec:universe-types}).
Specifically,
we look at Tarski-style universes that are closed under the type formers in an elementary topos~\cite{streicher:2005}.

Our final contribution is that all the constructions and proofs in this paper are formalized in a proof assistant.

\paragraph*{Formalization.}
The results in this paper are formalized in the \CoqLink proof assistant~\cite{team:2024}
using the \UniMath library~\cite{voevodsky:2024}.
The code is incorporated in \UniMath,
and it is available online.
Definitions and theorems are accompanied by links to their corresponding identifier in the formalization.
Links to the formalization are underlined and colored \link{blue}.

Throughout this paper we distinguish between proof relevant constructions
and theorems that are proof irrelevant.
A ``Construction'' for a ``Problem'' is proof relevant,
whereas a ``Proof'' for a ``Proposition'' or ``Theorem'' is proof irrelevant.

\paragraph*{Version History.}
This paper is an extended version of our LICS paper~\cite{weide:2025a}.
Compared to that version, we made the following changes.
\begin{itemize}
  \item We added \Cref{sec:cat-theory-uf} where the necessary preliminaries are given.
  \item In \Cref{sec:models}, we added a description of how type theory is interpreted in comprehension categories.
  \item The development of displayed biequivalences (\Cref{def:disp-psfunctor,def:disp-pstrans,def:disp-biequiv}) is written in higher generality.
  \item We extended the explanations of \Cref{exa:fin-coprod,exa:exact,exa:subobj-class} and \Cref{exa:nno-list}.
  \item We added proofs in \Cref{sec:bicat-comp-cat,sec:type-formers,sec:fin-lim,sec:lccc,sec:topos}.
  \item We added a characterization of propositions in a comprehension category via monomorphisms in \Cref{prop:hprop-mono,prop:hprop-fib-mono},
  and an explanation of how a subobject classifier gives rise to a universe of propositions.
  \item We added an extension of the internal language theorem to toposes with a universe in \Cref{sec:universes,sec:universe-types}.
\end{itemize}

\paragraph*{Overview.}
We start in \Cref{sec:cat-theory-uf} by recalling the basics of univalent foundations and how to develop category theory therein.
In this section, we also recall the notions that are used throughout this paper.
We follow this up in \Cref{sec:models} by discussing categorical models of dependent type theory in univalent foundations.
More specifically, we explain why we use comprehension categories instead of categories with families to model the internal language of univalent categories.

Next we study various bicategories of comprehension categories.
We start in \Cref{sec:bicat-comp-cat} with the bicategory of full univalent comprehension categories (\Cref{def:disp-bicat-full-comp-cat}),
and we prove that it is univalent (\Cref{prop:univ-bicat-full-comp-cat}).
The key idea behind the construction is to use displayed bicategories,
which we recall at the start of \Cref{sec:bicat-comp-cat}.
After that, we consider various type formers for comprehension categories in \Cref{sec:type-formers},
and we construct the bicategory of univalent full democratic finite limit comprehension categories (DFL comprehension categories) (\Cref{def:dfl-comp-cat}).
These are full univalent comprehension categories that are democratic
and that support unit types, binary product types, equalizer types, and $\Sigma$-types.
We also prove that the resulting bicategory is univalent (\Cref{prop:univ-dfl-compcat}),
and we give a sufficient condition for when a morphism of such comprehension categories is an adjoint equivalence (\Cref{prop:dfl-compcat-adjequiv}).

In the final sections of the paper, we study internal language theorems for univalent categories.
We start in \Cref{sec:fin-lim} with the internal language of univalent categories with finite limits,
which is given by extensional Martin-L\"of type theory with unit types, binary product types, and $\Sigma$-types (\Cref{constr:biequiv}).
More specifically, we construct a biequivalence between the bicategories of univalent categories with finite limits
and of DFL comprehension categories.
In \Cref{sec:lccc} we extend this biequivalence to locally Cartesian closed categories (\Cref{constr:lccc-biequiv}).
The main idea here is to use displayed biequivalences, which we recall at the beginning of that section.
Next we extend this biequivalence using local properties in \Cref{sec:topos}.
This allows us to deduce internal language theorems for various classes of toposes (\Cref{exa:internal-language-topos}).
Finally, we also consider categories with a universe in \Cref{sec:universes,sec:universe-types}.
In those sections, we define a notion of universe in categories with finite limits and comprehension categories,
and we extend the aforementioned biequivalence to toposes with a natural numbers object and a universe (\Cref{constr:topos-universe-biequiv}).
We finish this paper with a discussion of related work in \Cref{sec:related-work},
and we conclude in \Cref{sec:conclusion}.

\section{Category Theory in Univalent Foundations}
\label{sec:cat-theory-uf}
In this section,
we recall \textbf{univalent foundations}~\cite{program:2013},
and the development of category theory therein~\cite{ahrens:2015}.
Univalent foundations is a version of Martin-L\"of type theory~\cite{martin-lof:1982}
in which the \textbf{univalence axiom} holds~\cite[Axiom 2.10.3]{program:2013}.
This axiom allows us to identify types if they are equivalent,
which gives us a more extensional identity of types
while type checking remains decidable.
Note the difference with plain Martin-L\"of type theory,
where types are identified in a fully intensional way
meaning that types are only considered to be equal if they are presented as the same object.

To formulate the univalence axiom precisely,
we use the map $\idtoequiv : A = B \rightarrow A \simeq B$,
which is defined for all types $A, B : \UnivU$ in some universe $\UnivU$~\cite[Lemma 2.10.1]{program:2013}.
This map is defined using path induction,
and it sends the identity path to the identity equivalence.
The univalence axiom states that $\idtoequiv$ is an equivalence of types.

A consequence of the univalence axiom is that identity in univalent foundations is necessarily proof relevant,
because the identity type of the universe is equivalent to the type of equivalences.
It is thus sensible to classify types based on the complexity of their identity type.
The types with the simplest identity types are \textbf{propositions} and \textbf{contractible types}.
A \textbf{proposition} is a type for which every two inhabitants are equal,
and a \textbf{contractible type} is an inhabited proposition.
The identity type of a proposition is thus trivial, because it is equivalent to the unit type.
We also have the notion of \textbf{set}, which is a type $A$ for which $x = y$ is a proposition for all $x, y : A$.
Inhabitants of a set could be equal,
but there always is at most one proof of their equality.
\textbf{Homotopy levels}~\cite[Definition 7.1.1]{program:2013} generalize the notions of contractible types, propositions, and sets,
and for each $n \geq -2$ we have a notion of \textbf{$n$-types}.
We say that a type is a $(-2)$-type if it is contractible,
and a type $A$ is an $(n + 1)$-type if the type $x = y$ is an $n$-type for all $x, y : A$.

Another central concept in univalent foundations is the \textbf{propositional truncation}.
Given a type $A$, the propositional truncation is a proposition $\trunc{A}$ together with a map $\truncmap{-} : A \rightarrow \trunc{A}$
satisfying the following universal property:
if we have some proposition $P$ and a map $A \rightarrow P$,
then we have a unique map $\trunc{A} \rightarrow P$.

The propositional truncation allows us to distinguish between equipped structure and structure that merely exists,
which is important in mathematics~\cite{castellan:2021}.
For instance, let $T$ be the type of terminal objects in $\C$.
An inhabitant of $T$ equips $\C$ with a terminal object.
The truncation of $T$, on the other hand, expresses the mere existence of a terminal object in $\C$.
Constructively, one can generally not obtain a terminal object from an inhabitant of $\trunc{T}$.

\subsection{Univalent Categories}
Univalent foundations provides an interesting point of view on several areas of mathematics,
and among those is category theory.
Categories can be studied up to two notions of equivalence: isomorphism and adjoint equivalence.
This bifurcation is reflected in univalent foundations where we have two different notions of categories: setcategories and univalent categories~\cite{ahrens:2015}.
Setcategories are those categories whose types of objects form a set,
and univalent categories are those categories in which isomorphisms coincide with equality.
One uses setcategories to study category theory up to isomorphism,
whereas one uses univalent categories to study category theory up to adjoint equivalence~\cite{rasekh:2024}.
More specifically, setcategories are identified if they are isomorphic,
and univalent categories are identified if they are adjoint equivalent.

Throughout this paper, our focus is on univalent categories.
To define this notion, let us start by recalling categories~\cite[Definition 3.1]{ahrens:2015}.
A category is given by
a type of objects,
a \textbf{set} of morphisms indexed by pairs of objects
with identity morphisms and a composition operation
satisfying the usual unitality and associativity laws.
Composition is written in diagrammatic order:
if we have $f : x \rightarrow y$ and $g : y \rightarrow z$,
then we denote their composition by $f \cdot g : x \rightarrow z$.
We say that a category is \textbf{univalent} if the identity type of objects is equivalent to isomorphism of objects.

\begin{defiC}[{\cite[Definition 3.6]{ahrens:2015}}]
\label[defi]{def:univalent-cat}
Let $\C$ be a category.
For all objects $x, y : \C$ we have a map \conceptDef{$\idtoiso{x}{y}$}{CategoryTheory.Core.Univalence}{idtoiso}, which sends identities $x = y$ to isomorphisms $\iso{x}{y}$.
This map is defined using path induction, and it sends the identity path to the identity isomorphism.
We say that $\C$ is \conceptDef{univalent}{CategoryTheory.Core.Univalence}{is_univalent} if the map $\idtoiso{x}{y}$ is an equivalence of types for all $x, y : \C$.
\end{defiC}

Note that univalence can be defined for a wide variety of categorical structures~\cite{ahrens:2021a}.
By the univalence axiom, various categories, such as the category of sets or groups, are univalent~\cite{coquand:2013}.
Now let us look at some properties of univalent categories that we use throughout this paper.
First, note that the type of objects in a univalent category is not necessarily a set.
A counterexample is given by the category $\Set$ of sets and functions.
The univalence axiom implies that identity of sets corresponds to isomorphism,
and as a consequence, the category $\Set$ is univalent
and the type of sets is not a set itself.
Since identity of objects is equivalent to isomorphism
and since isomorphisms generally only form a set,
we do have that the type of objects is a $1$-type~\cite[Lemma 3.8]{ahrens:2015}.

\begin{propL}[\coqdocurl{CategoryTheory.Core.Univalence}{univalent_category_has_groupoid_ob}]
\label[propL]{prop:hlevel-obj}
The type of objects in a univalent category is a $1$-type.
\end{propL}

Another feature of univalent categories is that uniqueness up to isomorphism corresponds to actual uniqueness.
In category theory, objects such as limits and colimits are specified using universal mapping properties,
and that only guarantees their uniqueness up to isomorphism.
However, if we reason constructively,
there is a drawback to that:
the mere existence of limits in a setcategory
is, in general, insufficient to equip that category with limits
unless we assume the axiom of choice.
If we work with univalent categories,
it is possible to avoid the axiom of choice.
In such categories,
limits are unique up to identity,
which allows us to choose them if they exist.

\begin{prop}
\label[prop]{prop:limits-prop}
Let $\C$ be a univalent category.
\begin{itemize}
  \item[\link{\href{\coqdocurl{CategoryTheory.Limits.Terminal}{isaprop_Terminal}}{\labelitemi}}] The type of terminal objects for $\C$ is a proposition.
  \item[\link{\href{\coqdocurl{CategoryTheory.Limits.BinProducts}{isaprop_BinProduct}}{\labelitemi}}] The type expressing that $\C$ is equipped with binary products, is a proposition.
  \item[\link{\href{\coqdocurl{CategoryTheory.Limits.Equalizers}{isaprop_Equalizer}}{\labelitemi}}] The type expressing that $\C$ is equipped with equalizers, is a proposition.
\end{itemize}
\end{prop}

\begin{proof}
We only consider the case for terminal objects,
since the ideas are the same for binary products and equalizers.
The type of terminal objects is defined as a $\Sigma$-type,
whose inhabitants are pairs of an object $x : \C$ together with a proof that $x$ is terminal.
To check that two terminal objects are equal,
it suffices to check whether their underlying objects are equal,
because the type saying that some object is terminal,
is a proposition.
Since we assumed $\C$ to be univalent,
every isomorphism in $\C$ gives rise to an equality.
Hence, to check that two terminal objects $T_1$ and $T_2$ in $\C$ are equal,
it suffices to give an isomorphism between $T_1$ and $T_2$.
Since terminal objects are unique up to isomorphism,
we can construct the desired isomorphism,
and we can conclude that the type of terminal objects for $\C$ is a proposition.
\end{proof}

Note that \Cref{prop:limits-prop} generalizes to arbitrary (co)limits.
To illustrate how one can use \Cref{prop:limits-prop},
we explain how we can choose a terminal object in a univalent category that merely has one.
Let us write $T$ for the type of terminal objects in some given univalent category $\C$.
If we assume that $\C$ merely has a terminal object, then we have an inhabitant of $\trunc{T}$.
Since the type $T$ is a proposition,
we have a map $\trunc{T} \rightarrow T$ by the recursion principle of the propositional truncation.
Hence, we also have an inhabitant of $T$, which gives us a terminal object in $\C$.

We see the same phenomenon in the proof that essentially surjective and fully faithful functors are adjoint equivalences.
Suppose that we have a functor $\F : \C_1 \rightarrow \C_2$ that is both essentially surjective and fully faithful.
To show that $\F$ is an adjoint equivalence,
we need to construct a functor $\G : \C_2 \rightarrow \C_1$ that is an inverse of $\F$ up to isomorphism.
For each object $y : \C_2$, we must thus find a preimage $x : \C_1$.
However, these preimages are only unique up to isomorphism in general.
If we assume that $\C_1$ is univalent,
then the preimages are actually unique
and we can construct $\G$ while reasoning constructively.
All in all, we can again avoid the axiom of choice by assuming that some of the involved categories are univalent,
and we can constructively prove the following proposition~\cite[Lemma 6.8]{ahrens:2015}.

\begin{propL}[\coqdocurl{CategoryTheory.Equivalences.Core}{rad_equivalence_of_cats}]
\label[propL]{prop:eso-ff}
Suppose that we have categories $\C_1$ and $\C_2$ such that $\C_1$ is univalent.
Then a functor $\F : \C_1 \rightarrow \C_2$ is an adjoint equivalence
if $\F$ is essentially surjective and fully faithful.
\end{propL}

Finally, another key property of univalent categories is their structure identity principle~\cite[Theorem 6.17]{ahrens:2015}.
This principle allows us to identify univalent categories if we have an adjoint equivalence between them.
We phrase this structure identity principle using \textbf{univalent bicategories}~\cite[Definition 3.1]{ahrens:2021}.
Recall that a bicategory is given by a type of objects, a type of 1-cells indexed by two objects,
and a \textbf{set} of 2-cells indexed by two 1-cells.
If we have 1-cells $f : x \rightarrow y$ and $g : y \rightarrow z$, then we denote their composition by $f \cdot g : x \rightarrow z$.
We write $\lunitor : \id \cdot f \twocell f$ and $\runitor : f \cdot \id \twocell f$ for the left and right unitor in a bicategory respectively,
and we write $\lassociator : f \cdot (g \cdot h) \twocell (f \cdot g) \cdot h$ for the associator.
If we have 1-cells $f : x \rightarrow y$ and $g_1, g_2 : y \rightarrow z$
and a 2-cell $\tau : g_1 \twocell g_2$,
then we write $f \whiskerl \tau : f \cdot g_1 \twocell f \cdot g_2$ for the left whiskering.
Analogously, we write $\tau \whiskerr g : f_1 \cdot g \twocell f_2 \cdot g$
for the right whiskering of $\tau : f_1 \twocell f_2$ with a 1-cell $g$.
An example of a bicategory is $\UnivCat$, whose objects are univalent categories,
1-cells are functors,
and 2-cells are natural transformations.
Univalent bicategories are defined as follows.

\begin{defi}
\label[defi]{def:univalent-bicat}
Let $\B$ be a bicategory.
\begin{itemize}
  \item For all objects $x, y : \B$ we have a map \conceptDef{$\idtoadjequiv{x}{y}$}{Bicategories.Core.Univalence}{idtoiso_2_0}
    sending identities $x = y$ to adjoint equivalences $\adjequiv{x}{y}$.
    The map $\idtoadjequiv{x}{y}$ is defined using path induction,
    and it sends the identity path to the identity adjoint equivalence.
  \item For all 1-cells $f, g : x \rightarrow y$ we have a map \conceptDef{$\idtoinvcell{f}{g}$}{Bicategories.Core.Univalence}{idtoiso_2_1}
    sending identities $f = g$ to invertible 2-cells $\invcell{f}{g}$.
\end{itemize}
We say that $\B$ is \conceptDef{globally univalent}{Bicategories.Core.Univalence}{is_univalent_2_0} if the map $\idtoadjequiv{x}{y}$ is an equivalence for all $x, y : \B$,
and that $\B$ is \conceptDef{locally univalent}{Bicategories.Core.Univalence}{is_univalent_2_1} if the map $\idtoinvcell{f}{g}$ is an equivalence for all $f, g : x \rightarrow y$.
Finally, we say that $\B$ is \conceptDef{univalent}{Bicategories.Core.Univalence}{is_univalent_2} if it is both locally and globally univalent.
\end{defi}

\begin{propL}[\coqdocurl{Bicategories.Core.Examples.BicatOfUnivCats}{univalent_cat_is_univalent_2}]
\label[propL]{prop:bicat-cat-univ}
The bicategory $\UnivCat$ is univalent.
\end{propL}

\Cref{prop:bicat-cat-univ} tells us that the identity type of univalent categories is equivalent to adjoint equivalences,
and that the identity type of functors is equivalent to natural isomorphisms.
This structure identity principle thus allows us to treat adjoint equivalences as identities.
As a consequence,
we can prove some property for all adjoint equivalences
by only verifying it for the identity equivalence,
and we can transfer properties and structure along adjoint equivalences for free.
We make this reasoning principle precise in the following proposition~\cite[Propositions 3.4 and 3.5]{ahrens:2021}.

\begin{prop}
\label[prop]{prop:equivalence-induction}
Let $\B$ be a univalent bicategory.
\begin{itemize}
  \item[\link{\href{\coqdocurl{Bicategories.Core.Univalence}{J_2_0}}{\labelitemi}}] Let $\Pred$ be a type family taking objects $x, y : \B$ and adjoint equivalences $l : \adjequiv{x}{y}$.
    To prove $\Pred(x, y, l)$ for all $x, y : \B$ and $l : \adjequiv{x}{y}$,
    it suffices to prove $\Pred(x, x, \id[x])$ for all $x : \B$.
  \item[\link{\href{\coqdocurl{Bicategories.Core.Univalence}{J_2_1}}{\labelitemi}}] Let $\Pred$ be a type family taking objects $x, y : \B$, 1-cells $f, g : x \rightarrow y$, and invertible 2-cells $\tau : \invcell{f}{g}$.
    To prove $\Pred(x, y, f, g, \tau)$ for all objects $x, y : \B$, 1-cells $f, g : x \rightarrow y$, and invertible 2-cells $\tau : \invcell{f}{g}$,
    it suffices to prove $\Pred(x, y, f, f, \id[f])$ for all $x, y : \B$ and $f : x \rightarrow y$.
\end{itemize}
\end{prop}

\subsection{Displayed Categories}
\textbf{Displayed categories}~\cite{ahrens:2019a} are fundamental when studying category theory in univalent foundations.
Their applications are plentiful,
and our main interest stems from Grothendieck fibrations~\cite{grothendieck:1960}.
This is because
if we phrase Grothendieck fibrations in terms of displayed categories,
we can avoid referring to equality of objects when we express the existence of Cartesian lifts.
In the remainder of this section, we recall the concepts related to displayed categories that we use in this paper.

\begin{defi}
\label[defi]{def:disp-cat}
Let $\C$ be a category.
A \conceptDef{displayed category $\D$ over $\C$}{CategoryTheory.DisplayedCats.Core}{disp_cat} consists of
\begin{itemize}
  \item a type $\dob{\D}{x}$ for each $x : \C$;
  \item a set $\dmor{\xx}{\yy}{f}$ for all objects $\xx : \dob{\D}{x}$ and $\yy : \dob{\D}{y}$, and morphisms $f : x \rightarrow y$.
\end{itemize}
Inhabitants of $\dob{\D}{x}$ are called \textbf{objects over $x$},
and inhabitants of $\dmor{\xx}{\yy}{f}$ are called \textbf{morphisms from $\xx$ to $\yy$ over $f$}.
We also have a morphism $\idd[\xx] : \dmor{\xx}{\xx}{\id{x}}$ for every $\xx : \dob{\D}{x}$,
and if we have morphisms $\ff : \dmor{\xx}{\yy}{f}$ and $\gg : \dmor{\yy}{\zz}{g}$,
then we also have a composition $\ff \cdot \gg : \dmor{\xx}{\zz}{f \cdot g}$.
We require the usual associativity and unitality laws,
which are dependent equalities over the corresponding equality in $\C$~\cite[Definition 3.1]{ahrens:2019a}.

Let $\F : \C \rightarrow \C'$ be a functor,
and let $\D$ and $\D'$ be displayed categories over $\C$ and $\C'$ respectively.
A \conceptDef{displayed functor $\FF : \D \rightarrow \D'$ over $\F$}{CategoryTheory.DisplayedCats.Functors}{disp_functor} consists of
\begin{itemize}
  \item a map sending objects $\xx : \dob{\D}{x}$ to objects $\FF(\xx) : \dob{\D'}{\F(x)}$ for each $x : \C$;
  \item a map sending morphisms $\ff : \dmor{\xx}{\yy}{f}$ to morphisms $\FF(\ff) : \dmor{\FF(\xx)}{\FF(\yy)}{\F(f)}$
    for all objects $\xx : \dob{\D}{x}$ and $\yy : \dob{\D}{y}$, and morphisms $f : x \rightarrow y$.
\end{itemize}
We also require the functoriality laws~\cite[Definition 3.11]{ahrens:2019a}.

Given a natural transformation $\nt : \F \Rightarrow \G$
and displayed functors $\FF : \D \rightarrow \D'$ and $\GG : \D \rightarrow \D'$
over $\F$ and $\G$ respectively,
a \conceptDef{displayed natural transformation $\tautau : \FF \Rightarrow \GG$ over $\nt$}{CategoryTheory.DisplayedCats.NaturalTransformations}{disp_nat_trans}
is given by a family of morphisms $\dmor{\FF(\xx)}{\GG(\xx)}{\nt(x)}$ for each $x : \C$ and $\xx : \dob{\D}{x}$.
We also require that this family of morphisms is natural~\cite[Definition 3.12]{ahrens:2019a}.
\end{defi}

Every displayed category $\D$ over $\C$ gives rise to a category $\total{\D}$, called the \textbf{total category},
and a functor $\proj{\D} : \total{\D} \rightarrow \C$, called the \textbf{projection}.
A displayed functor $\FF : \D \rightarrow \D'$ over $\F$ gives rise to a functor $\total{\FF} : \total{\D} \rightarrow \total{\D'}$
that makes the following square commute up to definitional equality.
\[
\begin{tikzcd}
  {\total{\D}} & {\total{\D'}} \\
  \C & {\C'}
  \arrow["{\total{\FF}}", from=1-1, to=1-2]
  \arrow["{\proj{\D}}"', from=1-1, to=2-1]
  \arrow["{\proj{\D'}}", from=1-2, to=2-2]
  \arrow["\F"', from=2-1, to=2-2]
\end{tikzcd}
\]

In addition, every object $x : \C$ gives rise to a category $\fiber{\D}{x}$,
which we call the \textbf{fiber category}.
The type of objects of $\fiber{\D}{x}$ is $\dob{\D}{x}$,
and morphisms from $\xx : \dob{\D}{x}$ to $\yy : \dob{\D}{x}$
are morphisms $\dmor{\xx}{\yy}{\id[x]}$.
Every displayed functor $\FF : \D \rightarrow \D'$ over $\F$
gives rise to a functor $\fiberfunctor{\FF}{x} : \fiber{\D}{x} \rightarrow \fiber{\D'}{\F(x)}$.

Next we define \textbf{displayed isomorphisms}.
Given objects $\xx : \dob{\D}{x}$ and $\yy : \dob{\D}{y}$, and an isomorphism $f : \iso{x}{y}$,
a displayed isomorphism over $f$ consists of morphisms $\ff : \dmor{\xx}{\yy}{f}$ and $\gg : \dmor{\yy}{\xx}{f^{-1}}$ that compose to the identity.
We write $\isodisp{\xx}{\yy}{f}$ for the type of displayed isomorphisms from $\xx : \dob{\D}{x}$ to $\yy : \dob{\D}{y}$ over $f : \iso{x}{y}$.
Using \Cref{def:disp-cat}, we can define \textbf{displayed adjoint equivalences} over adjoint equivalences in the usual way.

Just like categories,
displayed categories come with a notion of univalence.

\begin{defiC}[{\cite[Definition 7.2]{ahrens:2019a}}]
\label[defi]{def:disp-univalence}
Suppose that we have a displayed category $\D$ over $\C$.
We say that $\D$ is \conceptDef{univalent}{CategoryTheory.DisplayedCats.Univalence}{is_univalent_disp}
if for all objects $x, y : \C$ with $p : x = y$ and objects $\xx : \dob{\D}{x}$ and $\yy : \dob{\D}{y}$ over $x$ and $y$ respectively,
the map $\xx = \yy \rightarrow \isodisp{\xx}{\yy}{\idtoiso{x}{y}(p)}$ defined by path induction is an equivalence of types.
\end{defiC}

If a displayed category $\D$ is univalent,
then each fiber category $\fiber{\D}{x}$ is univalent as well.
If both $\C$ and $\D$ are univalent,
then $\total{\D}$ also is univalent.
The next notion of interest is that of \textbf{cleavings}.
Note that in the following definition we do not refer to equality of objects,
but instead we only talk about objects in $\D$ over some object in $\C$.

\begin{defiC}[{\cite[Definition 5.3]{ahrens:2019a}}]
\label[defi]{def:fibration}
Let $\D$ be a displayed category over $\C$.
A morphism $\ff : \dmor{\xx}{\yy}{f}$ over $f : x \rightarrow y$ is said to be \conceptDef{Cartesian}{CategoryTheory.DisplayedCats.Fibrations}{is_cartesian}
if for all objects $\ww : \dob{\D}{w}$ and morphisms $g : w \rightarrow x$ and $\hh : \dmor{\ww}{\yy}{g \cdot f}$
there is a unique morphism $\gg : \dmor{\ww}{\xx}{g}$ making the triangle below commute.
\[
\begin{tikzcd}[column sep=large]
  \ww \\
  & \xx & \yy \\
  w & x & y
  \arrow["\gg"', dashed, from=1-1, to=2-2]
  \arrow["\hh", bend left=20, from=1-1, to=2-3]
  \arrow["\ff"', from=2-2, to=2-3]
  \arrow["g"', from=3-1, to=3-2]
  \arrow["f"', from=3-2, to=3-3]
\end{tikzcd}
\]

A \conceptDef{cleaving}{CategoryTheory.DisplayedCats.Fibrations}{cleaving} for $\D$ is a map assigning to each morphism $f : x \rightarrow y$ and $\yy : \dob{\D}{y}$
an object $\xx : \dob{\D}{x}$ and a Cartesian morphism $\ff : \dmor{\xx}{\yy}{f}$.
The morphism $\ff$ is the \textbf{Cartesian lift} of $f$,
and if $\D$ is equipped with a cleaving, then we say that $\D$ is a \textbf{cloven fibration}.

A displayed functor $\FF : \D \rightarrow \D'$ over $\F$ is said to be \conceptDef{Cartesian}{CategoryTheory.DisplayedCats.Fibrations}{is_cartesian_disp_functor} if it maps Cartesian morphisms to Cartesian morphisms.
\end{defiC}

If we have a displayed category $\D$ equipped with a cleaving,
then every morphism $f : x \rightarrow y$ induces a functor $\functorfiber{f} : \fiber{\D}{y} \rightarrow \fiber{\D}{x}$~\cite{jacobs:2001}.
This functor sends every object $\yy : \dob{\D}{y}$ to the domain of the Cartesian lift of $f$ along $\yy$.

Note that we phrased \Cref{def:fibration} as structure on a displayed category.
An alternative would be to only assume that Cartesian lifts merely exist,
which would give rise to the notion of \textbf{fibrations}.
However, since Cartesian lifts are unique up to isomorphism, they are actually unique if $\D$ is univalent.
A univalent displayed category $\D$ is thus a cloven fibration if it is a fibration.

We also use \textbf{fiberwise limits} in displayed categories.
Our focus is on several special cases of fiberwise limits, namely the following.

\begin{defi}
\label[defi]{def:fiberwise-limits}
Let $\D$ be a displayed category over $\C$ with a cleaving.
We say that
\begin{itemize}
  \item $\D$ is equipped with a \conceptDef{fiberwise terminal object}{CategoryTheory.DisplayedCats.Fiberwise.FiberwiseTerminal}{fiberwise_terminal}
    if for all $x : \C$ the category $\fiber{\D}{x}$ is equipped with a terminal object
    and if $\functorfiber{f} : \fiber{\D}{y} \rightarrow \fiber{\D}{x}$ preserves terminal objects for all $f : x \rightarrow y$;
  \item $\D$ is equipped with \conceptDef{fiberwise binary products}{CategoryTheory.DisplayedCats.Fiberwise.FiberwiseProducts}{fiberwise_binproducts}
    if for all $x : \C$ the category $\fiber{\D}{x}$ is equipped with binary products
    and if $\functorfiber{f} : \fiber{\D}{y} \rightarrow \fiber{\D}{x}$ preserves binary products for all $f : x \rightarrow y$;
  \item $\D$ is equipped with \conceptDef{fiberwise equalizers}{CategoryTheory.DisplayedCats.Fiberwise.FiberwiseEqualizers}{fiberwise_equalizers}
    if for all $x : \C$ the category $\fiber{\D}{x}$ is equipped with equalizers
    and if $\functorfiber{f} : \fiber{\D}{y} \rightarrow \fiber{\D}{x}$ preserves equalizers for all $f : x \rightarrow y$.
\end{itemize}
\end{defi}

Since fiber categories are univalent if $\D$ is so,
\Cref{prop:limits-prop} allows us to conclude that the type of terminal objects for $\fiber{\D}{x}$ is a proposition if $\D$ is univalent.
We can say the same for binary products and equalizers,
and thus we get an analogue of \Cref{prop:limits-prop} for the fiberwise limits that we consider in this paper.

\begin{prop}
\label[prop]{prop:fib-limits-prop}
Let $\D$ be a univalent displayed category over a category $\C$.
\begin{itemize}
  \item[\link{\href{\coqdocurl{CategoryTheory.DisplayedCats.Fiberwise.FiberwiseTerminal}{isaprop_fiberwise_terminal}}{\labelitemi}}] The type of fiberwise terminal objects for $\D$ is a proposition.
  \item[\link{\href{\coqdocurl{CategoryTheory.DisplayedCats.Fiberwise.FiberwiseProducts}{isaprop_fiberwise_binproducts}}{\labelitemi}}] The type expressing that $\D$ is equipped with fiberwise binary products, is a proposition.
  \item[\link{\href{\coqdocurl{CategoryTheory.DisplayedCats.Fiberwise.FiberwiseEqualizers}{isaprop_fiberwise_equalizers}}{\labelitemi}}] The type expressing that $\D$ is equipped with fiberwise equalizers, is a proposition.
\end{itemize}
\end{prop}

In this paper, we also use the \textbf{Beck--Chevalley condition}.
This condition is related to interpretation of quantifiers as adjoints in categorical logic~\cite{lawvere:1970,lawvere:2006}.
More specifically, dependent sums and dependent products are interpreted as left adjoints and right adjoints of weakening respectively.
The Beck--Chevalley condition expresses that quantifiers are preserved by substitution.
While we formulate the Beck--Chevalley condition for arbitrary squares of functors,
we only use it in the case that each of the involved functors is of the shape $\functorfiber{s} : \fiber{\D}{y} \rightarrow \fiber{\D}{x}$
for some cloven fibration $\D$ and some morphism $s : x \rightarrow y$.

\begin{defi}
\label[defi]{def:beck-chevalley}
Suppose that we have a square of functors that commutes up to natural isomorphism as follows.
\[
\begin{tikzcd}
  {\C_1} & {\C_2} \\
  {\C_3} & {\C_4}
  \arrow[""{name=0, anchor=center, inner sep=0}, "{\F_1}", from=1-1, to=1-2]
  \arrow["{\G_1}"', from=1-1, to=2-1]
  \arrow["{\G_2}", from=1-2, to=2-2]
  \arrow[""{name=1, anchor=center, inner sep=0}, "{\F_2}"', from=2-1, to=2-2]
  \arrow["\nt", shorten <=4pt, shorten >=4pt, Rightarrow, from=1, to=0]
\end{tikzcd}
\]
In addition, suppose that we have adjunctions $\adjunction{\LF_1}{\G_1}{\eta_1}{\varepsilon_1}$
and $\adjunction{\LF_2}{\G_2}{\eta_2}{\varepsilon_2}$.
This square satisfies the \conceptDef{Beck--Chevalley condition for left adjoints}{CategoryTheory.DisplayedCats.Fiberwise.DependentSums}{left_beck_chevalley}
if the following natural transformation is an isomorphism.
\[
\begin{tikzcd}[column sep = huge]
  {\F_2 \cdot \LF_2} & {\LF_1 \cdot \G_1 \cdot \F_2 \cdot \LF_2} & {\LF_1 \cdot \F_1 \cdot \G_2 \cdot \LF_2} & {\LF_1 \cdot \F_1}
  \arrow["{\eta_1 \whiskerr (\F_2 \cdot \LF_2)}", Rightarrow, from=1-1, to=1-2]
  \arrow["{\LF_1 \whiskerl \nt \whiskerr \LF_2}", Rightarrow, from=1-2, to=1-3]
  \arrow["{(\LF_1 \cdot \F_1) \whiskerl \varepsilon_2}", Rightarrow, from=1-3, to=1-4]
\end{tikzcd}
\]
Analogously, we define the Beck--Chevalley condition for right adjoints.
\end{defi}

Let us finish this section with an example of a displayed category.
The example of interest in this paper is the \textbf{arrow category}.

\begin{exa}[{\cite[Example 3.18]{ahrens:2019a}}]
\label[exa]{exa:arrow-cat}
Let $\C$ be a category.
We define a displayed category \conceptDef{$\ArrD{\C}$}{CategoryTheory.DisplayedCats.Codomain}{disp_codomain} over $\C$ as follows.
\begin{itemize}
  \item The objects over $y : \C$ are pairs $x : \C$ together with a morphism $g : x \rightarrow y$.
  \item The morphisms over $f : y_1 \rightarrow y_2$ from $g_1 : x_1 \rightarrow y_1$ to $g_2 : x_2 \rightarrow y_2$
    are morphisms $h : x_1 \rightarrow x_2$ making the following diagram commute.
    \[
      \begin{tikzcd}
	{x_1} & {x_2} \\
	{y_1} & {y_2}
	\arrow["h", from=1-1, to=1-2]
	\arrow["{g_1}"', from=1-1, to=2-1]
	\arrow["{g_2}", from=1-2, to=2-2]
	\arrow["f"', from=2-1, to=2-2]
      \end{tikzcd}
    \]
\end{itemize}
If $\C$ is univalent, then so is $\ArrD{\C}$.
The total category, which we denote by $\Arr{\C}$, is the arrow category of $\C$,
and the forgetful functor $\proj{\ArrD{\C}} : \Arr{\C} \rightarrow \C$ sends every morphism to its codomain.
The fiber category $\fiber{\ArrD{\C}}{x}$ is equivalent to the slice category $\slice{\C}{x}$.
Cartesian morphisms in $\ArrD{\C}$ correspond to pullback squares,
and if $\C$ has pullbacks,
then we can equip $\ArrD{\C}$ with a cleaving~\cite[Proposition 1.1.6]{jacobs:2001}.
Given a morphism $f : x \rightarrow y$ in $\C$,
the fiber functor $\functorfiber{f} : \slice{\C}{y} \rightarrow \slice{\C}{x}$ sends morphisms $g : w \rightarrow y$ to the pullback of $f$ and $g$.
If $\C$ is finitely complete, then $\ArrD{\C}$ has all fiberwise finite limits.

Every functor $\F : \C_1 \rightarrow \C_2$ gives rise to a displayed functor $\ArrD{\F} : \ArrD{\C_1} \rightarrow \ArrD{\C_2}$ over $\F$.
This functor sends morphisms $g : x \rightarrow y$ to $\F(g) : \F(x) \rightarrow \F(y)$.
Finally,
every natural transformation $\nt : \F \Rightarrow \G$ induces a displayed natural transformation $\ArrD{\nt} : \ArrD{\F} \Rightarrow \ArrD{\G}$ over $\id$.
\end{exa}

\section{Models of Type Theory in Univalent Foundations}
\label{sec:models}
Before phrasing internal language theorems of various classes of univalent categories,
we first discuss the categorical semantics of dependent type theory.
In this discussion, our main question is what would be suitable categorical structures to model dependent type theory in univalent categories.
We start by recalling \textbf{categories with families}~\cite[Definition 1]{dybjer:1995},
which were used by Clairambault and Dybjer to prove their internal language theorems~\cite[Theorem 6.1]{clairambault:2014}.
Let us write $\Fam$ for the category whose objects are pairs of a set $A : \Set$
and a family $B : A \rightarrow \Set$.

\begin{defiC}[{\cite[Definition 1]{dybjer:1995}}]
\label[defi]{def:cwf}
A \textbf{category with families} (CwF) consists of
\begin{itemize}
  \item a category $\C$ with a terminal object $\ECtx$;
  \item a functor $\CwFFunc : \Op{\C} \rightarrow \Fam$ that sends $\Gamma : \C$ to a set $\Ty_{\Gamma}$ with a family $\Tm_{\Gamma} : \Ty_{\Gamma} \rightarrow \Set$;
  \item for all $\Gamma : \C$ and $A : \Ty_{\Gamma}$ a context extension $\CtxExt{\Gamma}{A} : \C$.
\end{itemize}
The context comprehension operation needs to satisfy a universal property,
which can be found in the literature~\cite[Definition 2.2]{clairambault:2014}.
\end{defiC}

Here the category $\C$ is the category of contexts and context morphisms,
inhabitants of $\Ty_{\Gamma}$ and $\Tm_{\Gamma}(A)$ represent types in context $\Gamma$ and terms of type $A$ respectively.
Throughout the years CwFs have been used to define various models of type theory such as the cubical sets model~\cite{bezem:2013}.
One key feature of categories with families is that they can be presented as a generalized algebraic theory~\cite{dybjer:1995,bezem:2021}.
As a consequence, one can directly define the initial CwF (i.e., the syntax) using quotient inductive-inductive types~\cite{altenkirch:2016,altenkirch:2018}.
We can also use categories with families as a general framework to define various notions of models in categorical logic~\cite{castellan:2021}.

However, the notion of category with families that we gave in \Cref{def:cwf},
does not give a suitable framework to describe internal models of homotopy type theory~\cite{kraus:2021},
and in particular the internal language of univalent categories.
To understand why, let us consider the (univalent) category $\Set$ of sets.
One might hope that $\Set$ gives rise to a model of type theory in which types in the empty context are the same as sets,
but this is not the case if we were to use categories with families.
This is because in \Cref{def:cwf} we require there to be a set of types in every context $\Gamma$.
That requirement gets violated by the univalence axiom, which prevents the type of sets from being a set itself.
The same issue occurs when showing that every univalent category with finite limits gives rise to a category with families,
because the type of objects in a univalent category forms, in general, only a 1-type by \Cref{prop:hlevel-obj}.
For the same reason, one needs to assume uniqueness of identity proofs to construct the set model when defining the syntax of type theory as a quotient inductive-inductive type \cite[Section 6]{altenkirch:2016}.

All in all,
we either need to alter the notion of category with families to remove the restriction that types form a set,
or use another categorical structure to model dependent types in univalent categories.
In the categorical structure used, the collection of types in each context should not be required to form a set.
A wide variety of categorical models of dependent types,
such as categories with families~\cite{dybjer:1995}, categories with attributes~\cite{cartmell:1978}, C-systems~\cite{cartmell:1986,voevodsky:2017}, and natural models~\cite{awodey:2018},
do not satisfy this requirement.

\subsection{Comprehension Categories}
One categorical structure for interpreting dependent types that does satisfy our demands,
is given by \textbf{comprehension categories}~\cite{jacobs:1993,jacobs:2001}.
We discuss other possibilities in \Cref{rem:other-options}.
The notion of comprehension category is centered around that of \textbf{fibrations}~\cite{grothendieck:1960},
which are used to model the category of contexts, the category of types depending on contexts, and the substitution operation on types.
The other core ingredient of comprehension categories is given by the \textbf{comprehension functor},
which models context extension.
Following Ahrens and Lumsdaine, we define comprehension categories using displayed categories~\cite[Definition 6.1]{ahrens:2019a}.

\begin{defiC}[{\cite[Definition 4.1]{jacobs:1993}}]
\label[defi]{def:comp-cat}
A \conceptDef{comprehension category}{CategoryTheory.DisplayedCats.ComprehensionC}{comprehension_cat_structure} is given by
\begin{itemize}
  \item a category $\C$ with a terminal object $\ECtx$;
  \item a displayed category $\D$ over $\C$;
  \item a cleaving on $\D$;
  \item a displayed functor $\compfunctor : \D \rightarrow \ArrD{\C}$ over the identity
\end{itemize}
such that $\compfunctor$ preserves Cartesian morphisms.
Pictorially, we represent this data as follows.
\[
\begin{tikzcd}
  {\D} \ar[rr, "{\compfunctor}"] & & \ArrD{\C}
  \\
  &
  {\C}
\end{tikzcd}
\]
A comprehension category is called \textbf{full}
if $\compfunctor$ is a fully faithful functor.
We say that a comprehension category is \textbf{univalent}
if both $\C$ and $\D$ are univalent.
\end{defiC}

Since $\compfunctor : \D \rightarrow \ArrD{\C}$ is a displayed functor over the identity,
it maps objects $\xx : \dob{\D}{x}$ to objects in $\ArrD{\C}$ over $x$,
which are morphisms $f : y \rightarrow x$.
Every comprehension category $\compfunctor : \D \rightarrow \ArrD{\C}$
gives rise to a triangle of functors, as depicted below, that commutes up to definitional equality.
\[
\begin{tikzcd}
  {\total{\D}} && {\Arr{\C}} \\
  & \C
  \arrow["{\total{\compfunctor}}", from=1-1, to=1-3]
  \arrow["{\proj{\D}}"', from=1-1, to=2-2]
  \arrow["\cod", from=1-3, to=2-2]
\end{tikzcd}
\]

In the remainder of this paper, we use univalent full comprehension categories as our categorical model for dependent types.
The notion of univalent full comprehension category is less restrictive than that of a CwF,
because the former does not require the types to form a set.
Every univalent comprehension category gives rise to a pseudofunctor $\Ty : \Op{\C} \rightarrow \UnivCat$
which lands in the bicategory of univalent categories.
Since the objects of a univalent category generally only form a 1-type,
the types in some context $\Gamma$ are only required to be a 1-type.
This additional flexibility is advantageous, because it allows us to define the comprehension category of sets.

\begin{exa}
\label[exa]{exa:set-comp-cat}
We have the following comprehension category.
\[
\begin{tikzcd}
  {\ArrD{\Set}} \ar[rr, "{\id}"] & & \ArrD{\Set}
  \\
  &
  {\Set}
\end{tikzcd}
\]
This comprehension category is both full and univalent.
\end{exa}

We can generalize \Cref{exa:set-comp-cat} to arbitrary univalent categories with finite limits.

\begin{exa}
\label[exa]{exa:arrow-comp-cat}
If $\C$ is a univalent category with finite limits,
then we have the following comprehension category.
\[
\begin{tikzcd}
  {\ArrD{\C}} && {\ArrD{\C}} \\
  & \C
  \arrow["\id", from=1-1, to=1-3]
\end{tikzcd}
\]
Note that $\ArrD{\C}$ is a fibration if and only if $\C$ has pullbacks~\cite{jacobs:2001}.
The resulting comprehension category is full and univalent,
and we denote it by $\id : \ArrD{\C} \rightarrow \ArrD{\C}$.
\end{exa}

It is important to note that we do not assume that the cleaving in \cref{def:comp-cat} is split.
While in set-theoretic foundations splitness is necessary to model dependent type theory,
this is not so if we assume univalence.
There are rules in type theory saying that the substitution $\substTy{\id}{A}$ is equal to $A$,
and that substitution commutes with composition of context morphisms.
In the context of fibrations,
these rules translate to the Cartesian lift of the identity being equal to the identity,
and the Cartesian lift of a composition being equal to the composition of the Cartesian lifts.
These two requirements only hold up to isomorphism in general,
and they only hold up to equality if the fibration is split.
However, if we assume that the involved categories are univalent,
then we always have the desired identities,
because every isomorphism gives rise to an identity.
The coherence problem remains:
since one would not force the initial model to be a set,
there could be multiple ways to prove that two types are equal.
One can solve this problem for the groupoid-syntax of type theory with a universe and $\Pi$-types~\cite[Theorem 16]{altenkirch:2025}.

\begin{rem}
\label[rem]{rem:other-options}
Finally, let us remark that there are other notions of models of dependent types that also satisfy our demands,
such as \textbf{generalized category with families}~\cite[Definition 10]{coraglia:2024},
also known as \textbf{judgmental dependent type theories}~\cite[Definition 3.0.1]{coraglia:2021},
and \textbf{groupoid categories with families}~\cite[Definition 19]{altenkirch:2025}.
In fact, the bicategories of generalized categories with families and of comprehension categories are biequivalent~\cite[Theorem 3]{coraglia:2024}.
The notion of groupoid category with families is in spirit close to CwFs,
and they give us a syntactic viewpoint on higher-dimensional models of type theory.
Other alternatives are \textbf{display map categories}~\cite{taylor:1999}, \textbf{type categories}~\cite{pitts:2000}, and \textbf{D-categories}~\cite{ehrhard:1988}.
A 2-categorical comparison between various notions of categorical model for dependent type theory is given by Ahrens et al.~\cite{ahrens:2024a}.
The reason why we use comprehension categories is that the comprehension functor plays a prominent role in our proof
(see \Cref{prop:comprehension-eso}).
\end{rem}

\subsection{Interpreting Type Theory in Comprehension Categories}
Let us finish this section by recalling how to interpret type theory with explicit substitution~\cite{curien:1993,curien:2014} in a comprehension category $\D \rightarrow \ArrD{\C}$.
Specifically, we introduce notation that allows us to reason type theoretically about comprehension categories
meaning that we discuss how contexts, context morphisms, types, and terms are interpreted.
We also discuss various operations on types and terms,
such as substitution and variables.
Our development follows the work by Najmaei et al.~\cite{najmaei:2025},
and in particular, we consider a coercion operation on terms
that allows us to coerce terms of type $A$ to terms of type $B$ if we have a morphism from $A$ to $B$.
The material in this section can be used to interpret the groupoid syntax of type theory in comprehension categories~\cite{altenkirch:2025}.

\paragraph*{Contexts, types, and terms.}
We start by describing how the basic judgments are interpreted.
We interpret objects and morphisms in $\C$ as contexts and context morphisms respectively,
and types in a context $\Gamma$ are given by objects $A : \dob{\D}{\Gamma}$.
Terms of type $A$ in context $\Gamma$ are interpreted as sections of $\CtxPr[A] : \CtxExt{\Gamma}{A} \rightarrow \Gamma$,
i.e., a morphism $t : \Gamma \rightarrow \CtxExt{\Gamma}{A}$ such that the following composition is the identity.
\[
\begin{tikzcd}
  \Gamma & {\CtxExt{\Gamma}{A}} & \Gamma
  \arrow["t", from=1-1, to=1-2]
  \arrow["{\CtxPr[A]}", from=1-2, to=1-3]
\end{tikzcd}
\]
We write $\Tms{\Gamma}{A}$ for the collection of terms of type $A$ in context $\Gamma$.

\paragraph*{Operations on types.}
Next we look at substitution of types.
Given a context morphism $s : \Gamma \rightarrow \Delta$ and a type $A : \dob{\D}{\Delta}$,
we interpret the substitution $\substTy{s}{A} : \dob{\D}{\Gamma}$ as the domain of the following Cartesian lift.
\[
\begin{tikzcd}
  {\substTy{s}{A}} & A \\
  \Gamma & \Delta
  \arrow["\substTyMor{s}", from=1-1, to=1-2]
  \arrow["s"', from=2-1, to=2-2]
\end{tikzcd}
\]
We have isomorphisms $\iso{\substTy{\id[\Delta]}{A}}{A}$ and $\iso{\substTy{(s_1 \cdot s_2)}{A}}{\substTy{s_1}{\substTy{s_2}{A}}}$
in the fiber categories,
because taking Cartesian lifts is pseudofunctorial.
In addition,
every morphism $f : \dmor{A}{B}{\id}$ gives rise to a morphism $\substTy{s}{f} : \dmor{\substTy{s}{A}}{\substTy{s}{B}}{\id}$.
If $\D$ is univalent, then these isomorphisms induce actual identities.

\paragraph*{Operations on contexts.}
The empty context is the terminal object $\ECtx$ and we have unique maps $\ECtxMap[\Gamma] : \Gamma \rightarrow \ECtx$.
Context extension is interpreted using the functor $\compfunctor$:
if we have a type $A$ in context $\Gamma$,
then from $\compfunctor(A)$ we obtain a context $\CtxExt{\Gamma}{A}$ and a morphism $\CtxPr[A] : \CtxExt{\Gamma}{A} \rightarrow \Gamma$.
To interpret extension of context morphisms,
we assume that we have a context morphism $s : \Gamma \rightarrow \Delta$
and a term $t : \Tms{\Gamma}{\substTy{s}{A}}$.
We define $\PairSub{s}{t} : \Gamma \rightarrow \CtxExt{\Delta}{A}$ to be $t \cdot \compfunctor(\substTyMor{s})$.
\[
\begin{tikzcd}
  {\CtxExt{\Gamma}{\substTy{s}{A}}} & {\CtxExt{\Delta}{A}} \\
  \Gamma & \Delta
  \arrow["{\compfunctor(\substTyMor{s})}", from=1-1, to=1-2]
  \arrow["\CtxPr"', bend right=20, from=1-1, to=2-1]
  \arrow["\CtxPr", from=1-2, to=2-2]
  \arrow["t"', bend right=20, from=2-1, to=1-1]
  \arrow["s"', from=2-1, to=2-2]
\end{tikzcd}
\]

\paragraph*{Operations on terms.}
To interpret substitution of terms,
we use that $\compfunctor$ preserves Cartesian morphisms.
We thus know that the following square is a pullback,
since $\substTyMor{s}$ is Cartesian.
\[
\begin{tikzcd}
  {\CtxExt{\Gamma}{\substTy{s}{A}}} & {\CtxExt{\Delta}{A}} \\
  \Gamma & \Delta
  \arrow["\compfunctor(\substTyMor{s})", from=1-1, to=1-2]
  \arrow["{\CtxPr}"', from=1-1, to=2-1]
  \arrow["{\CtxPr}", from=1-2, to=2-2]
  \arrow["s"', from=2-1, to=2-2]
  \arrow["\lrcorner"{anchor=center, pos=0.125}, draw=none, from=1-1, to=2-2]
\end{tikzcd}
\]
If we have a context morphism $s : \Gamma \rightarrow \Delta$
and a term $t : \Tms{\Delta}{A}$,
we get another term $\substTm{s}{t} : \Tms{\Gamma}{\substTy{s}{A}}$.
\[
\begin{tikzcd}
  \Gamma & \Delta \\
  & {\CtxExt{\Gamma}{\substTy{s}{A}}} & {\CtxExt{\Delta}{A}} \\
  & \Gamma & \Delta
  \arrow["{s}", from=1-1, to=1-2]
  \arrow["{\substTm{s}{t}}"{description}, dashed, from=1-1, to=2-2]
  \arrow["\id"', bend right=30, from=1-1, to=3-2]
  \arrow["t", bend left=10, from=1-2, to=2-3]
  \arrow["\compfunctor(\substTyMor{s})", from=2-2, to=2-3]
  \arrow["{\CtxPr}"', from=2-2, to=3-2]
  \arrow["{\CtxPr}", from=2-3, to=3-3]
  \arrow["s"', from=3-2, to=3-3]
  \arrow["\lrcorner"{anchor=center, pos=0.125}, draw=none, from=2-2, to=3-3]
\end{tikzcd}
\]

We also use that $\compfunctor$ preserves Cartesian morphisms to interpret the variable rule.
More specifically, if we have a type $A$ in context $\Gamma$,
we have a term $\var$ in context $\CtxExt{\Gamma}{A}$ of type $\substTy{\CtxPr}{A}$,
which is defined as follows.
\[
\begin{tikzcd}
  {\CtxExt{\Gamma}{A}} \\
  & {\CtxExt{\CtxExt{\Gamma}{A}}{\substTy{\CtxPr}{A}}} & {\CtxExt{\Gamma}{A}} \\
  & {\CtxExt{\Gamma}{A}} & \Gamma
  \arrow["\var"{description}, dashed, from=1-1, to=2-2]
  \arrow["\id", bend left=20, from=1-1, to=2-3]
  \arrow["\id"', bend right=20, from=1-1, to=3-2]
  \arrow["\compfunctor(\substTyMor{\CtxPr})", from=2-2, to=2-3]
  \arrow["{\CtxPr}"', from=2-2, to=3-2]
  \arrow["{\CtxPr}", from=2-3, to=3-3]
  \arrow["{\CtxPr}"', from=3-2, to=3-3]
  \arrow["\lrcorner"{anchor=center, pos=0.125}, draw=none, from=2-2, to=3-3]
\end{tikzcd}
\]

\paragraph*{Coercion of terms.}
Finally, we consider a coercion operation on terms.
Following Coraglia and Emmenegger,
a morphism $\ff : \dmor{A}{B}{\id}$ in $\D$ indicates that we have a coercion from $A$ to $B$~\cite{coraglia:2024a}.
Note that such morphisms are not unique in general.
We can use such morphisms to coerce terms of type $A$ to terms of type $B$.
To do so,
we first observe that morphisms $\ff : \dmor{A}{B}{\id}$ in $\D$
give rise to morphisms $\compfunctor(\ff) : \CtxExt{\Gamma}{A} \rightarrow \CtxExt{\Gamma}{B}$
making the following triangle commute.
\[
\begin{tikzcd}
  {\CtxExt{\Gamma}{A}} && {\CtxExt{\Gamma}{B}} \\
  & \Gamma
  \arrow["{\compfunctor(f)}", from=1-1, to=1-3]
  \arrow["{\CtxPr[A]}"', from=1-1, to=2-2]
  \arrow["{\CtxPr[B]}", from=1-3, to=2-2]
\end{tikzcd}
\]
If we have types $A, B : \dob{\D}{\Gamma}$
and a morphism $f : \dmor{A}{B}{\id[\Gamma]}$,
then every term $t : \Tms{\Gamma}{A}$ gives rise to a term $\coeTm{t}{f} : \Tms{\Gamma}{B}$
defined by the following composition.
\[
\begin{tikzcd}
  \Gamma & {\CtxExt{\Gamma}{A}} & {\CtxExt{\Gamma}{B}}
  \arrow["t", from=1-1, to=1-2]
  \arrow["{\compfunctor(f)}", from=1-2, to=1-3]
\end{tikzcd}
\]
Coercion on terms is functorial
in the sense that we have $\coeTm{t}{\id[A]} = t$
and $\coeTm{t}{(f \cdot g)} = \coeTm{(\coeTm{t}{f})}{g}$.

\section{The Bicategory of Univalent Full Comprehension Categories}
\label{sec:bicat-comp-cat}
Now we start our study of internal language theorems for univalent categories.
Following Clairambault and Dybjer~\cite{clairambault:2014} we express our internal language theorems as biequivalences
between bicategories of models and of type theories.
Since we use univalent full comprehension categories as our models,
the bicategory $\FullCompCat$ of univalent full comprehension categories acts as the bicategory of models throughout this paper.
In this section, we take a closer look at $\FullCompCat$, and we have two goals:
define this bicategory and prove that it is univalent.

Before we do so, let us recall the bicategory of comprehension categories defined by Coraglia and Emmenegger~\cite[Definition 7]{coraglia:2024}.
The 1-cells
from a comprehension category $\compfunctor_1 : \D_1 \rightarrow \ArrD{\C_1}$ to $\compfunctor_2 : \D_2 \rightarrow \ArrD{\C_2}$
are given by \textbf{lax morphisms},
which consist of a functor $\F : \C_1 \rightarrow \C_2$ that preserves terminal objects,
a Cartesian displayed functor $\FF : \D_1 \rightarrow \D_2$ over $\F$,
and a displayed natural transformation $\F_\chi$ over the identity as follows.
\[
\begin{tikzcd}
  {\D_1} & {\ArrD{\C_1}} \\
  {\D_2} & {\ArrD{\C_2}}
  \arrow[""{name=0, anchor=center, inner sep=0}, "{\compfunctor_1}", from=1-1, to=1-2]
  \arrow["\FF"', from=1-1, to=2-1]
  \arrow["{\ArrD{\F}}", from=1-2, to=2-2]
  \arrow[""{name=1, anchor=center, inner sep=0}, "{\compfunctor_2}"', from=2-1, to=2-2]
  \arrow["\F_\chi", shorten <=4pt, shorten >=4pt, Rightarrow, from=0, to=1]
\end{tikzcd}
\]
Lax morphisms preserve substitution up to isomorphism
and context extension up to a morphism.
A 2-cell from $(\F, \FF, \F_\chi)$ to $(\G, \GG, \G_\chi)$
is given by
a natural transformation $\nt : \F \Rightarrow \G$,
a displayed natural transformation $\tautau : \FF \Rightarrow \GG$ over $\nt$
such that the compositions below are equal.
\[\begin{tikzcd}[column sep = 50pt]
    {\D_1} & {\ArrD{\C_1}} \\
    {\D_2} & {\ArrD{\C_2}}
    \arrow["{\compfunctor_1}", from=1-1, to=1-2]
    \arrow[""{name=0, anchor=center, inner sep=0}, "\FF"', bend right=50, from=1-1, to=2-1]
    \arrow[""{name=1, anchor=center, inner sep=0}, "\GG"{description}, bend left=50, from=1-1, to=2-1]
    \arrow[""{name=2, anchor=center, inner sep=0}, "{\ArrD{\G}}", from=1-2, to=2-2]
    \arrow["{\compfunctor_2}"', from=2-1, to=2-2]
    \arrow["\tautau", shorten <=2pt, shorten >=2pt, Rightarrow, from=0, to=1, shorten >=1ex]
    \arrow["{\G_\chi}", shorten <=7pt, shorten >=7pt, Rightarrow, from=1, to=2]
\end{tikzcd}
\quad \quad \quad
\begin{tikzcd}[column sep = 50pt]
  {\D_1} & {\ArrD{\C_1}} \\
  {\D_2} & {\ArrD{\C_2}}
  \arrow["{\compfunctor_1}", from=1-1, to=1-2]
  \arrow[""{name=0, anchor=center, inner sep=0}, "\FF"', from=1-1, to=2-1]
  \arrow[""{name=1, anchor=center, inner sep=0}, "{\ArrD{\G}}", bend left=50, from=1-2, to=2-2]
  \arrow[""{name=2, anchor=center, inner sep=0}, "{\ArrD{\F}}"{description}, bend right=50, from=1-2, to=2-2]
  \arrow["{\compfunctor_2}"', from=2-1, to=2-2]
  \arrow["{\F_{\chi}}", shorten <=7pt, shorten >=7pt, Rightarrow, from=0, to=2]
  \arrow["{\ArrD{\tau}}", shorten <=2pt, shorten >=2pt, Rightarrow, from=2, to=1, shorten <=1.5ex]
\end{tikzcd}
\]

The bicategory $\FullCompCat$ that we use, is slightly different from the bicategory defined by Coraglia and Emmenegger~\cite{coraglia:2024}.
Instead of all comprehension categories, $\FullCompCat$ only has univalent full comprehension categories as objects.
Following Clairambault and Dybjer~\cite{clairambault:2014}, we also restrict the morphisms: instead of all lax morphisms, $\FullCompCat$ only contains pseudo morphisms.
Here we say that a lax morphism $(\F, \FF, \F_\chi)$ is a \textbf{pseudo morphism} if $\F_\chi$ is a natural isomorphism~\cite{curien:2014}.
Such morphisms thus preserve context extension up to isomorphism.
We do not pose any restriction on the 2-cells.

To construct $\FullCompCat$, we use \textbf{displayed bicategories}~\cite[Definition 6.1]{ahrens:2021},
which allow us to construct this bicategory and to prove its univalence modularly.
These techniques have also been used to construct the bicategories of
monoidal categories~\cite{wullaert:2022},
double categories~\cite{vanderweide:2024},
and monads~\cite{vanderweide:2023a}.

We start this section by recalling some basic notions and theorems about displayed bicategories.
After that we use these techniques to define the bicategory $\FullCompCat$ of univalent full comprehension categories.
We finish by proving that $\FullCompCat$ is univalent.

\subsection{Displayed Bicategories}
To construct the bicategory $\FullCompCat$,
we make extensive use of displayed bicategories.
When we construct some bicategory using displayed bicategories,
we start with some given bicategory $\B$
and we define a displayed bicategory over it.
The objects, 1-cells, and 2-cells in this displayed bicategory describe
the structure and properties to be added to the objects, 1-cells, and 2-cells in $\B$.
Each displayed bicategory $\D$ gives rise to a total bicategory $\total{\D}$,
whose objects are objects in $\B$ together with the structure specified by the displayed bicategory.
The key feature of displayed bicategories is that we can modularly reason about the total bicategory:
we can prove that $\total{\D}$ is univalent by solely considering $\B$ and $\D$,
and we can construct adjoint equivalences in $\total{\D}$ from adjoint equivalences in $\B$ and $\D$.

To construct $\FullCompCat$, we start with the bicategory $\UnivCat$ of univalent categories, and functors.
We add the necessary structures to obtain full comprehension categories in layers,
and each layer is described by a displayed bicategory.
For instance, we construct a displayed bicategory over $\UnivCat$ whose objects over $\C$ are fibrations.
By showing that each of these displayed bicategories are univalent,
we get a simple and modular proof that $\FullCompCat$ is univalent.
In addition,
we use the notion of displayed adjoint equivalence to describe adjoint equivalences in $\FullCompCat$.

Displayed bicategories generalize displayed categories to the 2-dimensional setting.
Intuitively, a displayed bicategory $\D$ over a bicategory $\B$
represents structure and properties to be added to the objects, 1-cells, and 2-cells of $\B$.
Every displayed bicategory $\D$ gives rise to a \textbf{total bicategory} $\total{\D}$
and a \textbf{forgetful pseudofunctor} $\proj{} : \total{\D} \rightarrow \B$ that forgets the added structure and properties.
More precisely, we define displayed bicategories as follows.

\begin{defiC}[{\cite[Definition 6.1]{ahrens:2021}}]
\label[defi]{def:disp-bicat}
A \conceptDef{displayed bicategory}{Bicategories.DisplayedBicats.DispBicat}{disp_bicat} $\D$ over a bicategory $\B$ is given by
\begin{itemize}
  \item a type $\dob{\D}{x}$ of \textbf{objects over $x$} for all objects $x : \B$;
  \item a type $\dmor{\xx}{\yy}{f}$ of \textbf{morphisms over $f$} for all 1-cells $f : x \rightarrow y$ in $\B$ and objects $\xx$ and $\yy$ over $x$ and $y$ respectively;
  \item a set $\dtwo{\ff}{\gg}{\tau}$ of \textbf{2-cells over $\tau$} for all 2-cells $\tau : f \Rightarrow g$ and morphisms $\ff : \dmor{\xx}{\yy}{f}$ and $\gg : \dmor{\xx}{\yy}{g}$ over $f$ and $g$ respectively.
\end{itemize}
We write $\idd[\xx]$ and $\idd[\ff]$ for the identity 1-cells and 2-cells respectively,
and we write $\ff \cdot \gg$ and $\tautau \cdot \thetatheta$ for their compositions.
A complete list of operations and laws is found in the literature~\cite[Definition 6.1]{ahrens:2021}.

Every displayed bicategory $\D$ gives rise to a \conceptDef{total bicategory}{Bicategories.DisplayedBicats.DispBicat}{total_bicat} $\total{\D}$,
whose objects are pairs $x : \B$ and $\xx : \dob{\D}{x}$
and a \conceptDef{forgetful pseudofunctor}{Bicategories.PseudoFunctors.Examples.Projection}{pr1_psfunctor} $\proj{\D} : \total{\D} \rightarrow \B$.
\end{defiC}

To understand this notion, let us consider an example,
namely the displayed bicategory $\UnivCatT$ of terminal objects over the bicategory of univalent categories.
Here the objects over a univalent category $\C$ are terminal objects for $\C$,
the 1-cells over a functor $\F : \C_1 \rightarrow \C_2$ are proofs that witness that $\F$ preserves terminal objects
and the type of 2-cells over a natural transformation $\nt$ is the unit type.
The total bicategory of this displayed bicategory is the bicategory of univalent categories with a terminal object,
and the forgetful pseudofunctor maps a category with a terminal object to its underlying category.
Yet another example is the displayed bicategory $\UnivCatM$ of monoidal structures over $\UnivCat$~\cite{wullaert:2022}.
The objects over $\C$ are monoidal structures for $\C$,
the 1-cells over $\F$ are structures that witness that $\F$ is a lax monoidal functor,
and the 2-cells over $\nt$ are proofs that $\nt$ is a monoidal transformation.
The bicategory of univalent monoidal categories is the total bicategory of this displayed bicategory,
and we also have the forgetful pseudofunctor,
which maps every monoidal category to its underlying category.

Next we look at how to construct adjoint equivalences and invertible 2-cells in the total bicategory.
Here the main idea is that an adjoint equivalence in $\total{\D}$ can be obtained from an adjoint equivalence in $\B$
together with some additional structure in $\D$
and similarly for invertible 2-cells.
The necessary extra structure is given by \textbf{displayed adjoint equivalences} and \textbf{displayed invertible 2-cells} respectively.
These notions are defined similarly to their non-displayed versions,
except that we use displayed objects, 1-cells, and 2-cells.

\begin{defiC}[{\cite[Definitions 7.1 and 7.2]{ahrens:2021}}]
\label[defi]{def:disp-adjequiv}
Let $\D$ be a displayed bicategory over $\B$.
Given 1-cells $f, g : x \rightarrow y$ and an invertible 2-cell $\tau : f \Rightarrow g$,
a \conceptDef{displayed invertible 2-cell}{Bicategories.DisplayedBicats.DispBicat}{is_disp_invertible_2cell} over $\tau$ from $\ff : \dmor{\xx}{\yy}{f}$ to $\gg : \dmor{\xx}{\yy}{g}$
consists of displayed 2-cells $\tautau : \dtwo{\ff}{\gg}{\tau}$ and $\overline{\tau^{-1}} : \dtwo{\gg}{\ff}{\tau^{-1}}$
that compose to the identity.
We denote the type of invertible 2-cells over $\tau$ by $\dinvcell{\ff}{\gg}{\tau}$.

Given a 1-cell $\l : x \rightarrow y$,
a \conceptDef{displayed adjoint equivalence}{Bicategories.DisplayedBicats.DispAdjunctions}{disp_left_adjoint_equivalence}
over an adjoint equivalence $\adjunction{\l}{\r}{\eta}{\varepsilon}$
from $\xx : \dob{\D}{x}$ to $\yy : \dob{\D}{y}$
consists of displayed 1-cells $\ll : \dmor{\xx}{\yy}{\l}$ and $\rr : \dmor{\yy}{\xx}{\r}$
and displayed invertible 2-cells $\etaeta : \dtwo{\id}{\ll \cdot \rr}{\eta}$ and $\epsilonepsilon : \dtwo{\rr \cdot \ll}{\id}{\varepsilon}$
satisfying the usual triangle laws.
The type of adjoint equivalences over $\l$ is denoted by $\dadjequiv{\xx}{\yy}{\l}$.
\end{defiC}

Every displayed invertible 2-cell $\tautau$ over $\tau$ gives rise to an invertible 2-cell $(\tau , \tautau)$ in $\total{\D}$.
In addition, if we have a displayed adjoint equivalence $\ll$ over an adjoint equivalence $\adjunction{\l}{\r}{\eta}{\varepsilon}$,
then we get an adjoint equivalence in $\total{\D}$ whose left adjoint is given by $(\l, \ll)$.

Every displayed 1-cell in $\UnivCatT$ is a displayed adjoint equivalence,
because every adjoint equivalence of categories preserves terminal objects.
It also is the case that every displayed 2-cell in $\UnivCatT$ is invertible,
because the type of 2-cells over $\nt$ is the unit type.
Adjoint equivalences and invertible 2-cells in the total bicategory $\total{\UnivCatT}$
thus correspond to adjoint equivalences and natural isomorphisms respectively.
While every displayed 2-cell in $\UnivCatM$ is invertible,
displayed 1-cells in $\UnivCatM$ are not necessarily adjoint equivalences.
Instead, displayed adjoint equivalences over $\adjunction{\F}{\G}{\eta}{\varepsilon}$,
correspond to structures that witness that $\F$ is a strong monoidal functor.

Displayed bicategories also allow us to modularly prove that a bicategory is univalent.
Instead of proving directly that the total bicategory $\total{\D}$ is univalent,
it is often simpler to consider $\B$ and $\D$ separately
and to prove that both $\B$ and $\D$ are univalent.
For this purpose, we define \textbf{univalent displayed bicategories}.

\begin{defiC}[{\cite[Definition 7.3]{ahrens:2021}}]
\label[defi]{def:disp-univalent}
Let $\D$ be a displayed bicategory over $\B$.
\begin{itemize}
  \item We say that $\D$ is \conceptDef{locally univalent}{Bicategories.DisplayedBicats.DispUnivalence}{disp_univalent_2_1}
    if for all objects $x, y : \B$,
    1-cells $f, g : x \rightarrow y$ with $p : f = g$,
    displayed objects $\xx : \dob{\D}{x}$ and $\yy : \dob{\D}{y}$,
    and displayed 1-cells $\ff : \dmor{\xx}{\yy}{f}$ and $\gg : \dmor{\xx}{\yy}{g}$
    the map $\idtoinvcelldisp{\ff}{\gg} : \ff = \gg \rightarrow \dinvcell{\ff}{\gg}{\idtoinvcell{f}{g}(p)}$
    defined by path induction is an equivalence of types.
  \item We say that $\D$ is \conceptDef{globally univalent}{Bicategories.DisplayedBicats.DispUnivalence}{disp_univalent_2_0}
    if for all objects $x, y : \B$ with $p : x = y$
    and displayed objects $\xx : \dob{\D}{x}$ and $\yy : \dob{\D}{y}$ over $x$ and $y$ respectively,
    the map $\idtoadjequivdisp{\xx}{\yy} : \xx = \yy \rightarrow \dadjequiv{\xx}{\yy}{\idtoadjequiv{x}{y}(p)}$
    defined by path induction is an equivalence of types.
  \item Finally, we say that $\D$ is \conceptDef{univalent}{Bicategories.DisplayedBicats.DispUnivalence}{disp_univalent_2} if it is both locally and globally univalent.
\end{itemize}
\end{defiC}

Note that in \cref{def:disp-univalent} it suffices to restrict to identities
rather than arbitrary adjoint equivalences and invertible 2-cells in the base bicategory $\B$.
The resulting definition is equivalent to the one by Ahrens et al.~\cite[Definition 7.3]{ahrens:2021}.
The key theorem about displayed univalence is as follows~\cite[Theorem 7.4]{ahrens:2021}.

\begin{propL}[\coqdocurl{Bicategories.DisplayedBicats.DispUnivalence}{total_is_univalent_2}]
\label{prop:total-univalent}
Let $\D$ be a univalent displayed bicategory over a univalent bicategory $\B$.
Then the total bicategory $\total{\D}$ is univalent.
\end{propL}

Both $\UnivCatT$ and $\UnivCatM$ are univalent,
and hence their total bicategories are so as well.
Throughout this paper, we are interested in additional properties that displayed bicategories could satisfy.
The first property, which we call \conceptDef{locally preordered}{Bicategories.DisplayedBicats.DispBicat}{disp_2cells_isaprop},
says that each type $\dtwo{\ff}{\gg}{\tau}$ is a proposition.
The second property, which we call \conceptDef{locally groupoidal}{Bicategories.DisplayedBicats.DispBicat}{disp_locally_groupoid},
says that each displayed 2-cell $\dtwo{\ff}{\gg}{\tau}$ over an invertible 2-cell is invertible as well.
Finally, we say that a displayed bicategory \conceptDef{has trivial 2-cells}{Bicategories.DisplayedBicats.DispBicat}{disp_2cells_iscontr} if the type $\dtwo{\ff}{\gg}{\tau}$
is contractible for all displayed 1-cells $\ff : \dmor{\xx}{\yy}{f}$ and $\gg : \dmor{\xx}{\yy}{g}$, and 2-cells $\tau : f \Rightarrow g$.
Note that every displayed bicategory with trivial 2-cells is both locally preordered and groupoidal.
The displayed bicategory $\UnivCatT$ has trivial 2-cells,
because the type of 2-cells over $\nt$ is the unit type.
While $\UnivCatM$ is both locally preordered and groupoidal,
it does not have trivial 2-cells,
because not every natural transformation is monoidal.
The concrete examples of displayed bicategories that we consider in the remainder of this paper,
are locally preordered and groupoidal.

We end this recap with several constructions on displayed bicategories.

\begin{exa}
\label[exa]{exa:disp-bicat}
Let $\B$ be a bicategory and suppose that we have a type family $\Pred$ on the objects of $\B$.
We also assume that we have a type family $\Qred$ on triples of objects $x, y : \B$
such that $\Pred(x)$ and $\Pred(y)$ hold,
and 1-cells $f : x \rightarrow y$.
Suppose that $\Qred(\id[x])$ holds for all $x : \B$ such that $\Pred(x)$ holds
and that $\Qred(f \cdot g)$ holds whenever $\Qred(f)$ and $\Qred(g)$ hold.
We define the \conceptDef{displayed subbicategory}{Bicategories.DisplayedBicats.Examples.Sub1Cell}{disp_subbicat} $\subbicatD{\Pred}{\Qred}$ as the displayed bicategory whose
\begin{itemize}
  \item objects over $x$ are proofs of $\Pred(x)$;
  \item 1-cells over $f : x \rightarrow y$ are proofs of $\Qred(f)$;
  \item type of 2-cells over $\tau : f \Rightarrow g$ is the unit type.
\end{itemize}
This displayed bicategory is univalent whenever both $\Pred$ and $\Qred$ are valued in propositions.
We write $\subbicat{\Pred}{\Qred}$ for $\total{\subbicatD{\Pred}{\Qred}}$.
Note that $\subbicat{\Pred}{\Qred}$ has trivial 2-cells.

As a special case of the previous example we have the full subbicategory.
Given a type family $\Pred$ on the objects of $\B$,
the \conceptDef{displayed full subbicategory}{Bicategories.DisplayedBicats.Examples.FullSub}{disp_fullsubbicat} $\fullsubbicatD{\Pred}$ is the displayed bicategory whose
\begin{itemize}
  \item objects over $x$ are proofs of $\Pred(x)$;
  \item type of 1-cells over $f : x \rightarrow y$ is the unit type;
  \item type of 2-cells over $\tau : f \Rightarrow g$ is the unit type.
\end{itemize}
This displayed bicategory is univalent if the type family $\Pred$ is valued in propositions.
We denote $\total{\fullsubbicatD{\Pred}}$ by $\fullsubbicat{\Pred}$.

Suppose that we have displayed bicategories $\D_1$ and $\D_2$ over $\B$.
We define their \conceptDef{product}{Bicategories.DisplayedBicats.Examples.Prod}{disp_dirprod_bicat} $\D_1 \times \D_2$ as the displayed bicategory whose
\begin{itemize}
  \item objects over $x : \B$ are pairs of $\xx : \dob{\D_1}{x}$ and $\yy : \dob{\D_2}{x}$;
  \item 1-cells over $f : x \rightarrow y$ from $(\xx_1 , \yy_1)$ to $(\xx_2 , \yy_2)$
    are pairs $\ff : \dmor{\xx_1}{\xx_2}{f}$ and $\gg : \dmor{\yy_1}{\yy_2}{f}$;
  \item 2-cells over $\tau : f \Rightarrow g$ from $(\ffone , \ggone)$ to $(\fftwo , \ggtwo)$
    are pairs $\tautau : \dtwo{\ffone}{\fftwo}{\tau}$ and $\thetatheta : \dtwo{\ggone}{\ggtwo}{\tau}$.
\end{itemize}
If both $\D_1$ and $\D_2$ are univalent, then so is $\D_1 \times \D_2$.

Let $\D_1$ and $\D_2$ be displayed bicategories over $\B$ and $\total{\D_1}$ respectively.
We define the \conceptDef{$\Sigma$-displayed bicategory}{Bicategories.DisplayedBicats.Examples.Sigma}{sigma_bicat} $\sigmabicat{\D_1}{\D_2}$ as the displayed bicategory over $\B$ whose
\begin{itemize}
  \item objects over $x : \B$ are pairs of $\xx : \dob{\D_1}{x}$ and $\yy : \dobP{\D_2}{(x , \xx)}$;
  \item 1-cells over $f : x \rightarrow y$ from $(\xx_1 , \yy_1)$ to $(\xx_2 , \yy_2)$
    are pairs $\ff : \dmor{\xx_1}{\xx_2}{f}$ and $\gg : \dmor{\yy_1}{\yy_2}{(f, \ff)}$;
  \item 2-cells over $\tau : f \Rightarrow g$ from $(\ffone , \ggone)$ to $(\fftwo , \ggtwo)$
    are pairs $\tautau : \dtwo{\ffone}{\fftwo}{\tau}$ and $\thetatheta : \dtwo{\ggone}{\ggtwo}{(\tau, \tautau)}$.
\end{itemize}
If we have univalent displayed bicategories $\D_1$ and $\D_2$ that are both locally preordered and groupoidal,
then $\sigmabicat{\D_1}{\D_2}$ is univalent as well.
\end{exa}

\subsection{The Construction}
Now we construct the bicategory $\FullCompCat$ using displayed bicategories.
The idea is to build up this bicategory step by step.
Starting with univalent categories,
we layerwise add structures and properties to obtain full univalent comprehension categories.
Each displayed bicategory represents one piece of the necessary data,
and we use different layers to indicate the dependencies between the displayed bicategories.

More specifically, we start with the bicategory $\UnivCat$ of univalent categories.
On top of this bicategory, we define two displayed bicategories:
$\UnivCatT$, which adds a terminal object,
and $\UnivDispCat$, which adds a univalent displayed category (\cref{def:disp-bicat-terminal}).
We take their product, which we denote by $\UnivDispCatT$,
and we continue by defining the displayed bicategory $\CleavingT$ over $\total{\UnivDispCatT}$ (\cref{def:disp-bicat-cleaving}).
Here we select out those displayed categories that are equipped with a cleaving.

Objects of $\total{\CleavingT}$ consist of a univalent category $\C$ with a terminal object
and a univalent displayed category $\D$ over $\C$ with a cleaving.
On top of $\total{\CleavingT}$ we define yet another displayed bicategory, namely $\CompCat$ (\cref{def:disp-bicat-comp-cat}).
This bicategory corresponds to the one defined by Coraglia and Emmenegger~\cite[Definition 7]{coraglia:2024}
except for the fact that we require the involved categories to be univalent.
We finish the construction by taking a subbicategory to obtain $\FullCompCat$ (\cref{def:disp-bicat-full-comp-cat}).
A pictorial overview of the construction is found in \Cref{fig:construction}.
That diagram shows the total bicategories of each displayed bicategory used to construct $\FullCompCat$
and the pseudofunctors between them.

\begin{figure}
\[
\begin{tikzcd}
  & \total{\DFullCompCat} \\
  & \total{\CompCat} \\
  & \total{\CleavingT} \\
  \total{\UnivDispCat} & \total{\UnivDispCatT} & \total{\UnivCatT} \\
  & \UnivCat
  \arrow[from=1-2, to=2-2]
  \arrow[from=2-2, to=3-2]
  \arrow[from=3-2, to=4-2]
  \arrow[from=4-1, to=5-2]
  \arrow[from=4-2, to=4-1]
  \arrow[from=4-2, to=4-3]
  \arrow[from=4-3, to=5-2]
\end{tikzcd}
\]
\caption{The displayed bicategories used to construct $\FullCompCat$}
\label{fig:construction}
\end{figure}

We start by defining $\UnivCatT$ and $\UnivDispCatT$.

\begin{defi}
\label[defi]{def:disp-bicat-terminal}
We define the displayed bicategory \conceptDef{$\UnivCatT$}{Bicategories.Core.Examples.StructuredCategories}{disp_bicat_terminal_obj} over $\UnivCat$ as $\subbicatD{\Pred}{\Qred}$
where $\Pred$ expresses that a univalent category has a terminal object
and where $\Qred$ expresses that a functor preserves terminal objects.

We also define the displayed bicategory \conceptDef{$\UnivDispCat$}{Bicategories.DisplayedBicats.Examples.DispBicatOfDispCats}{disp_bicat_of_univ_disp_cats} over $\UnivCat$ as follows.
\begin{itemize}
  \item The objects over univalent categories $\C$ are univalent displayed categories $\D$ over $\C$.
  \item The 1-cells over functors $\F : \C_1 \rightarrow \C_2$
    from a displayed category $\D_1$ over $\C_1$
    to a displayed category $\D_2$ over $\C_2$
    are displayed functors $\FF$ from $\D_1$ to $\D_2$ over $\F$.
  \item The 2-cells over a natural transformation $\nt : \F \Rightarrow \G$
    from a displayed functor $\FF : \D_1 \rightarrow \D_2$ over $\F$
    to a displayed functor $\GG : \D_1 \rightarrow \D_2$ over $\G$
    are displayed natural transformations from $\FF$ to $\GG$ over $\nt$.
\end{itemize}
We define \conceptDef{$\UnivDispCatT$}{Bicategories.ComprehensionCat.BicatOfCompCat.DispCatTerminal}{disp_bicat_cat_with_terminal_disp_cat} to be $\UnivCatT \times \UnivDispCat$.
\end{defi}

Explicitly, objects in $\UnivCatT$ over $\C$ are terminal objects,
and 1-cells over $\F$ are proofs that $\F$ preserves terminal objects.
The type of 2-cells over a natural transformation $\nt$ is the unit type.
The objects of the total bicategory $\total{\UnivDispCatT}$ consist of a univalent category $\C$ with a terminal object $\ECtx$
and a displayed univalent category $\D$ over $\C$.
Morphisms from $(\C_1, \D_1)$ to $(\C_2, \D_2)$ consist of a functor $\F : \C_1 \rightarrow \C_2$ that preserves terminal objects
and a displayed functor $\FF : \D_1 \rightarrow \D_2$ over $\F$.
For each such morphism we have an isomorphism $\FTerm{\F} : \iso{\F(\ECtx_1)}{\ECtx_2}$
where $\ECtx_1$ and $\ECtx_2$ are the terminal objects of $\C_1$ and $\C_2$ respectively.
Finally, a 2-cell from $(\F, \FF)$ to $(\G, \GG)$ consists of a natural transformation $\nt : \F \Rightarrow \G$
and a displayed natural transformation $\ntnt : \FF \Rightarrow \GG$ over $\nt$.

Next we add a cleaving to the structure.

\begin{defi}
\label[defi]{def:disp-bicat-cleaving}
We define the displayed bicategory \conceptDef{$\CleavingT$}{Bicategories.ComprehensionCat.BicatOfCompCat.FibTerminal}{disp_bicat_cat_with_terminal_cleaving} over $\total{\UnivDispCatT}$ as $\subbicatD{\Pred}{\Qred}$
where $\Pred$ expresses that a displayed category is equipped with a cleaving
and where $\Qred$ expresses that a displayed functor is Cartesian.
\end{defi}

The objects of $\total{\CleavingT}$ consist of a univalent category $\C$ with a terminal object
and a displayed univalent category $\D$ over $\C$ together with a cleaving.
Morphisms from $(\C_1, \D_1)$ to $(\C_2, \D_2)$ consist of a functor $\F : \C_1 \rightarrow \C_2$ that preserves terminal objects
and a Cartesian displayed functor $\FF : \D_1 \rightarrow \D_2$ over $\F$.
The 2-cells in $\total{\CleavingT}$ are the same as 2-cells in $\total{\UnivDispCatT}$.

Now we define the bicategory of comprehension categories.

\begin{defi}
\label[defi]{def:disp-bicat-comp-cat}
We define the displayed bicategory \conceptDef{$\CompCat$}{Bicategories.ComprehensionCat.BicatOfCompCat.CompCat}{disp_bicat_comp_cat} over $\total{\CleavingT}$ as follows.
\begin{itemize}
  \item The objects over $(\C, \D)$ are Cartesian functors $\compfunctor : \D \rightarrow \ArrD{\C}$ over the identity.
  \item The 1-cells over $(\F, \FF)$ are displayed natural transformations as follows.
    \[
      \begin{tikzcd}
        {\D_1} & {\ArrD{\C_1}} \\
        {\D_2} & {\ArrD{\C_2}}
        \arrow[""{name=0, anchor=center, inner sep=0}, "{\compfunctor_1}", from=1-1, to=1-2]
        \arrow["\FF"', from=1-1, to=2-1]
        \arrow["{\ArrD{\F}}", from=1-2, to=2-2]
        \arrow[""{name=1, anchor=center, inner sep=0}, "{\compfunctor_2}"', from=2-1, to=2-2]
        \arrow["\F_\chi", shorten <=4pt, shorten >=4pt, Rightarrow, from=0, to=1]
      \end{tikzcd}
    \]
  \item The 2-cells over $(\nt, \ntnt)$ are proofs that the following compositions are equal.
    \[\begin{tikzcd}[column sep = 50pt]
        {\D_1} & {\ArrD{\C_1}} \\
        {\D_2} & {\ArrD{\C_2}}
        \arrow["{\compfunctor_1}", from=1-1, to=1-2]
        \arrow[""{name=0, anchor=center, inner sep=0}, "\FF"', bend right=50, from=1-1, to=2-1]
        \arrow[""{name=1, anchor=center, inner sep=0}, "\GG"{description}, bend left=50, from=1-1, to=2-1]
        \arrow[""{name=2, anchor=center, inner sep=0}, "{\ArrD{\G}}", from=1-2, to=2-2]
        \arrow["{\compfunctor_2}"', from=2-1, to=2-2]
        \arrow["\tautau", shorten <=2pt, shorten >=2pt, Rightarrow, from=0, to=1, shorten >=1ex]
        \arrow["{\G_\chi}", shorten <=7pt, shorten >=7pt, Rightarrow, from=1, to=2]
      \end{tikzcd}
      \quad \quad = \quad \quad
      \begin{tikzcd}[column sep = 50pt]
        {\D_1} & {\ArrD{\C_1}} \\
        {\D_2} & {\ArrD{\C_2}}
        \arrow["{\compfunctor_1}", from=1-1, to=1-2]
        \arrow[""{name=0, anchor=center, inner sep=0}, "\FF"', from=1-1, to=2-1]
        \arrow[""{name=1, anchor=center, inner sep=0}, "{\ArrD{\G}}", bend left=50, from=1-2, to=2-2]
        \arrow[""{name=2, anchor=center, inner sep=0}, "{\ArrD{\F}}"{description}, bend right=50, from=1-2, to=2-2]
        \arrow["{\compfunctor_2}"', from=2-1, to=2-2]
        \arrow["{\F_{\chi}}", shorten <=7pt, shorten >=7pt, Rightarrow, from=0, to=2]
        \arrow["{\ArrD{\tau}}", shorten <=2pt, shorten >=2pt, Rightarrow, from=2, to=1, shorten <=1.5ex]
      \end{tikzcd}
    \]
\end{itemize}
\end{defi}

Note that $\total{\CompCat}$ corresponds to the bicategory defined by Coraglia and Emmenegger~\cite[Definition 7]{coraglia:2024}
with the only difference being that we require the involved categories to be univalent.
We finish the construction by taking a subbicategory.

\begin{defi}
\label[defi]{def:disp-bicat-full-comp-cat}
We define the displayed bicategory \conceptDef{$\DFullCompCat$}{Bicategories.ComprehensionCat.BicatOfCompCat.FullCompCat}{disp_bicat_full_comp_cat} over $\total{\CompCat}$ as $\subbicatD{\Pred}{\Qred}$
where $\Pred$ expresses that a comprehension category is full
and where $\Qred$ expresses that a lax morphism is a pseudo morphism.
We write \conceptDef{$\FullCompCat$}{Bicategories.ComprehensionCat.BicatOfCompCat.FullCompCat}{bicat_full_comp_cat} for $\total{\DFullCompCat}$.
\end{defi}

Now we have the desired bicategory $\FullCompCat$,
whose objects are full univalent comprehension categories
and whose morphisms are pseudo morphisms.
We finish this section by verifying that this bicategory is univalent.

\begin{propL}[\coqdocurl{Bicategories.ComprehensionCat.BicatOfCompCat.FullCompCat}{is_univalent_2_bicat_full_comp_cat}]
\label[propL]{prop:univ-bicat-full-comp-cat}
The bicategory $\FullCompCat$ is univalent.
\end{propL}

\begin{proof}
We repeatedly use \cref{prop:total-univalent},
so we need to verify that $\UnivCat$ and each of the displayed bicategories defined in this section are univalent.
The bicategory $\UnivCat$ is known to be univalent~\cite[Proposition 3.19]{ahrens:2021}.
To prove that $\UnivDispCat$ is univalent, we use the same steps as the proof for $\UnivCat$.
The displayed bicategory $\UnivCatT$ is univalent, since it is defined as a subbicategory.
Here we use that the type of terminal objects in a univalent category is a proposition (\Cref{prop:limits-prop}).
Since both $\UnivDispCat$ and $\UnivCatT$ are univalent, $\UnivDispCatT$ is univalent as well.
Next we note that $\CleavingT$ also is univalent, because it is defined as a subbicategory
and because the type of cleavings for a univalent displayed category is a proposition.
For a proof that the displayed bicategory $\CompCat$ is univalent, we refer the reader to the formalization.
Finally, $\DFullCompCat$ is defined as a subbicategory as well, and thus it is univalent.
All in all, we conclude that $\FullCompCat$ is univalent.
\end{proof}

\section{Basic Type Formers for Comprehension Categories}
\label{sec:type-formers}
Up to now we have constructed the univalent bicategory $\FullCompCat$
whose objects are full univalent comprehension categories.
However, such comprehension categories only model the basic judgments of Martin-L\"of type theory,
and they do not interpret any type former.
For internal language theorems, we also need to include the relevant types in our models.

Let us focus on the internal language of univalent categories with finite limits for now.
Recall that Clairambault and Dybjer~\cite[Theorem 6.1]{clairambault:2014} constructed
a biequivalence between the bicategory of categories with finite limits and democratic CwFs
that support extensional identity types and $\Sigma$-types.
For comprehension categories,
we consider democracy and similar type formers,
namely unit types, binary product types, equalizer types, and $\Sigma$-types.
We call comprehension categories that support all these type formers \textbf{univalent full democratic finite limit comprehension categories},
or, for short, \textbf{DFL comprehension categories}.

Note that our choice of type formers differs from Clairambault and Dybjer.
While we explicitly require unit types and binary product types in the structure,
they derive these types using democracy and $\Sigma$-types respectively.
In addition, we use equalizer types instead of extensional identity types,
which are equivalent~\cite[Theorem 10.5.10]{jacobs:2001}.
The reason for these differences is that they slightly simplify the proofs in \Cref{sec:fin-lim}.

The main goal in this section is to construct the univalent bicategory $\DFLCompCat$ of DFL comprehension categories,
and for that we use displayed bicategories.
More specifically, we define a displayed bicategory over $\FullCompCat$ for each of the necessary type formers.
For instance,
we define a displayed bicategory $\DUnit$ over $\FullCompCat$,
whose displayed objects over $\compfunctor : \D \rightarrow \ArrD{\C}$ represent unit types for $\compfunctor$.
We also define the displayed bicategories $\DProdType$ and $\DEqualizer$,
whose displayed objects over $\compfunctor : \D \rightarrow \ArrD{\C}$ give us fiberwise binary products and equalizers.
We use fiberwise binary products and equalizers to interpret product types and extensional identity types respectively.
In addition,
we define the displayed bicategories $\SigmaType$ and $\Democracy$.
Displayed objects in $\SigmaType$ over $\compfunctor$ give us strong $\Sigma$-types for $\compfunctor$,
and objects over $\compfunctor$ in $\Democracy$ are proofs that $\compfunctor$ is democratic.
An overview of the construction can be found in \Cref{fig:construction-dfl-compcat}.
We finish this section with a sufficient condition for when a 1-cell in this bicategory is an adjoint equivalence.

\begin{figure}
\[
\begin{tikzcd}
  && \DFLCompCatD \\
  \DUnit & \DProdType & \DEqualizer & \SigmaType & \Democracy \\
  && \FullCompCat
  \arrow[from=1-3, to=2-1]
  \arrow[from=1-3, to=2-2]
  \arrow[from=1-3, to=2-3]
  \arrow[from=1-3, to=2-4]
  \arrow[from=1-3, to=2-5]
  \arrow[from=2-1, to=3-3]
  \arrow[from=2-2, to=3-3]
  \arrow[from=2-3, to=3-3]
  \arrow[from=2-4, to=3-3]
  \arrow[from=2-5, to=3-3]
\end{tikzcd}
\]
\caption{Construction of $\DFLCompCatD$}
\label{fig:construction-dfl-compcat}
\end{figure}

We start by defining the displayed bicategories $\DProdType$ and $\DEqualizer$ directly.
To do so, we use fiberwise limits (\Cref{def:fiberwise-limits}).

\begin{defi}
\label[defi]{def:product-equalizer-type}
We define the displayed bicategory \conceptDef{$\DProdType$}{Bicategories.ComprehensionCat.TypeFormers.ProductTypes}{disp_bicat_of_prod_type} over $\FullCompCat$
as the subbicategory $\subbicatD{\Pred}{\Qred}$.
Here $\Pred$ expresses that $\compfunctor : \D \rightarrow \ArrD{\C}$ has fiberwise binary products,
and $\Qred$ expresses that, given a morphism $(\F, \FF, \F_\chi)$ from $\compfunctor_1 : \D_1 \rightarrow \ArrD{\C_1}$ to $\compfunctor_2 : \D_2 \rightarrow \ArrD{\C_2}$,
the fiber functor $\fiberfunctor{\FF}{x}$ preserves binary products for all $x : \C_1$.

We also define the displayed bicategory \conceptDef{$\DEqualizer$}{Bicategories.ComprehensionCat.TypeFormers.EqualizerTypes}{disp_bicat_of_equalizer_type} over $\FullCompCat$
as the subbicategory $\subbicatD{\Pred}{\Qred}$.
Here $\Pred$ expresses that $\compfunctor : \D \rightarrow \ArrD{\C}$ has fiberwise equalizers,
and $\Qred$ expresses that, given a morphism $(\F, \FF, \F_\chi)$ from $\compfunctor_1 : \D_1 \rightarrow \ArrD{\C_1}$ to $\compfunctor_2 : \D_2 \rightarrow \ArrD{\C_2}$,
the functor $\fiberfunctor{\FF}{x}$ preserves equalizers for all $x : \C_1$.
\end{defi}

The comprehension category $\id : \ArrD{\C} \rightarrow \ArrD{\C}$ from \Cref{exa:arrow-comp-cat} supports both binary product types and equalizer types.
The fiber of $\ArrD{\C}$ along $x$ is the slice category $\slice{\C}{x}$,
which inherits finite limits from $\C$.
Given a morphism $f : x_1 \rightarrow x_2$
the functor $\functorfiber{f}$ sends morphisms $g : y \rightarrow x_2$ to the pullback of $g$ along $f$.
Since this functor is right adjoint to composition,
it preserves all limits,
and in particular, binary products and equalizers.

For unit types, we require something stronger than just fiberwise terminal objects,
since we also require that the comprehension functor preserves terminal objects.
To understand why, suppose that we have a full comprehension category $\compfunctor : \D \rightarrow \ArrD{\C}$
and that $\D$ has a fiberwise terminal object $\FibTerm$.
We write $\FibTerm[\Gamma]$ for the terminal object in the fiber $\fiber{\D}{\Gamma}$.
We have two ways to interpret terms of type $A$ in context $\Gamma$ in such a comprehension category:
either as sections of $\CtxPr[A]$
or as morphisms from $\FibTerm[\Gamma]$ to $A$ in $\D$ over the identity.
If we require that the comprehension category is full and that $\compfunctor$ preserves terminal objects,
then these two interpretations agree.
Note that $\compfunctor$ maps a terminal object $\FibTerm[\Gamma] : \dob{\D}{\Gamma}$ to a terminal object
if and only if $\CtxPr[{\FibTerm[\Gamma]}]$ is an isomorphism.

\begin{defi}
\label[defi]{def:unit-type}
We define the displayed bicategory \conceptDef{$\DUnit$}{Bicategories.ComprehensionCat.TypeFormers.UnitTypes}{disp_bicat_of_strong_unit_type} over $\FullCompCat$
as the subbicategory $\subbicatD{\Pred}{\Qred}$.
Here $\Pred$ expresses that $\compfunctor : \D \rightarrow \ArrD{\C}$ has fiberwise terminal objects
and that the morphism $\CtxPr[{\FibTerm[\Gamma]}]$ is an isomorphism for every $\Gamma : \C$,
and $\Qred$ expresses that, given a morphism $(\F, \FF, \F_\chi)$ from $\compfunctor_1 : \D_1 \rightarrow \ArrD{\C_1}$ to $\compfunctor_2 : \D_2 \rightarrow \ArrD{\C_2}$,
the functor $\fiberfunctor{\FF}{x}$ preserves terminal objects for all $x : \C_1$.
\end{defi}

If $\C$ is a category with finite limits,
then $\id : \ArrD{\C} \rightarrow \ArrD{\C}$ supports unit types.
The slice category $\slice{\C}{x}$ always has a terminal object,
namely the identity $\id : x \rightarrow x$.
We denote this terminal object by $\FibTerm[x]$.
In fact, every isomorphism $f : y \rightarrow x$ is a terminal object in $\slice{\C}{x}$.
The functor $\functorfiber{f}$ preserves terminal objects,
since the pullback of an isomorphism is again an isomorphism.
Finally, the morphism $\CtxPr[{\FibTerm[x]}]$ is an isomorphism,
because it is the identity.

Next we look at $\Sigma$-types, and these are interpreted as left adjoints of weakening~\cite{lawvere:2006}.
In addition, we require the Beck-Chevalley condition for left adjoints (\Cref{def:beck-chevalley}) to guarantee stability under substitution.
Since we also want to have first and second projections for our $\Sigma$-types,
we look at \textbf{strong $\Sigma$-types}~\cite[Definition 10.5.2]{jacobs:2001}.

\begin{defi}
\label[defi]{def:sigma-type}
A full univalent comprehension category $\compfunctor : \D \rightarrow \ArrD{\C}$ has \conceptDef{$\Sigma$-types}{Bicategories.ComprehensionCat.TypeFormers.SigmaTypes}{comp_cat_dependent_sum} if
\begin{itemize}
  \item each functor $\functorfiber{\CtxPr[A]} : \fiber{\D}{\Gamma} \rightarrow \fiber{\D}{\CtxExt{\Gamma}{A}}$ has a left adjoint $\DepSum{A}$ with unit $\eta$ and counit $\varepsilon$
    where $A : \dob{\D}{\Gamma}$;
  \item the Beck--Chevalley condition for left adjoints holds for every pullback square.
\end{itemize}
In addition, we say that $\compfunctor : \D \rightarrow \ArrD{\C}$ has \conceptDef{strong $\Sigma$-types}{Bicategories.ComprehensionCat.TypeFormers.SigmaTypes}{strong_dependent_sums} if
for all contexts $\Gamma : \C$ and types $A : \dob{\D}{\Gamma}$ and $B : \dob{\D}{\CtxExt{\Gamma}{A}}$
the following map is an isomorphism.
\[
\begin{tikzcd}[column sep = 45pt]
  {\CtxExt{\CtxExt{\Gamma}{A}}{B}} & {\CtxExt{\CtxExt{\Gamma}{A}}{\substTy{\CtxPr[A]}{\DepSumType{A}{B}}}} & {\CtxExt{\Gamma}{\DepSumType{A}{B}}}
  \arrow["{\CtxExt{\CtxExt{\Gamma}{A}}{\eta(B)}}", from=1-1, to=1-2]
  \arrow["{\compfunctor(\substTyMor{CtxPr[A]})}", from=1-2, to=1-3]
\end{tikzcd}
\]

Finally, we define the displayed bicategory \conceptDef{$\SigmaType$}{Bicategories.ComprehensionCat.TypeFormers.SigmaTypes}{disp_bicat_of_sigma_type} over $\FullCompCat$ as $\fullsubbicatD{\Pred}$
where $\Pred$ says that a full univalent comprehension category has strong $\Sigma$-types.
\end{defi}

Note that in \Cref{def:sigma-type} we do not require that the morphisms preserve $\Sigma$-types.
This is automatic since we have an isomorphism $\iso{\CtxExt{\CtxExt{\Gamma}{A}}{B}}{\CtxExt{\Gamma}{\DepSumType{A}{B}}}$~\cite[Proposition 3.5]{clairambault:2014}.
In addition, the type expressing that a full univalent comprehension category $\compfunctor : \D \rightarrow \ArrD{\C}$ has strong $\Sigma$-types, is a proposition.
This is because the type of left adjoints for a functor is a proposition~\cite[Lemma 5.2]{ahrens:2015}.

The comprehension category $\id : \ArrD{\C} \rightarrow \ArrD{\C}$ supports strong $\Sigma$-types.
Since the functor $\functorfiber{f}$ maps morphisms to their pullback along $f$,
it is right adjoint to the composition functor.
Concretely,
if we have $f : y \rightarrow x$,
then the functor $\DepSum{f}$ maps $g : z \rightarrow y$ to their composition $g \cdot f : z \rightarrow x$.
One can show that the Beck--Chevalley condition is satisfied for each pullback square~\cite{seely:1983,seely:1984}.
Note that the resulting $\Sigma$-types are strong,
because the morphism $\CtxExt{\CtxExt{\Gamma}{A}}{B} \rightarrow \CtxExt{\Gamma}{\DepSumType{A}{B}}$
is equal to the identity,
and thus an isomorphism.

Finally, we consider democracy for comprehension categories~\cite[Definition 2.6]{clairambault:2014}.
Roughly speaking, democracy for comprehension categories expresses that the contexts are generated by the types,
and one can intuitively capture this idea with the following phrase:
a comprehension category is democratic if every context has a representative.
We make this idea formal in the following definition.

\begin{defi}
\label[defi]{def:democracy-disp-bicat}
A full univalent comprehension category $\compfunctor : \D \rightarrow \ArrD{\C}$
is called \conceptDef{democratic}{Bicategories.ComprehensionCat.TypeFormers.Democracy}{is_democratic}
if for all contexts $\Gamma : \C$
there is a type $\DemTy{\Gamma} : \dobP{\D}{\ECtx}$
and an isomorphism $\DemIso{\Gamma} : \iso{\Gamma}{\CtxExt{\ECtx}{\DemTy{\Gamma}}}$.
We define the displayed bicategory \conceptDef{$\Democracy$}{Bicategories.ComprehensionCat.TypeFormers.Democracy}{disp_bicat_of_democracy} over $\FullCompCat$ as $\fullsubbicatD{\Pred}$
where $\Pred$ says that a full univalent comprehension category is democratic.
\end{defi}

Note that the comprehension category $\id : \ArrD{\C} \rightarrow \ArrD{\C}$ is democratic,
because every $x : \C$ is represented by the morphism $\ECtxMap : x \rightarrow \ECtx$.
Note that an equivalent way to formulate democracy is by saying that $\fiberfunctor{\compfunctor}{\ECtx}$ is split essentially surjective.
Let us make a couple of remarks on this definition.
First, the displayed bicategory $\Democracy$ is univalent.
This is because being democratic is a proposition.
Note that we use that $\D$ is univalent in the proof.

\begin{propL}[\coqdocurl{Bicategories.ComprehensionCat.TypeFormers.Democracy}{isaprop_is_democratic}]
\label[propL]{prop:isaprop-democratic}
The type expressing that a full univalent comprehension category $\compfunctor : \D \rightarrow \ArrD{\C}$
is democratic,
is a proposition.
\end{propL}

\begin{proof}
Suppose that we have a context $\Gamma : \C$, types $A_1, A_2 : \dobP{\D}{\ECtx}$,
and isomorphisms $\gamma_1 : \iso{\Gamma}{\CtxExt{\ECtx}{A_1}}$
and $\gamma_2 : \iso{\Gamma}{\CtxExt{\ECtx}{A_2}}$.
We must show that the pairs $(A_1, \gamma_1)$ and $(A_2, \gamma_2)$ are equal.
Since $\D$ is univalent,
it suffices to construct an isomorphism $i$ that makes the following triangle commute.
\[
\begin{tikzcd}
  {\CtxExt{\ECtx}{A_1}} && {\CtxExt{\ECtx}{A_2}} \\
  & \Gamma
  \arrow["{\CtxExt{\ECtx}{i}}", from=1-1, to=1-3]
  \arrow["{\gamma_1}", from=2-2, to=1-1]
  \arrow["{\gamma_2}"', from=2-2, to=1-3]
\end{tikzcd}
\]
Note that we have the following isomorphism:
\[
\begin{tikzcd}
  {\CtxExt{\ECtx}{A_1}} & \Gamma & {\CtxExt{\ECtx}{A_2}}
  \arrow["{\gamma_1^{-1}}", from=1-1, to=1-2]
  \arrow["{\gamma_2}", from=1-2, to=1-3]
\end{tikzcd}
\]
Since $\compfunctor$ is fully faithful,
we get the desired isomorphism $i : \iso{A_1}{A_2}$.
\end{proof}

Second, we do not require the morphisms to preserve democracy,
whereas Clairambault and Dybjer~\cite[Definition 3.6]{clairambault:2014} add such a requirement.
This requirement is in fact redundant, as witnessed by the following proposition.

\begin{propL}[\coqdocurl{Bicategories.ComprehensionCat.TypeFormers.Democracy}{all_functors_democratic}]
\label[propL]{prop:preservation-of-democracy}
Suppose that we have democratic full comprehension categories $\compfunctor_1 : \D_1 \rightarrow \ArrD{\C_1}$ and $\compfunctor_2 : \D_2 \rightarrow \ArrD{\C_2}$,
and a morphism $(\F, \FF, \F_\chi)$ from $\compfunctor_1$ to $\compfunctor_2$.
Then
for each $\Gamma : \C_1$
there is an isomorphism $d_{\Gamma} : \iso{\F(\DemTy{\Gamma})}{\substTy{\ECtxMap[\F(\ECtx)]}{\DemTy{\F(\Gamma)}}}$
making the following diagram commute.
\[
\begin{tikzcd}[column sep = 50pt]
  {\F(\Gamma)}
  &
  & {\F(\CtxExt{\ECtx}{\DemTy{\Gamma}})} \\
  {\CtxExt{\ECtx}{\DemTy{\F(\Gamma)}}}
  & {\CtxExt{\F \> \ECtx}{\substTy{\ECtxMap[\F(\ECtx)]}{\DemTy{\F(\Gamma)}}}}
  & {\CtxExt{\F \> \ECtx}{\F(\DemTy{\Gamma})}}
  \arrow["{\F(\DemIso{\Gamma})}", from=1-1, to=1-3]
  \arrow["{\DemIso{\F(\Gamma)}}"', from=1-1, to=2-1]
  \arrow["{\F_\chi(\DemTy{\Gamma})}", from=1-3, to=2-3]
  \arrow["{\PairSub{\ECtxMap[\F \> \ECtx]}{\var}}", from=2-2, to=2-1]
  \arrow["{\CtxExt{\F \> \ECtx}{d_{\Gamma}}}", from=2-3, to=2-2]
\end{tikzcd}
\]
\end{propL}

\begin{proof}
Note that the maps $\DemIso{\F(\Gamma)}$ and $\F(\DemIso{\Gamma})$ are isomorphisms,
since functors preserve isomorphisms and by definition of democracy.
In addition, $\F_\chi(\DemTy{\Gamma})$ is an isomorphism, because we look at pseudo morphisms.
Finally, one can also show that $\PairSub{\ECtxMap[\F \> \ECtx]}{\var}$ is an isomorphism using that $\F$ preserves terminal objects.
Since $\compfunctor$ is fully faithful, there must be a unique isomorphism $d_{\Gamma}$ making the diagram above commute.
\end{proof}

Now we have all the necessary ingredients to define the bicategory $\DFLCompCat$
and to prove that it is univalent.

\begin{defi}
\label[defi]{def:dfl-comp-cat}
We define the displayed bicategory \conceptDef{$\DFLCompCatD$}{Bicategories.ComprehensionCat.DFLCompCat}{disp_bicat_of_dfl_full_comp_cat} over $\FullCompCat$ to be
\[
\DUnit \times \DProdType \times \DEqualizer \times \SigmaType \times \Democracy.
\]
We write \conceptDef{$\DFLCompCat$}{Bicategories.ComprehensionCat.DFLCompCat}{bicat_of_dfl_full_comp_cat} for the total bicategory $\total{\DFLCompCatD}$.
\end{defi}

\begin{propL}[\coqdocurl{Bicategories.ComprehensionCat.DFLCompCat}{is_univalent_2_bicat_of_dfl_full_comp_cat}]
\label[propL]{prop:univ-dfl-compcat}
The bicategory $\DFLCompCat$ is univalent.
\end{propL}

\begin{proof}
It suffices to show that each of the displayed bicategories defined in this section is univalent.
Using \Cref{prop:fib-limits-prop}, we get that $\DUnit$, $\DProdType$, and $\DEqualizer$ are univalent.
The displayed bicategory $\SigmaType$ also is univalent,
because the type expressing that a full univalent comprehension category has dependent sums,
is a proposition.
Since we define $\Democracy$ as a full subbicategory, we get from \Cref{prop:isaprop-democratic} that $\Democracy$ is univalent.
\end{proof}

Before continuing, we briefly describe the objects, morphisms, and 2-cells in $\DFLCompCat$.
Objects in $\DFLCompCat$ are democratic full univalent comprehension categories $\compfunctor : \D \rightarrow \ArrD{\C}$
that support unit types, binary product types, equalizer types, and $\Sigma$-types.
We call such comprehension categories DFL comprehension categories for short.
We write $\FibTerm[\Gamma]$ for the terminal object in $\dob{\D}{\Gamma}$,
and we write $A \times B$ for the binary product of two types $A, B : \dob{\D}{\Gamma}$.
The $\Sigma$-type of $A : \dob{\D}{\Gamma}$ and $B : \dob{\D}{\CtxExt{\Gamma}{A}}$ is denoted by $\DepSumType{A}{B}$,
and, given a context $\Gamma : \C$, we write $\DemTy{\Gamma} : \dobP{\D}{\ECtx}$ for its representative
and $\DemIso{\Gamma} : \iso{\Gamma}{\CtxExt{\ECtx}{\DemTy{\Gamma}}}$ for the isomorphism.
Finally, we have a type $\ExtId{t_1}{t_2}$ for all terms $t_1, t_2 : \Tms{\Gamma}{A}$.
This is because $t_1$ and $t_2$ give rise to morphisms $\overline{t_1}, \overline{t_2} : \FibTerm[\Gamma] \rightarrow A$ in $\D$ over the identity,
and the type $\ExtId{t_1}{t_2}$ is defined to be the equalizer of $\overline{t_1}$ and $\overline{t_2}$.
We can show that $\ExtId{t_1}{t_2}$ is stable under substitution.

A morphism from $\compfunctor_1 : \D_1 \rightarrow \ArrD{\C_1}$ to $\compfunctor_2 : \D_2 \rightarrow \ArrD{\C_2}$
consists of a morphism $(\F, \FF, \F_\chi)$ of full univalent comprehension categories from $\compfunctor_1$ to $\compfunctor_2$
such that $\fiberfunctor{\FF}{x}$ preserves finite limits for all $x : \C_1$.
Finally, 2-cells in $\DFLCompCat$ are the same as 2-cells in $\FullCompCat$.
We can also show that the comprehension category constructed in \Cref{exa:arrow-comp-cat} gives rise to DFL comprehension category.

\begin{exa}
\label[exa]{exa:arrow-dfl-comp-cat}
Let $\C$ be a univalent category with finite limits.
The comprehension category $\id : \ArrD{\C} \rightarrow \ArrD{\C}$ is univalent,
because $\C$ is univalent,
and it is full,
because $\id$ is fully faithful.
Because the slice category $\slice{\C}{x}$ inherits its limits from $\C$
and because the functors $\functorfiber{f}$ are right adjoints,
this comprehension category has fiberwise terminal objects, binary products, and equalizers.
We can also show that $\id : \ArrD{\C} \rightarrow \ArrD{\C}$ supports strong $\Sigma$-types,
since the left adjoint of $\functorfiber{f}$ is composition,
and it is democratic,
because every $x : \C$ is represented by $\ECtxMap : x \rightarrow \ECtx$.
Hence, $\id : \ArrD{\C} \rightarrow \ArrD{\C}$ is a DFL comprehension category.
\end{exa}

We end this section by looking at adjoint equivalences in $\DFLCompCat$.
More specifically, we show that a 1-cell $(\F, \FF, \F_\chi)$ is an adjoint equivalence
if $\F$ and $\FF$ are.
We use this condition in \Cref{sec:fin-lim} to establish the biequivalence between the bicategory of univalent categories with finite limits and $\DFLCompCat$.

\begin{propL}[\coqdocurl{Bicategories.ComprehensionCat.DFLCompCat}{dfl_full_comp_cat_left_adjoint_equivalence}]
\label[propL]{prop:dfl-compcat-adjequiv}
A 1-cell $(\F, \FF, \F_\chi)$ in $\DFLCompCat$ from $\compfunctor_1 : \D_1 \rightarrow \ArrD{\C_1}$ to $\compfunctor_2 : \D_2 \rightarrow \ArrD{\C_2}$
is an adjoint equivalence
if $\F$ is an adjoint equivalence of categories
and $\FF$ is a displayed adjoint equivalence over $\F$.
\end{propL}

\begin{proof}
To prove this proposition,
we use displayed adjoint equivalences to construct adjoint equivalences in total bicategories (\Cref{def:disp-adjequiv}).
We start by noting that adjoint equivalences in $\UnivCat$ are the same as adjoint equivalences of categories,
and that displayed adjoint equivalences in $\UnivDispCat$ are the same as displayed adjoint equivalences of displayed categories.
Next we note that there is a displayed adjoint equivalence in $\UnivCatT$ over each adjoint equivalence in $\UnivCat$,
since equivalences preserve terminal objects.

Now we look at displayed adjoint equivalences in $\CleavingT$.
More specifically, we construct a displayed adjoint equivalence in $\CleavingT$ over each adjoint equivalence $(\F, \FF)$ in $\total{\UnivDispCatT}$.
Here we use \Cref{prop:equivalence-induction},
and thus it suffices to only consider identity adjoint equivalences in $\total{\UnivDispCatT}$.
We can directly construct a displayed adjoint equivalence of the identity,
since the identity functor is Cartesian.

Using the same idea, we can show that a morphism $(\F, \FF, \F_\chi)$ in $\CompCat$
is an adjoint equivalence if $\F$ and $\FF$ are adjoint equivalences
and if $\F_\chi$ is a natural isomorphism.
As a consequence, a 1-cell $(\F, \FF, \F_\chi)$ in $\DFullCompCat$ is an adjoint equivalence
if $\F$ and $\FF$ are,
since $\F_\chi$ is a natural isomorphism by the definition of $\DFullCompCat$.

Finally, since $\SigmaType$ and $\Democracy$ are defined as full subbicategories,
there is a displayed adjoint equivalence in those displayed bicategories over every adjoint equivalence in $\DFullCompCat$.
For $\DUnit$, $\DProdType$, and $\DEqualizer$, we use the same idea as for $\CleavingT$.
From this, we conclude that a 1-cell $(\F, \FF, \F_\chi)$ in $\DFLCompCat$ is an adjoint equivalence
if $\F$ is an adjoint equivalence
and $\FF$ is a displayed adjoint equivalence over $\F$.
\end{proof}

\section{The Internal Language of Categories with Finite Limits}
\label{sec:fin-lim}
Now we have everything in place to prove our analogue of the first internal language theorem by Clairambault and Dybjer~\cite[Theorem 6.1]{clairambault:2014}.
More specifically, our goal in this section is to show that $\DFLCompCat$ is biequivalent to the bicategory $\FinLim$ of univalent categories with finite limits.
This theorem thus says that the internal language of univalent categories with finite limits is
extensional Martin-L\"of type theory with unit types, binary product types, and $\Sigma$-types.
Note that we already showed in \Cref{exa:arrow-dfl-comp-cat} how categories with finite limits give rise to DFL comprehension categories.

Let us start this section by recalling the definitions of $\FinLim$ and of biequivalences.
Recall that a category has finite limits if it has a terminal object and pullbacks
and that a functor preserves finite limits if it preserves terminal objects and pullbacks.

\begin{defi}
\label[defi]{def:bicat-fin-lim}
We define the bicategory \conceptDef{$\FinLim$}{Bicategories.Core.Examples.StructuredCategories}{bicat_of_univ_cat_with_finlim} as the subbicategory $\subbicat{\Pred}{\Qred}$ of $\UnivCat$
where $\Pred$ says that a univalent category has finite limits
and $\Qred$ says that a functor preserves finite limits.
\end{defi}

\begin{defiC}[{\cite[Definition 6.2.8]{johnson:2021}}]
\label[defi]{def:biequiv}
A \conceptDef{biequivalence}{Bicategories.PseudoFunctors.Biequivalence}{is_biequivalence} from a bicategory $\B_1$ to a bicategory $\B_2$
consists of pseudofunctors $\F : \B_1 \rightarrow \B_2$ and $\G : \B_2 \rightarrow \B_1$
and pseudotransformations $\eta : \id[\B_2] \Rightarrow \G \cdot \F$ and $\varepsilon : \F \cdot \G \Rightarrow \id[\B_1]$
such that $\eta$ and $\varepsilon$ are adjoint equivalences in the bicategory of pseudofunctors.
The type of biequivalences from $\B_1$ to $\B_2$ is denoted by $\B_1 \simeq \B_2$.
\end{defiC}

Since $\eta$ and $\varepsilon$ are adjoint equivalences,
we have
pseudotransformations $\eta^{-1}$ and $\varepsilon^{-1}$
and four invertible modifications
$m_1 : \modif{\eta \cdot \eta^{-1}}{\id,}$
$m_2 : \modif{\eta^{-1} \cdot \eta}{\id,}$
$m_3 : \modif{\varepsilon \cdot \varepsilon^{-1}}{\id,}$
and $m_4 : \modif{\varepsilon^{-1} \cdot \varepsilon}{\id.}$
To construct a pseudotransformation $\zeta : \F \Rightarrow \G$
that is an adjoint equivalence,
we need to give another pseudotransformation $\xi : \G \Rightarrow \F$
and invertible modifications $m_1 : \modif{\zeta \cdot \xi}{\id[\F]}$ and $m_2 : \modif{\xi \cdot \zeta}{\id[\G].}$
Note that it is not necessary to verify the triangle coherences,
because equivalences can be refined into adjoint equivalences.
However, we have a simpler way to verify that a pseudotransformation is an adjoint equivalence.
Instead of constructing the aforementioned transformations and modifications,
it suffices to verify that $\zeta$ is a pointwise adjoint equivalence~\cite[Proposition 6.2.16]{johnson:2021}.

\begin{propL}[\coqdocurl{Bicategories.Transformations.PseudoTransformation}{pointwise_adjequiv_to_adjequiv}]
\label[propL]{prop:pointwise-adj-equiv}
Let $\B_1$ and $\B_2$ be bicategories such that $\B_2$ is univalent,
and let $\F, \G : \B_1 \rightarrow \B_2$ be pseudofunctors.
A pseudotransformation $\zeta : \F \Rightarrow \G$
is an adjoint equivalence in the bicategory of pseudofunctors
if each $\zeta(x)$ is an adjoint equivalence.
\end{propL}

\begin{proof}[Proof sketch]
We prove this the same way as \Cref{prop:dfl-compcat-adjequiv},
and we need that $\B_2$ is univalent so that we can use \Cref{prop:equivalence-induction}.
Here we use the construction of the bicategory of pseudofunctors via displayed bicategories~\cite[Definition 9.6]{ahrens:2021}.
\end{proof}

Now we set out to construct a biequivalence between $\DFLCompCat$ and $\FinLim$.
To do so, we use \Cref{prop:pointwise-adj-equiv}, and it suffices to construct
\begin{itemize}
  \item a pseudofunctor $\FinLimToCompCat : \FinLim \rightarrow \DFLCompCat$ (\Cref{constr:finlim-to-compcat});
  \item a pseudofunctor $\CompCatToFinLim : \DFLCompCat \rightarrow \FinLim$ (\Cref{constr:compcat-to-finlim});
  \item a pseudotransformation $\FinLimCompCatUnit : \FinLimToCompCat \cdot \CompCatToFinLim \Rightarrow \id[\FinLim]$ (\Cref{constr:unit});
  \item a pseudotransformation $\FinLimCompCatCounit : \id[\DFLCompCat] \Rightarrow \CompCatToFinLim \cdot \FinLimToCompCat$ (\Cref{constr:counit}).
\end{itemize}
We must also show that $\FinLimCompCatUnit$ and $\FinLimCompCatCounit$ are pointwise adjoint equivalences (\Cref{prop:unit-pointwise-adjequiv,prop:counit-pointwise-adjequiv}).
Finally, we construct the desired biequivalence in \Cref{constr:biequiv}.

\subsection{From Categories to Comprehension Categories}
We start by constructing $\FinLimToCompCat : \FinLim \rightarrow \DFLCompCat$.
This pseudofunctor sends every univalent category $\C$ with finite limits to the following comprehension category.
\[
\begin{tikzcd}
  {\ArrD{\C}} && {\ArrD{\C}} \\
  & \C
  \arrow["\id", from=1-1, to=1-3]
\end{tikzcd}
\]
Note that this comprehension category is full, because the identity functor is fully faithful,
and it is univalent, because $\C$ is so.

As discussed in \Cref{exa:arrow-dfl-comp-cat},
the comprehension category $\id : \ArrD{\C} \rightarrow \ArrD{\C}$ has the desired type formers.
Since $\C$ is finitely complete, the displayed category $\ArrD{\C}$ has all fiberwise finite limits.
In addition, we can equip $\ArrD{\C}$ with a cleaving for the same reason.
The comprehension category $\id : \ArrD{\C} \rightarrow \ArrD{\C}$ also is democratic,
because every context $\Gamma$ is represented by the type $\ECtxMap[\Gamma] : \Gamma \rightarrow \ECtx$.
Finally, the substitution functor $\functorfiber{s}$ has a left adjoint for each morphism $s : \Gamma \rightarrow \Delta$.
This left adjoint is given by composition, and it satisfies the Beck--Chevalley condition~\cite{seely:1983,seely:1984}.

\begin{problem}
\label[problem]{prob:finlim-to-compcat}
To construct a pseudofunctor $\FinLimToCompCat : \FinLim \rightarrow \DFLCompCat$.
\end{problem}

\begin{construction}{\coqdocurl{Bicategories.ComprehensionCat.Biequivalence.FinLimToDFLCompCat}{finlim_to_dfl_comp_cat_psfunctor}}{prob:finlim-to-compcat}
\label{constr:finlim-to-compcat}
We already described the action of $\FinLimToCompCat$ on objects.
Next we describe the action on 1-cells.
Given a functor $\F : \C_1 \rightarrow \C_2$,
we get the following morphism of comprehension categories.
\[
\begin{tikzcd}
  {\ArrD{\C_1}} & {\ArrD{\C_1}} \\
  {\ArrD{\C_2}} & {\ArrD{\C_2}}
  \arrow[""{name=0, anchor=center, inner sep=0}, "\id", from=1-1, to=1-2]
  \arrow["{\ArrD{\F}}"', from=1-1, to=2-1]
  \arrow["{\ArrD{\F}}", from=1-2, to=2-2]
  \arrow[""{name=1, anchor=center, inner sep=0}, "\id"', from=2-1, to=2-2]
  \arrow["\id", shorten <=4pt, shorten >=4pt, Rightarrow, from=0, to=1]
\end{tikzcd}
\]
The functor $\ArrD{\F}$ preserves fiberwise finite limits, because $\F$ preserves finite limits.
Finally, every natural transformation $\nt : \F \Rightarrow \G$ gives rise to the 2-cell $(\nt, \ArrD{\nt})$.
\end{construction}

The difference between our construction of $\FinLimToCompCat$ and the one by Clairambault and Dybjer~\cite[Proposition 5.2]{clairambault:2014} lies in the fact
that we use the arrow category of $\C$ while they need to use the right adjoint splitting~\cite{benabou:1985,hofmann:1994}.
To understand this difference, we reformulate their construction using displayed categories.
Let $\slice{\C}{\Gamma}$ be the displayed category over $\C$ whose objects over $\Delta$ are context morphisms $\Delta \rightarrow \Gamma$.
The construction by Clairambault and Dybjer uses the displayed category $\D$ over $\C$ whose objects over $\Gamma : \C$
are displayed functors $\FF : \slice{\C}{\Gamma} \rightarrow \ArrD{\C}$ over the identity.
If we work in set-theoretic foundations,
then this displayed category is a split fibration over $\C$ equivalent to $\ArrD{\C}$.
\[
\begin{tikzcd}
  \D & {\ArrD{\C}} & {\ArrD{\C}} \\
  & \C
  \arrow["\simeq", from=1-1, to=1-2]
  \arrow["\id", from=1-2, to=1-3]
\end{tikzcd}
\]
In univalent foundations,
we need to assume that the objects of $\C$ form a set
to guarantee that $\D$ is a split fibration~\cite{najmaei:2026a}.

\subsection{From Comprehension Categories to Categories}
Next we construct the pseudofunctor $\CompCatToFinLim : \DFLCompCat \rightarrow \FinLim$.
This pseudofunctor sends every DFL comprehension category $\compfunctor : \D \rightarrow \ArrD{\C}$ to the univalent category $\C$ of contexts~\cite[Proposition 4.1]{clairambault:2014}.
The main work thus lies in showing that $\C$ is finitely complete.

To prove that the category of contexts is finitely complete,
we show that the comprehension functor $\compfunctor : \D \rightarrow \ArrD{\C}$ of a DFL comprehension category
is an adjoint equivalence of displayed categories.
This gives us an equivalence $\adjequiv{\fiber{\D}{\Gamma}}{\slice{\C}{\Gamma}}$ of fiber categories for every context $\Gamma$,
and, in particular, we have equivalences $\fiber{\D}{\ECtx} \simeq \slice{\C}{\ECtx} \simeq \C$.
Since $\D$ has fiberwise finite limits,
each fiber category $\fiber{\D}{\Gamma}$ is finitely complete,
and thus $\C$ also has all finite limits.

To prove that $\compfunctor : \D \rightarrow \ArrD{\C}$ is an adjoint equivalence,
we show that it is both essentially surjective and fully faithful.
Note that this argument is constructively valid,
because $\D$ is univalent.
We also assumed that the comprehension category $\compfunctor : \D \rightarrow \ArrD{\C}$ is full,
meaning that the functor $\compfunctor : \D \rightarrow \ArrD{\C}$ is fully faithful.
Hence, it suffices to prove that $\compfunctor$ is essentially surjective,
and for that we follow Clairambault and Dybjer~\cite[Proposition 2.9]{clairambault:2014}.

\begin{propL}[\coqdocurl{Bicategories.ComprehensionCat.ComprehensionEso}{comprehension_eso}]
\label[propL]{prop:comprehension-eso}
Let $\compfunctor : \D \rightarrow \ArrD{\C}$ be a DFL comprehension category.
Then the functor $\compfunctor$ is essentially surjective.
\end{propL}

\begin{proof}
To prove this, we show that the fiber functor $\fiberfunctor{\compfunctor}{\Gamma}$ is essentially surjective for each $\Gamma : \C$.
Let $s : \Delta \rightarrow \Gamma$ be a context morphism.
Our goal is to construct a type $A$ such that $\CtxPr[A]$ is isomorphic to $s$ in the slice category $\slice{\C}{\Gamma}$.

We start by introducing some notation.
Since our comprehension category is democratic,
we get types $\DemTy{\Gamma} : \dobP{\D}{\ECtx}$ and $\DemTy{\Delta} : \dobP{\D}{\ECtx}$,
and isomorphisms $\DemIso{\Gamma} : \iso{\Gamma}{\CtxExt{\ECtx}{\DemTy{\Gamma}}}$
and $\DemIso{\Delta} : \iso{\Delta}{\CtxExt{\ECtx}{\DemTy{\Delta}}}$.
Let us write $\widetilde{\Delta} : \dob{\D}{\Gamma}$ and $\widetilde{\Gamma} : \dob{\D}{\CtxExt{\Gamma}{\widetilde{\Delta}}}$
for $\substTy{\ECtxMap}{\DemTy{\Delta}}$ and $\substTy{\ECtxMap}{\DemTy{\Gamma}}$ respectively.

Next we construct two terms of type $\widetilde{\Gamma}$ in context $\CtxExt{\Gamma}{\widetilde{\Delta}}$.
To do so, we construct the two morphisms from $\CtxExt{\Gamma}{\widetilde{\Delta}}$ to $\CtxExt{\CtxExt{\Gamma}{\widetilde{\Delta}}}{\widetilde{\Gamma}}$.
First, we define the map $l$ using the following diagram and the universal mapping property of the pullback.
\[
\begin{tikzcd}
  {\CtxExt{\Gamma}{\widetilde{\Delta}}} & \Gamma \\
  & {\CtxExt{\CtxExt{\Gamma}{\widetilde{\Delta}}}{\widetilde{\Gamma}}} & {\CtxExt{\ECtx}{\DemTy{\Gamma}}} \\
  & {\CtxExt{\Gamma}{\widetilde{\Delta}}} & \ECtx
  \arrow["\CtxPr", from=1-1, to=1-2]
  \arrow["l"{description}, dashed, from=1-1, to=2-2]
  \arrow["\id"', bend right=30, from=1-1, to=3-2]
  \arrow["{\DemIso{\Gamma}}", bend left=30, from=1-2, to=2-3]
  \arrow["\compfunctor(\substTyMor{\ECtxMap})", from=2-2, to=2-3]
  \arrow["\CtxPr"', from=2-2, to=3-2]
  \arrow["\CtxPr", from=2-3, to=3-3]
  \arrow["\ECtxMap"', from=3-2, to=3-3]
\end{tikzcd}
\]
Second, we also define the map $r$ using the universal mapping property of the pullback.
\[
\begin{tikzcd}
  {\CtxExt{\Gamma}{\widetilde{\Delta}}} & {\CtxExt{\ECtx}{\DemTy{\Delta}}} & \Delta & \Gamma \\
  &&& {\CtxExt{\CtxExt{\Gamma}{\widetilde{\Delta}}}{\widetilde{\Gamma}}} & {\CtxExt{\ECtx}{\DemTy{\Gamma}}} \\
  &&& {\CtxExt{\Gamma}{\widetilde{\Delta}}} & \ECtx
  \arrow["\compfunctor(\substTyMor{\ECtxMap})", from=1-1, to=1-2]
  \arrow["r"{description}, dashed, from=1-1, to=2-4]
  \arrow["\id"', bend right=30, from=1-1, to=3-4]
  \arrow["{\DemIso{\Delta}^{-1}}", from=1-2, to=1-3]
  \arrow["s", from=1-3, to=1-4]
  \arrow["{\DemIso{\Gamma}}", bend left=30, from=1-4, to=2-5]
  \arrow["\compfunctor(\substTyMor{\ECtxMap})", from=2-4, to=2-5]
  \arrow["\CtxPr"', from=2-4, to=3-4]
  \arrow["\CtxPr", from=2-5, to=3-5]
  \arrow["\ECtxMap"', from=3-4, to=3-5]
\end{tikzcd}
\]
Note that $l$ and $r$ are terms by construction.

Since extensional identity types can be derived from equalizers,
we have a type $\ExtId{l}{r}$ in context $\CtxExt{\Gamma}{\widetilde{\Delta}}$.
We define the desired type $A$ to be $\DepSumType{\widetilde{\Delta}}{(\ExtId{l}{r})}$, which is a type in context $\Gamma$.
Note that we have the following isomorphism, because we assumed $\Sigma$-types to be strong.
\[
\iso{\CtxExt{\Gamma}{\DepSumType{\widetilde{\Delta}}{(\ExtId{l}{r})}}}{\CtxExt{\CtxExt{\Gamma}{\widetilde{\Delta}}}{(\ExtId{l}{r})}}
\]

We also have a morphism $f$ from $\CtxExt{\CtxExt{\Gamma}{\widetilde{\Delta}}}{(\ExtId{l}{r})}$ to $\Delta$,
which is defined as follows.
\[
\begin{tikzcd}[column sep = large]
  {\CtxExt{\CtxExt{\Gamma}{\widetilde{\Delta}}}{(\ExtId{l}{r})}} & {\CtxExt{\Gamma}{\widetilde{\Delta}}} & {\CtxExt{\ECtx}{\DemTy{\Delta}}} & \Delta
  \arrow["\CtxPr", from=1-1, to=1-2]
  \arrow["{\compfunctor(\substTyMor{\ECtxMap})}", from=1-2, to=1-3]
  \arrow["{\DemIso{\Delta}^{-1}}", from=1-3, to=1-4]
\end{tikzcd}
\]
To construct the inverse of $f$, we define a context morphism $h : \Delta \rightarrow \CtxExt{\Gamma}{\widetilde{\Delta}}$
and a term $t$ of type $\substTm{h}{l} = \substTm{h}{r}$ in context $\Delta$.
We define $h$ using the universal property of the pullback.
\[
\begin{tikzcd}
  \Delta \\
  & {\CtxExt{\Gamma}{\widetilde{\Delta}}} & {\CtxExt{\ECtx}{\DemTy{\Delta}}} \\
  & \Gamma & \ECtx
  \arrow["h"{description}, dashed, from=1-1, to=2-2]
  \arrow["{\DemIso{\Delta}}", bend left=20, from=1-1, to=2-3]
  \arrow["s"', bend right=30, from=1-1, to=3-2]
  \arrow["\compfunctor(\substTyMor{\ECtxMap})", from=2-2, to=2-3]
  \arrow["\CtxPr"', from=2-2, to=3-2]
  \arrow["\CtxPr", from=2-3, to=3-3]
  \arrow["\ECtxMap"', from=3-2, to=3-3]
\end{tikzcd}
\]
To define the term $t$,
it suffices to prove that the underlying morphisms of $\CtxPr[{\FibTerm[\Delta]}] \cdot \substTm{h}{l}$ and $\CtxPr[{\FibTerm[\Delta]}] \cdot \substTm{h}{r}$ are equal.
For a proof, we refer the reader to the formalization.

All in all, we constructed a type $A$ and an isomorphism between $\CtxExt{\Gamma}{A}$ and $\Delta$.
This isomorphism lifts to the slice category $\slice{\C}{\Gamma}$,
and thus $\fiberfunctor{\compfunctor}{\Gamma}$ is essentially surjective.
\end{proof}

Now we construct the desired pseudofunctor.

\begin{problem}
\label[problem]{prob:compcat-to-finlim}
To construct a pseudofunctor $\CompCatToFinLim : \DFLCompCat \rightarrow \FinLim$.
\end{problem}

\begin{construction}{\coqdocurl{Bicategories.ComprehensionCat.Biequivalence.DFLCompCatToFinLim}{dfl_comp_cat_to_finlim_psfunctor}}{prob:compcat-to-finlim}
\label{constr:compcat-to-finlim}
We map every DFL comprehension category $\chi : \D \rightarrow \ArrD{\C}$ to the category $\C$.
Note that $\C$ is univalent, because we assumed our comprehension category to be univalent.
In addition, $\C$ has all finite limits,
because we have equivalences $\C \simeq \slice{\C}{\ECtx} \simeq \fiber{\D}{\ECtx}$
and because $\D$ has all fiberwise finite limits.

Every 1-cell $(\F, \FF, \F_\chi)$ from a DFL comprehension category $\compfunctor_1 : \D_1 \rightarrow \ArrD{\C_1}$ to $\compfunctor_2 : \D_2 \rightarrow \ArrD{\C_2}$
is mapped to $\F$ by $\CompCatToFinLim$,
so we need to show that $\F$ preserves finite limits.
To do so, we use the following natural isomorphisms.
\[
\begin{tikzcd}
  {\fiber{\D_1}{\ECtx}} & {\slice{\C_1}{\ECtx}} & {\C_1} \\
  {\fiber{\D_2}{\F \> \ECtx}} & {\slice{\C_2}{\F \> \ECtx}} & {\C_2}
  \arrow[""{name=0, anchor=center, inner sep=0}, "\simeq", from=1-1, to=1-2]
  \arrow["{\fiberfunctor{\FF}{\ECtx}}"', from=1-1, to=2-1]
  \arrow[""{name=1, anchor=center, inner sep=0}, "\simeq", from=1-2, to=1-3]
  \arrow[from=1-2, to=2-2]
  \arrow["\F", from=1-3, to=2-3]
  \arrow[""{name=2, anchor=center, inner sep=0}, "\simeq"', from=2-1, to=2-2]
  \arrow[""{name=3, anchor=center, inner sep=0}, "\simeq"', from=2-2, to=2-3]
  \arrow["\cong"{description}, shorten <=4pt, shorten >=4pt, Rightarrow, from=0, to=2]
  \arrow["\cong"{description}, shorten <=4pt, shorten >=4pt, Rightarrow, from=1, to=3]
\end{tikzcd}
\]
The functor $\fiberfunctor{\FF}{\ECtx}$ preserves finite limits by assumption,
and thus $\F$ preserves finite limits as well.
Finally, every 2-cell $(\nt, \tautau)$ is mapped to $\nt$.
\end{construction}

In \Cref{constr:compcat-to-finlim} we transferred properties and structure along adjoint equivalences and natural isomorphisms.
More specifically, we used that a category $\C'$ has finite limits if we have an equivalence $\C \simeq \C'$ such that $\C$ has all finite limits.
We also used that $\G$ preserves finite limits if $\F$ preserves finite limits and if we have a diagram as follows.
\[
\begin{tikzcd}
  {\C_1} & {\C_1'} \\
  {\C_2} & {\C_2'}
  \arrow[""{name=0, anchor=center, inner sep=0}, "\simeq", from=1-1, to=1-2]
  \arrow["\F"', from=1-1, to=2-1]
  \arrow["\G", from=1-2, to=2-2]
  \arrow[""{name=1, anchor=center, inner sep=0}, "\simeq"', from=2-1, to=2-2]
  \arrow["\cong"{description}, shorten <=4pt, shorten >=4pt, Rightarrow, from=0, to=1]
\end{tikzcd}
\]
Since these categories are univalent,
both principles follow directly from \Cref{prop:equivalence-induction}.

\subsection{The Pseudotransformations}
We finish the construction by defining the necessary pseudotransformations
and by proving that they are adjoint equivalences.

\begin{problem}
\label[problem]{prob:unit}
 To construct a pseudotransformation $\FinLimCompCatUnit : \FinLimToCompCat \cdot \CompCatToFinLim \Rightarrow \id[\FinLim]$.
\end{problem}

\begin{construction}{\coqdocurl{Bicategories.ComprehensionCat.Biequivalence.Unit}{finlim_dfl_comp_cat_unit}}{prob:unit}
\label{constr:unit}
Given a univalent category $\C$ with finite limits,
the functor $\FinLimCompCatUnit(\C) : \C \rightarrow \C$ is defined to be the identity.
If we have a functor $\F : \C_1 \rightarrow \C_2$,
the naturality square $\FinLimCompCatUnit(\F)$, depicted below, is defined to be identity pointwise.
\[
\begin{tikzcd}
  {\C_1} & {\C_1} \\
  {\C_2} & {\C_2}
  \arrow[""{name=0, anchor=center, inner sep=0}, "\id", from=1-1, to=1-2]
  \arrow["\F"', from=1-1, to=2-1]
  \arrow["\F", from=1-2, to=2-2]
  \arrow[""{name=1, anchor=center, inner sep=0}, "\id"', from=2-1, to=2-2]
  \arrow["\id", shorten <=4pt, shorten >=4pt, Rightarrow, from=0, to=1]
\end{tikzcd}
\]
This data satisfies the coherence laws for pseudotransformations.
\end{construction}

\begin{propL}[\coqdocurl{Bicategories.ComprehensionCat.Biequivalence.Unit}{finlim_dfl_comp_cat_unit_pointwise_equiv}]
\label[propL]{prop:unit-pointwise-adjequiv}
The pseudotransformation $\FinLimCompCatUnit$ is a pointwise adjoint equivalence.
\end{propL}

\begin{proof}
This follows from the fact that adjoint equivalences in $\FinLim$ are the same as adjoint equivalences of categories.  
\end{proof}

The construction of the other pseudotransformation is more interesting,
and we use \Cref{prop:comprehension-eso} to prove that it is an adjoint equivalence.

\begin{problem}
\label[problem]{prob:counit}
To construct a pseudotransformation $\FinLimCompCatCounit : \id[\DFLCompCat] \Rightarrow \CompCatToFinLim \cdot \FinLimToCompCat$.
\end{problem}

\begin{construction}{\coqdocurl{Bicategories.ComprehensionCat.Biequivalence.Counit}{finlim_dfl_comp_cat_counit}}{prob:counit}
\label{constr:counit}
Let $\compfunctor : \D \rightarrow \ArrD{\C}$ be a DFL comprehension category.
Note that $\CompCatToFinLim \cdot \FinLimToCompCat$ maps $\compfunctor$ to $\id : \ArrD{\C} \rightarrow \ArrD{\C}$.
As such, we need to construct a functor $\F : \C \rightarrow \C$ and a displayed functor $\FF : \D \rightarrow \ArrD{\C}$.
We take $\F$ to be the identity and $\FF$ to be $\compfunctor$.
Note that $\compfunctor$ preserves fiberwise finite limits,
because it is an adjoint equivalence by \Cref{prop:comprehension-eso}.

Now suppose that we have a morphism $(\F, \FF, \F_\chi)$
from $\compfunctor_1 : \D_1 \rightarrow \ArrD{\C_1}$
to $\compfunctor_2 : \D_2 \rightarrow \ArrD{\C_2}$.
To construct the naturality square $\FinLimCompCatCounit(\F, \FF, \F_\chi)$,
we need a natural transformation $\tau : \F \Rightarrow \F$
and a displayed natural isomorphism $\tautau$ over $\tau$ as follows.
\[
\begin{tikzcd}
  {\D_1} & {\ArrD{\C_1}} \\
  {\D_2} & {\ArrD{\C_2}}
  \arrow[""{name=0, anchor=center, inner sep=0}, "{\compfunctor_1}", from=1-1, to=1-2]
  \arrow["\FF"', from=1-1, to=2-1]
  \arrow["{\ArrD{\F}}", from=1-2, to=2-2]
  \arrow[""{name=1, anchor=center, inner sep=0}, "{\compfunctor_2}"', from=2-1, to=2-2]
  \arrow["\tautau", shorten <=4pt, shorten >=4pt, Rightarrow, from=0, to=1]
\end{tikzcd}
\]
For $\tau$ we take the identity
and for $\tautau$ we take $\F_\chi$.
Note that $\tautau$ is a natural isomorphism,
because we restricted ourselves to pseudo morphisms.
\end{construction}

\begin{propL}[\coqdocurl{Bicategories.ComprehensionCat.Biequivalence.Counit}{finlim_dfl_comp_cat_counit_pointwise_equiv}]
\label[propL]{prop:counit-pointwise-adjequiv}
The pseudotransformation $\FinLimCompCatCounit$ is a pointwise adjoint equivalence.
\end{propL}

\begin{proof}
By \Cref{prop:dfl-compcat-adjequiv},
it suffices to show that $\id : \C \rightarrow \C$ and $\compfunctor : \D \rightarrow \ArrD{\C}$ are adjoint equivalences
for each DFL comprehension category $\compfunctor : \D \rightarrow \ArrD{\C}$.
One can verify directly that the identity is an adjoint equivalence,
and we know from \Cref{prop:comprehension-eso} that $\compfunctor$ is a displayed adjoint equivalence over the identity.
\end{proof}

We conclude this section by putting all parts of the desired biequivalence together.

\begin{problem}
\label[problem]{prob:biequiv}
To construct a biequivalence between the bicategories $\DFLCompCat$ and $\FinLim$.
\end{problem}

\begin{construction}{\coqdocurl{Bicategories.ComprehensionCat.Biequivalence.Biequiv}{finlim_biequiv_dfl_comp_cat_psfunctor}}{prob:biequiv}
\label[constructionCounter]{constr:biequiv}
The necessary pseudofunctors are given in \Cref{constr:finlim-to-compcat,constr:compcat-to-finlim},
and the pseudotransformations are given in \Cref{constr:unit,constr:counit}.
These pseudotransformations are shown to be pointwise equivalences in \Cref{prop:unit-pointwise-adjequiv,prop:counit-pointwise-adjequiv},
so by \Cref{prop:pointwise-adj-equiv} they are adjoint equivalences in the bicategory of pseudofunctors.
Hence, we have a biequivalence between $\DFLCompCat$ and $\FinLim$.
\end{construction}

\section{The Internal Language of Locally Cartesian Closed Categories}
\label{sec:lccc}
Up to now we have only considered the internal language theorem for univalent categories with finite limits,
and we showed that their internal language is given by
extensional Martin-L\"of type theory with unit types, binary product types, and $\Sigma$-types.
In the remainder of this paper, we extend this theorem to classes of categories with more structure
and to more expressive versions of type theory.
We start by considering locally Cartesian closed categories,
and we show that their internal language also supports $\Pi$-types.
This statement was the theorem originally proven by Seely~\cite[Theorem 6.3]{seely:1984}
and later corrected by Clairambault and Dybjer~\cite[Theorem 6.1]{clairambault:2014}.

Our goal in this section is to construct a biequivalence between
the bicategories $\LCCC$ of univalent locally Cartesian closed categories
and $\DFLCompCatPi$ of DFL comprehension categories that support $\Pi$-types.
Let us start by defining these. 

\begin{defi}
\label[defi]{def:lccc-bicat}
A category $\C$ with finite limits is said to be \conceptDef{locally Cartesian closed}{CategoryTheory.LocallyCartesianClosed.LocallyCartesianClosed}{is_locally_cartesian_closed} if
the pullback functor $\functorfiber{f} : \slice{\C}{y} \rightarrow \slice{\C}{x}$ has a right adjoint
for all morphisms $f : x \rightarrow y$.
A functor $\F : \C_1 \rightarrow \C_2$ \conceptDef{preserves exponentials}{CategoryTheory.LocallyCartesianClosed.Preservation}{preserves_locally_cartesian_closed}
if the Beck--Chevalley condition for right adjoints holds for the following square.
\[
\begin{tikzcd}
  {\slice{\C_1}{y}} & {\slice{\C_2}{\F(y)}} \\
  {\slice{\C_1}{x}} & {\slice{\C_2}{\F(x)}}
  \arrow[""{name=0, anchor=center, inner sep=0}, "{\slice{\F}{y}}", from=1-1, to=1-2]
  \arrow["{\functorfiber{f}}"', from=1-1, to=2-1]
  \arrow["{\functorfiber{\F(f)}}", from=1-2, to=2-2]
  \arrow[""{name=1, anchor=center, inner sep=0}, "{\slice{\F}{x}}"', from=2-1, to=2-2]
  \arrow["\cong"{description}, shorten <=4pt, shorten >=4pt, Rightarrow, from=0, to=1]
\end{tikzcd}
\]
Here the functor $\slice{\F}{x} : \slice{\C_1}{x} \rightarrow \slice{\C_2}{\F(x)}$ sends a morphism $g : w \rightarrow x$
to $\F(g) : \F(w) \rightarrow \F(x)$.

We define the displayed bicategory \conceptDef{$\LCCCD$}{Bicategories.Core.Examples.StructuredCategories}{disp_bicat_univ_lccc} over $\FinLim$ as the subbicategory $\subbicatD{\Pred}{\Qred}$
where $\Pred$ says that a univalent category with finite limits is locally Cartesian closed
and where $\Qred$ says that a functor that preserves finite limits also preserves exponentials.
The total bicategory $\total{\LCCCD}$ is denoted by \conceptDef{$\LCCC$}{Bicategories.Core.Examples.StructuredCategories}{bicat_of_univ_lccc}.
\end{defi}

\begin{defi}
\label[defi]{def:pi-types}
Let $\compfunctor : \D \rightarrow \ArrD{\C}$ be a full univalent comprehension category.
We say that $\compfunctor$ \conceptDef{supports $\Pi$-types}{Bicategories.ComprehensionCat.TypeFormers.PiTypes}{comp_cat_dependent_prod} if
\begin{itemize}
  \item for every $A : \dob{\D}{\Gamma}$ the functor $\functorfiber{\CtxPr[A]} : \fiber{\D}{\Gamma} \rightarrow \fiber{\D}{\CtxExt{\Gamma}{A}}$
  has a right adjoint $\DepProd{A}$ with unit $\eta$ and counit $\varepsilon$;
  \item the Beck--Chevalley condition for right adjoints holds for every pullback square.
\end{itemize}
\end{defi}

If $\compfunctor : \D \rightarrow \ArrD{\C}$ supports $\Pi$-types,
then the functor $\functorfiber{s} : \dob{\D}{\Gamma} \rightarrow \dob{\D}{\Delta}$ has a right adjoint
whenever we have a commutative triangle as follows.
\[
\begin{tikzcd}
  \Delta && {\CtxExt{\Gamma}{A}} \\
  & \Gamma
  \arrow["\cong", from=1-1, to=1-3]
  \arrow["s"', from=1-1, to=2-2]
  \arrow["{\CtxPr[A]}", from=1-3, to=2-2]
\end{tikzcd}
\]
We use this observation to say when a 1-cell
$(\F, \FF, \F_\chi)$ from $\compfunctor_1 : \D_1 \rightarrow \ArrD{\C_1}$ to $\compfunctor_2 : \D_2 \rightarrow \ArrD{\C_2}$
preserves $\Pi$-types.
This is because $(\F, \FF, \F_\chi)$ only preserves context extension up to isomorphism
meaning that we have the following commutative triangle.
\[
\begin{tikzcd}
  {\F(\CtxExt{\Gamma}{A})} && {\CtxExt{\F(\Gamma)}{\FF(A)}} \\
  & {\F(\Gamma)}
  \arrow["\cong", from=1-1, to=1-3]
  \arrow["{\F(\CtxPr[A])}"', from=1-1, to=2-2]
  \arrow["{\CtxPr[\FF(A)]}", from=1-3, to=2-2]
\end{tikzcd}
\]
If $\compfunctor_2 : \D_2 \rightarrow \ArrD{\C_2}$ has $\Pi$-types,
then $\functorfiber{\CtxPr[\FF(A)]}$ has a right adjoint
and thus $\functorfiber{\F(\CtxPr[A])}$ as well.

\begin{defi}
\label[defi]{def:pi-types-bicat}
Let $(\F, \FF, \F_\chi)$ be a 1-cell from $\compfunctor_1 : \D_1 \rightarrow \ArrD{\C_1}$ to $\compfunctor_2 : \D_2 \rightarrow \ArrD{\C_2}$.
We say that $(\F, \FF, \F_\chi)$ \conceptDef{preserves $\Pi$-types}{Bicategories.ComprehensionCat.TypeFormers.PiTypes}{preserves_comp_cat_dependent_prod} if
for all objects $x, y : \C_1$ and morphisms $f : x \rightarrow y$ the Beck-Chevalley condition for right adjoints holds for the following square.
\[
\begin{tikzcd}
  {\fiber{\D_1}{y}} & {\fiber{\D_2}{\F(y)}} \\
  {\fiber{\D_1}{x}} & {\fiber{\D_2}{\F(x)}}
  \arrow[""{name=0, anchor=center, inner sep=0}, "{\fiberfunctor{\FF}{y}}", from=1-1, to=1-2]
  \arrow["{\functorfiber{f}}"', from=1-1, to=2-1]
  \arrow["{\functorfiber{\F(f)}}", from=1-2, to=2-2]
  \arrow[""{name=1, anchor=center, inner sep=0}, "{\fiberfunctor{\FF}{x}}"', from=2-1, to=2-2]
  \arrow["\cong"{description}, shorten <=4pt, shorten >=4pt, Rightarrow, from=0, to=1]
\end{tikzcd}
\]
Finally, we define the displayed bicategory \conceptDef{$\DFLCompCatPiD$}{Bicategories.ComprehensionCat.TypeFormers.PiTypes}{disp_bicat_of_pi_type} over $\DFLCompCat$ to be $\subbicatD{\Pred}{\Qred}$
where $\Pred$ says that $\compfunctor : \D \rightarrow \ArrD{\C}$ supports $\Pi$-types
and $\Qred$ says that $(\F, \FF, \F_\chi)$  preserves $\Pi$-types.
We write \conceptDef{$\DFLCompCatPi$}{Bicategories.ComprehensionCat.TypeFormers.PiTypes}{bicat_of_pi_type_dfl_full_comp_cat} for the total bicategory $\total{\DFLCompCatPiD}$.
\end{defi}

Before we continue, let us make a couple of remarks on \Cref{def:lccc-bicat,def:pi-types-bicat}.
First, both $\LCCC$ and $\DFLCompCatPi$ are univalent.
Since $\subbicatD{\Pred}{\Qred}$ is univalent if both the families $\Pred$ and $\Qred$ are valued in propositions (\Cref{exa:disp-bicat}),
it follows that the displayed bicategories $\LCCCD$ and $\DFLCompCatPiD$ are univalent,
and thus $\LCCC$ and $\DFLCompCatPi$ are so as well.
Second, the displayed 2-cells of $\LCCCD$ and $\DFLCompCatPiD$ are trivial,
since the type $\dtwo{\ff}{\gg}{\tau}$ is the unit type for all displayed 1-cells $\ff : \dmor{\xx}{\yy}{f}$ and $\gg : \dmor{\xx}{\yy}{g}$, and 2-cells $\tau : f \Rightarrow g$.

To prove that $\LCCC$ and $\DFLCompCatPi$ are biequivalent, we extend the biequivalence defined in \Cref{sec:fin-lim} using \textbf{displayed biequivalences}~\cite[Definition 8.9]{ahrens:2021}.
The notion of displayed biequivalences lifts the notion of biequivalence to the setting of displayed bicategories,
and we can use them to construct biequivalences modularly.
More specifically,
if we have a biequivalence from $\B_1$ to $\B_2$ and displayed bicategories $\D_1$ and $\D_2$ over $\B_1$ and $\B_2$ respectively,
then a displayed biequivalence specifies the data and properties necessary to extend the given biequivalence to the total bicategories $\total{\D_1}$ and $\total{\D_2}$.

\subsection{Displayed Biequivalences}
Our main goal is to show that we can use displayed biequivalences to modularly construct biequivalences of bicategories,
which we show in \Cref{constr:total-biequiv}.
Since we already have a biequivalence between $\DFLCompCat$ and $\FinLim$,
it suffices to only consider the displayed bicategories $\LCCCD$ and $\DFLCompCatPiD$
in order to construct a biequivalence between $\LCCC$ and $\DFLCompCatPi$.
To do so,
we have to show,
among others,
that $\id : \ArrD{\C} \rightarrow \ArrD{\C}$ supports $\Pi$-types
whenever $\C$ is locally Cartesian closed,
and that $\C$ is locally Cartesian closed
whenever $\chi : \D \rightarrow \ArrD{\C}$ supports $\Pi$-types.
Note that our approach is more modular compared to directly constructing a biequivalence $\LCCC \simeq \DFLCompCatPi$,
because we can reuse \Cref{constr:biequiv}.

We start by recalling displayed biequivalences.
In the next sections, we give concrete examples of displayed biequivalences.
To define this notion, Ahrens et al.~\cite[Definition 8.9]{ahrens:2021}
first defined displayed pseudofunctors, displayed pseudotransformations, and displayed invertible modifications.
Recall that every pseudofunctor $\F : \B_1 \rightarrow \B_2$ comes with invertible 2-cells
$\identitor{\F}(x) : \id[\F(x)] \twocell \F(\id[x])$
and $\compositor{\F}(f, g) : \F(f) \cdot \F(g) \twocell \F(f \cdot g)$,
which we call the \textbf{identitor} and \textbf{compositor} respectively.

\begin{defiC}[{\cite[Definition 8.2]{ahrens:2021}}]
\label[defi]{def:disp-psfunctor}
Let $\B_1$ and $\B_2$ be bicategories,
and let $\D_1$ and $\D_2$ be displayed bicategories over $\B_1$ and $\B_2$ respectively.
A \conceptDef{displayed pseudofunctor}{Bicategories.DisplayedBicats.DispPseudofunctor}{disp_psfunctor} $\FF : \dmor{\D_1}{\D_2}{\F}$ from $\D_1$ to $\D_2$ over $\F : \B_1 \rightarrow \B_2$
consists of
\begin{itemize}
  \item a map sending $\xx : \dob{\D_1}{x}$ to $\FF(\xx) : \dob{\D_2}{\F(x)}$ for each $x : \B_1$;
  \item a map sending $\ff : \dmor{\xx}{\yy}{f}$ to $\FF(\ff) : \dmor{\FF(\xx)}{\FF(\yy)}{\F(f)}$
    for all objects $\xx$ over $x$ and $\yy$ over $y$, and all 1-cells $f : x \rightarrow y$;
  \item a map sending $\tautau : \dtwo{\ff}{\gg}{\tau}$ to $\FF(\tautau) : \dtwo{\FF(\ff)}{\FF(\gg)}{\F(\tau)}$
    for all 1-cells $\ff$ and $\gg$ over $f$ and $g$ respectively;
  \item a displayed invertible 2-cell $\identitor{\FF}(\xx) : \dtwo{\idd[\FF(\xx)]}{\FF(\idd[\xx])}{\identitor{\F}(x)}$ for each object $\xx : \dob{\D_1}{x}$;
  \item a displayed invertible 2-cell $\compositor{\FF}(\ff, \gg) : \dtwo{\FF(\ff) \cdot \FF(\gg)}{\FF(\ff \cdot \gg)}{\compositor{\F}(f, g)}$
    for all 1-cells $\ff : \dmor{\xx}{\yy}{f}$ and $\gg : \dmor{\yy}{\zz}{g}$.
\end{itemize}
This data is required to satisfy various laws,
for which we refer the reader to the literature~\cite[Definition 8.2]{ahrens:2021}.
\end{defiC}

Let us give an example of a displayed pseudofunctor.
We define $\UnivCatOb$ to be the displayed bicategory over $\UnivCat$ such that
\begin{itemize}
  \item objects over $\C$ are objects $x : \C$;
  \item 1-cells over $\F : \C_1 \rightarrow \C_2$ from $x : \C_1$ to $y : \C_2$ are isomorphisms $f : \iso{\F(x)}{y}$;
  \item 2-cells over $\nt : \F \twocell \G$ from $f : \iso{\F(x)}{y}$ to $g : \iso{\G(x)}{y}$ are proofs that the following diagram commutes.
    \[
      \begin{tikzcd}
	{\F(x)} && {\G(x)} \\
	& y
	\arrow["{\nt(x)}", from=1-1, to=1-3]
	\arrow["f"', from=1-1, to=2-2]
	\arrow["g", from=1-3, to=2-2]
      \end{tikzcd}
    \]
\end{itemize}
We define $\UnivCatObTwo$ to be $\UnivCatOb \times \UnivCatOb$.
Concretely, objects over $\C$ in $\UnivCatObTwo$ are pairs of objects $x, y : \C$.
We have several examples of displayed pseudofunctors between the displayed bicategories $\UnivCatOb$ and $\UnivCatObTwo$.
For instance, we have a displayed pseudofunctor $\diagD$ from $\UnivCatOb$ to $\UnivCatObTwo$ that maps $x$ to the pair $(x, x)$.
The actions on displayed 1-cells and 2-cells are defined similarly.
We also have two displayed pseudofunctors from $\UnivCatObTwo$ to $\UnivCatOb$:
one maps $(x, y)$ to $x$ and the other maps $(x, y)$ to $y$.
We denote these displayed pseudofunctors by $\fstD$ and $\sndD$ respectively.
None of these displayed pseudofunctors are part of a displayed biequivalence.

Another example of a displayed pseudofunctor is the identity.
Specifically,
if we have a displayed bicategory $\D$ over $\B$,
then we have a displayed pseudofunctor $\idd : \dmor{\D}{\D}{\id}$,
which maps $\xx$ to $\xx$.
We can also compose displayed pseudofunctors.
Given $\FF : \dmor{\D_1}{\D_2}{\F}$ and $\GG : \dmor{\D_2}{\D_3}{\G}$,
we have a composition $\FF \cdot \GG : \dmor{\D_1}{\D_3}{\F \cdot \G}$,
which maps $\xx$ to $\GG(\FF(\xx))$.

If $\tau$ is a pseudotransformation from $\F$ to $\G$,
then every 1-cell $f : x \rightarrow y$
gives rise to an invertible 2-cell $\tau(f) : \tau(x) \cdot \G(f) \twocell \F(f) \cdot \tau(y)$
witnessing naturality.

\begin{defiC}[{\cite[Definition 8.3]{ahrens:2021}}]
\label[defi]{def:disp-pstrans}
Let $\B_1$ and $\B_2$ be bicategories,
and let $\D_1$ and $\D_2$ be displayed bicategories over $\B_1$ and $\B_2$ respectively.
Suppose that we also have pseudofunctors $\F, \G : \B_1 \rightarrow \B_2$
and displayed pseudofunctors $\FF : \dmor{\D_1}{\D_2}{\F}$ and $\GG : \dmor{\D_1}{\D_2}{\G}$.
A \conceptDef{displayed pseudotransformation}{Bicategories.DisplayedBicats.DispTransformation}{disp_pstrans} $\tautau : \dtwo{\FF}{\GG}{\tau}$ over $\tau : \F \Rightarrow \G$
consists of a displayed 1-cell $\tautau(\xx) : \dmor{\FF(\xx)}{\GG(\xx)}{\tau(x)}$ for each $\xx$ over $x : \B_1$
and a displayed invertible 2-cell $\tautau(\ff) : \dtwo{\tautau(x) \cdot \GG(\ff)}{\FF(\ff) \cdot \tautau(y)}{\tau(f)}$ for each $\ff : \dmor{\xx}{\yy}{f}$.
This data is subject to various laws,
which the reader can find in the literature~\cite[Definition 8.3]{ahrens:2021}.
\end{defiC}

We have a displayed pseudotransformation $\diagfstD$ from $\diagD \cdot \fstD : \dmor{\UnivCatOb}{\UnivCatOb}{\id \cdot \id}$ to $\idd : \dmor{\UnivCatOb}{\UnivCatOb}{\id}$
over the pseudotransformation $\lunitor : \id \cdot \id \twocell \id$,
which is defined to be the identity on every object.
Given a category $\C$ and an object $x : \C$,
we need to construct a morphism $\diagfstD(x) : \dmor{x}{x}{\id}$.
Concretely, that means we need to give an isomorphism $\iso{x}{x}$,
for which we take the identity isomorphism.
We can also define a displayed pseudotransformation $\diagfstD^{-1} : \dtwo{\idd}{\diagD \cdot \fstD}{\linvunitor}$.
In the same way we can define displayed pseudotransformations $\diagsndD : \dtwo{\diagD \cdot \sndD}{\idd}{\lunitor}$ and $\diagsndD^{-1} : \dtwo{\idd}{\diagD \cdot \sndD}{\linvunitor}$.

We also have the identity and composition operation on pseudotransformations.
If we have a displayed pseudofunctor $\FF : \dmor{\D_1}{\D_2}{\F}$,
then we have a displayed pseudotransformation $\idd : \dtwo{\FF}{\FF}{\id}$,
which is defined to be $\idd[\FF(\xx)] : \dmor{\FF(\xx)}{\FF(\xx)}{\id[\F(x)]}$ for each $\xx$.
Given pseudotransformations $\tautau : \dtwo{\FF}{\GG}{\tau}$ and $\thetatheta : \dtwo{\GG}{\HH}{\theta}$,
then we also have a displayed pseudotransformation $\tautau \cdot \thetatheta : \dtwo{\FF}{\HH}{\tau \cdot \theta}$.
We define $\tautau \cdot \thetatheta$ to be $\tautau(x) \cdot \thetatheta(x) : \dmor{\FF(x)}{\HH(x)}{\tau(x) \cdot \theta(x)}$ for each $x$.

\begin{defiC}[{\cite[Definition 8.5]{ahrens:2021}}]
\label[defi]{def:disp-invmodif}
Suppose that we have bicategories $\B_1$ and $\B_2$,
pseudofunctors $\F, \G : \B_1 \rightarrow \B_2$,
and pseudotransformations $\tau_1, \tau_2 : \F \twocell \G$.
In addition,
we assume that we have
displayed bicategories $\D_1$ and $\D_2$ over $\B_1$ and $\B_2$ respectively,
displayed pseudofunctors $\FF : \dmor{\D_1}{\D_2}{\F}$ and $\GG : \dmor{\D_1}{\D_2}{\G}$,
and displayed pseudotransformations $\tautau_1 : \dtwo{\FF}{\GG}{\tau_1}$ and $\tautau_2 : \dtwo{\FF}{\GG}{\tau_2}$.
A \conceptDef{displayed invertible modification}{Bicategories.DisplayedBicats.DispModification}{disp_invmodification} $\mm$
over an invertible modification $m : \modif{\tau_1}{\tau_2}$ consists of
displayed invertible 2-cells $\mm(\xx) : \dtwo{\tautau_1(\xx)}{\tautau_2(\xx)}{m(x)}$ for each $\xx$ over $x$.
This data is required to satisfy a naturality law,
which is given in the literature~\cite[Definition 8.4]{ahrens:2021}.
\end{defiC}

Note that we have an invertible modification $m : \modif{\lunitor \cdot \linvunitor}{\id,}$
which is defined to be the identity 2-cell pointwise.
We can construct a displayed invertible modification $\mm$ from $\diagfstD \cdot \diagfstD^{-1}$ to the identity.
To do so,
let us assume that we have some category $\C$ with an object $x : \C$.
From $\C$ and $x$ we get two displayed 1-cells,
namely $(\diagfstD \cdot \diagfstD^{-1})(\C, x) : \dmor{x}{x}{\id[\C]}$
and $\idd(\C, x) : \dmor{x}{x}{\id[\C]}$.
Both these displayed 1-cells are the identity isomorphism.
Hence, we get the desired displayed 2-cell
from the fact that the following diagram commutes.
\[
\begin{tikzcd}
  x && x \\
  & x
  \arrow["\id", from=1-1, to=1-3]
  \arrow["\id"', from=1-1, to=2-2]
  \arrow["\id", from=1-3, to=2-2]
\end{tikzcd}
\]
Analogously, we can construct a displayed invertible modification from $\diagsndD \cdot \diagsndD^{-1}$ to the identity.
We can also define the identity and composition of displayed invertible modifications,
but since we do not use these operations in the remainder,
we do not discuss their definitions.

Next we put these notions together to obtain the notion of displayed biequivalence.

\begin{defiC}[{\cite[Definition 8.9]{ahrens:2021}}]
\label[defi]{def:disp-biequiv}
Let $\B_1$ and $\B_2$ be bicategories,
and let $\D_1$ and $\D_2$ be displayed bicategories over $\B_1$ and $\B_2$ respectively.
Suppose that we have a biequivalence $(\F, \G, \eta, \varepsilon)$ where $\F : \B_1 \rightarrow \B_2$.
A \conceptDef{displayed biequivalence}{Bicategories.DisplayedBicats.DispBiequivalence}{disp_is_biequivalence_data} over $(\F, \G, \eta, \varepsilon)$ consists of
\begin{itemize}
  \item a displayed pseudofunctor $\FF : \dmor{\D_1}{\D_2}{\F}$;
  \item a displayed pseudofunctor $\GG : \dmor{\D_2}{\D_1}{\G}$;
  \item displayed pseudotransformations $\etaeta : \dtwo{\idd}{\GG \cdot \FF}{\eta}$ and $\etaeta^{-1} : \dtwo{\GG \cdot \FF}{\idd}{\eta^{-1}}$;
  \item displayed pseudotransformations $\epsilonepsilon : \dtwo{\FF \cdot \GG}{\idd}{\varepsilon}$ and $\epsilonepsilon^{-1} : \dtwo{\idd}{\FF \cdot \GG}{\varepsilon^{-1}}$;
  \item displayed invertible modifications witnessing that $\etaeta$ and $\etaeta^{-1}$ form an equivalence;
  \item displayed invertible modifications witnessing that $\epsilonepsilon$ and $\epsilonepsilon^{-1}$ form an equivalence.
\end{itemize}
\end{defiC}

In the remainder of this paper,
we discuss numerous examples of displayed biequivalences.
The reason why we are interested in displayed biequivalences,
is because they give rise to biequivalences on the total bicategories.
We use this feature to construct biequivalences in a modular way,
because we can extend a given biequivalence to displayed bicategories.
To do so, we use the following construction~\cite[Construction 8.11]{ahrens:2021}.

\begin{problem}
\label[problem]{prob:total-biequiv}
Given bicategories $\B_1$ and $\B_2$,
displayed bicategories $\D_1$ and $\D_2$ over $\B_1$ and $\B_2$ respectively,
and a displayed biequivalence $(\FF, \GG, \etaeta, \epsilonepsilon)$ over a biequivalence $(\F, \G, \eta, \varepsilon)$,
to construct a biequivalence $(\total{\FF}, \total{\GG}, \total{\etaeta}, \total{\epsilonepsilon})$ between the total bicategories $\total{\D_1}$ and $\total{\D_2}$.
\end{problem}

\begin{construction}{\coqdocurl{Bicategories.DisplayedBicats.DispBiequivalence}{total_is_biequivalence}}{prob:total-biequiv}
\label[constructionCounter]{constr:total-biequiv}
Each of the components is constructed by pairing.
For instance, the pseudofunctor $\total{\FF}$ sends a pair $(x, \xx)$ in $\total{\D_1}$ to $(\F(x), \FF(\xx))$ in $\total{\D_2}$.
The actions on 1-cells and 2-cells are similar.
\end{construction}

In the remainder, we only consider displayed bicategories that are locally preordered and groupoidal.
These conditions give us a more convenient way to construct displayed biequivalences.
In the construction, we use an analogue of \Cref{prop:pointwise-adj-equiv} for displayed pseudotransformations,
which says that it suffices to check that $\etaeta$ and $\epsilonepsilon$ are pointwise displayed adjoint equivalences.

\begin{problem}
\label[problem]{prop:disp-pointwise-adjequiv}
Suppose that we have univalent bicategories $\B_1$ and $\B_2$
and displayed bicategories $\D_1$ and $\D_2$ over $\B_1$ and $\B_2$ respectively,
such that $\D_1$ and $\D_2$ are both locally preordered and groupoidal.
Given a biequivalence $(\F, \G, \eta, \varepsilon)$ where $\F : \B_1 \rightarrow \B_2$
and
\begin{itemize}
  \item a displayed pseudofunctor $\FF : \dmor{\D_1}{\D_2}{\F}$;
  \item a displayed pseudofunctor $\GG : \dmor{\D_2}{\D_1}{\G}$;
  \item a displayed pseudotransformation $\etaeta : \dtwo{\idd}{\GG \cdot \FF}{\eta}$;
  \item a displayed pseudotransformation $\epsilonepsilon : \dtwo{\FF \cdot \GG}{\idd}{\varepsilon}$
\end{itemize}
such that $\etaeta$ and $\epsilonepsilon$ are pointwise displayed adjoint equivalences,
to construct a displayed biequivalence over $(\F, \G, \eta, \varepsilon)$.
\end{problem}

\begin{construction}{\coqdocurl{Bicategories.DisplayedBicats.DispPseudoNaturalAdjequiv}{make_disp_biequiv_pointwise_adjequiv}}{prop:disp-pointwise-adjequiv}
\label{constr:disp-biequiv-simpl}
To construct the desired displayed biequivalence,
it suffices to construct displayed pseudotransformations $\etaeta^{-1}$ and $\epsilonepsilon^{-1}$ and the necessary invertible modifications.
We consider a more general statement,
namely that if a displayed pseudotransformation $\tautau : \dtwo{\FF}{\GG}{\tau}$ is a pointwise displayed adjoint equivalence,
then we have a displayed pseudotransformation $\tautau^{-1} : \dtwo{\GG}{\FF}{\tau^{-1}}$
where $\tau : \F \twocell \G$ is a pseudonatural adjoint equivalence.
For this statement, we only give a sketch.
Since we assumed the involved bicategories to be univalent,
the bicategory of pseudofunctors is univalent as well~\cite[Theorem 9.8]{ahrens:2021},
and thus we can assume that $\tau$ is the identity by \Cref{prop:equivalence-induction}.
We define $\tautau^{-1}(x)$ to be the inverse of $\tautau(x)$
where we use that $\tautau(x)$ is a displayed adjoint equivalence.
In the same way, we define the necessary displayed invertible modifications.
\end{construction}

In \Cref{constr:disp-biequiv-simpl}, we used that the involved bicategories are univalent
and that the involved displayed bicategories are locally preordered and groupoidal.
These assumptions helped to simplify the necessary constructions on a technical level,
because, for instance,
we do not have to verify the coherences necessary for displayed pseudotransformations or modifications.
In addition, we can simplify the notions of displayed pseudofunctor and displayed pseudotransformation
if we assume the involved displayed bicategories are locally preordered and groupoidal.
Specifically, if $\D_2$ is both locally preordered and groupoidal,
then, to construct a displayed pseudofunctor $\FF : \dmor{\D_1}{\D_2}{\F}$, it suffices to give
\begin{itemize}
  \item a map sending $\xx : \dob{\D_1}{x}$ to $\FF(\xx) : \dob{\D_2}{\F(x)}$ for each $x : \B_1$;
  \item a map sending $\ff : \dmor{\xx}{\yy}{f}$ to $\FF(\ff) : \dmor{\FF(\xx)}{\FF(\yy)}{\F(f)}$
    for all objects $\xx$ over $x$ and $\yy$ over $y$, and all 1-cells $f : x \rightarrow y$;
  \item a map sending $\tautau : \dtwo{\ff}{\gg}{\tau}$ to $\FF(\tautau) : \dtwo{\FF(\ff)}{\FF(\gg)}{\F(\tau)}$
    for all 1-cells $\ff$ and $\gg$ over $f$ and $g$ respectively;
  \item a displayed 2-cell $\identitor{\FF}(\xx) : \dtwo{\idd[\FF(\xx)]}{\FF(\idd[\xx])}{\identitor{\F}(x)}$ for each object $\xx : \dob{\D_1}{x}$;
  \item a displayed 2-cell $\compositor{\FF}(\ff, \gg) : \dtwo{\FF(\ff) \cdot \FF(\gg)}{\FF(\ff \cdot \gg)}{\compositor{\F}(f, g)}$
    for all 1-cells $\ff : \dmor{\xx}{\yy}{f}$ and $\gg : \dmor{\yy}{\zz}{g}$.
\end{itemize}
For a displayed pseudotransformation $\tautau : \dtwo{\FF}{\GG}{\tau}$, it suffices to give
\begin{itemize}
  \item a displayed 1-cell $\tautau(\xx) : \dmor{\FF(\xx)}{\GG(\xx)}{\tau(x)}$ for each $\xx$ over $x : \B_1$;
  \item a displayed 2-cell $\tautau(\ff) : \dtwo{\tautau(x) \cdot \GG(\ff)}{\FF(\ff) \cdot \tautau(y)}{\tau(f)}$ for each $\ff : \dmor{\xx}{\yy}{f}$.
\end{itemize}  
If $\D_2$ further has trivial 2-cells, then the necessary displayed 2-cells exist automatically.

In the remainder, we use the following example of a displayed biequivalence.
This example says that if we have a biequivalence from $\D_1$ to $\D_2$ and from $\D_1'$ to $\D_2'$,
then we also have a biequivalence from $\D_1 \times \D_1'$ to $\D_2 \times \D_2'$.

\begin{exa}
\label[exa]{exa:prod-biequiv}
Suppose that we have displayed bicategories $\D_1$ and $\D_1'$ over $\B_1$,
and displayed bicategories $\D_2$ and $\D_2'$ over $\B_2$.
We assume that each of these displayed bicategories is locally preordered and groupoidal,
and we also assume to have a biequivalence $(\F, \G, \eta, \varepsilon)$ from $\B_1$ to $\B_2$.
Finally, suppose that we have displayed biequivalences $(\FF, \GG, \etaeta, \epsilonepsilon)$ from $\D_1$ to $\D_2$ over $(\F, \G, \eta, \varepsilon)$
and $(\FF', \GG', \etaeta', \epsilonepsilon')$ from $\D_1'$ to $\D_2'$ over $(\F, \G, \eta, \varepsilon)$.
Then we also have their \conceptDef{product displayed biequivalence}{Bicategories.DisplayedBicats.ProductDispBiequiv}{prod_disp_is_biequivalence_data} $(\FF \times \FF', \GG \times \GG', \etaeta \times \etaeta', \epsilonepsilon \times \epsilonepsilon')$,
which goes from $\D_1 \times \D_1'$ to $\D_2 \times \D_2'$
and lives over $(\F, \G, \eta, \varepsilon)$.
Both the pseudofunctors $\FF \times \FF'$ and $\GG \times \GG'$,
and the pseudotransformations $\etaeta \times \etaeta'$ and $\epsilonepsilon \times \epsilonepsilon'$
are defined by pairing.
\end{exa}

\subsection{Extending the Biequivalence to Locally Cartesian Closed Categories}
We finish this section by constructing the desired displayed biequivalence between $\LCCCD$ and $\DFLCompCatPiD$
over the biequivalence defined in \Cref{constr:biequiv}.
Recall that this biequivalence consists of pseudofunctors
$\FinLimToCompCat : \FinLim \rightarrow \DFLCompCat$ (\Cref{constr:finlim-to-compcat})
and $\CompCatToFinLim : \DFLCompCat \rightarrow \FinLim$ (\Cref{constr:compcat-to-finlim}),
and pseudotransformations $\FinLimCompCatUnit : \FinLimToCompCat \cdot \CompCatToFinLim \Rightarrow \id[\FinLim]$ (\Cref{constr:unit})
and $\FinLimCompCatCounit : \id[\DFLCompCat] \Rightarrow \CompCatToFinLim \cdot \FinLimToCompCat$ (\Cref{constr:counit}).
To construct the desired displayed biequivalence, we need to construct
\begin{itemize}
  \item a displayed pseudofunctor $\FinLimToCompCatPi$ from $\LCCCD$ to $\DFLCompCatPiD$ over $\FinLimToCompCat$ (\Cref{constr:disp-psfunctor-to-compcat});
  \item a displayed pseudofunctor $\CompCatToFinLimPi$ from $\DFLCompCatPiD$ to $\LCCCD$ over $\CompCatToFinLim$ (\Cref{constr:disp-psfunctor-to-cat});
  \item a displayed pseudotransformation $\FinLimCompCatUnitPi$ from $\FinLimToCompCatPi \cdot \CompCatToFinLimPi$ to $\idd$ over $\FinLimCompCatUnit$ (\Cref{prop:disp-pstrans-unit-pi});
  \item a displayed pseudotransformation $\FinLimCompCatCounitPi$ from $\idd$ to $\CompCatToFinLimPi \cdot \FinLimToCompCatPi$ over $\FinLimCompCatCounit$ (\Cref{prop:disp-pstrans-counit-pi})
\end{itemize}
and we need to verify that $\FinLimCompCatUnitPi$ and $\FinLimCompCatCounitPi$ are pointwise adjoint equivalences.
We start with $\FinLimToCompCatPi$.

\begin{problem}
\label[problem]{prob:disp-psfunctor-to-compcat}
To construct a displayed pseudofunctor $\FinLimToCompCatPi : \dmor{\LCCCD}{\DFLCompCatPiD}{\FinLimToCompCat}$.
\end{problem}

\begin{construction}{\coqdocurl{Bicategories.ComprehensionCat.Biequivalence.PiTypesBiequiv}{finlim_biequiv_dfl_comp_cat_disp_psfunctor_pi_types}}{prob:disp-psfunctor-to-compcat}
\label{constr:disp-psfunctor-to-compcat}
Let $\C$ be a univalent locally Cartesian closed category.
We need to verify that $\id : \ArrD{\C} \rightarrow \ArrD{\C}$ supports $\Pi$-types.
The necessary right adjoints exist, because $\C$ is locally Cartesian closed.
In addition, the Beck-Chevalley condition for $\Pi$-types follows from the Beck--Chevalley condition of $\Sigma$-types~\cite{seely:1983}.
We must also show that for every functor $\F : \C_1 \rightarrow \C_2$ that preserves exponentials,
the displayed functor $\ArrD{\F}$ preserves $\Pi$-types.
This follows by definition.
\end{construction}

\begin{problem}
\label[problem]{prob:disp-psfunctor-to-cat}
To construct a displayed pseudofunctor $\CompCatToFinLimPi : \dmor{\DFLCompCatPiD}{\LCCCD}{\CompCatToFinLim}$.
\end{problem}

\begin{construction}{\coqdocurl{Bicategories.ComprehensionCat.Biequivalence.PiTypesBiequiv}{dfl_comp_cat_to_finlim_disp_psfunctor_pi_types}}{prob:disp-psfunctor-to-cat}
\label{constr:disp-psfunctor-to-cat}
Let $\compfunctor : \D \rightarrow \ArrD{\C}$ be a DFL comprehension category that supports $\Pi$-types.
We need to show that the category $\C$ is locally Cartesian closed,
so we need to prove that the pullback functor $\functorfiber{s} : \slice{\C}{\Gamma} \rightarrow \slice{\C}{\Delta}$ has a right adjoint for every $s : \Delta \rightarrow \Gamma$.
By \Cref{prop:comprehension-eso} we can assume that $s$ is isomorphic to a map $\CtxPr[A]$ for some type $A : \dob{\D}{\Gamma}$,
and thus $\functorfiber{s}$ has a right adjoint since $\compfunctor$ supports $\Pi$-types.
\end{construction}

\begin{propL}[\coqdocurl{Bicategories.ComprehensionCat.Biequivalence.PiTypesBiequiv}{finlim_dfl_comp_cat_unit_pi_types}]
\label[propL]{prop:disp-pstrans-unit-pi}
We have a displayed pseudotransformation $\FinLimCompCatUnitPi : \FinLimToCompCatPi \cdot \CompCatToFinLimPi \Rightarrow \idd$ over $\FinLimCompCatUnit$.
\end{propL}

\begin{proof}
Recall that $\FinLimCompCatUnit$ was defined to be the identity functor pointwise.
The desired pseudotransformation exists, because the identity functor preserves exponentials.
\end{proof}

\begin{propL}[\coqdocurl{Bicategories.ComprehensionCat.Biequivalence.PiTypesBiequiv}{finlim_dfl_comp_cat_counit_pi_types}]
\label[propL]{prop:disp-pstrans-counit-pi}
We have a displayed pseudotransformation $\FinLimCompCatCounitPi : \idd \Rightarrow \CompCatToFinLimPi \cdot \FinLimToCompCatPi$ over $\FinLimCompCatCounit$.
\end{propL}

\begin{proof}
To construct the desired pseudotransformation,
we need to show that $\FinLimCompCatCounit$ preserves $\Pi$-types pointwise.
Recall that $\FinLimCompCatCounit$ is a pointwise adjoint equivalence by \Cref{prop:counit-pointwise-adjequiv}.
Using equivalence induction~(\Cref{prop:equivalence-induction}),
we can show that every adjoint equivalence of DFL comprehension categories preserves $\Pi$-types,
because the identity preserves $\Pi$-types.
From this, we conclude that the desired pseudotransformation exists.
\end{proof}

\begin{problem}
\label[problem]{prob:lccc-biequiv}
To construct a biequivalence between $\LCCC$ and $\DFLCompCatPi$.
\end{problem}

\begin{construction}{\coqdocurl{Bicategories.ComprehensionCat.Biequivalence.PiTypesBiequiv}{internal_language_lccc}}{prob:lccc-biequiv}
\label{constr:lccc-biequiv}
The desired pseudofunctors and pseudotransformations
are constructed in \Cref{constr:disp-psfunctor-to-compcat,constr:disp-psfunctor-to-cat}
and in \Cref{prop:disp-pstrans-unit-pi,prop:disp-pstrans-counit-pi} respectively.
Both $\FinLimCompCatUnitPi$ and $\FinLimCompCatCounitPi$ are pointwise adjoint equivalences by construction.
\end{construction}

\section{The Internal Language of Toposes}
\label{sec:topos}
Up to now we have considered the internal language of univalent categories with finite limits
and univalent locally Cartesian closed categories.
In this section we extend these internal language theorems to various classes of toposes,
among which are pretoposes, $\Pi$-pretoposes, and elementary toposes.
Note that each of these classes is either an extension of categories with finite limits or of locally Cartesian closed categories,
so,
compared to those,
their internal languages come with additional type formers.
Compared to categories with finite limits,
pretoposes extend the internal language with the empty type, disjoint sum types, and quotients,
whereas elementary toposes extend the internal language of locally Cartesian closed categories with a subobject classifier type.

Each of the aforementioned type formers is treated similarly in the semantics.
For instance, a full comprehension category $\compfunctor : \D \rightarrow \ArrD{\C}$ supports binary sum types if each fiber $\fiber{\D}{\Gamma}$ has binary coproducts,
and if the functor $\functorfiber{s} : \fiber{\D}{\Delta} \rightarrow \fiber{\D}{\Gamma}$ preserves binary coproducts for each $s : \Gamma \rightarrow \Delta$.
The type of natural numbers is treated the same: a comprehension category $\compfunctor : \D \rightarrow \ArrD{\C}$ supports a natural numbers type
if each fiber $\fiber{\D}{\Gamma}$ has a natural numbers object,
and if $\functorfiber{s} : \fiber{\D}{\Delta} \rightarrow \fiber{\D}{\Gamma}$  preserves natural numbers objects for each $s : \Gamma \rightarrow \Delta$.
We can say the same about empty types, subobject classifier types, and quotient types.
Hence, a full comprehension category supports such a type former if each fiber has some categorical structure,
and if the substitution functors preserve that structure.

This idea was used by Maietti to prove soundness and completeness theorems for various dependent type theories~\cite[Theorems 5.8 and 5.28]{maietti:2005}.
More specifically, Maietti defines the notion of \textbf{local property}~\cite[Proposition 2.13]{maietti:2005},
which encapsulates the pattern we mentioned before,
and interprets numerous type formers, like quotients and disjoint sum types, with suitable local properties.
Specifically,
a local property is a property of categories that is closed under slicing.
Since a wide variety of categorical structures can be captured using local properties,
this gives a powerful method to study the semantics of various flavors of type theory.

Our goal in this section is to extend the biequivalence defined in \Cref{constr:biequiv} along local properties.
To see what that entails, let us consider an example.
Recall that a category $\C$ with finite limits and finite coproducts is called \textbf{extensive}
if coproducts are disjoint and stable under pullback.
Extensiveness gives an instance of a local property, since this property is closed under slicing.
Using this property we can restrict $\FinLim$ to the bicategory of univalent extensive categories,
and $\DFLCompCat$ to the bicategory DFL comprehension categories for which each fiber is extensive
and each substitution functor preserves binary coproducts.
We show in this section that these two bicategories are biequivalent,
and we prove a general statement using arbitrary local properties.
To do so, we use the same methods as in \Cref{sec:lccc}.

We start this section by discussing local properties.
More specifically, we define this notion
and we extend the biequivalence from \Cref{constr:biequiv} along local properties.
After that we give numerous examples of local properties,
and we instantiate the extended biequivalence to obtain internal language theorems for various classes of toposes.

\subsection{Local Properties}
Now we define local properties.
The idea is that a local property is a property of categories that is closed under slicing.

\begin{defi}
\label[defi]{def:local-property}
A \conceptDef{local property}{Bicategories.ComprehensionCat.LocalProperty.LocalProperties}{local_property} $\LocalProp$ is given by
\begin{itemize}
  \item a proposition $\LocalPropCat{\LocalProp}{\C}$ for each univalent category $\C$ with finite limits;
  \item a proposition $\LocalPropFunctor{\LocalProp}{\F}$ for each functor $\F : \C_1 \rightarrow \C_2$
    such that $\F$ preserves finite limits
    and such that $\LocalPropCat{\LocalProp}{\C_1}$ and $\LocalPropCat{\LocalProp}{\C_2}$ hold.
\end{itemize}
We require the following axioms on this data.
\begin{itemize}
  \item For each univalent category $\C$ with finite limits we have $\LocalPropFunctor{\LocalProp}{\id[\C]}$ if $\LocalPropCat{\LocalProp}{\C}$ holds.
  \item Given functors $\F : \C_1 \rightarrow \C_2$ and $\G : \C_2 \rightarrow \C_3$ that preserve finite limits,
    we have $\LocalPropFunctor{\LocalProp}{\F \cdot \G}$ if both $\LocalPropFunctor{\LocalProp}{\F}$ and $\LocalPropFunctor{\LocalProp}{\G}$ hold.
  \item If $\C$ is a univalent category with finite limits such that $\LocalPropCat{\LocalProp}{\C}$, we have $\LocalPropCat{\LocalProp}{\slice{\C}{x}}$ for each object $x : \C$.
  \item If $\C$ is a univalent category with finite limits such that $\LocalPropCat{\LocalProp}{\C}$, we have $\LocalPropFunctor{\LocalProp}{f^*}$ for each morphism $f : x \rightarrow y$.
  \item If $\F : \C_1 \rightarrow \C_2$ is a functor preserving finite limits such that $\LocalPropFunctor{\LocalProp}{\F}$,
    then we also have $\LocalPropFunctor{\LocalProp}{\slice{\F}{x}}$ for each $x : \C_1$.
\end{itemize}
\end{defi}

There are several remarks to make on \Cref{def:local-property}.
First, we require $\LocalPropCat{\LocalProp}{\C}$ and $\LocalPropFunctor{\LocalProp}{\F}$ to be propositions.
This requirement is satisfied for all our examples of interest, because we work with univalent categories.
For instance, we have a local property that expresses that a category has a subobject classifier.
The type of subobject classifiers in a univalent category is a proposition, because subobject classifiers are unique up to isomorphism,
and thus unique if we assume our category to be univalent.
If our focus had been on categories that are not necessarily univalent,
then we would need to drop the requirement that $\LocalPropCat{\LocalProp}{\C}$ is a proposition,
because, for instance, chosen subobject classifiers are not guaranteed to be unique.
In addition, one would need to add the requirements to \Cref{def:local-property} saying that $\LocalPropCat{\LocalProp}{\C}$ and $\LocalPropFunctor{\LocalProp}{\F}$
are closed under adjoint equivalence and natural isomorphism respectively
to get a suitable notion of local property.

Second,
since we are working with univalent categories,
a local property is automatically invariant under equivalence.
More specifically,
if we have an adjoint equivalence $\C_1 \simeq \C_2$ between univalent categories,
then $\LocalPropCat{\LocalProp}{\C_1}$ holds if and only if $\LocalPropCat{\LocalProp}{\C_2}$ holds.
We can say the same for functors.
More specifically, suppose that we have a diagram as follows.
\[
\begin{tikzcd}
  {\C_1} & {\C_1'} \\
  {\C_2} & {\C_2'}
  \arrow[""{name=0, anchor=center, inner sep=0}, "\simeq", from=1-1, to=1-2]
  \arrow["\F"', from=1-1, to=2-1]
  \arrow["\G", from=1-2, to=2-2]
  \arrow[""{name=1, anchor=center, inner sep=0}, "\simeq"', from=2-1, to=2-2]
  \arrow["\cong"{description}, shorten <=4pt, shorten >=4pt, Rightarrow, from=0, to=1]
\end{tikzcd}
\]
Then $\LocalPropFunctor{\LocalProp}{\F}$ holds if and only if $\LocalPropFunctor{\LocalProp}{\G}$ holds.
Both these statements follow from \Cref{prop:equivalence-induction}.
If we did not assume univalence,
then we would need to add both these requirements to \Cref{def:local-property},
and we would need to verify them for each instance of a local property.

Finally, our definition is slightly different compared to the one by Maietti.
We require the functors $\slice{\F}{x} : \slice{\C_1}{x} \rightarrow \slice{\C_2}{\F(x)}$ to be structure preserving whenever $\F$ is,
while Maietti does not add this requirement.
The reason for this difference is that our focus is on biequivalences,
and thus we need to take morphisms into account.

Next we define the necessary displayed bicategories.

\begin{defi}
\label[defi]{def:local-property-disp-bicat}
Let $\LocalProp$ be a local property.
We define the displayed bicategory \conceptDef{$\LPropD{\LocalProp}$}{Bicategories.ComprehensionCat.LocalProperty.CatWithProp}{bicat_of_univ_cat_with_cat_property} over $\FinLim$ to be $\subbicatD{\LocalPropCatT{\LocalProp}}{\LocalPropFunctorT{\LocalProp}}$.

A DFL comprehension category $\compfunctor : \D \rightarrow \ArrD{\C}$ satisfies $\LocalProp$
if we have $\LocalPropCat{\LocalProp}{\fiber{\D}{\Gamma}}$ for each $\Gamma : \C$
and we have $\LocalPropFunctor{\LocalProp}{\functorfiber{s}}$ for each $s : \Gamma \rightarrow \Delta$.
A 1-cell $(\F, \FF, \F_\chi)$ preserves $\LocalProp$
if we have $\LocalPropFunctor{\LocalProp}{\fiberfunctor{\FF}{\Gamma}}$ for each $\Gamma$.
Now we define the displayed bicategory \conceptDef{$\DFLCompCatPropD{\LocalProp}$}{Bicategories.ComprehensionCat.LocalProperty.DFLCompCatWithProp}{disp_bicat_of_cat_property_dfl_full_comp_cat} over $\DFLCompCat$ to be $\subbicatD{\Qred}{\Qred'}$.
Here $\Qred$ says that a DFL comprehension category satisfies $\LocalProp$
and $\Qred'$ says that a 1-cell preserves $\LocalProp$.
\end{defi}

An object over a univalent category $\C$ with finite limits in $\LPropD{\LocalProp}$ is thus a proof that $\LocalPropCat{\LocalProp}{\C}$ holds,
and a morphism over a functor $\F$ that preserves finite limits is a proof of $\LocalPropFunctor{\LocalProp}{\F}$.
Objects in $\DFLCompCatPropD{\LocalProp}$ are DFL comprehension categories that satisfy $\LocalProp$ fiberwise,
and morphisms over $(\F, \FF, \F_\chi)$ are proofs that for each $\Gamma$ we have $\LocalPropFunctor{\LocalProp}{\fiberfunctor{\FF}{\Gamma}}$.
Note that both $\LPropD{\LocalProp}$ and $\DFLCompCatPropD{\LocalProp}$ are univalent.
This is because we define them as subbicategories,
and because local properties are valued in propositions.

Our goal is to construct a displayed biequivalence between $\LPropD{\LocalProp}$ and $\DFLCompCatPropD{\LocalProp}$
over the biequivalence constructed in \Cref{constr:biequiv}.

\begin{problem}
\label[problem]{prob:local-property-biequiv}
To construct a displayed biequivalence between $\LPropD{\LocalProp}$ and $\DFLCompCatPropD{\LocalProp}$
over the biequivalence $(\FinLimToCompCat, \CompCatToFinLim, \FinLimCompCatUnit, \FinLimCompCatCounit)$
for each local property $\LocalProp$.
\end{problem}

\begin{construction}{\coqdocurl{Bicategories.ComprehensionCat.Biequivalence.LocalProperty}{finlim_biequiv_dfl_comp_cat_psfunctor_local_property}}{prob:local-property-biequiv}
\label{constr:local-property-biequiv}
We first construct a displayed pseudofunctor
from $\LPropD{\LocalProp}$ to $\DFLCompCatPropD{\LocalProp}$
over $\FinLimToCompCat$.
To do so, we must show that the comprehension category $\id : \ArrD{\C} \rightarrow \ArrD{\C}$ satisfies $\LocalProp$.
This is so, because local properties are closed under slicing.

Next we construct a displayed pseudofunctor
from $\DFLCompCatPropD{\LocalProp}$ to $\LPropD{\LocalProp}$
over $\CompCatToFinLim$.
Concretely, we show that whenever we have some DFL comprehension category $\compfunctor : \D \rightarrow \ArrD{\C}$ that satisfies $\LocalProp$,
then we have $\LocalPropCat{\LocalProp}{\C}$.
This follows from the fact that the property $\LocalPropCatT{\LocalProp}$ is preserved under equivalence,
and because we have $\LocalPropCat{\LocalProp}{\fiber{\D}{\ECtx}}$.
Analogously,
we show for each 1-cell $(\F, \FF, \F_\chi)$ of DFL comprehension categories
we have that $\LocalPropFunctor{\LocalProp}{\F}$.

Finally, we must construct displayed pseudotransformations over $\FinLimCompCatUnit$ and $\FinLimCompCatCounit$.
For $\FinLimCompCatUnit$,
it follows from the fact that we have $\LocalPropFunctor{\LocalProp}{\id}$.
For $\FinLimCompCatCounit$,
we use that every adjoint equivalence satisfies $\LocalPropFunctorT{\LocalProp}$,
which follows from \Cref{prop:equivalence-induction}
and the fact that we have $\LocalPropFunctor{\LocalProp}{\id}$.
\end{construction}

\subsection{Extending the Biequivalence to Toposes}
Now our goal is to instantiate \Cref{constr:local-property-biequiv} to various classes of toposes,
and to do so, we first define the local properties necessary to define these classes of toposes.
We start with the conjunction on local properties.

\begin{exa}
\label[exa]{exa:local-prop-comb}
Given local properties $\LocalProp$ and $\LocalProp'$,
we define their conjunction \conceptDef{$\LocalAnd{\LocalProp}{\LocalProp'}$}{Bicategories.ComprehensionCat.LocalProperty.Examples}{local_property_conj} as follows.
\begin{itemize}
  \item The type $\LocalPropCat{(\LocalAnd{\LocalProp}{\LocalProp'})}{\C}$ is defined to be $\LocalPropCat{\LocalProp}{\C} \land \LocalPropCat{\LocalProp'}{\C}$.
  \item A functor $\F$ satisfies $\LocalPropFunctor{(\LocalAnd{\LocalProp}{\LocalProp'})}{\F}$
    if we have both $\LocalPropFunctor{\LocalProp}{\F}$ and $\LocalPropFunctor{\LocalProp'}{\F}$.
\end{itemize}
\end{exa}

Next we look at several type formers.
The first one is given by finite coproducts.
To interpret the empty type, we use strict initial objects,
and we say that an initial object is strict if every morphism into it is an isomorphism.
Binary coproduct types can be interpreted in a category if that category has binary coproducts
that are stable under pullback.
Another property of interest is \textbf{extensiveness}~\cite[Definition 2.1]{carboni:1993},
which allows us to interpret finite disjoint coproducts.
Recall that a category is \textbf{extensive} if it has finite coproducts
that are stable under pullback
and that are disjoint.
We say that a coproduct is disjoint if the pullback of the two coproduct inclusions is an initial object.

\begin{exa}
\label[exa]{exa:fin-coprod}
We define the local property \conceptDef{$\StrictInitial$}{Bicategories.ComprehensionCat.LocalProperty.Examples}{strict_initial_local_property}
such that a category $\C$ satisfies $\LocalPropCatT{\StrictInitial}$ if $\C$ has a strict initial object,
and a functor $\F$ satisfies $\LocalPropFunctorT{\StrictInitial}$ if $\F$ preserves initial objects.

Next we define the local property \conceptDef{$\StableCoproduct$}{Bicategories.ComprehensionCat.LocalProperty.Examples}{stable_bincoproducts_local_property}.
A category $\C$ satisfies $\LocalPropCatT{\StableCoproduct}$ if $\C$ has binary coproducts that are stable under pullback,
and a functor $\F$ satisfies $\LocalPropFunctorT{\StableCoproduct}$ if $\F$ preserves binary coproducts.

Finally, we define the local property \conceptDef{$\Extensive$}{Bicategories.ComprehensionCat.LocalProperty.Examples}{lextensive_local_property}.
Categories satisfy $\LocalPropCatT{\Extensive}$ if they are extensive,
and functors satisfy $\LocalPropFunctorT{\Extensive}$ if they preserve finite coproducts.
\end{exa}

The next type former of interest is given by quotient types.
We can interpret these in \textbf{exact categories}~\cite{barr:1971}.
Recall that a morphism $f : x \rightarrow y$ in a category $\C$ is a \textbf{regular epimorphism}~\cite[Definition 4.3.1]{borceux:1994a}
if there exists $w : \C$ and morphisms $g_1, g_2 : w \rightarrow x$ such that $f$ is the coequalizer of $g_1$ and $g_2$.
A category with finite limits is said to be \textbf{regular}~\cite[Definition 2.1.1]{borceux:1994a}
if regular epimorphisms are stable under pullback and
if for every morphism $f : x \rightarrow y$ the kernel pair, which is the pullback of $f$ along $f$, has a coequalizer.
An internal relation on $x$ is given by a monomorphism $m : r \rightarrow x \times x$,
and each internal relation gives rise to a relation $\relation{r}{w}$ on the type $w \rightarrow x$ of morphisms for each $w$.
This relation relates $f, g : w \rightarrow x$ if there is $k : w \rightarrow r$ making the following diagram commute.
\[
\begin{tikzcd}
  & w \\
  & r \\
  x & {x \times x} & x
  \arrow["k"{description}, dashed, from=1-2, to=2-2]
  \arrow["f"', from=1-2, to=3-1]
  \arrow["g", from=1-2, to=3-3]
  \arrow["m"{description}, from=2-2, to=3-2]
  \arrow["{\pi_1}", from=3-2, to=3-1]
  \arrow["{\pi_2}"', from=3-2, to=3-3]
\end{tikzcd}
\]
An internal equivalence relation $m : r \rightarrow x \times x$ is an internal relation for which each relation $\relation{r}{w}$ is an equivalence relation.
We say that a regular category is \textbf{exact}~\cite[Definition 2.6.1]{borceux:1994a} if
for every internal equivalence relation $m : r \rightarrow x \times x$
we have a morphism $f : x \rightarrow y$ such that $m$ is the kernel pair of $f$.

\begin{exa}
\label[exa]{exa:exact}
We define the local property \conceptDef{$\Regular$}{Bicategories.ComprehensionCat.LocalProperty.Examples}{regular_local_property} such that
categories satisfy $\LocalPropCatT{\Regular}$ if they are regular
and functors satisfy $\LocalPropFunctorT{\Regular}$ if they preserve regular epimorphisms.

The local property \conceptDef{$\Exact$}{Bicategories.ComprehensionCat.LocalProperty.Examples}{exact_local_property} is defined analogously.
Categories satisfy $\LocalPropCatT{\Exact}$ if they are exact
and functors satisfy $\LocalPropFunctorT{\Exact}$ if they preserve regular epimorphisms.
\end{exa}

Next we look at the local property for \textbf{subobject classifiers},
which represent an impredicative universe of all propositions.
Recall that a subobject classifier is given by a morphism that classifies all monomorphisms~\cite{johnstone:2002}.
Specifically,
a subobject classifier in a category $\C$ with a terminal object $\ECtx$ is given by an object $\Omega$ with a morphism $\top : \ECtx \rightarrow \Omega$
such that for every monomorphism $m : x \rightarrow y$ there is a unique $\chi_m : y \rightarrow \Omega$
making the following square a pullback.
\[
\begin{tikzcd}
  x & \ECtx \\
  y & \Omega
  \arrow[from=1-1, to=1-2]
  \arrow["m"', from=1-1, to=2-1]
  \arrow["\lrcorner"{anchor=center, pos=0.125}, draw=none, from=1-1, to=2-2]
  \arrow["\top", from=1-2, to=2-2]
  \arrow["{\chi_m}"', from=2-1, to=2-2]
\end{tikzcd}
\]

\begin{exa}
\label[exa]{exa:subobj-class}
We define the local property \conceptDef{$\SubobjClass$}{Bicategories.ComprehensionCat.LocalProperty.Examples}{subobject_classifier_local_property} such that
$\LocalPropCat{\SubobjClass}{\C}$ is the type of subobject classifiers in some category $\C$
and $\F$ satisfies $\LocalPropFunctor{\SubobjClass}{\F}$ if $\F$ preserves subobject classifiers.
\end{exa}

To see how subobject classifiers give rise to a universe of propositions,
we show that propositions in a DFL comprehension category $\compfunctor : \D \rightarrow \ArrD{\C}$ correspond to monomorphisms in $\C$.
This observation was already made by Maietti~\cite{maietti:2005},
who showed that internal propositions in the syntactic category correspond to monomorphisms.
Recall that a type $A$ is said to be a proposition if for all $x, y : A$ we have $x = y$.
We translate that notion to comprehension categories as follows.

\begin{defi}
\label[defi]{def:hprop-compcat}
Let $\compfunctor : \D \rightarrow \ArrD{\C}$ be a DFL comprehension category,
and let $A : \dob{\D}{\Gamma}$ be a type in context $\Gamma$.
Note that we have two terms of type $\substTy{\CtxPr}{\substTy{\CtxPr}{A}}$ in context $\CtxExt{\CtxExt{\Gamma}{A}}{\substTy{\CtxPr}{A}}$,
namely $l \defeq \var$ and $r \defeq \substTm{\CtxPr}{\var}$.
We say that $A$ is an \conceptDef{internal proposition}{Bicategories.ComprehensionCat.HPropMono}{is_hprop_ty}
if we have a term $t : \Tms{\CtxExt{\CtxExt{\Gamma}{A}}{\substTy{\CtxPr}{A}}}{\ExtId{l}{r}}$.
\end{defi}

The notion of internal proposition expresses that we have a term $\Gamma , x : A , y : A \vdash \ExtId{x}{y}$.
To show that internal propositions correspond to monomorphisms,
we first show that internal propositions can be characterized via their display maps.
Specifically, a type $A$ is an internal proposition if and only if $\CtxPr[A]$ is a monomorphism.
Throughout that proof, we use the concrete definitions of the underlying morphisms of $l$ and $r$ in \Cref{def:hprop-compcat}.
Recall that we defined $l$ as the following morphism.
\[
\begin{tikzcd}
  {\CtxExt{\CtxExt{\Gamma}{A}}{\substTy{\CtxPr}{A}}} \\
  & {\CtxExt{\CtxExt{\CtxExt{\Gamma}{A}}{\substTy{\CtxPr}{A}}}{\substTy{\CtxPr}{\substTy{\CtxPr}{A}}}} & {\CtxExt{\CtxExt{\Gamma}{A}}{\substTy{\CtxPr}{A}}} \\
  & {\CtxExt{\CtxExt{\Gamma}{A}}{\substTy{\CtxPr}{A}}} & {\CtxExt{\Gamma}{A}}
  \arrow["l"{description}, dashed, from=1-1, to=2-2]
  \arrow["\id", bend left=12, from=1-1, to=2-3]
  \arrow["\id"', bend right=20, from=1-1, to=3-2]
  \arrow["\compfunctor(\substTyMor{\CtxPr})", from=2-2, to=2-3]
  \arrow["{\CtxPr}"', from=2-2, to=3-2]
  \arrow["{\CtxPr}", from=2-3, to=3-3]
  \arrow["{\CtxPr}"', from=3-2, to=3-3]
  \arrow["\lrcorner"{anchor=center, pos=0.125}, draw=none, from=2-2, to=3-3]
\end{tikzcd}
\]
The term $r$ is defined as follows
\[
\begin{tikzcd}
  {\CtxExt{\CtxExt{\Gamma}{A}}{\substTy{\CtxPr}{A}}} & {\CtxExt{\Gamma}{A}} \\
  & {\CtxExt{\CtxExt{\CtxExt{\Gamma}{A}}{\substTy{\CtxPr}{A}}}{\substTy{\CtxPr}{\substTy{\CtxPr}{A}}}} & {\CtxExt{\CtxExt{\Gamma}{A}}{\substTy{\CtxPr}{A}}} \\
  & {\CtxExt{\CtxExt{\Gamma}{A}}{\substTy{\CtxPr}{A}}} & {\CtxExt{\Gamma}{A}}
  \arrow["{\CtxPr}", from=1-1, to=1-2]
  \arrow["{r}"{description}, dashed, from=1-1, to=2-2]
  \arrow["\id"', bend right=30, from=1-1, to=3-2]
  \arrow["r'", bend left=10, from=1-2, to=2-3]
  \arrow["\compfunctor(\substTyMor{\CtxPr})", from=2-2, to=2-3]
  \arrow["{\CtxPr}"', from=2-2, to=3-2]
  \arrow["{\CtxPr}", from=2-3, to=3-3]
  \arrow["\CtxPr"', from=3-2, to=3-3]
  \arrow["\lrcorner"{anchor=center, pos=0.125}, draw=none, from=2-2, to=3-3]
\end{tikzcd}
\]
where we defined $r'$ to be the following morphism.
\[
\begin{tikzcd}
  {\CtxExt{\Gamma}{A}} \\
  & {\CtxExt{\CtxExt{\Gamma}{A}}{\substTy{\CtxPr}{A}}} & {\CtxExt{\Gamma}{A}} \\
  & {\CtxExt{\Gamma}{A}} & \Gamma
  \arrow["r'"{description}, dashed, from=1-1, to=2-2]
  \arrow["\id", bend left=20, from=1-1, to=2-3]
  \arrow["\id"', bend right=20, from=1-1, to=3-2]
  \arrow["\compfunctor(\substTyMor{\CtxPr})", from=2-2, to=2-3]
  \arrow["{\CtxPr}"', from=2-2, to=3-2]
  \arrow["{\CtxPr}", from=2-3, to=3-3]
  \arrow["{\CtxPr}"', from=3-2, to=3-3]
  \arrow["\lrcorner"{anchor=center, pos=0.125}, draw=none, from=2-2, to=3-3]
\end{tikzcd}
\]

\begin{propL}[\coqdocurl{Bicategories.ComprehensionCat.HPropMono}{is_hprop_ty_weq_mono_ty}]
\label[propL]{prop:hprop-mono}
Let $\compfunctor : \D \rightarrow \ArrD{\C}$ be a DFL comprehension category.
A type $A$ is an internal proposition if and only if $\CtxPr[A]$ is a monomorphism.
\end{propL}

\begin{proof}
First, we assume that $A$ is an internal proposition.
Suppose that we have morphisms $s_1, s_2 : \Delta \rightarrow \CtxExt{\Gamma}{A}$ such that $s_1 \cdot \CtxPr = s_2 \cdot \CtxPr$.
We need to show that $s_1 = s_2$.
Since $\compfunctor$ preserves Cartesian morphisms,
we have the morphism $f$, depicted below, by the universal mapping property of the pullback.
\[
\begin{tikzcd}
  \Delta \\
  & {\CtxExt{\CtxExt{\Gamma}{A}}{\substTy{\CtxPr}{A}}} & {\CtxExt{\Gamma}{A}} \\
  & {\CtxExt{\Gamma}{A}} & \Gamma
  \arrow["{s_2}", from=1-1, to=2-3,bend left=20]
  \arrow["{s_1}"', from=1-1, to=3-2,bend right=20]
  \arrow["{\compfunctor(\substTyMor{\CtxPr})}", from=2-2, to=2-3]
  \arrow["f"{description}, dashed, from=1-1, to=2-2]
  \arrow["\CtxPr"', from=2-2, to=3-2]
  \arrow["\CtxPr", from=2-3, to=3-3]
  \arrow["\CtxPr"', from=3-2, to=3-3]
\end{tikzcd}
\]
In addition,
since $A$ is an internal proposition,
we have that $l = r$ by the reflection rule.
Hence, the following compositions of morphisms are equal.
\[
\begin{tikzcd}
  \Delta & {\CtxExt{\CtxExt{\Gamma}{A}}{\substTy{\CtxPr}{A}}} & {\CtxExt{\CtxExt{\CtxExt{\Gamma}{A}}{\substTy{\CtxPr}{A}}}{\substTy{\CtxPr}{\substTy{\CtxPr}{A}}}} & {\CtxExt{\CtxExt{\Gamma}{A}}{\substTy{\CtxPr}{A}}} & {\CtxExt{\Gamma}{A}}
  \arrow["f", from=1-1, to=1-2]
  \arrow["l", shift left=2, from=1-2, to=1-3]
  \arrow["r"', shift right=2, from=1-2, to=1-3]
  \arrow["{\compfunctor(\substTyMor{\CtxPr})}", from=1-3, to=1-4]
  \arrow["{\compfunctor(\substTyMor{\CtxPr})}", from=1-4, to=1-5]
\end{tikzcd}
\]
Since $f \cdot l \cdot \compfunctor(\substTyMor{\CtxPr}) \cdot \compfunctor(\substTyMor{\CtxPr}) = s_2$
and $f \cdot r \cdot \compfunctor(\substTyMor{\CtxPr}) \cdot \compfunctor(\substTyMor{\CtxPr}) = s_1$,
it follows that $s_1 = s_2$.

Next we assume that $\CtxPr[A]$ is a monomorphism,
and our goal is to construct a term $p : \Tms{\CtxExt{\CtxExt{\Gamma}{A}}{\substTy{\CtxPr}{A}}}{\ExtId{l}{r}}$.
It suffices to show that $l = r$,
and by the universal mapping property of the pullback,
we need to show that $l \cdot \compfunctor(\substTyMor{\CtxPr}) = r \cdot \compfunctor(\substTyMor{\CtxPr})$
and that $l \cdot \CtxPr = r \cdot \CtxPr$.
Note that we have $l \cdot \CtxPr = r \cdot \CtxPr$, because both compositions are the identity.
To prove that $l \cdot \compfunctor(\substTyMor{\CtxPr}) = r \cdot \compfunctor(\substTyMor{\CtxPr})$,
we use that $l \cdot \compfunctor(\substTyMor{\CtxPr}) = \id$
and $r \cdot \compfunctor(\substTyMor{\CtxPr}) = \CtxPr \cdot r'$,
so we need to show that $\CtxPr \cdot r' = \id$.
Again we use the universal mapping property of the pullback,
and it suffices to show that $\CtxPr \cdot r' \cdot \compfunctor(\substTyMor{\CtxPr}) = \compfunctor(\substTyMor{\CtxPr})$
and $\CtxPr \cdot r' \cdot \CtxPr = \CtxPr$.
The second equality follows, because $r' \cdot \CtxPr$ is an identity.
Since $r' \cdot \compfunctor(\substTyMor{\CtxPr})$ is an identity as well
and since we assumed $\CtxPr[A]$ to be a monomorphism,
it is sufficient to show that the following compositions are equal.
\[
\begin{tikzcd}
  {\CtxExt{\CtxExt{\Gamma}{A}}{\substTy{\CtxPr}{A}}} & {\CtxExt{\Gamma}{A}} & \Gamma
  \arrow["\CtxPr", shift left=2, from=1-1, to=1-2]
  \arrow["{\compfunctor(\substTyMor{\CtxPr})}"', shift right=2, from=1-1, to=1-2]
  \arrow["{\CtxPr[A]}", from=1-2, to=1-3]
\end{tikzcd}
\]
This equality follows from the fact that $\compfunctor(\substTyMor{\CtxPr})$ is a morphism in the arrow category.
\end{proof}

\begin{propL}[\coqdocurl{Bicategories.ComprehensionCat.HPropMono}{subsingleton_weq_mono_ty}]
\label[propL]{prop:hprop-fib-mono}
Let $\compfunctor : \D \rightarrow \ArrD{\C}$ be a DFL comprehension category,
and let $A : \dob{\D}{\Gamma}$ be a type.
The morphism $\CtxPr[A]$ is a monomorphism if and only if the unique map from $A$ to $\FibTerm[\Gamma]$
in the fiber category $\fiber{\D}{\Gamma}$ is a monomorphism.
\end{propL}

\begin{proof}
We write $f$ for the unique map from $A$ to $\FibTerm[\Gamma]$.
Note that we have a commuting triangle as follows. 
\[
\begin{tikzcd}
  {\CtxExt{\Gamma}{A}} && {\CtxExt{\Gamma}{\FibTerm[\Gamma]}} \\
  & \Gamma
  \arrow["{\compfunctor(f)}", from=1-1, to=1-3]
  \arrow["{\CtxPr[A]}"', from=1-1, to=2-2]
  \arrow["{\CtxPr[{\FibTerm[\Gamma]}]}", from=1-3, to=2-2]
\end{tikzcd}
\]
Since $\CtxPr[{\FibTerm[\Gamma]}]$ is an isomorphism,
$\CtxPr[A]$ is a monomorphism if and only if $\compfunctor(f)$ is one.
By definition $\compfunctor$ is fully faithful, and hence it reflects monomorphisms.
Since $\compfunctor$ is essentially surjective by \Cref{prop:comprehension-eso},
it also is an adjoint equivalence,
and thus it preserves monomorphisms.
Hence, $f$ is a monomorphism if and only if $\compfunctor(f)$ is one.
The desired proposition now follows from the fact that $\compfunctor(f)$ is a monomorphism if and only if $\CtxPr[A]$ is one.
\end{proof}

If a DFL comprehension category $\compfunctor : \D \rightarrow \ArrD{\C}$ has a subobject classifier type $\Omega$,
then terms $\Tms{\Gamma}{\Omega}$ correspond to morphisms $\FibTerm[\Gamma] \rightarrow \Omega$ in $\fiber{\D}{\Gamma}$.
By the universal property of subobject classifier,
morphisms $\FibTerm[\Gamma] \rightarrow \Omega$ correspond to monomorphisms into $\FibTerm[\Gamma]$.
By \Cref{prop:hprop-mono,prop:hprop-fib-mono} such monomorphisms correspond to internal propositions.

Finally, we look at the types of natural numbers.
Such types can be interpreted in categories with a \textbf{parameterized natural numbers object}~\cite[Definition 2.1]{maietti:2010}.
A parameterized natural numbers object in a category $\C$ with binary products and a terminal object $\ECtx$
is given by an object $n : \C$ together with morphisms $z : \ECtx \rightarrow n$ and $s : n \rightarrow n$
such that for all objects $b, y : \C$ together with morphisms $z' : b \rightarrow y$ and $s' : y \rightarrow y$
there is a unique $f : b \times n \rightarrow y$ making the following diagram commute.
\[
\begin{tikzcd}[column sep = large]
  b & {b \times n} & {b \times n} \\
  & y & y
  \arrow["{\langle \id , \ECtxMap \cdot z \rangle}", from=1-1, to=1-2]
  \arrow["{z'}"', from=1-1, to=2-2]
  \arrow["f"', from=1-2, to=2-2]
  \arrow["{\id \times s}"', from=1-3, to=1-2]
  \arrow["f", from=1-3, to=2-3]
  \arrow["{s'}", from=2-3, to=2-2]
\end{tikzcd}
\]
Note that parameterized natural numbers objects satisfy a stronger universal mapping property than ordinary ones,
and that the latter only gives a suitable notion in Cartesian closed categories.
For this reason, we use the parameterized version.

\begin{exa}
\label[exa]{exa:nno-list}
We define the local property \conceptDef{$\NNO$}{Bicategories.ComprehensionCat.LocalProperty.Examples}{parameterized_NNO_local_property}
such that $\LocalPropCat{\NNO}{\C}$ is the type of parameterized natural numbers objects in $\C$
and a functor $\F$ satisfies $\LocalPropFunctor{\NNO}{\F}$ if $\F$ preserves parameterized natural numbers objects.
\end{exa}

\begin{rem}
\label[rem]{rem:pi-local}
Being locally Cartesian closed gives us yet another local property~\cite{maietti:2005}.
This is because a category $\C$ is locally Cartesian closed
if and only if every slice category $\slice{\C}{x}$ is Cartesian closed.
While we could use this observation to directly establish an internal language theorem for locally Cartesian closed categories,
we chose not to do so.
If we were to use local properties to establish such a theorem,
then we would get a non-standard presentation of $\Pi$-types in comprehension categories:
instead of representing $\Pi$-types as right adjoints to weakening,
we would get that each fiber category is locally Cartesian closed.
The work we did in \Cref{sec:lccc} allowed us to connect locally Cartesian closed categories
to comprehension categories with a standard presentation of $\Pi$-types.
\end{rem}

Now we are ready to define bicategories of various classes of toposes.
These are
\textbf{pretoposes}~\cite{moerdijk:2002},
\textbf{arithmetic pretoposes}~\cite[Definition 2.6]{maietti:2010},
\textbf{$\Pi$-pretoposes}~\cite{moerdijk:2002},
\textbf{elementary toposes}~\cite{maclane:1994},
and \textbf{elementary toposes with NNO}~\cite{maclane:1994}.
We define these classes using displayed bicategories and by defining suitable local properties.

\begin{defi}
\label[defi]{def:bicat-topos}
We define the bicategories of various classes of toposes and of their corresponding comprehension categories.
In \Cref{tab:topos-local-prop}, we define the local properties that we use.
The bicategories of toposes and comprehension categories are defined in \Cref{tab:topos-bicat,tab:topos-comp-cat-bicat}.
Each of these is defined as the total bicategory of the displayed bicategory in the right column.
\end{defi}

\begin{table}[]
\begin{center}
\begin{tabular}{c | c}
Local Property                & Definition                                                                 \\
\hline
\conceptDef{$\PretopP$}{Bicategories.ComprehensionCat.LocalProperty.Examples}{pretopos_local_property}   & $\LocalAnd{\Extensive}{\Exact}$                                            \\
\conceptDef{$\APretopP$}{Bicategories.ComprehensionCat.LocalProperty.Examples}{pretopos_with_nat_local_property}  & $\LocalAnd{\PretopP}{\NNO}$                                                \\
\conceptDef{$\EltopP$}{Bicategories.ComprehensionCat.LocalProperty.Examples}{topos_local_property}    & $\LocalAnd{\PretopP}{\SubobjClass}$                                        \\
\conceptDef{$\EltopNNOP$}{Bicategories.ComprehensionCat.LocalProperty.Examples}{topos_with_NNO_local_property} & $\LocalAnd{\EltopP}{\NNO}$
\end{tabular}
\end{center}
\caption{Local properties for toposes}
\label{tab:topos-local-prop}
\end{table}

\begin{table}[]
\begin{center}
\begin{tabular}{c | c}
Bicategory of Toposes        & Displayed Bicategory                \\
\hline
\conceptDef{$\Pretop$}{Bicategories.ComprehensionCat.Biequivalence.InternalLanguageTopos}{bicat_of_univ_pretopos}   & $\LPropD{\PretopP}$                 \\
\conceptDef{$\APretop$}{Bicategories.ComprehensionCat.Biequivalence.InternalLanguageTopos}{bicat_of_univ_pretopos_with_nat}  & $\LPropD{\APretopP}$                \\
\conceptDef{$\PiPretop$}{Bicategories.ComprehensionCat.Biequivalence.InternalLanguageTopos}{bicat_of_univ_pi_pretopos} & $\LPropD{\PretopP} \times \LCCCD$   \\
\conceptDef{$\Eltop$}{Bicategories.ComprehensionCat.Biequivalence.InternalLanguageTopos}{bicat_of_univ_topos}    & $\LPropD{\EltopP} \times \LCCCD$    \\
\conceptDef{$\EltopNNO$}{Bicategories.ComprehensionCat.Biequivalence.InternalLanguageTopos}{bicat_of_univ_topos_with_NNO} & $\LPropD{\EltopNNOP} \times \LCCCD$
\end{tabular}
\end{center}
\caption{Bicategories of univalent toposes}
\label{tab:topos-bicat}
\end{table}

\begin{table}[]
\begin{center}
\begin{tabular}{c | c}
Bicategory of Comprehension Categories  & Displayed Bicategory                                 \\
\hline
\conceptDef{$\PretopCC$}{Bicategories.ComprehensionCat.Biequivalence.InternalLanguageTopos}{univ_pretopos_language}            & $\DFLCompCatPropD{\PretopP}$                         \\
\conceptDef{$\APretopCC$}{Bicategories.ComprehensionCat.Biequivalence.InternalLanguageTopos}{univ_pretopos_with_nat_language}           & $\DFLCompCatPropD{\APretopP}$                        \\
\conceptDef{$\PiPretopCC$}{Bicategories.ComprehensionCat.Biequivalence.InternalLanguageTopos}{univ_pi_pretopos_language}          & $\DFLCompCatPropD{\PretopP} \times \DFLCompCatPiD$   \\
\conceptDef{$\EltopCC$}{Bicategories.ComprehensionCat.Biequivalence.InternalLanguageTopos}{univ_topos_language}             & $\DFLCompCatPropD{\EltopP} \times \DFLCompCatPiD$    \\
\conceptDef{$\EltopNNOCC$}{Bicategories.ComprehensionCat.Biequivalence.InternalLanguageTopos}{univ_topos_with_NNO_language}          & $\DFLCompCatPropD{\EltopNNOP} \times \DFLCompCatPiD$
\end{tabular}
\end{center}
\caption{Bicategories of full univalent comprehension categories}
\label{tab:topos-comp-cat-bicat}
\end{table}

Note that one can show that the category of sets is a $\Pi$W-pretopos,
and a topos if one assumes resizing axioms~\cite{rijke:2015}.
Since the local properties $\EltopP$ and $\EltopNNOP$ do not mention exponentials,
they only express that a univalent locally Cartesian closed category is a topos (with a natural numbers object).
Using the material in this section we directly obtain internal language theorems for each of these classes of toposes.

\begin{exa}
\label[exa]{exa:internal-language-topos}
By \Cref{exa:prod-biequiv,constr:local-property-biequiv},
we get the following biequivalences.
\begin{itemize}
  \item[\link{\href{\coqdocurl{Bicategories.ComprehensionCat.InternalLanguageTopos.Pretopos}{internal_language_univ_pretopos}}{\labelitemi}}]
    $\biequiv{\Pretop}{\PretopCC}$;
  \item[\link{\href{\coqdocurl{Bicategories.ComprehensionCat.InternalLanguageTopos.PretoposNat}{internal_language_univ_pretopos_with_nat}}{\labelitemi}}]
    $\biequiv{\APretop}{\APretopCC}$;
  \item[\link{\href{\coqdocurl{Bicategories.ComprehensionCat.InternalLanguageTopos.PiPretopos}{internal_language_univ_pi_pretopos}}{\labelitemi}}]
    $\biequiv{\PiPretop}{\PiPretopCC}$;
  \item[\link{\href{\coqdocurl{Bicategories.ComprehensionCat.InternalLanguageTopos.Topos}{internal_language_univ_topos}}{\labelitemi}}]
    $\biequiv{\Eltop}{\EltopCC}$;
  \item[\link{\href{\coqdocurl{Bicategories.ComprehensionCat.InternalLanguageTopos.ToposNat}{internal_language_univ_topos_with_NNO}}{\labelitemi}}]
    $\biequiv{\EltopNNO}{\EltopNNOCC}$.
\end{itemize}
\end{exa}

\section{The Internal Language of Categories with a Family of Types}
\label{sec:universes}
Up to now,
we have considered a wide variety of type formers,
including $\Pi$-types, quotients, and the natural numbers,
and we proved internal language theorems for various classes of categories supporting some of these types.
The final type former that we consider in this paper,
is universe types for elementary toposes.
Specifically, our goal
is to extend the biequivalence $\EltopNNO \simeq \EltopNNOCC$ to include a universe
closed under all type formers in a topos~\cite{streicher:2005}.

Our focus is on Tarski-style universes~\cite{martin-lof:1984a},
which are given by a type of codes
and a map that translates each code to a type in the appropriate context.
In type theory,
we represent such a universe by a type $\UnivU$
and a type former that assigns to each term $t : \Tms{\Gamma}{\UnivU}$
a type $\El(t)$ in context $\Gamma$.
There are various introduction rules for $\UnivU$,
representing which type formers are contained in the universe.
For instance, one can postulate the existence of a code $\NatC : \UnivU$
such that $\El(\NatC) = \Nat$,
or that $\UnivU$ is closed under $\Sigma$-types.
Note that we do not have an elimination rule for $\UnivU$.

We use two steps to extend the biequivalence $\EltopNNO \simeq \EltopNNOCC$ with a universe type.
First, we consider a family of types,
which is the main goal of this section.
Such a family gives us a Tarski-style universe that does not have any introduction rule
meaning that it is not guaranteed to be closed under type formers.
In the next section, we further consider various type formers for such a family of types
following Streicher~\cite{streicher:2005}.
We start by defining a notion of family of types in categories with finite limits,
and we show that this gives rise to a displayed bicategory~(\Cref{def:disp-bicat-finlim-univ}).
Next we do the same for comprehension categories~(\Cref{def:comp-cat-universe}),
and we end this section by constructing the displayed biequivalence
between the resulting displayed bicategories~(\Cref{constr:universe-biequiv}).

\subsection{Categories with a Family of Types}
The notion of universes in categories is well-established,
and it is commonly used in areas like algebraic set theory~\cite{awodey:2014,vandenberg:2009a,joyal:1995}.
In that setting,
one works with a category $\C$ together with a predicate $\Pred$ on the morphisms in $\C$,
and morphisms satisfying that predicate are called \textbf{small}.
This class of maps is required to be closed under pullback,
which expresses that smallness is preserved under substitution.
A universe for $\Pred$ is a morphism $p : e \rightarrow u$
such that each small map is some pullback of $p$.
There are various instances of such a universe,
and a concrete example comes from set theory,
assuming that we have some inaccessible cardinal $\kappa$~\cite{joyal:1995}.
A function satisfies the predicate $\Pred$
if each fiber has cardinality less than $\kappa$.
The set $u$ is defined as the set of all subsets of $\kappa$
whose cardinality is less than $\kappa$,
and $e$ consists of pairs $A \in u$ together with $a \in A$.
Other instances include universes of presheaves~\cite{awodey:2024,hofmann:1997b,slattery:2026}.

To understand this notion of universe,
we translate it to type theory
using the internal language of categories with finite limits.
A morphism $p : e \rightarrow u$ gives us a type $\mathcal{U}$ in the empty context
and a type $\UnivUP$ in the context $\mathcal{U}$.
The resulting type $\mathcal{U}$ gives rise to a Tarski-style universe.
Since every term $t : \Tms{\Gamma}{\mathcal{U}}$ induces a context morphism $s : \Gamma \rightarrow \mathcal{U}$,
the substitution of $\UnivUP$ along $s$ gives a type in context $\Gamma$.
Note that we are not guaranteed that $\mathcal{U}$ is closed under any type former,
meaning that it only gives us a family of types.

We take a different, but equivalent approach.
Our approach is based on directly translating the type theoretic rules for Tarski-style universes to categories.
While the resulting notion is slightly more complex,
it is closer to the usual type theoretic presentation.
This style is also advantageous if one has a concrete description of $\El$,
which one has for the universe of iterative sets~\cite{gratzer:2024,program:2013}
or for inductive-recursive universes~\cite{dybjer:2000}.
In \Cref{rem:universe-def} we explain why we chose this approach in more detail.

To define the displayed bicategory $\FinLimUniv$
whose objects over a category $\C : \FinLim$ are families of types for $\C$,
we take two steps.
These two steps reflect that a family of types consists of two ingredients.
The first is a type $\CatUniv$ of codes,
which we represent in a category $\C$ as an object $\CatUniv : \C$.
Hence,
we construct a displayed bicategory $\FinLimOb$
whose objects over $\C$ are objects $\CatUniv : \C$.

\begin{defi}
\label[defi]{def:finlim-ob}
We define the displayed bicategory \conceptDef{$\FinLimOb$}{Bicategories.ComprehensionCat.Universes.CatWithUniv.CatWithOb}{disp_bicat_finlim_ob} over $\FinLim$
as the displayed bicategory whose
\begin{itemize}
  \item objects over $\C$ are objects $\CatUniv : \C$;
  \item 1-cells over $\F : \C_1 \rightarrow \C_2$ from $\CatUniv_1 : \C_1$ to $\CatUniv_2 : \C_2$
    are isomorphisms $\FunctorUniv{\F} : \iso{\F(\CatUniv_1)}{\CatUniv_2}$;
  \item 2-cells over $\nt : \F \twocell \G$
    from $\FunctorUniv{\F} : \iso{\F(\CatUniv_1)}{\CatUniv_2}$
    to $\FunctorUniv{\G} : \iso{\G(\CatUniv_1)}{\CatUniv_2}$
    are proofs that the following triangle commutes.
    \[
      \begin{tikzcd}
	{\F(\CatUniv_1)} && {\G(\CatUniv_1)} \\
	& {\CatUniv_2}
	\arrow["{\nt(\CatUniv_1)}", from=1-1, to=1-3]
	\arrow["{\FunctorUniv{\F}}"', from=1-1, to=2-2]
	\arrow["{\FunctorUniv{\G}}", from=1-3, to=2-2]
      \end{tikzcd}
    \]
\end{itemize}
\end{defi}

\begin{propL}[\coqdocurl{Bicategories.ComprehensionCat.Universes.CatWithUniv.CatWithOb}{disp_univalent_2_disp_bicat_finlim_ob}]
\label[propL]{prop:finlim-ob-univ}
The displayed bicategory $\FinLimOb$ is univalent.
\end{propL}

\begin{proof}
We need to prove that $\FinLimOb$ is both locally and globally univalent.
The local univalence of $\FinLimOb$ follows from the fact that all 2-cells in $\FinLimOb$ are invertible
and that 2-cells $\dtwo{\FunctorUniv{\F}}{\FunctorUniv{\F}'}{\id}$ are the same as identities $\FunctorUniv{\F} = \FunctorUniv{\F}'$
for all $\FunctorUniv{\F}, \FunctorUniv{\F}' : \iso{\F(\CatUniv_1)}{\CatUniv_2}$.
To prove that $\FinLimOb$ is globally univalent,
we first note that adjoint equivalences over $\id[\C]$ from $\CatUniv_1$ to $\CatUniv_2$ are the same as isomorphisms $\iso{\CatUniv_1}{\CatUniv_2}$.
Since $\C$ is univalent,
identities $\CatUniv_1 = \CatUniv_2$ also are the same as isomorphisms $\iso{\CatUniv_1}{\CatUniv_2}$.
Hence, $\FinLimOb$ is univalent.
\end{proof}

The second ingredient to a family of types is a map that assigns
to each term $t : \Tms{\Gamma}{\UnivU}$
a type $\El(t)$ in context $\Gamma$.
This assignment must be stable under substitution
in the sense that $\substTy{s}{\El(t)}$ and $\El(\substTm{s}{t})$ are isomorphic.
In addition,
we require various coherence axioms for these isomorphisms.
We collect these features in the notion of an \textbf{elements map}.

\begin{defi}
\label[defi]{def:cat-el-map}
Let $\C$ be a univalent category with finite limits and let $\CatUniv$ be an object in $\C$.
An \conceptDef{elements map}{Bicategories.ComprehensionCat.Universes.CatWithUniv.UniverseInCat}{cat_stable_el_map} $\CatEl$ for $\CatUniv$ consists of
\begin{itemize}
  \item a map that assigns to each morphism $t : \Gamma \rightarrow \CatUniv$ an object $\CatEl(t) : \C$ and morphism $\CatElM(t) : \CatEl(t) \rightarrow \Gamma$;
  \item for all morphisms $t : \Gamma \rightarrow \CatUniv$ and $s : \Delta \rightarrow \Gamma$
    there is a morphism $\CatElSub(s, t) : \CatEl(s \cdot t) \rightarrow \CatEl(t)$
    making the following square a pullback
    \[
      \begin{tikzcd}[column sep=large]
	{\CatEl(s \cdot t)} & {\CatEl(t)} \\
	\Delta & \Gamma & \CatUniv
	\arrow["{\CatElSub(s, t)}", from=1-1, to=1-2]
	\arrow["{\CatElM(s \cdot t)}"', from=1-1, to=2-1]
	\arrow["\lrcorner"{anchor=center, pos=0.125}, draw=none, from=1-1, to=2-2]
	\arrow["{\CatElM(t)}", from=1-2, to=2-2]
	\arrow["s"', from=2-1, to=2-2]
	\arrow["t"', from=2-2, to=2-3]
      \end{tikzcd}
    \]
\end{itemize}
If we have morphisms $t_1, t_2 : \Gamma \rightarrow \CatUniv$,
then every equality $t_1 = t_2$ gives rise to an isomorphism $\CatElEq : \iso{\CatEl(t_1)}{\CatEl(t_2)}$ making the following triangle commute.
\[
\begin{tikzcd}
  {\CatEl(t_1)} && {\CatEl(t_2)} \\
  & \Gamma
  \arrow["\CatElEq", from=1-1, to=1-3]
  \arrow["{\CatElM(t_1)}"', from=1-1, to=2-2]
  \arrow["{\CatElM(t_2)}", from=1-3, to=2-2]
\end{tikzcd}
\]
An elements map is called \conceptDef{coherent}{Bicategories.ComprehensionCat.Universes.CatWithUniv.UniverseInCat}{is_coherent_cat_stable_el_map}
whenever the following diagrams commute.
\[
\begin{tikzcd}[column sep=large]
  {\CatEl(\comp{\id[]}{t})} & {\CatEl(t)}
  \arrow["{\CatElSub(\id, t)}", shift left=2, from=1-1, to=1-2]
  \arrow["\CatElEq"', shift right=2, from=1-1, to=1-2]
\end{tikzcd}
\quad \quad \quad
\begin{tikzcd}[column sep=huge]
  {\CatEl((s_1 \cdot s_2) \cdot t)} & {\CatEl(t)} \\
  {\CatEl(s_1 \cdot (s_2 \cdot t))} & {\CatEl(s_2 \cdot t)}
  \arrow["{\CatElSub(s_1 \cdot s_2, t)}", from=1-1, to=1-2]
  \arrow["\CatElEq"', from=1-1, to=2-1]
  \arrow["{\CatElSub(s_1, s_2 \cdot t)}"', from=2-1, to=2-2]
  \arrow["{\CatElSub(s_2, t)}"', from=2-2, to=1-2]
\end{tikzcd}
\]
\end{defi}

To understand \Cref{def:cat-el-map},
we view this definition from a type theoretic perspective.
Recall that in a univalent category $\C$ with finite limits,
contexts are interpreted as objects in $\C$,
types in context $\Gamma$ as pairs of an object $A$ together with a morphism $\CtxPr : A \rightarrow \Gamma$,
and terms of such a type are interpreted as sections of $\CtxPr$.
The object $\CatUniv$ gives rise to an object in the empty context
using the unique map $\ECtxMap : \CatUniv \rightarrow \ECtx$.
Since the following square is a pullback,
the object $\Gamma \times \CatUniv$ interprets the universe type in arbitrary contexts $\Gamma$.
\[
\begin{tikzcd}
  {\Gamma \times \CatUniv} & \CatUniv \\
  \Gamma & \ECtx
  \arrow["{\pi_2}", from=1-1, to=1-2]
  \arrow["{\pi_1}"', from=1-1, to=2-1]
  \arrow["\lrcorner"{anchor=center, pos=0.125}, draw=none, from=1-1, to=2-2]
  \arrow["\ECtxMap", from=1-2, to=2-2]
  \arrow["\ECtxMap"', from=2-1, to=2-2]
\end{tikzcd}
\]
Since sections of the map $\pi_1 : \Gamma \times \CatUniv \rightarrow \Gamma$ correspond to morphisms $t : \Gamma \rightarrow \CatUniv$,
terms of type $\CatUniv$ in context $\Gamma$ are the same as such morphisms.
Hence,
the maps $\CatEl$ and $\CatElM$
assign to each term of type $\CatUniv$ in context $\Gamma$
a type in context $\Gamma$.

From the morphism $\CatElSub$
we get that $\CatEl(s \cdot t)$ is isomorphic to the pullback of $\CatEl(t)$ along $s$.
Since substitution of types is given by pullback
and substitution of terms by composition,
we thus get an interpretation of the expected stability rule for $\El$ up to isomorphism.
The coherences in \Cref{def:cat-el-map} tell us how to simplify $\CatElSub(s, t)$ in case $s$ is either an identity or a composition.
Note that similar coherences are given for the groupoid syntax of type theory~\cite[Definition 7]{altenkirch:2025}.

Our notion of family of types corresponds to the more standard notion that uses representable classes of small maps,
where a family of types is given by a morphism.
Specifically, we show that every morphism gives rise to a family of types in our sense,
and that this assignment is an equivalence of types.

\begin{problem}
\label[problem]{prob:el-map-equiv}
Given a univalent category $\C$ with finite limits and an object $\CatUniv : \C$,
to construct an equivalence between the type of pairs of objects $e : \C$ and morphisms $f : e \rightarrow \CatUniv$
and the type of coherent elements maps for $\CatUniv$.
\end{problem}

\begin{construction}{\coqdocurl{Bicategories.ComprehensionCat.Universes.CatWithUniv.UniverseAsMorphism}{cat_stable_el_map_coherent_mor_weq}}{prob:el-map-equiv}
\label{constr:el-map-equiv}
Given a coherent elements map for $\CatUniv$, we define $e$ to be $\CatEl(\id[u])$
and $f$ to be $\CatElM(\id[u])$ as in the following diagram.
\[
\begin{tikzcd}
  {\CatEl(\CatUniv)} \\
  \CatUniv & \CatUniv
  \arrow["{\CatElM(\id[\CatUniv])}"', from=1-1, to=2-1]
  \arrow["{\id[\CatUniv]}"', from=2-1, to=2-2]
\end{tikzcd}
\]

For the converse,
note that if we have a morphism $f : e \rightarrow \CatUniv$,
then each $t : \Gamma \rightarrow \CatUniv$ gives rise to a pullback as follows.
\[
\begin{tikzcd}
  x & e \\
  \Gamma & \CatUniv
  \arrow["{p_2}", from=1-1, to=1-2]
  \arrow["{p_1}"', from=1-1, to=2-1]
  \arrow["\lrcorner"{anchor=center, pos=0.125}, draw=none, from=1-1, to=2-2]
  \arrow["f", from=1-2, to=2-2]
  \arrow["t"', from=2-1, to=2-2]
\end{tikzcd}
\]
We define $\CatEl_f(t)$ to be $x$ and $\CatElM_f(t)$ to be $p_1$.
Suppose that we have morphisms $t : \Gamma \rightarrow \CatUniv$ and $s : \Delta \rightarrow \Gamma$.
Then we have the following pullback squares.
\[
\begin{tikzcd}
  x & y & e \\
  \Delta & \Gamma & \CatUniv
  \arrow["{p_2}", from=1-1, to=1-2]
  \arrow["{p_1}"', from=1-1, to=2-1]
  \arrow["\lrcorner"{anchor=center, pos=0.125}, draw=none, from=1-1, to=2-2]
  \arrow["{q_2}", from=1-2, to=1-3]
  \arrow["{q_1}"{description}, from=1-2, to=2-2]
  \arrow["\lrcorner"{anchor=center, pos=0.125}, draw=none, from=1-2, to=2-3]
  \arrow["f", from=1-3, to=2-3]
  \arrow[from=2-1, to=2-2]
  \arrow["t"', from=2-2, to=2-3]
\end{tikzcd}
\]
We define $\CatElSub_f(s, t)$ to be $p_2$.
This data constitutes a coherent elements map.

We only give a sketch of a proof that these maps are inverses.
If we have a map $f : e \rightarrow \CatUniv$,
then we get the following pullback square.
\[
\begin{tikzcd}
  e & e \\
  \CatUniv & \CatUniv
  \arrow["\id", from=1-1, to=1-2]
  \arrow["f"', from=1-1, to=2-1]
  \arrow["\lrcorner"{anchor=center, pos=0.125}, draw=none, from=1-1, to=2-2]
  \arrow["f", from=1-2, to=2-2]
  \arrow["\id"', from=2-1, to=2-2]
\end{tikzcd}
\]
Hence, we have an isomorphism between $\CatEl_f(\id[\CatUniv])$ and $f$ in $\slice{\C}{\CatUniv}$,
giving us the desired identity.
In addition,
every coherent elements map for $\CatUniv$
gives rise to a diagram as follows.
\[
\begin{tikzcd}[column sep=huge]
  {\CatEl(t)} & {\CatEl(t \cdot \id[u])} & {\CatEl(\id[u])} \\
  & \Gamma & \CatUniv & \CatUniv
  \arrow["\CatElEq", from=1-1, to=1-2]
  \arrow["{\CatElM(t)}"', from=1-1, to=2-2]
  \arrow["{\CatElSub(t, \id[u])}", from=1-2, to=1-3]
  \arrow["{\CatElM(t \cdot \id[u])}"{description}, from=1-2, to=2-2]
  \arrow["\lrcorner"{anchor=center, pos=0.125}, draw=none, from=1-2, to=2-3]
  \arrow["{\CatElM(\id[u])}", from=1-3, to=2-3]
  \arrow["t"', from=2-2, to=2-3]
  \arrow["{\id[u]}"', from=2-3, to=2-4]
\end{tikzcd}
\]
Since $\CatElEq$
is an isomorphism between
$\CatEl(t)$ and the pullback of $\CatElM(\id[u])$ along $t$,
the other desired equality follows as well.
\end{construction}

\begin{rem}
\label[rem]{rem:universe-def}
We use a non-standard notion of family of types,
because it gives one more control over the map $\CatEl$.
By picking a convenient representation for $\CatEl$,
it becomes easier to establish that our family of types is closed under various type formers.
We illustrate this advantage via an example.
Inductive-recursive types~\cite{dybjer:1999,dybjer:2000} give us a convenient and flexible way to construct universes in the category of sets.
Specifically, we can construct a type $\CatUniv$ together with a map $\CatElIR$ sending inhabitants of $\CatUniv$ to actual types using induction-recursion.
The type $\CatUniv$ is defined inductively,
and the constructors of $\CatUniv$ say which types are contained in the universe.
The map $\CatElIR$ is defined as a recursive function whose domain is $\CatUniv$,
and this function maps each constructor of $\CatUniv$ to its corresponding type.
Note that $\CatUniv$ and $\CatElIR$ are mutually dependent in the sense that constructors of $\CatUniv$ can refer to $\CatElIR$.
For instance,
to say that our universe $\CatUniv$ is closed under $\Sigma$-types,
we postulate that we have a constructor $\SigmaCode(a, b) : \CatUniv$
for all $a : \CatUniv$ and $b : \CatElIR(a) \rightarrow \CatUniv$.
We also need to specify how $\CatElIR$ acts on the constructors
for which we say that $\CatElIR(\SigmaCode(a, b)) = \Sigma (x : \CatElIR(a)), \CatElIR(b(x))$.

Using an inductive-recursive universe,
we can construct a coherent elements map for $\CatUniv$.
Doing so requires us to show that every $t : \Gamma \rightarrow \CatUniv$ gives rise to a map $\CatElM(t) : \CatEl(t) \rightarrow \Gamma$.
We take $\CatEl(t)$ to be $\Sigma (x : \Gamma), \CatElIR(t(x))$ and $\CatElM(t)$ to be the first projection.
If we instead construct a map $f : e \rightarrow \CatUniv$ and use \Cref{constr:el-map-equiv},
then we would define $e$ to be $\Sigma (x : \CatUniv), \CatElIR(x)$ and $f$ to be the first projection.
However, the resulting elements map is slightly more complicated,
because it is constructed as the following pullback as we showed in \Cref{constr:el-map-equiv}.
\[
\begin{tikzcd}[column sep = large]
  x & \sum (x : \CatUniv), \CatElIR(x) \\
  \Gamma & \CatUniv
  \arrow["{p_2}", from=1-1, to=1-2]
  \arrow["{p_1}"', from=1-1, to=2-1]
  \arrow["\lrcorner"{anchor=center, pos=0.125}, draw=none, from=1-1, to=2-2]
  \arrow["\pi_1", from=1-2, to=2-2]
  \arrow["t"', from=2-1, to=2-2]
\end{tikzcd}
\]
While both elements maps are equivalent,
the latter one requires us to take additional isomorphisms into account
when establishing that the resulting universe is closed under type formers.
Note that we can say the same for universes of iterative sets~\cite{gratzer:2024}.
\end{rem}

Our next goal is to describe the displayed 1-cells and 2-cells in the desired displayed bicategory $\FinLimUniv$.
To do so, we define when functors and natural transformations preserve elements maps.

\begin{defi}
\label[defi]{def:functor-finlim-univ}
Let $\C_1$ and $\C_2$ be univalent categories with finite limits,
let $\CatUniv_1 : \C_1$ and $\CatUniv_2 : \C_2$ be objects,
and let $\CatEl_1$ and $\CatEl_2$ be coherent elements maps for $\CatUniv_1$ and $\CatUniv_2$ respectively.
If we have a functor $\F : \C_1 \rightarrow \C_2$ that preserves finite limits
and an isomorphism $\FunctorUniv{\F} : \iso{\F(\CatUniv_1)}{\CatUniv_2}$,
then an \conceptDef{elements preservation family}{Bicategories.ComprehensionCat.Universes.CatWithUniv.UniverseInCat}{functor_stable_el_map} for $(\F, \FunctorUniv{\F})$
consists of isomorphisms $\FunctorEl{\F}(t) : \iso{\F(\CatEl_1(t))}{\CatEl_2(\F(t) \cdot \FunctorUniv{\F})}$
for each $t : \Gamma \rightarrow \CatUniv_1$
such that the following diagrams commute
for all $s : \Delta \rightarrow \Gamma$ and $t : \Gamma \rightarrow \CatUniv_1$.
\[
\begin{tikzcd}
  {\F(\CatEl_1(t))} && {\CatEl_2(\F(t) \cdot \FunctorUniv{\F})} \\
  & {\F(\Gamma)}
  \arrow["{\FunctorEl{\F}(t)}", from=1-1, to=1-3]
  \arrow["{\F(\CatElM_1(t))}"', from=1-1, to=2-2]
  \arrow["{\CatElM_2(\F(t) \cdot \FunctorUniv{\F})}", from=1-3, to=2-2]
\end{tikzcd}
\]
\[
\begin{tikzcd}[column sep=2.3cm]
  {\F(\CatEl_1(s \cdot t))} && {\F(\CatEl_1(t))} \\
  {\CatEl_2(\F(s \cdot t) \cdot \FunctorUniv{\F})} & {\CatEl_2(\F(s) \cdot (\F(t) \cdot \FunctorUniv{\F}))} & {\CatEl_2(\F(t) \cdot \FunctorUniv{\F})}
  \arrow["{\F(\CatElSub(s, t))}", from=1-1, to=1-3]
  \arrow["{\FunctorEl{\F}(s \cdot t)}"', from=1-1, to=2-1]
  \arrow["{\FunctorEl{\F}(t)}", from=1-3, to=2-3]
  \arrow["\CatElEq"', from=2-1, to=2-2]
  \arrow["{\CatElSub_2(\F(s), \F(t) \cdot \FunctorUniv{\F})}"', from=2-2, to=2-3]
\end{tikzcd}
\]
\end{defi}

There are a couple of things to note about \Cref{def:functor-finlim-univ}.
The elements map is only preserved by $\F$ via the isomorphism $\FunctorEl{\F}(t) : \iso{\F(\CatEl_1(t))}{\CatEl_2(\F(t) \cdot \FunctorUniv{\F})}$
rather than an equality.
Note that the composition $\F(t) \cdot \FunctorUniv{\F}$,
displayed below,
indeed gives a morphism into $\CatUniv_2$.
\[
\begin{tikzcd}
  {\F(\Gamma)} & {\F(\CatUniv_1)} & {\CatUniv_2}
  \arrow["{\F(t)}", from=1-1, to=1-2]
  \arrow["{\FunctorUniv{\F}}", from=1-2, to=1-3]
\end{tikzcd}
\]
In addition, we require two coherences for elements preservation families.
The first coherence says that $\FunctorEl{\F}(t)$ is an isomorphism in the slice category $\slice{\C_2}{\F(\Gamma)}$,
and the second coherence expresses that $\F$ preserves $\CatElSub_1$.
For natural transformations, preservation of the elements map is defined as follows.

\begin{defi}
\label[defi]{def:nat-trans-finlim-univ}
Suppose that we have coherent elements maps $\CatEl_1$ and $\CatEl_2$ for $\CatUniv_1 : \C_1$ and $\CatUniv_2 : \C_2$ respectively.
Let $\F, \G : \C_1 \rightarrow \C_2$ be functors that preserve elements maps,
and let $\nt$ be a natural transformation such that $\FunctorUniv{\F} = \nt(\CatUniv_1) \cdot \FunctorUniv{\G}$.
We say that $\nt$ \conceptDef{preserves elements maps}{Bicategories.ComprehensionCat.Universes.CatWithUniv.UniverseInCat}{nat_trans_preserves_el}
if for all $t : \Gamma \rightarrow \CatUniv$
the following diagram commutes.
\[
\begin{tikzcd}[column sep=2.7cm]
  {\F(\CatEl_1(t))} && {\G(\CatEl_1(t))} \\
  {\CatEl_2(\F(t) \cdot \FunctorUniv{\F})} & {\CatEl_2(\nt(\Gamma) \cdot (\G(t) \cdot \FunctorUniv{\G}))} & {\CatEl_2(\G(t) \cdot \FunctorUniv{\G})}
  \arrow["{\nt(\CatEl_1(t))}", from=1-1, to=1-3]
  \arrow["{\FunctorEl{\F}(t)}"', from=1-1, to=2-1]
  \arrow["{\FunctorEl{\G}(t)}", from=1-3, to=2-3]
  \arrow["\CatElEq"', from=2-1, to=2-2]
  \arrow["{\CatElSub_2(\nt(\Gamma), \G(t) \cdot \FunctorUniv{\G})}"', from=2-2, to=2-3]
\end{tikzcd}
\]
\end{defi}

Finally,
we put these notions together to obtain the desired displayed bicategory.

\begin{defi}
\label[defi]{def:disp-bicat-finlim-univ}
We define \conceptDef{$\FinLimEl$}{Bicategories.ComprehensionCat.Universes.CatWithUniv.UniverseDispBicat}{disp_bicat_finlim_el} as the displayed bicategory over $\total{\FinLimOb}$ whose
\begin{itemize}
  \item objects over $(\C, \CatUniv)$ are coherent elements maps for $\CatUniv$;
  \item 1-cells over $(\F, \FunctorUniv{\F})$ are elements preservation families for $(\F, \FunctorUniv{\F})$;
  \item 2-cells over $\nt$ are proofs that $\nt$ preserves elements maps.
\end{itemize}
We also define the displayed bicategory \conceptDef{$\FinLimUniv$}{Bicategories.ComprehensionCat.Universes.CatWithUniv.UniverseDispBicat}{disp_bicat_finlim_universe}
to be $\sigmabicat{\FinLimOb}{\FinLimEl}$.
\end{defi}

One can directly show that the displayed bicategory $\FinLimUniv$ is both locally preordered and groupoidal,
and that it is univalent.
In addition, adjoint equivalences in $\FinLimUniv$ can be characterized as adjoint equivalences in $\FinLim$.

\begin{propL}[\coqdocurl{Bicategories.ComprehensionCat.Universes.CatWithUniv.UniverseDispBicat}{disp_univalent_2_disp_bicat_finlim_universe}]
\label[propL]{prop:disp-bicat-finlim-univ-univalent}
The displayed bicategory $\FinLimUniv$ is univalent.
\end{propL}

\begin{proof}
By \Cref{prop:finlim-ob-univ}, $\FinLimOb$ is univalent.
One can show that $\FinLimEl$ is univalent as well,
and since both $\FinLimOb$ and $\FinLimEl$ are locally preordered and groupoidal,
we get that $\FinLimUniv$ is univalent~\cite[Proposition 7.9]{ahrens:2021}.
\end{proof}

\begin{propL}[\coqdocurl{Bicategories.ComprehensionCat.Universes.CatWithUniv.UniverseDispBicat}{disp_left_adjoint_equivalence_finlim_universe}]
\label[propL]{prop:disp-bicat-finlim-univ-adjequiv}
Every displayed 1-cell in $\FinLimUniv$ over an adjoint equivalence is a displayed adjoint equivalence.
\end{propL}

\begin{proof}
Since $\FinLimUniv$ is univalent,
this statement follows from the fact that all displayed 1-cells in $\FinLimUniv$ over the identity are adjoint equivalences.
\end{proof}
  
\subsection{Comprehension Categories with a Family of Types}
Next we consider families of types in comprehension categories,
and we show that we have a displayed bicategory $\DFLCompCatUniv$ over $\DFLCompCat$
whose objects over $\compfunctor : \D \rightarrow \ArrD{\C}$ are families of types.
Note that,
while we are only interested in DFL comprehension categories,
we define the relevant notions for arbitrary comprehension categories.
Our development is similar to the treatment of families of types in categories with finite limits,
and in this section we take the same steps.

A universe for a comprehension category consists of an object $\CatUniv$ in the empty context
together with a map that assigns to every term of type $\CatUniv$ an actual type in the same context,
which is required to be coherently stable under substitution.
Note that if we do not require any closure condition,
then we only get a family of types.
Throughout this section,
we call $\CatUniv$ the universe type of such a family of types.

We construct $\DFLCompCatUniv$ in two steps,
and the first step is to define a displayed bicategory $\DFLCompCatOb$ over $\DFLCompCat$
whose objects over $\compfunctor : \D \rightarrow \ArrD{\C}$ are objects $\CatUniv : \dobP{\D}{\ECtx}$.
Such an object represents our universe type.
To define $\DFLCompCatOb$,
we use that every morphism $(\F, \FF, \F_\chi)$ from a comprehension category $\compfunctor_1 : \D_1 \rightarrow \ArrD{\C_1}$ to $\compfunctor_2 : \D_2 \rightarrow \ArrD{\C_2}$
comes with an isomorphism $\FTerm{\F} : \iso{\F \> \ECtx}{\ECtx}$,
since $\F$ is required to preserve terminal objects.

\begin{defi}
\label[defi]{def:comp-cat-ob}
We define a displayed bicategory \conceptDef{$\DFLCompCatOb$}{Bicategories.ComprehensionCat.Universes.CompCatUniv.CompCatOb}{disp_bicat_comp_cat_with_ob}
over $\DFLCompCat$ such that
\begin{itemize}
  \item the objects over a DFL comprehension category $\compfunctor : \D \rightarrow \ArrD{\C}$ are objects $\CatUniv : \dobP{\D}{\ECtx}$;
  \item the 1-cells over a morphism $(\F, \FF, \F_\chi)$ from $\CatUniv_1 : \dobP{\D_1}{\ECtx}$ to $\CatUniv_2 : \dobP{\D_2}{\ECtx}$
    are isomorphisms $\CompFunctorUniv{\FF} : \isodisp{\FF(\CatUniv_1)}{\CatUniv_2}{\FTerm{\F}}$;
  \item the 2-cells over $(\nt, \tautau)$ from $\CompFunctorUniv{\FF} : \isodisp{\FF(\CatUniv_1)}{\CatUniv_2}{\FTerm{\F}}$ to $\CompFunctorUniv{\GG} : \isodisp{\GG(\CatUniv_1)}{\CatUniv_2}{\FTerm{\G}}$
    are proofs that the following triangle commutes.
    \[
      \begin{tikzcd}
	{\FF(\CatUniv_1)} && {\GG(\CatUniv_1)} \\
	& {\CatUniv_2}
	\arrow["{\tautau(\CatUniv_1)}", from=1-1, to=1-3]
	\arrow["{\CompFunctorUniv{\FF}}"', from=1-1, to=2-2]
	\arrow["{\CompFunctorUniv{\GG}}", from=1-3, to=2-2]
      \end{tikzcd}
    \]
\end{itemize}
\end{defi}

Let us make some remarks on \Cref{def:comp-cat-ob}.
Displayed 2-cells in $\DFLCompCatOb$ are expressed as a heterogeneous equality,
because $\tautau(\CatUniv) \cdot \CompFunctorUniv{\GG}$
and $\CompFunctorUniv{\FF}$
are morphisms
over $\tau \> \ECtx \cdot \FTerm{\G} : \F \> \ECtx \rightarrow \ECtx$,
and $\FTerm{\F} : \F \> \ECtx \rightarrow \ECtx$
respectively.
These two morphisms are equal,
because their codomain is the terminal object.

In addition,
the object $\CatUniv : \dobP{\D}{\ECtx}$ gives rise to a type in every context via weakening.
We define $\CatUnivC[\Gamma] : \dob{\D}{\Gamma}$ to be $\substTy{\ECtxMap[\Gamma]}{\CatUniv}$,
which is the following Cartesian lift.
\[
\begin{tikzcd}
  {\CatUnivC[\Gamma]} & \CatUniv \\
  \Gamma & \ECtx
  \arrow["{\substTyMor{\ECtxMap[\Gamma]}}", from=1-1, to=1-2]
  \arrow["{\ECtxMap[\Gamma]}"', from=2-1, to=2-2]
\end{tikzcd}
\]
Often we write $\CatUnivC$ instead of $\CatUnivC[\Gamma]$ if $\Gamma$ is clear from the context.
Note that the type $\CatUnivC$ is stable under substitution.
Specifically,
if we have a context morphism $s : \Delta \rightarrow \Gamma$,
then we have an isomorphism $\CompCatUnivSub{s} : \iso{\substTy{s}{\CatUnivC[\Gamma]}}{\CatUnivC[\Delta]}$ in the fiber $\fiber{\D}{\Delta}$.
Since $\CatUnivC[\Delta]$ is defined to be $\substTy{\ECtxMap[\Delta]}{\CatUniv}$,
and $\substTy{s}{\CatUnivC[\Gamma]}$ is isomorphic to $\substTy{s}{\substTy{\ECtxMap[\Gamma]}{\CatUniv}}$,
we get the desired isomorphism from the fact that substitution is pseudofunctorial.

Finally,
the displayed bicategory $\DFLCompCatOb$ is both locally preordered and groupoidal.
It also is univalent.

\begin{propL}[\coqdocurl{Bicategories.ComprehensionCat.Universes.CompCatUniv.CompCatOb}{disp_univalent_2_disp_bicat_comp_cat_with_ob}]
\label[propL]{prop:comp-cat-ob-univalent}
The displayed bicategory $\DFLCompCatOb$ is univalent.
\end{propL}

\begin{proof}
We prove this in the same way as \Cref{prop:finlim-ob-univ}.
\end{proof}

We also define the notion of elements maps for comprehension category.
Such a map interprets the type former $\El$ in a comprehension category:
it assigns to each term of the type $\CatUniv$
a type in the appropriate context.

\begin{defi}
\label[defi]{def:comp-cat-el-map}
Let $\compfunctor : \D \rightarrow \ArrD{\C}$ be a comprehension category
and let $\CatUniv : \dobP{\D}{\ECtx}$ be an object.
An \conceptDef{elements map}{Bicategories.ComprehensionCat.Universes.CompCatUniv.UniverseType}{comp_cat_stable_el_map} $\CompCatElTy$ for $\CatUniv$ consists of
\begin{itemize}
  \item a map that assigns to each term $t : \Tms{\Gamma}{\CatUnivC}$ a type $\CompCatElTy(t) : \dob{\D}{\Gamma}$;
  \item for all context morphisms $s : \Delta \rightarrow \Gamma$
    and terms $t : \Tms{\Gamma}{\CatUnivC}$
    an isomorphism $\CompCatElSub(s, t) : \iso{\substTy{s}{\CompCatElTy(t)}}{\CompCatElTy(\coeTm{\substTm{s}{t}}{\CompCatUnivSub{s}})}$
    in the fiber category $\fiber{\D}{\Delta}$.
\end{itemize}
Often we write $\CompCatElSub$ instead of $\CompCatElSub(s, t)$ whenever $s$ and $t$ are clear from the context.
Every equality $t_1 = t_2$ between terms $t_1, t_2 : \Tms{\Gamma}{\CatUnivC}$
gives rise to an isomorphism $\CompCatElEq : \iso{\CompCatElTy(t_1)}{\CompCatElTy(t_2)}$ in the fiber $\fiber{\D}{\Gamma}$.
An elements map is said to be \conceptDef{coherent}{Bicategories.ComprehensionCat.Universes.CompCatUniv.UniverseType}{comp_cat_coherent_el_map}
whenever the following diagrams commute.
\[
\begin{tikzcd}
  {\CompCatElTy(t)} && {\substTy{\id}{\CompCatElTy(t)}} \\
  & {\CompCatElTy(\coeTm{\substTm{\id}{t}}{\CompCatUnivSub{\id}})}
  \arrow["{\cong}", from=1-1, to=1-3]
  \arrow["{\CompCatElEq}"', from=1-1, to=2-2]
  \arrow["{\CompCatElSub(\id, t)}", from=1-3, to=2-2]
\end{tikzcd}
\]
\[
\begin{tikzcd}[column sep=large]
  {\substTy{s_1}{\substTy{s_2}{\CompCatElTy(t)}}} && {\substTy{(s_1 \cdot s_2)}{\CompCatElTy(t)}} \\
  {\substTy{s_1}{\CompCatElTy(\coeTm{\substTm{s_2}{t}}{\CompCatUnivSub{s_2}})}} & \CompCatElTy(\coeTm{\substTm{s_1}{\coeTm{\substTm{s_2}{t}}{\CompCatUnivSub{s_2}}}}{\CompCatUnivSub{s_1}}) & {\CompCatElTy(\coeTm{\substTm{(s_1 \cdot s_2)}{t}}{\CompCatUnivSub{s_1 \cdot s_2}})}
  \arrow["\cong", from=1-1, to=1-3]
  \arrow["{\substTy{s_1}{\CompCatElSub}}"', from=1-1, to=2-1]
  \arrow["{\CompCatElSub}", from=1-3, to=2-3] 
  \arrow["{\CompCatElSub}"', from=2-1, to=2-2] 
  \arrow["\CompCatElEq"', from=2-2, to=2-3]
\end{tikzcd}
\]
\end{defi}

Note that the map $\CompCatElTy$ is stable under substitution,
but only up to a coherent isomorphism.
To express the stability, we need to use the coercion operation on terms.
If we have $t : \Tms{\Gamma}{\CatUnivC[\Gamma]}$ and $s : \Delta \rightarrow \Gamma$,
then the term $\substTm{s}{t}$ is of type $\substTy{s}{\CatUnivC[\Gamma]}$
rather than $\CatUnivC[\Delta]$.
Since we have an isomorphism $\CompCatUnivSub{s} : \iso{\substTy{s}{\CatUnivC[\Gamma]}}{\CatUnivC[\Delta]}$,
we can coerce $\substTm{s}{t}$ into a term of type $\CatUnivC[\Delta]$.
We also require two coherence conditions for elements maps,
and these describe the isomorphism $\CompCatUnivSub{s}$ if $s$ is the identity or composition.
Our coherence conditions agree with the ones given by Altenkirch et al.~\cite{altenkirch:2025}.

Before we define what it means for functors to preserve elements maps,
we introduce some notation.
First, we describe how functors between comprehension categories act on terms.
The term $\FTms{\F}(t)$ is defined to be $\F(t) \cdot \F_\chi$.
To see that this indeed gives rise to a section,
note that we have a commutative diagram as follows.
\[
\begin{tikzcd}[column sep=large]
  {\F(\Gamma)} & {\F(\CtxExt{\Gamma}{A})} && {\CtxExt{\F(\Gamma)}{\FF(A)}} \\
  && {\F(\Gamma)}
  \arrow["{\F(t)}", from=1-1, to=1-2]
  \arrow["\id"', from=1-1, to=2-3]
  \arrow["{\F_\chi}", from=1-2, to=1-4]
  \arrow["{\F(\CtxPr)}"{description}, from=1-2, to=2-3]
  \arrow["\CtxPr", from=1-4, to=2-3]
\end{tikzcd}
\]

Next note that we have an isomorphism $\substFunctor{\FF}(s, A) : \iso{\FF(\substTy{s}{A})}{\substTy{\F(s)}{\FF(A)}}$
as indicated in the diagram below.
\[
\begin{tikzcd}[column sep=huge]
  {\FF(\substTy{s}{A})} \\
  {\substTy{\F(s)}{\FF(A)}} & {\FF(A)} \\
  {\F(\Delta)} & {\F(\Gamma)}
  \arrow["{\substFunctor{\FF}(s, A)}"', dashed, from=1-1, to=2-1]
  \arrow["{\FF(\substTyMor{s})}", from=1-1, to=2-2]
  \arrow["{\substTyMor{\F(s)}}"', from=2-1, to=2-2]
  \arrow["\F(s)"', from=3-1, to=3-2]
\end{tikzcd}
\]
The morphism $\substTyMor{\F(s)}$ is Cartesian by definition,
and $\FF(\substTyMor{s})$ is so as well,
because $\FF$ preserves Cartesian morphisms.
Since Cartesian lifts are unique up to isomorphism,
we get $\substFunctor{\FF}(s, A) : \iso{\FF(\substTy{s}{A})}{\substTy{\F(s)}{\FF(A)}}$ as desired.

Finally,
suppose that we have an isomorphism $\CompFunctorUniv{\FF} : \isodisp{\FF(\CatUniv)}{\CatUniv'}{\FTerm{\F}}$,
which witnesses that $\FF$ preserves the universe type in the empty context.
Then we can extend it to an isomorphism $\CompFunctorUnivC{\FF} : \iso{\FF(\CatUnivC[\Gamma])}{\CatUnivC[\F(\Gamma)]}$
that shows that $\FF$ preserves universe type in arbitrary contexts $\Gamma$.
We define $\CompFunctorUnivC{\FF}$ as the following composition of morphisms.
\[
\begin{tikzcd}[column sep=huge]
  {\FF(\substTy{\ECtxMap[\Gamma]}{\CatUniv})} & {\substTy{\F(\ECtxMap[\Gamma])}{\FF(\CatUniv)}} & {\substTy{(\F(\ECtxMap[\Gamma]) \cdot \FTerm{\F})}{\CatUniv}} & {\substTy{\ECtxMap[\F(\Gamma)]}{\CatUniv}}
  \arrow["{\substFunctor{\FF}(\ECtxMap[\Gamma], \CatUniv)}", from=1-1, to=1-2]
  \arrow["{\substTy{\F(\ECtxMap[\Gamma])}{\CompFunctorUniv{\FF}}}", from=1-2, to=1-3]
  \arrow["{\cong}", from=1-3, to=1-4]
\end{tikzcd}
\]
Since the codomain of both $\F(\ECtxMap[\Gamma])$ and $\ECtxMap[\F(\Gamma)]$ is a terminal object,
they are equal,
and thus we have an isomorphism $\iso{\substTy{\F(\ECtxMap[\Gamma])}{\CatUniv}}{\substTy{\ECtxMap[\F(\Gamma)]}{\CatUniv}}$.
Now we are ready to define when functors preserve elements maps.

\begin{defi}
\label[defi]{def:comp-cat-functor-el-map}
Suppose that we have comprehension categories $\compfunctor_1 : \D_1 \rightarrow \ArrD{\C_1}$ and $\compfunctor_2 : \D_2 \rightarrow \ArrD{\C_2}$,
and elements maps $\CompCatElTy_1$ and $\CompCatElTy_2$ for $\CatUniv_1 : \dobP{\D_1}{\ECtx}$ and $\CatUniv_2 : \dobP{\D_2}{\ECtx}$ respectively.
Let $(\F, \FF, \F_\chi)$ be a morphism of comprehension categories
from $\compfunctor_1 : \D_1 \rightarrow \ArrD{\C_1}$ to $\compfunctor_2 : \D_2 \rightarrow \ArrD{\C_2}$,
and let $\CompFunctorUniv{\FF} : \isodisp{\FF(\CatUniv_1)}{\CatUniv_2}{\FTerm{\F}}$ be an isomorphism.
An \conceptDef{elements preservation family}{Bicategories.ComprehensionCat.Universes.CompCatUniv.UniverseType}{comp_cat_functor_preserves_univ_type}
$\CompCatFunctorEl{\F}$ for $(\F, \FF, \F_\chi)$
consists of an isomorphism $\CompCatFunctorEl{\F}(t) : \iso{\FF(\CompCatElTy_1(t))}{\CompCatElTy_2(\coeTm{\FTms{\F}(t)}{\CompFunctorUnivC{\FF}})}$
for each $t : \Tms{\Gamma}{\CatUnivC[1]}$
such that the following diagram commutes for all $s : \Delta \rightarrow \Gamma$ and $t : \Tms{\Gamma}{\CatUnivC[1]}$.
\[
\begin{tikzcd}
  {\FF(\substTy{s}{\CompCatElTy_1(t)})}
  & {\substTy{\F(s)}{\FF(\CompCatElTy_1(t))}} \\
  {\FF(\CompCatElTy_1(\coeTm{\substTm{s}{t}}{\CompCatUnivSub{s}}))}
  & {\substTy{\F(s)}{\CompCatElTy_2(\coeTm{\FTms{\F}(t)}{\CompFunctorUnivC{\FF}})}} \\
  {\CompCatElTy_2(\coeTm{\FTms{\F}(\coeTm{\substTm{s}{t}}{\CompCatUnivSub{s}})}{\CompFunctorUnivC{\FF}})}
  & {\CompCatElTy_2(\coeTm{\substTm{\F(s)}{\coeTm{\FTms{\F}(t)}{\CompFunctorUnivC{\FF}}}}{\CompCatUnivSub{\F(s)}})}
  \arrow["{\substFunctor{\FF}}", from=1-1, to=1-2]
  \arrow["{\FF(\CompCatElSub_1)}"', from=1-1, to=2-1]
  \arrow["{\substTy{\F(s)}{\CompCatFunctorEl{\F}}}", from=1-2, to=2-2]
  \arrow["{\CompCatElSub_2}", from=2-2, to=3-2]
  \arrow["{\CompCatFunctorEl{\F}}"', from=2-1, to=3-1]
  \arrow["{\CompCatElEq}"', from=3-1, to=3-2]
\end{tikzcd}
\]
\end{defi}

Next we define when natural transformations of comprehension categories preserve elements maps.
Let $(\F, \FF, \F_\chi)$ and $(\G, \GG, \G_\chi)$ be morphisms of comprehension categories
from $\compfunctor_1 : \D_1 \rightarrow \ArrD{\C_1}$ to $\compfunctor_2 : \D_2 \rightarrow \ArrD{\C_2}$.
If we have a transformation $(\nt, \tautau)$ from $(\F, \FF, \F_\chi)$ to $(\G, \GG, \G_\chi)$,
then each $\Gamma : \C$ and $A : \dob{\D_1}{\Gamma}$ give rise to a morphism $\tautau(A) : \dmor{\FF(A)}{\GG(A)}{\tau(\Gamma)}$.
Since $\D_2$ is a fibration,
we can equivalently represent that morphism as $\factmor{\tautau(A)} : \dmor{\FF(A)}{\substTy{\tau(\Gamma)}{\GG(A)}}{\id}$,
which lives in the fiber category $\fiber{\D_2}{\F(\Gamma)}$.

\begin{defi}
\label[defi]{def:comp-cat-nat-trans-el-map}
Let $(\F, \FF, \F_\chi)$ and $(\G, \GG, \G_\chi)$ be morphisms of comprehension categories,
and let $\CompFunctorUniv{\FF} : \isodisp{\FF(\CatUniv_1)}{\CatUniv_2}{\FTerm{\F}}$
and $\CompFunctorUniv{\GG} : \isodisp{\GG(\CatUniv_1)}{\CatUniv_2}{\FTerm{\G}}$ be isomorphisms.
Suppose that we have elements preservation families
$\CompCatFunctorEl{\F}$ and $\CompCatFunctorEl{\G}$
for $(\F, \FF, \F_\chi)$ and $(\G, \GG, \G_\chi)$ respectively.
Let $(\nt, \tautau)$ be a transformation from $(\F, \FF, \F_\chi)$ to $(\G, \GG, \G_\chi)$
such that $\tautau(\CatUniv_1) \cdot \CompFunctorUniv{\GG} = \CompFunctorUniv{\FF}$.
We say that $(\nt, \tautau)$ \conceptDef{preserves elements maps}{Bicategories.ComprehensionCat.Universes.CompCatUniv.UniverseType}{comp_cat_nat_trans_preserves_univ_type}
whenever the following diagram commutes for all $t : \Tms{\Gamma}{\CatUnivC[1]}$.
\[
\begin{tikzcd}
  {\FF(\CompCatElTy_1(t))}
  & {\CompCatElTy_2(\coeTm{\FTms{\F}(t)}{\CompFunctorUnivC{\FF}})} \\
  {\substTy{\tau(\Gamma)}{\GG(\CompCatElTy_1(t))}}
  & \\
  {\substTy{\tau(\Gamma)}{\CompCatElTy_2(\coeTm{\FTms{\G}(t)}{\CompFunctorUnivC{\GG}})}}
  & {\CompCatElTy_2(\coeTm{\substTm{\tau(\Gamma)}{\coeTm{\FTms{\G}(t)}{\CompFunctorUnivC{\GG}}}}{\CompCatUnivSub{\nt(\Gamma)}})}
  \arrow["{\CompCatFunctorEl{\F}}", from=1-1, to=1-2]
  \arrow["{\factmor{\tautau(\CompCatElTy(t))}}"', from=1-1, to=2-1]
  \arrow["{\CompCatElEq}", from=1-2, to=3-2]
  \arrow["{\substTy{\tau(\Gamma)}{\CompCatFunctorEl{\G}}}"', from=2-1, to=3-1]
  \arrow["{\CompCatElSub}"', from=3-1, to=3-2]
\end{tikzcd}
\]
\end{defi}

To state the diagram in \Cref{def:comp-cat-nat-trans-el-map},
we need an equality as follows.
\[
\coeTm{\FTms{\F}(t)}{\CompFunctorUnivC{\FF}}
=
\coeTm{\substTm{\tau(\Gamma)}{\coeTm{\FTms{\G}(t)}{\CompFunctorUnivC{\GG}}}}{\CompCatUnivSub{\nt(\Gamma)}}
\]
This equation requires the assumption that $\tautau(\CatUniv) \cdot \CompFunctorUniv{\GG} = \CompFunctorUniv{\FF}$.
Now we put everything together to get the desired displayed bicategory.

\begin{defi}
\label[defi]{def:comp-cat-universe}
We define a displayed bicategory \conceptDef{$\DFLCompCatEl$}{Bicategories.ComprehensionCat.Universes.CompCatUniv.CompCatWithUniv}{disp_bicat_comp_cat_ob_univ}
over $\total{\DFLCompCatOb}$ such that
\begin{itemize}
  \item the objects over a pair $\compfunctor : \D \rightarrow \ArrD{\C}$ and $\CatUniv$ are coherent elements maps for $\CatUniv$;
  \item the morphisms over a pair $(\F, \FF, \F_\chi)$ and $\CompFunctorUniv{\FF}$ are elements preservation families for $\CompFunctorUniv{\FF}$;
  \item the 2-cells over $(\nt, \tautau)$ are proofs that $(\nt, \tautau)$ preserves elements maps.
\end{itemize}
We define the displayed bicategory \conceptDef{$\DFLCompCatUniv$}{Bicategories.ComprehensionCat.Universes.CompCatUniv.CompCatWithUniv}{disp_bicat_dfl_full_comp_cat_with_univ}
to be $\sigmabicat{\DFLCompCatOb}{\DFLCompCatEl}$.
\end{defi}

The displayed bicategories $\DFLCompCatEl$ and $\DFLCompCatUniv$ are locally preordered and groupoidal.
They also are univalent,
and we have a simple check to determine whether some displayed 1-cell is an adjoint equivalence.

\begin{propL}[\coqdocurl{Bicategories.ComprehensionCat.Universes.CompCatUniv.CompCatWithUniv}{disp_univalent_2_disp_bicat_comp_cat_ob_univ}]
\label[propL]{prop:comp-cat-universe}
The displayed bicategory $\DFLCompCatUniv$ is univalent.
\end{propL}

\begin{proof}
This proposition is proven similarly to \Cref{prop:disp-bicat-finlim-univ-univalent}.
\end{proof}

\begin{propL}[\coqdocurl{Bicategories.ComprehensionCat.Universes.CompCatUniv.CompCatWithUniv}{disp_left_adjoint_equivalence_comp_cat_universe}]
\label[propL]{prop:disp-bicat-comp-cat-univ-adjequiv}
Every displayed 1-cell in $\DFLCompCatUniv$ over an adjoint equivalence is a displayed adjoint equivalence.
\end{propL}

\begin{proof}
Similar to \Cref{prop:disp-bicat-finlim-univ-adjequiv}.
\end{proof}

\subsection{Extending the Biequivalence to Families of Types}
We end this section by showing that $\DFLCompCatUniv$ and $\FinLimUniv$ are biequivalent.
To do so,
we construct a displayed biequivalence over the biequivalence $(\FinLimToCompCat, \CompCatToFinLim, \FinLimCompCatUnit, \FinLimCompCatCounit)$
that we constructed in \Cref{constr:biequiv}.
By \Cref{constr:disp-biequiv-simpl} it suffices to give
\begin{itemize}
  \item a displayed pseudofunctor $\FinLimToCompCatUniv$ from $\FinLimUniv$ to $\DFLCompCatUniv$ over $\FinLimToCompCat$~(\Cref{constr:finlim-to-compcat-universe});
  \item a displayed pseudofunctor $\CompCatToFinLimUniv$ from $\DFLCompCatUniv$ to $\FinLimUniv$ over $\CompCatToFinLim$~(\Cref{constr:compcat-to-finlim-universe});
  \item a displayed pseudotransformation $\FinLimCompCatUnitUniv$ from $\FinLimToCompCatUniv \cdot \CompCatToFinLimUniv$ to $\idd$ over $\FinLimCompCatUnit$~(\Cref{constr:universe-unit});
  \item a displayed pseudotransformation $\FinLimCompCatCounitUniv$ from $\idd$ to $\CompCatToFinLimUniv \cdot \FinLimToCompCatUniv$ over $\FinLimCompCatCounit$ (\Cref{constr:universe-counit})
\end{itemize}
such that $\FinLimCompCatUnitUniv$ and $\FinLimCompCatCounitUniv$ are pointwise displayed adjoint equivalences.
Note that by \Cref{prop:disp-bicat-comp-cat-univ-adjequiv} and \Cref{prop:disp-bicat-finlim-univ-adjequiv},
all displayed morphisms in $\DFLCompCatUniv$ and $\FinLimUniv$  over adjoint equivalences are displayed adjoint equivalences,
so it suffices to construct the aforementioned pseudofunctors and pseudotransformations.

While each of the displayed bicategories that we considered in \Cref{sec:lccc,sec:topos} has trivial 2-cells,
neither $\DFLCompCatUniv$ nor $\FinLimUniv$ has trivial 2-cells.
Their types of 2-cells are only guaranteed to be propositions
and they are not necessarily inhabited.
The construction of the desired pseudofunctors and pseudotransformations thus becomes more complicated.
To define a displayed pseudofunctor,
one must construct a suitable action on displayed 2-cells
and 2-cells that witness the preservation of the identity and composition,
and for a displayed pseudotransformation,
one must construct the displayed 2-cells that witness naturality.
Below we give a sketch of each of the components of the displayed biequivalence.
We only describe the action on objects for each of the pseudofunctors and pseudotransformations.

\begin{problem}
\label[problem]{prob:finlim-to-compcat-universe}
To construct $\FinLimToCompCatUniv : \dmor{\FinLimUniv}{\DFLCompCatUniv}{\FinLimToCompCat}$.
\end{problem}

\begin{construction}{\coqdocurl{Bicategories.ComprehensionCat.Universes.Biequiv.ToCompCatUniv}{finlim_to_dfl_comp_cat_disp_psfunctor_universe}}{prob:finlim-to-compcat-universe}
\label{constr:finlim-to-compcat-universe}
Let $\C$ be a univalent category with finite limits,
let $\CatUniv : \C$ be an object,
and let $\CatEl$ be a coherent elements map for $\CatUniv$.
From $\CatUniv$ we get an object $\CatUniv : \dobP{\ArrD{\C}}{\ECtx}$,
which is determined by the unique map $\CatUniv \rightarrow \ECtx$.
To construct an elements map for $\CatUniv$,
we use that terms $t : \Tms{\Gamma}{\CatUnivC[\Gamma]}$ in $\id : \ArrD{\C} \rightarrow \ArrD{\C}$ correspond to morphisms $\TmToMor{t} : \Gamma \rightarrow \CatUniv$.
This is because we have the following pullback square.
\[
\begin{tikzcd}
  \Gamma \\
  & {\CatUniv \times \Gamma} & \Gamma \\
  & \CatUniv & \ECtx
  \arrow["t"{description}, dashed, from=1-1, to=2-2]
  \arrow["\id", bend left=20, from=1-1, to=2-3]
  \arrow["{\TmToMor{t}}"', bend right=20, from=1-1, to=3-2]
  \arrow["{\pi_2}", from=2-2, to=2-3]
  \arrow["{\pi_1}"', from=2-2, to=3-2]
  \arrow["\lrcorner"{anchor=center, pos=0.125}, draw=none, from=2-2, to=3-3]
  \arrow["\ECtxMap", from=2-3, to=3-3]
  \arrow["\ECtxMap"', from=3-2, to=3-3]
\end{tikzcd}
\]
Hence, every term $t : \Tms{\Gamma}{\CatUnivC[\Gamma]}$
gives rise to a type $\CatEl(\TmToMor{t})$ in context $\Gamma$.
Since substitution in the comprehension category $\id : \ArrD{\C} \rightarrow \ArrD{\C}$ is interpreted as pullback,
this assignment is coherently stable under substitution.
Hence, we get a coherent elements map for $\id : \ArrD{\C} \rightarrow \ArrD{\C}$.
\end{construction}

\begin{problem}
\label[problem]{prob:compcat-to-finlim-universe}
To construct a displayed pseudofunctor $\CompCatToFinLimUniv : \dmor{\DFLCompCatUniv}{\FinLimUniv}{\CompCatToFinLim}$.
\end{problem}

\begin{construction}{\coqdocurl{Bicategories.ComprehensionCat.Universes.Biequiv.ToCatFinLimUniv}{dfl_comp_cat_to_finlim_disp_psfunctor_universe}}{prob:compcat-to-finlim-universe}
\label{constr:compcat-to-finlim-universe}
Suppose that we have a DFL comprehension category $\compfunctor : \D \rightarrow \ArrD{\C}$,
an object $\CatUniv : \dobP{\D}{\ECtx}$,
and a coherent elements map $\CompCatElTy$ for $\CatUniv$.
Note that we have an object $\CtxExt{\ECtx}{\CatUniv} : \C$,
and our goal is to construct a coherent elements map for $\CtxExt{\ECtx}{\CatUniv}$.
Morphisms $t : \Gamma \rightarrow \CtxExt{\ECtx}{\CatUniv}$ are the same as terms $\MorToTm{t} : \Tms{\Gamma}{\CatUnivC[\Gamma]}$,
because we have the following pullback square.
\[
\begin{tikzcd}
  \Gamma \\
  & {\CtxExt{\Gamma}{\CatUnivC[\Gamma]}} & {\CtxExt{\ECtx}{\CatUniv}} \\
  & {\Gamma} & \ECtx
  \arrow["{\MorToTm{t}}"{description}, dashed, from=1-1, to=2-2]
  \arrow["t", bend left=20, from=1-1, to=2-3]
  \arrow["\id"', bend right=20, from=1-1, to=3-2]
  \arrow["{\compfunctor(\substTyMor{\ECtxMap})}", from=2-2, to=2-3]
  \arrow["{\CtxPr}"', from=2-2, to=3-2]
  \arrow["\ECtxMap", from=2-3, to=3-3]
  \arrow[""{name=0, anchor=center, inner sep=0}, "\ECtxMap"', from=3-2, to=3-3]
  \arrow["\lrcorner"{anchor=center, pos=0.125}, draw=none, from=2-2, to=0]
\end{tikzcd}
\]
Hence, to define the desired elements map,
we send each morphism $t : \Gamma \rightarrow \CtxExt{\ECtx}{\CatUniv}$
to $\CompCatElTy(\MorToTm{t})$.
This assignment is coherently stable under substitution,
because $\CompCatElTy$ is so.
\end{construction}

\begin{problem}
\label[problem]{prob:universe-unit}
To construct a displayed pseudotransformation $\FinLimCompCatUnitUniv : \dtwo{\FinLimToCompCatUniv \cdot \CompCatToFinLimUniv}{\idd}{\FinLimCompCatUnit}$.
\end{problem}

\begin{construction}{\coqdocurl{Bicategories.ComprehensionCat.Universes.Biequiv.UnitForUniv}{dfl_comp_cat_to_finlim_disp_psfunctor_unit}}{prob:universe-unit}
\label{constr:universe-unit}
Let $\C$ be a univalent category with finite limits with an elements map $\CatEl$ for $\CatUniv : \C$.
Recall that the pseudotransformation $\FinLimCompCatUnit(\C) : \CompCatToFinLimUniv(\FinLimToCompCatUniv(\C)) \rightarrow \C$ is defined to be the identity functor pointwise.
We first note that the universe type in $\FinLimToCompCatUniv(\C)$ is $\CatUniv \rightarrow \ECtx$,
and thus the universe type in $\CompCatToFinLimUniv(\FinLimToCompCatUniv(\C))$ is $\CatUniv$.
Hence,
we can take $\CompFunctorUniv{\FinLimCompCatUnit(\C)}$ to be the identity.
To construct the necessary elements preservation family,
we need to construct for every $t : \Gamma \rightarrow \CatUniv$
an isomorphism between $\CatEl(\TmToMor{\MorToTm{t}})$
and $\CatEl(t \cdot \id)$.
The desired isomorphism is given by $\CatElEq$,
since one can show that $\TmToMor{\MorToTm{t}} = t \cdot \id$.
\end{construction}

\begin{problem}
\label[problem]{prob:universe-counit}
To construct a displayed pseudotransformation $\FinLimCompCatCounitUniv : \dtwo{\idd}{\CompCatToFinLimUniv \cdot \FinLimToCompCatUniv}{\FinLimCompCatCounit}$.
\end{problem}

\begin{construction}{\coqdocurl{Bicategories.ComprehensionCat.Universes.Biequiv.CounitForUniv}{dfl_comp_cat_to_finlim_disp_psfunctor_counit}}{prob:universe-counit}
\label{constr:universe-counit}
Suppose that we have a DFL comprehension category $\compfunctor : \D \rightarrow \ArrD{\C}$,
an object $\CatUniv : \dobP{\D}{\ECtx}$,
and a coherent elements map $\CompCatElTy$ for $\CatUniv$.
Note that the underlying displayed functor of $\FinLimCompCatCounit(\compfunctor) : \D \rightarrow \ArrD{\C}$ is $\compfunctor$,
and that the universe type in $\FinLimToCompCatUniv(\CompCatToFinLimUniv(\compfunctor))$ is $\CtxExt{\ECtx}{\CatUniv}$.
We can thus define $\CompFunctorUniv{\FinLimCompCatCounit(\compfunctor)}$ to be the identity.
The construction of the necessary elements preservation family is similar to \Cref{constr:universe-unit}.
\end{construction}

\begin{problem}
\label[problem]{prob:universe-biequiv}
To construct a displayed biequivalence $(\FinLimToCompCatUniv, \CompCatToFinLimUniv, \FinLimCompCatUnitUniv, \FinLimCompCatCounitUniv)$ from $\FinLimUniv$ to $\DFLCompCatUniv$ over $(\FinLimToCompCat, \CompCatToFinLim, \FinLimCompCatUnit, \FinLimCompCatCounit)$.
\end{problem}

\begin{construction}{\coqdocurl{Bicategories.ComprehensionCat.Universes.UniverseBiequiv}{finlim_biequiv_dfl_comp_cat_psfunctor_universe}}{prob:universe-biequiv}
\label{constr:universe-biequiv}
The necessary displayed pseudofunctors are constructed in \Cref{constr:finlim-to-compcat-universe,constr:compcat-to-finlim-universe},
and the necessary displayed pseudotransformations are constructed in \Cref{constr:universe-unit,constr:universe-counit}.
Note that these pseudotransformations are pointwise displayed adjoint equivalences by \Cref{prop:disp-bicat-comp-cat-univ-adjequiv,prop:disp-bicat-finlim-univ-adjequiv},
and thus we have a displayed biequivalence by \Cref{constr:disp-biequiv-simpl}.
\end{construction}

\section{The Internal Language of Toposes with a Universe}
\label{sec:universe-types}
We finish our development by considering elementary toposes with a natural numbers object
and a universe that is closed under the type formers in the internal language of such toposes.
In \Cref{sec:topos,sec:universes},
we already gave the necessary prerequisites:
by \Cref{exa:internal-language-topos},
we have a biequivalence $\biequiv{\EltopNNO}{\EltopNNOCC}$,
and by \Cref{constr:universe-biequiv},
we have a displayed biequivalence $\biequiv{\FinLimUniv}{\DFLCompCatUniv}$.
We can combine these biequivalences
to obtain a biequivalence $\biequiv{\EltopNNOUniv}{\EltopNNOCCUniv}$
between the bicategories of toposes with a natural numbers object and a universe object
and DFL comprehension categories with $\Pi$-types and a universe object,
and that satisfy $\EltopNNOP$ fiberwise.

However,
this universe is not required to be closed under any type former
that one might expect.
We are not guaranteed that, for instance,
our universe contains the terminal object
or that it is closed under $\Sigma$-types.
Hence,
the main goal of this section is to extend the biequivalence
with type formers in the universe.
Our focus lies on the closure conditions for universes in toposes
given by Streicher~\cite{streicher:2005}.
Specifically,
we look at
the natural numbers object,
the subobject classifier,
propositional resizing,
$\Sigma$-types,
and $\Pi$-types.

To construct the desired biequivalence,
we use another construction,
namely \textbf{reindexing}.
If we have a biequivalence $\F : \B_1 \rightarrow \B_2$,
then every displayed bicategory $\D$ over $\B_2$
gives rise to a displayed bicategory $\reindex{\F}{\D}$ over $\B_1$
together with a displayed biequivalence from $\reindex{\F}{\D}$ to $\D$ over $\F$.
This construction allows us to only construct a single displayed bicategory in this section,
namely $\EltopNNOTopUnivD$ over the bicategory $\EltopNNOUniv$
whose objects over an elementary topos with a universe $\CatUniv$
represent that $\CatUniv$ is closed under the aforementioned type formers.
Using reindexing, we obtain a biequivalent displayed bicategory over $\EltopNNOCCUniv$.
We also have a description of when a universe $\CatUniv$ in some DFL comprehension category $\compfunctor : \D \rightarrow \ArrD{\C}$
is closed under type formers.
Specifically, $\CatUniv$ is closed under some type former
if the universe $\CtxExt{\ECtx}{\CatUniv}$ in the category $\C$ of contexts is closed under it.

To see why the resulting description of type formers is sensible,
we consider an example.
Suppose that we have some DFL comprehension category $\compfunctor : \D \rightarrow \ArrD{\C}$
that comes with a fiberwise natural numbers object $\Nat$
and an elements map $\CatEl$ for $\CatUniv : \dobP{\D}{\ECtx}$.
To say that $\CatEl$ contains $\Nat$,
we need a term $\NatCode : \Tms{\ECtx}{\CatUniv}$
and an isomorphism $\NatIso : \iso{\CatEl(\NatCode)}{\Nat}$ in the fiber category $\fiber{\D}{\ECtx}$.
We saw in \Cref{constr:compcat-to-finlim-universe}
that terms $\NatCode : \Tms{\ECtx}{\CatUniv}$
are the same as morphisms $\ECtx \rightarrow \CtxExt{\ECtx}{\CatUniv}$,
In addition,
isomorphisms $\NatIso : \iso{\CatEl(\NatCode)}{\Nat}$ in $\fiber{\D}{\ECtx}$
are the same as isomorphisms $\iso{\CtxExt{\ECtx}{\CatEl(\NatCode)}}{\CtxExt{\ECtx}{\Nat}}$ in the slice category $\slice{\C}{\ECtx}$
because $\compfunctor$ is fully faithful.
Since $\CtxExt{\ECtx}{\Nat}$ is a natural numbers object in $\slice{\C}{\ECtx}$,
we can thus express that $\CatEl$ contains $\Nat$ by only referring to $\C$,
which corresponds to the description obtained from reindexing.

\subsection{Reindexing Displayed Bicategories}
We start this section by discussing reindexing of displayed bicategories.
Specifically, we construct from each displayed bicategory $\D$ over $\B_2$
and each pseudofunctor $\F : \B_1 \rightarrow \B_2$
a new displayed bicategory $\reindex{\F}{\D}$ over $\B_1$,
which is the reindexing of $\D$ along $\F$.
We also show that we have a displayed biequivalence between $\reindex{\F}{\D}$ and $\D$ over $\F$
if $\F$ is a biequivalence.

The type $\dmor{\xx}{\yy}{f}$ of morphisms in $\reindex{\F}{\D}$ is defined to be $\dmor{\xx}{\yy}{\F(f)}$.
However,
a difficulty comes up when one tries to define the identity and composition in $\reindex{\F}{\D}$.
To define the displayed identity,
we need a morphism from $\xx$ to $\xx$ over $\F(\id)$.
Since we only have an invertible 2-cell $\identitor{\F} : \invcell{\id}{\F(\id)}$ in general,
we cannot necessarily transport $\idd : \dmor{\xx}{\xx}{\id}$ to obtain a morphism of type $\dmor{\xx}{\xx}{\F(\id)}$.
We have the same problem for composition.
If we have $\ff : \dmor{\xx}{\yy}{\F(f)}$ and $\gg : \dmor{\yy}{\zz}{\F(g)}$,
then their composition in $\reindex{\F}{\D}$ should be of type $\dmor{\xx}{\zz}{\F(f \cdot g)}$.
However,
we only have a morphism $\ff \cdot \gg : \dmor{\xx}{\zz}{\F(f) \cdot \F(g)}$,
and an invertible 2-cell $\compositor{\F} : \invcell{\F(f) \cdot \F(g)}{\F(f \cdot g)}$.

To define the identity and composition in $\reindex{\F}{\D}$,
we assume that $\D$ comes with additional structure,
namely a \textbf{local isocleaving}~\cite{buckley:2014}.
If a displayed bicategory is equipped with a local isocleaving,
then one is able to transport displayed 1-cells along invertible 2-cells in the base.
Local isocleavings are defined in the usual way via a lifting property.

\begin{defi}
\label[defi]{def:local-isocleaving}
Let $\D$ be a displayed bicategory over $\B$.
A \conceptDef{local isocleaving}{Bicategories.DisplayedBicats.CleavingOfBicat}{local_iso_cleaving} for $\D$
is a map that assigns to every invertible 2-cell $\tau : f \twocell g$
and $\gg : \dmor{\xx}{\yy}{g}$
a displayed 1-cell $\substTy{\tau}{\gg} : \dmor{\xx}{\yy}{f}$
and an invertible 2-cell $\substTyMor{\tau} : \dtwo{\substTy{\tau}{\gg}}{\gg}{\tau}$.
\end{defi}

We can visualize local isocleavings using the diagram below.
\[
\begin{tikzcd}[row sep = huge,column sep = huge]
  {\xx} & {\yy} \\
  x & y
  \arrow[""{name=0, anchor=center, inner sep=0}, "\gg"',  bend right=40, from=1-1, to=1-2]
  \arrow[""{name=1, anchor=center, inner sep=0}, "\substTy{\tau}{\gg}", bend left=40, dashed, from=1-1, to=1-2]
  \arrow[""{name=2, anchor=center, inner sep=0}, "g"', bend right=40, from=2-1, to=2-2]
  \arrow[""{name=3, anchor=center, inner sep=0}, "f", bend left=40, from=2-1, to=2-2]
  \arrow["\tau", shorten <=2pt, shorten >=2pt, Rightarrow, from=3, to=2, shorten >=1ex]
  \arrow["\substTyMor{\tau}", shorten <=2pt, shorten >=2pt, Rightarrow, from=1, to=0, shorten >=1ex]
\end{tikzcd}
\]
Just like every displayed category over a univalent category can be equipped with an isocleaving~\cite[Construction 5.12]{ahrens:2019a},
we can equip every displayed bicategory $\D$ over a locally univalent bicategory $\B$ with a local isocleaving.
By \Cref{prop:equivalence-induction},
we can assume that $\tau$ is an identity,
and then we can take $\substTy{\id}{\gg}$ to be $\gg$
and $\substTyMor{\id} : \dtwo{\gg}{\gg}{\id}$ to be the identity.

\begin{defi}
\label[defi]{def:reindex-disp-bicat}
Suppose that we have bicategories $\B_1$ and $\B_2$,
a pseudofunctor $\F : \B_1 \rightarrow \B_2$,
and a displayed bicategory $\D$ over $\B_2$
that is locally preordered
and equipped with a local isocleaving.
We define the \conceptDef{reindexed displayed bicategory}{Bicategories.DisplayedBicats.Examples.Reindex}{reindex_disp_bicat},
denoted by $\reindex{\F}{\D}$,
as the displayed bicategory over $\B_1$ whose
\begin{itemize}
  \item objects over $x : \B_1$ are objects $\dob{\D}{\F(x)}$;
  \item morphisms over $f : x \rightarrow y$ from $\xx : \dob{\D}{\F(x)}$ to $\yy : \dob{\D}{\F(y)}$
    are morphisms $\dmor{\xx}{\yy}{\F(f)}$;
  \item 2-cells over $\tau : f \twocell g$ from $\ff : \dmor{\xx}{\yy}{\F(f)}$ to $\gg : \dmor{\xx}{\yy}{\F(g)}$
    are 2-cells $\tautau : \dtwo{\ff}{\gg}{\tau}$.
\end{itemize}
The displayed identity is defined as the following 1-cell.
\[
\begin{tikzcd}[row sep = huge,column sep = huge]
  {\xx} & {\xx} \\
  x & x
  \arrow[""{name=0, anchor=center, inner sep=0}, "\idd"',  bend right=30, from=1-1, to=1-2]
  \arrow[""{name=1, anchor=center, inner sep=0}, "\substTy{(\identitor{\F}^{-1})}{\idd}", bend left=30, dashed, from=1-1, to=1-2]
  \arrow[""{name=2, anchor=center, inner sep=0}, "\id"', bend right=30, from=2-1, to=2-2]
  \arrow[""{name=3, anchor=center, inner sep=0}, "\F(\id)", bend left=30, from=2-1, to=2-2]
  \arrow["\identitor{\F}^{-1}", shorten <=2pt, shorten >=2pt, Rightarrow, from=3, to=2, shorten >=1ex]
  \arrow["\substTyMor{\identitor{\F}^{-1}}", shorten <=2pt, shorten >=2pt, Rightarrow, from=1, to=0, shorten >=1ex]
\end{tikzcd}
\]
We define the displayed composition of $\ff : \dmor{\xx}{\yy}{\F(f)}$ and $\gg : \dmor{\yy}{\zz}{\F(g)}$ as follows.
\[
\begin{tikzcd}[row sep = huge,column sep = huge]
  {\xx} & {\zz} \\
  x & z
  \arrow[""{name=0, anchor=center, inner sep=0}, "\ff \cdot \gg"',  bend right=30, from=1-1, to=1-2]
  \arrow[""{name=1, anchor=center, inner sep=0}, "\substTy{(\compositor{\F}^{-1})}{\ff \cdot \gg}", bend left=30, dashed, from=1-1, to=1-2]
  \arrow[""{name=2, anchor=center, inner sep=0}, "\F(f) \cdot \F(g)"', bend right=30, from=2-1, to=2-2]
  \arrow[""{name=3, anchor=center, inner sep=0}, "\F(f \cdot g)", bend left=30, from=2-1, to=2-2]
  \arrow["\compositor{\F}^{-1}", shorten <=2pt, shorten >=2pt, Rightarrow, from=3, to=2, shorten >=1ex]
  \arrow["\substTyMor{\compositor{\F}^{-1}}", shorten <=2pt, shorten >=2pt, Rightarrow, from=1, to=0, shorten >=1ex]
\end{tikzcd}
\]
To define the right unitor in $\reindex{\F}{\D}$,
we need to make a 2-cell from $\substTy{(\compositor{\F}^{-1})}{\ff \cdot \idd^*}$ to $\ff$ over $\F(\runitor)$
where we write $\idd^*$ for $\substTy{(\identitor{\F}^{-1})}{\idd}$.
We have the following composition of 2-cells.
\[
\begin{tikzcd}[column sep=huge]
  \xx & \yy & \yy
  \arrow[""{name=0, anchor=center, inner sep=0}, "\substTy{(\compositor{\F}^{-1})}{\ff \cdot \idd^*}", bend left=50, from=1-1, to=1-3]
  \arrow[""{name=1, anchor=center, inner sep=0}, "\idd"', bend right=30, from=1-2, to=1-3]
  \arrow[""{name=2, anchor=center, inner sep=0}, "{\idd^*}", bend left=30, from=1-2, to=1-3]
  \arrow[""{name=3, anchor=center, inner sep=0}, "\ff"{description}, from=1-1, to=1-2]
  \arrow[""{name=4, anchor=center, inner sep=0}, "\ff"', bend right=50, from=1-1, to=1-3]
  \arrow["\substTyMor{\compositor{\F}^{-1}}"{description}, shorten <=2pt, shorten >=2pt, Rightarrow, from=0, to=3, shorten >=1ex]
  \arrow["\substTyMor{\identitor{\F}^{-1}}", shorten <=2pt, shorten >=2pt, Rightarrow, from=2, to=1, shorten >=1ex]
  \arrow["\overline{\runitor}", shorten <=5pt, shorten >=2pt, Rightarrow, from=3, to=4, shorten >=1ex]
\end{tikzcd}
\]
This composition lives over the following 2-cell.
\[
\begin{tikzcd}[column sep=huge]
  \F(x) & \F(y) & \F(y)
  \arrow[""{name=0, anchor=center, inner sep=0}, "\F(f \cdot \id)", bend left=50, from=1-1, to=1-3]
  \arrow[""{name=1, anchor=center, inner sep=0}, "\id"', bend right=30, from=1-2, to=1-3]
  \arrow[""{name=2, anchor=center, inner sep=0}, "{\F(\id)}", bend left=30, from=1-2, to=1-3]
  \arrow[""{name=3, anchor=center, inner sep=0}, "f"{description}, from=1-1, to=1-2]
  \arrow[""{name=4, anchor=center, inner sep=0}, "f"', bend right=50, from=1-1, to=1-3]
  \arrow["\compositor{\F}^{-1}"{description}, shorten <=2pt, shorten >=2pt, Rightarrow, from=0, to=3, shorten >=1ex]
  \arrow["\identitor{\F}^{-1}", shorten <=2pt, shorten >=2pt, Rightarrow, from=2, to=1, shorten >=1ex]
  \arrow["\runitor", shorten <=5pt, shorten >=2pt, Rightarrow, from=3, to=4, shorten >=1ex]
\end{tikzcd}
\]
This 2-cell is equal to $\F(\runitor)$ as desired,
since $\F$ is a pseudofunctor.
The other operations on 2-cells are defined similarly.
\end{defi}

Since we assumed $\D$ to be locally preordered,
$\reindex{\F}{\D}$ is so as well,
because it inherits its 2-cells from $\D$.
For this reason, all bicategorical laws of $\reindex{\F}{\D}$ hold by local preorderedness.
We can also show that $\reindex{\F}{\D}$ is univalent.

\begin{propL}[\coqdocurl{Bicategories.DisplayedBicats.Examples.Reindex}{disp_univalent_2_reindex_disp_bicat}]
\label[propL]{prop:reindex-disp-bicat-univalent}
If $\D$ is univalent,
then so is $\reindex{\F}{\D}$.
\end{propL}

\begin{proof}
This proposition follows from the fact
that invertible 2-cells and adjoint equivalences in $\reindex{\F}{\D}$
correspond to invertible 2-cells and adjoint equivalences in $\D$ respectively.
\end{proof}

Next we show that we can lift every biequivalence $(\F, \G, \eta, \varepsilon)$ with $\F : \B_1 \rightarrow \B_2$
to a displayed biequivalence $(\FF, \GG, \etaeta, \epsilonepsilon)$ where $\FF : \dmor{\reindex{\F}{\D}}{\D}{\F}$.
We make several simplifying assumptions in this construction,
namely that $\D$ has contractible 2-cells
and that $\B_1$ and $\B_2$ are univalent.
However, to define $\epsilonepsilon$,
we need to make a further assumption:
there needs to be invertible 2-cells witnessing the triangle identities for $(\F, \G, \eta, \varepsilon)$~\cite[Definition 2.1]{gurski:2012}.

To see why, we describe how $\FF$ and $\GG$ are defined concretely.
The displayed pseudofunctor $\FF$ sends $\xx : \dob{\reindex{\F}{\D}}{x}$ to $\xx : \dob{\D}{\F(x)}$.
To define $\GG$,
we need to map every $\xx : \dob{\D}{x}$ to some $\GG(\xx) : \dob{\reindex{\F}{\D}}{\G(x)}$.
Here we use that $\B_2$ is univalent.
Specifically,
since we have an adjoint equivalence $\eta(x) : \adjequiv{x}{\F(\G(x))}$,
every $\xx : \dob{\D}{x}$ can be lifted to obtain $\substTy{\eta^{-1}(x)}{\xx} : \dob{\D}{\F(\G(x))}$ as follows.
\[
\begin{tikzcd}[column sep=large]
  {\substTy{\eta^{-1}(x)}{\xx}} & \xx \\
  {\F(\G(x))} & x
  \arrow["{\substTyMor{\eta^{-1}(x)}}", dashed, from=1-1, to=1-2]
  \arrow["{\eta^{-1}(x)}"', from=2-1, to=2-2]
\end{tikzcd}
\]
The existence of this lift follows from \Cref{prop:equivalence-induction}.
It suffices to assume that $\eta^{-1}(x)$ is the identity,
and we take $\substTy{\eta^{-1}(x)}{\xx}$ to be $\xx$ and $\substTyMor{\eta^{-1}(x)}$ to be the identity.
Note that it would also have been sufficient to assume that $\D$ is equipped with a global isocleaving~\cite{buckley:2014}.

For the counit,
we need to find for each $\xx : \dob{\D}{\F(x)}$ a displayed morphism $\epsilonepsilon(\xx)$.
\[
\begin{tikzcd}[column sep=large]
  {\substTy{\eta^{-1}(\F(x))}{\xx}} & \xx \\
  {\F(\G(\F(x)))} & {\F(x)}
  \arrow["{\epsilonepsilon(\xx)}", dashed, from=1-1, to=1-2]
  \arrow["{\F(\varepsilon(x))}"', from=2-1, to=2-2]
\end{tikzcd}
\]
Here one would expect to use the morphism $\dmor{\substTy{\eta(x)}{\xx}}{\xx}{\eta^{-1}(\F(x))}$.
However,
this morphism does not lie in the fiber of $\F(\varepsilon(x))$,
but instead it lies in the fiber of $\eta^{-1}(\F(x))$.
Hence,
to get the desired morphism,
we need to have an invertible 2-cell between $\F(\varepsilon(x))$ and $\eta^{-1}(\F(x))$,
which means that our biequivalence must be sufficiently coherent.

To solve this problem,
we use that every biequivalence can be refined into a coherent one
where the triangle identities are satisfied up to coherent invertible 2-cell~\cite[Theorem 3.2]{gurski:2012}.
Similar statements have been considered for quasi-invertible maps in homotopy type theory~\cite[Theorem 4.2.3]{program:2013},
and for adjoint equivalence in a bicategory~\cite[Theorem 1.9]{gurski:2012}.
Since we only need the triangle identities,
we discuss those and not any coherence.

\begin{problem}
\label{prob:coherent-biequiv}
Given a biequivalence $(\F, \G, \eta, \varepsilon)$ where $\F : \B_1 \rightarrow \B_2$,
to construct a biequivalence $(\F, \G, \eta, \varepsilon')$
and for each $x : \B_1$
an invertible 2-cell $\theta : \invcell{\F(\varepsilon'(x))}{\eta(\F(x))}$.
\end{problem}

\begin{construction}{\coqdocurl{Bicategories.PseudoFunctors.BiequivalenceCoherent}{coherent_is_biequivalence_adjoints}}{prob:coherent-biequiv}
\label{constr:coherent-biequiv}
We only show how $\varepsilon'$ is defined.
Note that we have pseudotransformations $\eta^{-1} : \G \cdot \F \twocell \id$ and $\varepsilon : \F \cdot \G \twocell \id$.
Since $\varepsilon$ is an equivalence,
we also have a pseudotransformation $\varepsilon^{-1} : \id \twocell \F \cdot \G$.
We have the following composition.
\[
\begin{tikzcd}[column sep = 50pt, row sep = large]
  {\F \cdot \G} & {\id \cdot (\F \cdot \G)} & {(\F \cdot \G) \cdot (\F \cdot \G)} & {\F \cdot (\G \cdot (\F \cdot \G))} \\
  {\F \cdot ((\G \cdot \F) \cdot \G)} & {\F \cdot (\id \cdot \G)} & {\F \cdot \G} & \id
  \arrow["\linvunitor", Rightarrow, from=1-1, to=1-2]
  \arrow["{\varepsilon^{-1} \whiskerr (\F \cdot \G)}", Rightarrow, from=1-2, to=1-3]
  \arrow["\lassociator", Rightarrow, from=1-3, to=1-4]
  \arrow["{\F \whiskerl \rassociator}"{description}, Rightarrow, from=1-4, to=2-1]
  \arrow["{\F \whiskerl (\eta^{-1} \whiskerr \G)}"', Rightarrow, from=2-1, to=2-2]
  \arrow["{\F \whiskerl \lunitor}"', Rightarrow, from=2-2, to=2-3]
  \arrow["\varepsilon"', Rightarrow, from=2-3, to=2-4]
\end{tikzcd}
\]
Here the pseudotransformations $\lunitor$, $\linvunitor$, $\lassociator$, and $\rassociator$ are defined to be the identity 1-cell pointwise,
and the whiskering operations are defined in the usual way.
The composition above is a composition of adjoint equivalences,
and we define $\varepsilon'$ as the inverse.
\end{construction}

\begin{problem}
\label{prob:reindex-biequiv}
Given univalent bicategories $\B_1$ and $\B_2$,
a biequivalence $(\F, \G, \eta, \varepsilon)$ where $\F : \B_1 \rightarrow \B_2$,
and a displayed bicategory $\D$ over $\B_2$ with trivial 2-cells,
to construct a displayed biequivalence $(\FF, \GG, \etaeta, \epsilonepsilon)$
over $(\F, \G, \eta, \varepsilon)$ where $\FF : \dmor{\reindex{\F}{\D}}{\D}{\F}$.
\end{problem}

\begin{construction}{\coqdocurl{Bicategories.DisplayedBicats.ReindexBiequivalence}{reindex_disp_is_biequivalence_univ_coh}}{prob:reindex-biequiv}
\label{constr:reindex-biequiv}
By \Cref{constr:coherent-biequiv},
we can assume without loss of generality
that we have invertible 2-cells
$\theta : \invcell{\F(\varepsilon'(x))}{\eta(\F(x))}$
for each $x : \B_1$.
As discussed before,
$\FF$ sends $\xx : \dob{\reindex{\F}{\D}}{x}$ to $\xx : \dob{\D}{\F(x)}$.
For $\GG$ we use that $\B_2$ is univalent
and that we have an adjoint equivalence $\eta(x) : \adjequiv{x}{\F(\G(x))}$,
meaning that every $\xx : \dob{\D}{x}$ gives rise to an object $\substTy{\eta^{-1}(x)}{\xx} : \dob{\D}{\F(\G(x))}$ as follows.
\[
\begin{tikzcd}[column sep=large]
  {\substTy{\eta^{-1}(x)}{\xx}} & \xx \\
  {\F(\G(x))} & x
  \arrow["{\substTyMor{\eta^{-1}(x)}}", dashed, from=1-1, to=1-2]
  \arrow["{\eta^{-1}(x)}"', from=2-1, to=2-2]
\end{tikzcd}
\]
We define $\etaeta^{-1}(\xx) : \dmor{\substTy{\eta^{-1}(x)}{\xx}}{\xx}{\eta^{-1}(x)}$ to be $\substTyMor{\eta^{-1}(x)}$.
For $\epsilonepsilon$,
we use that $\D$ can be equipped with a local isocleaving
since $\B_2$ is univalent.
We define $\epsilonepsilon(\xx)$ to be $\substTyMor{\eta^{-1}(\F(x))}$.
\[
\begin{tikzcd}[row sep = 45pt, column sep = huge]
  {\substTy{\eta^{-1}(\F(x))}{\xx}} & {\xx} \\
  {\F(\G(\F(x)))} & {\F(x)}
  \arrow[""{name=0, anchor=center, inner sep=0}, "{\substTyMor{\eta^{-1}(\F(x))}}"', bend right=25, from=1-1, to=1-2, pos=.4]
  \arrow[""{name=1, anchor=center, inner sep=0}, "{\substTy{\theta(x)}{\substTyMor{\eta^{-1}(\F(x))}}}", bend left=25, dashed, from=1-1, to=1-2, pos=.4]
  \arrow[""{name=2, anchor=center, inner sep=0}, "{\eta^{-1}(\F(x))}"', bend right=25, from=2-1, to=2-2]
  \arrow[""{name=3, anchor=center, inner sep=0}, "{\F(\varepsilon(x))}", bend left=25, from=2-1, to=2-2]
  \arrow["{\theta(x)}", shorten <=2pt, shorten >=2pt, Rightarrow, from=3, to=2, shorten >=1ex]
  \arrow["{\substTyMor{\theta(x)}}", shorten <=2pt, shorten >=2pt, Rightarrow, from=1, to=0, shorten >=1ex]
\end{tikzcd}
\]
Since $\substTyMor{\eta^{-1}(\F(x))}$ is a displayed adjoint equivalence,
$\epsilonepsilon(\xx)$ is so as well.
\end{construction}

\subsection{Type Formers in Universes}
Now we have the necessary tools to
prove the internal language theorem for elementary toposes with a natural numbers object and a universe.
Our starting point are the bicategories $\EltopNNOUniv$ and $\EltopNNOCCUniv$.
The objects of $\EltopNNOUniv$ are univalent elementary toposes with a natural numbers object
and a universe that is not required to be closed under any type former.
Objects of $\EltopNNOCCUniv$ are DFL comprehension categories with $\Pi$-types and a universe object,
and that satisfy $\EltopNNOP$ fiberwise.

\begin{defi}
\label[defi]{def:topos-univ-obj}
We define the bicategory \conceptDef{$\EltopNNOUniv$}{Bicategories.ComprehensionCat.InternalLanguageTopos.ToposNatUniv}{bicat_of_univ_topos_with_NNO_univ}
to be $\total{(\EltopNNO \times \FinLimUniv)}$,
and we define \conceptDef{$\EltopNNOCCUniv$}{Bicategories.ComprehensionCat.InternalLanguageTopos.ToposNatUniv}{univ_topos_with_NNO_univ_language}
to be $\total{(\EltopNNOCC \times \DFLCompCatUniv)}$.
By \Cref{exa:prod-biequiv,exa:internal-language-topos,constr:universe-biequiv}
we have a biequivalence $(\FinLimToCompCatUnivTop, \CompCatToFinLimUnivTop, \FinLimCompCatUnitUnivTop, \FinLimCompCatCounitUnivTop)$
where $\FinLimToCompCatUnivTop : \EltopNNOUniv \rightarrow \EltopNNOCCUniv$.
\end{defi}

As discussed before,
categorical developments of universes are usually based on a notion of small maps,
and we can express that the universe is closed under certain type formers via suitable closure properties.
As an example, we consider $\Sigma$-types,
which are interpreted in a category with finite limits via composition.
If we have a category $\C$ together with a predicate $\Pred$ on the morphisms of $\C$,
then saying that $\Pred$ is closed under composition expresses that our universe is closed under $\Sigma$-types.
Similarly, we can define when our universe is closed under other type formers,
like the subobject classifier and $\Pi$-types.

Our approach is slightly different.
Since we are interested in models of type theory,
we want our categorical models to interpret the expected type theoretic rules.
Such rules are expressed as operations on terms that are stable under substitution up to isomorphism.
However,
if one uses small maps,
then one only talks about a \textbf{property} on the maps.
Closure properties of small maps thus express properties of our universe
rather than giving operations.
Angiuli and Gratzer use the same ideas to give a notion of universe in categories with families~\cite[Structure 6.4.17]{angiuli2024principles},
although they require strict stability under substitution.

Depending on the context,
one uses a different collection of axioms for small maps~\cite{joyal:1995,vandenberg:2009a,awodey:2014},
and we follow Streicher~\cite{streicher:2005}.
Streicher considers the following closure properties for classes of small maps in an elementary topos:
\begin{itemize}
  \item small maps are closed under pullback;
  \item small maps are closed under composition;
  \item every monomorphism is small;
  \item the unique map from $\Nat$ to $\ECtx$ is small;
  \item the unique map from $\Omega$ to $\ECtx$ is small;
  \item small maps are closed under dependent products.
\end{itemize}
From a type theoretic point of view,
each of these conditions expresses that our universe contains a certain type former.
Since small maps are closed under composition and dependent products,
our universe is closed under $\Sigma$-types and $\Pi$-types.
The universe also is closed under substitution,
because small maps are closed under pullback.
The universe also contains $\Nat$ and $\Omega$,
because the maps from these objects to $\ECtx$ are small.
Since monomorphisms correspond to internal propositions by \Cref{prop:hprop-mono},
our universe satisfies propositional resizing,
meaning that it contains every proposition.

In the remainder of this section,
we define when a universe contains $\Nat$ and $\Omega$,
and when it is closed under $\Sigma$-types, $\Pi$-types, and propositional resizing.
We also study the relation between our development
and the closure properties of small maps that are used for classes of small maps as given by Streicher~\cite{streicher:2005}.
To do so,
we use that each universe with an elements map gives rise to a class of small maps.

\begin{defi}
\label[defi]{def:small-map}
Let $\C$ be a univalent category with an object $\CatUniv$
and an elements map $\CatEl$ for $\CatUniv$.
A morphism $f : x \rightarrow y$ is said to be \conceptDef{small}{Bicategories.ComprehensionCat.Universes.SmallMaps}{is_small_map}
if there merely exists a morphism $a : y \rightarrow \CatUniv$
and an isomorphism $h : \iso{\CatEl(a)}{x}$ making the following diagram commute.
\[
\begin{tikzcd}
  {\CatEl(a)} && x \\
  & y
  \arrow["h", from=1-1, to=1-3]
  \arrow["{\CatElM(t)}"', from=1-1, to=2-2]
  \arrow["f", from=1-3, to=2-2]
\end{tikzcd}
\]
An object $x : \C$ is said to be \conceptDef{small}{Bicategories.ComprehensionCat.Universes.SmallMaps}{small_object}
if the map $\ECtxMap : x \rightarrow \ECtx$ is small.
\end{defi}

It is important to note that being small is a property for morphisms rather than structure,
as we do not require $a$ and $h$ to be chosen.
Small maps are closed under isomorphism.
One can prove this property using that $\C$ is univalent,
but for a concrete proof,
we use the following diagram.
\[
\begin{tikzcd}
  {\CatEl(a)} & x & {x'} \\
  & y
  \arrow["h", from=1-1, to=1-2]
  \arrow["{\CatElM(a)}"', from=1-1, to=2-2]
  \arrow["{h'}", from=1-2, to=1-3]
  \arrow["f"{description}, from=1-2, to=2-2]
  \arrow["{f'}", from=1-3, to=2-2]
\end{tikzcd}
\]
If $a$ and $h$ witness that $f$ is small,
then $a$ and $h \cdot h'$ witness that $f'$ is small.

To see why small maps are closed under pullbacks,
we use the following diagram.
\[
\begin{tikzcd}
  {\CatEl(s \cdot a)} & x & y & {\CatEl(a)} \\
  & \Delta & \Gamma && \CatUniv
  \arrow["\cong", from=1-1, to=1-2]
  \arrow["{\CatElM(s \cdot a)}"', from=1-1, to=2-2]
  \arrow["h", from=1-2, to=1-3]
  \arrow["g"{description}, from=1-2, to=2-2]
  \arrow["\lrcorner"{anchor=center, pos=0.125}, draw=none, from=1-2, to=2-3]
  \arrow["f"{description}, from=1-3, to=2-3]
  \arrow["\cong"', from=1-4, to=1-3]
  \arrow["{\CatElM(a)}", from=1-4, to=2-3]
  \arrow["s"', from=2-2, to=2-3]
  \arrow["a"', from=2-3, to=2-5]
\end{tikzcd}
\]
If we assume $f : y \rightarrow \Gamma$ to be small,
then we have a morphism $a : \Gamma \rightarrow \CatUniv$ together with an isomorphism $\iso{y}{\CatEl(a)}$.
Since $\CatEl(s \cdot a)$ is the pullback of $s$ and $\CatElM(a)$,
we must have an isomorphism $\iso{\CatEl(s \cdot a)}{x}$,
which proves that $g$ is indeed small.

\subsubsection{Natural Numbers}
\label{sec:nat-type-univ}
Now we define various displayed bicategories over $\EltopNNOUniv$,
whose objects over $(\C, \CatUniv)$ are structures that represent that $\CatUniv$ is closed under some type former.
We start with the natural numbers object.

\begin{defi}
\label[defi]{def:small-nno}
Let $\C$ be a univalent category with finite limits,
and let $\CatEl$ be an elements map for an object $\CatUniv : \C$.
Suppose that we have a parameterized natural numbers object $\Nat$ in $\C$.
A \conceptDef{code for the natural numbers}{Bicategories.ComprehensionCat.Universes.CatTypes.Constant}{pnno_in_cat_univ} for $\CatEl$
consists of a morphism $\NatCode : \ECtx \rightarrow \CatUniv$
together with an isomorphism $\NatIso : \iso{\CatEl(\NatCode)}{\Nat}$.
\end{defi}

Note that we do not require a stability property in \Cref{def:small-nno}.
This is because the morphism $\NatCode : \ECtx \rightarrow \CatUniv$ gives rise to a morphism $\NatCodeCtx : \Gamma \rightarrow \CatUniv$ by composition,
which is automatically stable under substitution.
In the language of small maps,
one expresses that the universe contains the natural numbers object
by saying that $\Nat$ is small.
If $\CatEl$ has a code for the natural numbers,
then this closure property follows.

\begin{propL}[\coqdocurl{Bicategories.ComprehensionCat.Universes.SmallMaps}{nno_small}]
\label[propL]{prop:small-nno}  
If $\CatEl$ has a code for the natural numbers,
then $\Nat$ is small.
\end{propL}

\begin{proof}
The desired code and isomorphism are given by $\NatCode$ and $\NatIso$ respectively.
\end{proof}

Recall that every functor $\F : \C_1 \rightarrow \C_2$ that preserves natural numbers objects,
comes with an isomorphism $\FNat{\F} : \iso{\F(\Nat_1)}{\Nat_2}$
where $\Nat_1$ and $\Nat_2$ are the natural numbers objects of $\C_1$ and $\C_2$ respectively.
Now we define the first of the desired displayed bicategories.

\begin{defi}
\label[defi]{def:small-nno-disp-bicat}
Suppose that the elements maps $\CatEl_1$ and $\CatEl_2$ for $\CatUniv_1 : \C_1$ and $\CatUniv_2 : \C_2$ have codes for the natural numbers.
Let $\F : \C_1 \rightarrow \C_2$ be a functor that preserves elements maps
and the natural numbers objects.
We say that $\F$ \conceptDef{preserves codes for natural numbers}{Bicategories.ComprehensionCat.Universes.CatTypes.Constant}{functor_preserves_pnno_in_cat_univ}
if the following diagrams commute.
\[
\begin{tikzcd}
  {\F \> \ECtx} & {\F(\CatUniv_1)} \\
  {\ECtx} & {\CatUniv_2}
  \arrow["{\F(\NatCode_1)}", from=1-1, to=1-2]
  \arrow["{\ECtxMap}"', from=1-1, to=2-1]
  \arrow["{\FunctorUniv{\F}}", from=1-2, to=2-2]
  \arrow["{\NatCode_2}"', from=2-1, to=2-2]
\end{tikzcd}
\]
\[
\begin{tikzcd}[column sep=large]
  {\F(\CatEl_1(\NatCode_1))} &&& {\F(\Nat_1)} \\
  {\CatEl_2(\F(\NatCode_1) \cdot \FunctorUniv{\F})} & {\CatEl_2(\ECtxMap \cdot \NatCode_2)} & {\CatEl_2(\NatCode_2)} & {\Nat_2}
  \arrow["{\F(\NatIso_1)}", from=1-1, to=1-4]
  \arrow["{\FunctorEl{\F}}"', from=1-1, to=2-1]
  \arrow["{\FNat{\F}}", from=1-4, to=2-4]
  \arrow["\CatElEq"', from=2-1, to=2-2]
  \arrow["\CatElSub"', from=2-2, to=2-3]
  \arrow["{\NatIso_2}"', from=2-3, to=2-4]
\end{tikzcd}
\]
We define the displayed bicategory \conceptDef{$\EltopNNOUnivNat$}{Bicategories.ComprehensionCat.InternalLanguageTopos.ToposNatUniv}{disp_bicat_topos_with_NNO_univ_nno_type}
over $\EltopNNOUniv$ whose
\begin{itemize}
  \item objects over a univalent elementary topos $\C$ with an elements map $\CatEl$
    are codes for the natural numbers;
  \item morphisms over $(\F, \FunctorEl{\F})$ are proofs that $\F$ preserves codes for natural numbers.
\end{itemize}
The type of 2-cells over an arbitrary natural transformation is the unit type.
\end{defi}

\subsubsection{Subobject Classifiers}
\label{sec:omega-type-univ}
We use the same ideas to define when a universe contains the subobject classifier.
Specifically,
we require there to be a term in the empty context
whose associated type is isomorphic to the subobject classifier.
Smallness of $\Omega$ is proven in a similar way.

\begin{defi}
\label[defi]{def:small-omega}
Let $\C$ be a univalent category with finite limits,
and let $\CatEl$ be an elements map for an object $\CatUniv : \C$.
Suppose that we have a subobject classifier $\Omega$ in $\C$.
A \conceptDef{subobject classifier code}{Bicategories.ComprehensionCat.Universes.CatTypes.Constant}{subobject_classifier_in_cat_univ} for $\CatEl$
consists of a morphism $\SubObjCode : \ECtx \rightarrow \CatUniv$
together with an isomorphism $\SubObjIso : \iso{\CatEl(\SubObjCode)}{\Omega}$.
\end{defi}

\begin{propL}[\coqdocurl{Bicategories.ComprehensionCat.Universes.SmallMaps}{subobject_classifier_small}]
\label[propL]{prop:small-omega}
If $\CatEl$ has a subobject classifier code,
then $\Omega$ is small.
\end{propL}

\begin{proof}
The smallness of $\Omega$ follows from $\SubObjCode$ and $\SubObjIso$.
\end{proof}

To construct the desired displayed bicategory,
we define when functors preserve codes for subobject classifier.
Here we use that we have an isomorphism $\FOmega{\F} : \iso{\F(\Omega_1)}{\Omega_2}$ whenever $\F$ preserves subobject classifiers.

\begin{defi}
\label[defi]{def:small-omega-disp-bicat}
Suppose that the elements maps $\CatEl_1$ and $\CatEl_2$ for $\CatUniv_1 : \C_1$ and $\CatUniv_2 : \C_2$ have codes for subobject classifiers.
If $\F : \C_1 \rightarrow \C_2$ is a functor that preserves elements maps
and subobject classifiers,
we say that $\F$ \conceptDef{preserves codes for subobject classifiers}{Bicategories.ComprehensionCat.Universes.CatTypes.Constant}{functor_preserves_subobject_classifier_in_cat_univ}
if the following diagrams commute.
\[
\begin{tikzcd}
  {\F \> \ECtx} & {\F(\CatUniv_1)} \\
  {\ECtx} & {\CatUniv_2}
  \arrow["{\F(\SubObjCode_1)}", from=1-1, to=1-2]
  \arrow["{\ECtxMap}"', from=1-1, to=2-1]
  \arrow["{\FunctorUniv{\F}}", from=1-2, to=2-2]
  \arrow["{\SubObjCode_2}"', from=2-1, to=2-2]
\end{tikzcd}
\]
\[
\begin{tikzcd}[column sep=large]
  {\F(\CatEl_1(\SubObjCode_1))} &&& {\F(\Omega_1)} \\
  {\CatEl_2(\F(\SubObjCode_1) \cdot \FunctorUniv{\F})} & {\CatEl_2(\ECtxMap \cdot \SubObjCode_2)} & {\CatEl_2(\SubObjCode_2)} & {\Omega_2}
  \arrow["{\F(\SubObjIso_1)}", from=1-1, to=1-4]
  \arrow["{\FunctorEl{\F}}"', from=1-1, to=2-1]
  \arrow["{\FOmega{\F}}", from=1-4, to=2-4]
  \arrow["\CatElEq"', from=2-1, to=2-2]
  \arrow["\CatElSub"', from=2-2, to=2-3]
  \arrow["{\SubObjIso_2}"', from=2-3, to=2-4]
\end{tikzcd}
\]
We define the displayed bicategory \conceptDef{$\EltopNNOUnivOmega$}{Bicategories.ComprehensionCat.InternalLanguageTopos.ToposNatUniv}{disp_bicat_topos_with_NNO_univ_subobj_classifier}
over $\EltopNNOUniv$ whose
\begin{itemize}
  \item objects over a univalent elementary topos $\C$ with an elements map $\CatEl$
    are codes for subobject classifiers;
  \item morphisms over $(\F, \FunctorEl{\F})$ are proofs that $\F$ preserves codes for subobject classifiers.
\end{itemize}
The type of 2-cells over an arbitrary natural transformation is the unit type.
\end{defi}

\subsubsection{Propositional Resizing}
\label{sec:resizing-univ}
Next we consider closure conditions for propositional resizing,
and here we meet a complication compared to the natural numbers and subobject classifier.
Whereas it was not necessary to require any stability conditions for $\Nat$ and $\Omega$,
we need to consider suitable stability conditions for resizing codes.
The reason is that propositional resizing takes types in arbitrary contexts,
but for $\Nat$ and $\Omega$ it was sufficient to only consider a type in the empty context.

\begin{defi}
\label[defi]{def:resizing}
Let $\C$ be a univalent category with finite limits,
and let $\CatEl$ be an elements map for an object $\CatUniv : \C$.
A \conceptDef{family of resizing codes}{Bicategories.ComprehensionCat.Universes.CatTypes.Resizing}{cat_univ_stable_codes_resizing} for $\CatEl$
consists of a map that assigns to each monomorphism $m : x \rightarrow y$
a morphism $\ResizeCode(m) : y \rightarrow \CatUniv$
and an isomorphism $\ResizeIso(m) : \iso{\CatEl(\ResizeCode(m))}{x}$ making the following diagram commute.
\[
\begin{tikzcd}
  {\CatEl(\ResizeCode(m))} && x \\
  & y
  \arrow["{\ResizeIso(m)}", from=1-1, to=1-3]
  \arrow["{\CatElM(\ResizeCode(m))}"', from=1-1, to=2-2]
  \arrow["m", hook, from=1-3, to=2-2]
\end{tikzcd}
\]
We require the following diagram to commute for each $s : w \rightarrow y$ and $m : x \rightarrow y$
\[
\begin{tikzcd}
  w && y \\
  & \CatUniv
  \arrow["s", from=1-1, to=1-3]
  \arrow["{\ResizeCode(m')}"', from=1-1, to=2-2]
  \arrow["{\ResizeCode(m)}", from=1-3, to=2-2]
\end{tikzcd}
\]
where $m'$ is defined using the following pullback.
\[
\begin{tikzcd}
  {\substTy{s}{x}} & x \\
  w & y
  \arrow["{\substTyMor{s}}", from=1-1, to=1-2]
  \arrow["{m'}"', hook, from=1-1, to=2-1]
  \arrow["\lrcorner"{anchor=center, pos=0.125}, draw=none, from=1-1, to=2-2]
  \arrow["m", hook, from=1-2, to=2-2]
  \arrow["s"', from=2-1, to=2-2]
\end{tikzcd}
\]
\end{defi}

Let us make a couple of remarks on \Cref{def:resizing}.
A monomorphism $m : x \rightarrow y$ represents an internal proposition in context $y$ by \Cref{prop:hprop-mono},
and a morphism $\ResizeCode(m) : y \rightarrow \CatUniv$ represents a term of type $\CatUniv$ in context $y$.
Hence, a family of resizing codes assigns to each internal proposition in context $y$
a term of type $\CatUniv$ in context $y$
whose associated type is isomorphic to $x$.
Such a family thus interprets propositional resizing.

In addition, we only required a stability condition for $\ResizeCode(m)$ in \Cref{def:resizing},
because stability for $\ResizeIso(m)$ is derivable.
Specifically,
this stability requirement means that the upper square of the following diagram commutes.
\[
\begin{tikzcd}
  {\CatEl(\ResizeCode(m'))} & {\CatEl(s \cdot \ResizeCode(m))} \\
  & {\CatEl(\ResizeCode(m))} \\
  {\substTy{s}{x}} & x \\
  w & y
  \arrow["\CatElEq", from=1-1, to=1-2]
  \arrow["{\ResizeIso(m')}"', from=1-1, to=3-1]
  \arrow["{\CatElSub(s, \ResizeCode(m))}", from=1-2, to=2-2]
  \arrow["{\ResizeIso(m)}", from=2-2, to=3-2]
  \arrow["{\substTyMor{s}}", from=3-1, to=3-2]
  \arrow["{m'}"', hook, from=3-1, to=4-1]
  \arrow["m", hook, from=3-2, to=4-2]
  \arrow[""{name=0, anchor=center, inner sep=0}, "s"', from=4-1, to=4-2]
  \arrow["\lrcorner"{anchor=center, pos=0.125}, draw=none, from=3-1, to=0]
\end{tikzcd}
\]
We can prove that the upper square indeed commutes,
because $m$ is a monomorphism
and $\ResizeIso(m)$ is an isomorphism in the slice category.
The desired stability property follows from the fact that the following diagram commutes.
\[
\begin{tikzcd}[column sep = 60pt, row sep=huge]
  {\CatEl(\ResizeCode(m'))} & {\CatEl(s \cdot \ResizeCode(m))} & {\CatEl(\ResizeCode(m))} \\
  & w & y
  \arrow["\CatElEq", from=1-1, to=1-2]
  \arrow["{\CatElM(\ResizeCode(m'))}"', from=1-1, to=2-2]
  \arrow["{\CatElSub(s, \ResizeCode(m))}", from=1-2, to=1-3]
  \arrow["{\CatElM(s \cdot \ResizeCode(m))}"{description}, from=1-2, to=2-2]
  \arrow["{\CatElM(\ResizeCode(m))}", from=1-3, to=2-3]
  \arrow["s"', from=2-2, to=2-3]
\end{tikzcd}
\]
Type theoretically,
monomorphisms correspond to internal propositions (\Cref{prop:hprop-mono}),
and our proof uses that all inhabitants of a proposition are equal.

From our notion of resizing codes,
one can deduce the analogous closure condition for small maps:
every monomorphism is small.

\begin{propL}[\coqdocurl{Bicategories.ComprehensionCat.Universes.SmallMaps}{all_monics_small}]
\label[propL]{prop:resizing}
If $\CatEl$ is equipped with a family of resizing codes,
then each monomorphism is small.
\end{propL}

\begin{proof}
The smallness of a monomorphism $m$ is witnessed by $\ResizeCode(m)$ and $\ResizeIso(m)$.
\end{proof}

With these notions in place,
we are ready to define the displayed bicategory of resizing codes for a universe.
Recall that every functor $\F$ that preserves finite limits,
also preserves monomorphisms.

\begin{defi}
\label[defi]{def:resizing-disp-bicat}
Suppose that the elements maps $\CatEl_1$ and $\CatEl_2$ for $\CatUniv_1 : \C_1$ and $\CatUniv_2 : \C_2$ have resizing codes,
and let $\F : \C_1 \rightarrow \C_2$ be a functor that preserves finite limits and elements maps.
We say that $\F$ \conceptDef{preserves resizing codes}{Bicategories.ComprehensionCat.Universes.CatTypes.Resizing}{functor_preserves_stable_resizing_codes}
if the following diagram commutes for every monomorphism $m : x \rightarrow y$.
\[
\begin{tikzcd}[row sep = large,column sep = large]
  {\F(y)} && {\F(\CatUniv_1)} \\
  & {\CatUniv_2}
  \arrow["{\F(\ResizeCode(m))}", from=1-1, to=1-3]
  \arrow["{\ResizeCode(\F(m))}"', from=1-1, to=2-2]
  \arrow["{\FunctorUniv{\F}}", from=1-3, to=2-2]
\end{tikzcd}
\]
We define the displayed bicategory \conceptDef{$\EltopNNOUnivResize$}{Bicategories.ComprehensionCat.InternalLanguageTopos.ToposNatUniv}{disp_bicat_topos_with_NNO_univ_resizing}
over $\EltopNNOUniv$ whose
\begin{itemize}
  \item objects over a univalent elementary topos $\C$ with an elements map $\CatEl$
    are resizing codes;
  \item morphisms over $(\F, \FunctorEl{\F})$ are proofs that $\F$ preserves resizing codes.
\end{itemize}
The type of 2-cells over an arbitrary natural transformation is the unit type.
\end{defi}

Whereas in \Cref{def:small-nno-disp-bicat,def:small-omega-disp-bicat}
we require functors to both preserve the code and the isomorphism,
we do not do so for resizing codes in \Cref{def:resizing-disp-bicat}.
Again the reason is that the analogous preservation requirement would hold vacuously,
because all inhabitants of propositions are equal.

\subsubsection{$\Sigma$-types}
\label{sec:sum-type-univ}
Now we consider closure conditions for when a universe contains $\Sigma$-types.
Intuitively,
we want to say
whenever we have types $A$ and $B$ in $\CatUniv$
where $B$ might depend on $A$,
then their dependent sum also is in $\CatUniv$.
Since $\Sigma$-types are interpreted using composition,
this closure condition is translated as follows to the language of categories with finite limits.

\begin{defi}
\label[defi]{def:small-sigma}
Let $\C$ be a univalent category with finite limits
with an elements map $\CatEl$ for $\CatUniv : \C$.
A \conceptDef{family of $\Sigma$-codes}{Bicategories.ComprehensionCat.Universes.CatTypes.Sigma}{cat_univ_stable_codes_sigma} for $\CatEl$
is a map that assigns to each $a : y \rightarrow \CatUniv$ and $b : \CatEl(a) \rightarrow \CatUniv$
a morphism $\SigmaCode(a, b) : y \rightarrow \CatUniv$
and an isomorphism $\SigmaIso(a, b) : \iso{\CatEl(\SigmaCode(a, b))}{\CatEl(b)}$
making the left triangle in the diagram below commute.
\[
\begin{tikzcd}[row sep = large, column sep = huge]
  {\CatEl(\SigmaCode(a, b))} & {\CatEl(b)} \\
  & {\CatEl(a)} \\
  & y & \CatUniv
  \arrow["{\SigmaIso(a, b)}", dashed, from=1-1, to=1-2]
  \arrow["{\CatElM(\SigmaCode(a, b))}"', from=1-1, to=3-2]
  \arrow["{\CatElM(b)}"{description}, from=1-2, to=2-2]
  \arrow["{\CatElM(a)}"{description}, from=2-2, to=3-2]
  \arrow["{b}", bend left=30, from=2-2, to=3-3]
  \arrow["a", shift left, bend left = 10, from=3-2, to=3-3]
  \arrow["{\SigmaCode(a, b)}"', shift right, bend right = 10, dashed, from=3-2, to=3-3]
\end{tikzcd}
\]
We also require a stability condition for $\SigmaCode$ and $\SigmaIso$,
which says that for all morphisms $s : x \rightarrow y$,
$a : y \rightarrow \CatUniv$, and $b : \CatEl(a) \rightarrow \CatUniv$
we have that $s \cdot \SigmaCode(a, b) = \SigmaCode(s \cdot a, \CatElSub \cdot b)$,
which we display in the diagram below.
\[
\begin{tikzcd}[row sep = large, column sep = huge]
  {\CatEl(\CatElSub \cdot b)} & {\CatEl(b)} \\
  {\CatEl(s \cdot a)} & {\CatEl(a)} \\
  x & y & \CatUniv
  \arrow["\CatElSub", from=1-1, to=1-2]
  \arrow["{\CatElM(\CatElSub \cdot b)}"', from=1-1, to=2-1]
  \arrow["{\CatElM(b)}"', from=1-2, to=2-2]
  \arrow["\CatElSub"{description}, from=2-1, to=2-2]
  \arrow["{\CatElM(s \cdot a)}"', from=2-1, to=3-1]
  \arrow["\lrcorner"{anchor=center, pos=0.125}, draw=none, from=1-1, to=2-2]
  \arrow["\lrcorner"{anchor=center, pos=0.125}, draw=none, from=2-1, to=3-2]
  \arrow["{\CatElM(a)}"', from=2-2, to=3-2]
  \arrow["b", shorten >=1ex, bend left=30, from=2-2, to=3-3]
  \arrow["s"', from=3-1, to=3-2]
  \arrow["{\SigmaCode(s \cdot a, \CatElSub \cdot b)}"', shorten >=0.5ex, bend right=40, from=3-1, to=3-3]
  \arrow["a", shorten >=0.3ex, shift left, bend left=10, from=3-2, to=3-3]
  \arrow["{\SigmaCode(a, b)}"', shorten >=0.3ex, shift right, bend right=10, from=3-2, to=3-3]
\end{tikzcd}
\]
We also require the following diagram to commute.
\[
\begin{tikzcd}[column sep = 60pt]
  {\CatEl(s \cdot \SigmaCode(a, b))} && {\CatEl(\SigmaCode(a, b))} \\
  {\CatEl(\SigmaCode(s \cdot a, \CatElSub \cdot b))} & {\CatEl(\CatElSub \cdot b)} & {\CatEl(b)}
  \arrow["\CatElSub", from=1-1, to=1-3]
  \arrow["\CatElEq"', from=1-1, to=2-1]
  \arrow["{\SigmaIso(a, b)}", from=1-3, to=2-3]
  \arrow["{\SigmaIso(s \cdot a, \CatElSub \cdot b)}"', from=2-1, to=2-2]
  \arrow["\CatElSub"', from=2-2, to=2-3]
\end{tikzcd}
\]
\end{defi}

Whereas stability for $\ResizeIso$ was derivable in \Cref{def:resizing},
we cannot derive the second coherence condition in \Cref{def:small-sigma}.
Intuitively,
this is because $\Sigma$-types are not necessarily propositions.

If a universe is equipped with a family of $\Sigma$-codes,
then its associated class of small maps satisfies the expected closure condition.

\begin{propL}[\coqdocurl{Bicategories.ComprehensionCat.Universes.SmallMaps}{all_composition_small}]
\label[propL]{prop:small-sigma}
If a universe is equipped with a family of $\Sigma$-codes,
then the composition of small maps is again small.
\end{propL}

\begin{proof}
Suppose that $f : x \rightarrow y$ and $g : y \rightarrow z$ are small.
Then we have morphisms $c_f : y \rightarrow \CatUniv$ and $c_g : z \rightarrow \CatUniv$
and isomorphisms $h_f : \iso{\CatEl(c_f)}{x}$ and $h_g : \iso{\CatEl(c_g)}{y}$
making the following diagrams commute.
\[
\begin{tikzcd}
  & {\CatEl(c_f)} && x \\
  {\CatEl(c_g)} && y \\
  & z
  \arrow["{h_f}", from=1-2, to=1-4]
  \arrow["{\CatElM(c_f)}"{description}, from=1-2, to=2-3]
  \arrow["f", from=1-4, to=2-3]
  \arrow["{h_g}"{description}, from=2-1, to=2-3]
  \arrow["{\CatElM(c_g)}"', from=2-1, to=3-2]
  \arrow["g", from=2-3, to=3-2]
\end{tikzcd}
\]
To see that $f \cdot g$ is small,
we use the following diagram.
\[
\begin{tikzcd}
  {\CatEl(\SigmaCode(c_g, h_g \cdot c_f))} & {\CatEl(h_g \cdot c_f)} && {\CatEl(c_f)} && x \\
  && {\CatEl(c_g)} && y \\
  &&& z && \CatUniv
  \arrow["{\SigmaIso}", from=1-1, to=1-2] 
  \arrow["{\CatElM(\SigmaCode(c_g, h_g \cdot c_f))}"', bend right=20, from=1-1, to=3-4]
  \arrow["\CatElSub", from=1-2, to=1-4]
  \arrow["{\CatElM(h_g \cdot c_f)}"{description}, from=1-2, to=2-3]
  \arrow["{h_f}", from=1-4, to=1-6]
  \arrow["{\CatElM(c_f)}"{description}, from=1-4, to=2-5]
  \arrow["f", from=1-6, to=2-5]
  \arrow["{h_g}"{description}, from=2-3, to=2-5]
  \arrow["{\CatElM(c_g)}"{description}, from=2-3, to=3-4]
  \arrow["g"{description}, from=2-5, to=3-4]
  \arrow["{c_f}", from=2-5, to=3-6]
  \arrow["{c_g}", from=3-4, to=3-6]
  \arrow["{\SigmaCode(c_g, h_g \cdot c_f)}"', shift right=1, bend right=20, from=3-4, to=3-6]
  \arrow["\lrcorner"{anchor=center, pos=0.06}, draw=none, from=1-2, to=2-5]
\end{tikzcd}
\]
The desired code for $f \cdot g$ is $\SigmaCode(c_g, h_g \cdot c_f)$.
We also need to show $\CatEl(\SigmaCode(c_g, h_g \cdot c_f))$ is isomorphic to $x$,
and for that we prove that $\SigmaIso \cdot \CatElSub \cdot h_f$ is an isomorphism.
Note that $\SigmaIso$ and $h_f$ are isomorphisms by assumption.
To see why $\CatElSub$ is an isomorphism,
we use that it is the pullback of $h_g$
and that pullbacks of isomorphisms are again isomorphisms.
\end{proof}

To construct the displayed bicategory of $\Sigma$-codes,
we define when functors preserve $\Sigma$-codes.
Here we use the same ideas as for \Cref{def:small-nno-disp-bicat,def:small-omega-disp-bicat}:
functors need to preserve both $\SigmaCode$ and $\SigmaIso$.

\begin{defi}
\label[defi]{def:sigma-code-disp-bicat}
Suppose that we have elements maps $\CatEl_1$ and $\CatEl_2$ for $\CatUniv_1 : \C_1$ and $\CatUniv_2 : \C_2$ respectively,
and a functor $\F : \C_1 \rightarrow \C_2$
with an elements preservation family $\FunctorEl{\F}$.
Let $\SigmaCode_1$ and $\SigmaCode_2$ be $\Sigma$-codes for $\CatEl_1$ and $\CatEl_2$ respectively.
We say that $\F$ \conceptDef{preserves $\Sigma$-codes}{Bicategories.ComprehensionCat.Universes.CatTypes.Sigma}{functor_preserves_stable_codes_sigma}
whenever the following diagrams commute for all $a : y \rightarrow \CatUniv_1$ and $b : \CatEl(a) \rightarrow \CatUniv_1$.
\[
\begin{tikzcd}
  {\F(y)} && {\F(\CatUniv_1)} \\
  & {\CatUniv_2}
  \arrow["{\F(\SigmaCode_1(a,b))}", from=1-1, to=1-3]
  \arrow["{\SigmaCode_2(a',b')}"', from=1-1, to=2-2]
  \arrow["{\FunctorUniv{\F}}", from=1-3, to=2-2]
\end{tikzcd}
\]
\[
\begin{tikzcd}
  {\F(\CatEl_1(\SigmaCode_1(a, b)))} &&& {\F(\CatEl_1(b))} \\
  {\CatEl_2(\F(\SigmaCode_1(a, b)) \cdot \FunctorUniv{\F})} \\
  {\CatEl_2(\SigmaCode_2(a', b'))} & {\CatEl_2(b')} & {\CatEl_2(\FunctorEl{\F}^{-1} \cdot (\F(b) \cdot \FunctorUniv{\F}))} & {\CatEl_2(\F(b) \cdot \FunctorUniv{\F})}
  \arrow["{{\F(\SigmaIso_1(a, b))}}", from=1-1, to=1-4]
  \arrow["{{\FunctorEl{\F}}}"', from=1-1, to=2-1]
  \arrow["{{\FunctorEl{\F}}}", from=1-4, to=3-4]
  \arrow["\CatElEq"', from=2-1, to=3-1]
  \arrow["{{\SigmaIso_2(a', b')}}"', from=3-1, to=3-2]
  \arrow["\CatElEq"', from=3-2, to=3-3]
  \arrow["\CatElSub"', from=3-3, to=3-4]
\end{tikzcd}
\]
Here we define $a'$ to be $\F(a) \cdot \FunctorUniv{\F}$
and $b'$ to be $(\FunctorEl{\F}^{-1} \cdot \F(b)) \cdot \FunctorUniv{\F}$,
as displayed below.
\[
\begin{tikzcd}[row sep = large,column sep = huge]
  {\CatEl(\F(a) \cdot \FunctorUniv{\F})} && {\F(\CatEl(a))} \\
  & {\F(y)} && {\F(\CatUniv_1)} & {\CatUniv_2}
  \arrow["{\CatElM(\F(a) \cdot \FunctorUniv{\F})}"', from=1-1, to=2-2]
  \arrow["{\FunctorEl{\F}^{-1}}", from=1-1, to=1-3]
  \arrow["{\F(\CatElM(a))}"{description}, from=1-3, to=2-2]
  \arrow["{\F(b)}", from=1-3, to=2-4]
  \arrow["{\F(a)}"', from=2-2, to=2-4]
  \arrow["{\SigmaCode_2(a',b')}"', bend right=20, from=2-2, to=2-5]
  \arrow["{\FunctorUniv{\F}}"', from=2-4, to=2-5]
\end{tikzcd}
\]
We define the displayed bicategory \conceptDef{$\EltopNNOUnivSigma$}{Bicategories.ComprehensionCat.InternalLanguageTopos.ToposNatUniv}{disp_bicat_topos_with_NNO_univ_sigma}
over $\EltopNNOUniv$ whose
\begin{itemize}
  \item objects over a univalent elementary topos $\C$ with an elements map $\CatEl$
    are $\Sigma$-codes for $\CatEl$;
  \item morphisms over $(\F, \FunctorEl{\F})$ are proofs that $\F$ preserves $\Sigma$-codes.
\end{itemize}
The type of 2-cells over an arbitrary natural transformation is the unit type.
\end{defi}

\subsubsection{$\Pi$-types}
\label{sec:prod-type-univ}
The last type former that we consider,
is the $\Pi$-type former,
and our next goal is to define when a universe is closed under dependent products.
To do so,
we use similar ideas as for $\Sigma$-types:
we say that we have an operation on morphisms to the universe,
and we require this operation to be stable under substitution up to isomorphism.
Functors are required to preserve the operation and the isomorphism.
In the following,
we use some additional notation related to locally Cartesian closed categories.
Specifically, we need that dependent products are functorial in both arguments
and that they are stable under substitution.

\begin{nota}
Let $\C$ be a locally Cartesian closed category.
Given a morphism $f : y \rightarrow z$, we write $\Pi_f : \slice{\C}{y} \rightarrow \slice{\C}{z}$
for the right adjoint of $\functorfiber{f} : \slice{\C}{z} \rightarrow \slice{\C}{y}$,
which we call the dependent product.
Specifically,
if we have morphisms $f : y \rightarrow z$ and $g : x \rightarrow y$,
then $\Pi_f g : \slice{\C}{z}$ is a morphism,
which we write as $\prodMor{f}{g} : \Pi_f g \rightarrow z$.

The dependent product operation is functorial in both arguments.
Suppose that we have morphisms $h_x$ and $h_y$ as in the following commutative diagram.
\[
\begin{tikzcd}[row sep = large]
  {y_1} && {\substTy{h_x}{y_1}} & {y_2} \\
  {x_1} && {x_2} \\
  & z
  \arrow["{g_1}"', from=1-1, to=2-1]
  \arrow["{\substTyMor{h_x}}"', from=1-3, to=1-1]
  \arrow["{}"{description}, from=1-3, to=2-3]
  \arrow["{h_y}", from=1-3, to=1-4]
  \arrow["{g_2}", from=1-4, to=2-3]
  \arrow["{f_1}"', from=2-1, to=3-2]
  \arrow[""{name=0, anchor=center, inner sep=0}, "{h_x}"{description}, from=2-3, to=2-1]
  \arrow["{f_2}", from=2-3, to=3-2]
  \arrow["\lrcorner"{anchor=center, pos=0.125, rotate=-90}, draw=none, from=1-3, to=0]
\end{tikzcd}
\]
Then we have a morphism $\Pi_{h_x} h_y : \Pi_{f_1} g_1 \rightarrow \Pi_{f_2} g_2$.
This morphism is defined using the unit and counit of the adjunction.
If both $h_x$ and $h_y$ are isomorphisms,
then $\Pi_{h_x} h_y$ is an isomorphism as well.

Dependent products also are stable under pullback.
Stability means that
whenever we have morphisms $s : z' \rightarrow z$,
$f : y \rightarrow z$ and
$g : x \rightarrow y$,
then we have an isomorphism $\prodStable{s}{f}{g} : \iso{\substTy{s}{\Pi_f g}}{\Pi_{f'} g'}$
where $f'$ and $g'$ are defined using the following pullbacks.
\[
\begin{tikzcd}
  {\substTy{\substTyMor{s}}{g}} & x \\
  {\substTy{s}{f}} & y \\
  {z'} & z
  \arrow["{\substTyMor{\substTyMor{s}}}", from=1-1, to=1-2]
  \arrow["{g'}"', from=1-1, to=2-1]
  \arrow["\lrcorner"{anchor=center, pos=0.125}, draw=none, from=1-1, to=2-2]
  \arrow["g", from=1-2, to=2-2]
  \arrow["{\substTyMor{s}}", from=2-1, to=2-2]
  \arrow["{f'}"', from=2-1, to=3-1]
  \arrow["\lrcorner"{anchor=center, pos=0.125}, draw=none, from=2-1, to=3-2]
  \arrow["f", from=2-2, to=3-2]
  \arrow["s"', from=3-1, to=3-2]
\end{tikzcd}
\]
The desired isomorphism is constructed in \Cref{constr:disp-psfunctor-to-compcat}.
\end{nota}

\begin{defi}
\label[defi]{def:small-pi}
Let $\C$ be a univalent locally Cartesian closed category
with an elements map $\CatEl$ for $\CatUniv : \C$.
A \conceptDef{family of $\Pi$-codes}{Bicategories.ComprehensionCat.Universes.CatTypes.PiTypesBasics}{cat_univ_stable_codes_pi} for $\CatEl$
is a map that assigns to each $a : y \rightarrow \CatUniv$ and $b : \CatEl(a) \rightarrow \CatUniv$
a morphism $\PiCode(a, b) : y \rightarrow \CatUniv$
together with
an isomorphism $\PiIso(a, b) : \iso{\CatEl(\PiCode(a, b))}{\Pi_{\CatEl(a)} \CatEl(b)}$
that makes the following diagram commute.
\[
\begin{tikzcd}
  {\CatEl(\PiCode(a, b))} && {\prod_{\CatEl(a)} \CatEl(b)} \\
  & y
  \arrow["{\PiIso(a, b)}", from=1-1, to=1-3]
  \arrow["{\CatElM(\PiCode(a, b))}"', from=1-1, to=2-2]
  \arrow["{\prodMor{a}{b}}", from=1-3, to=2-2]
\end{tikzcd}
\]
We also require a stability condition for $\PiCode$ and $\PiIso$,
which says that for all morphisms $s : x \rightarrow y$,
$a : y \rightarrow \CatUniv$, and $b : \CatEl(a) \rightarrow \CatUniv$
we have that $s \cdot \PiCode(a, b) = \PiCode(s \cdot a, \CatElSub \cdot b)$,
which we display in the diagram below.
\[
\begin{tikzcd}[row sep = large, column sep = huge]
  {\CatEl(\CatElSub \cdot b)} & {\CatEl(b)} \\
  {\CatEl(s \cdot a)} & {\CatEl(a)} \\
  x & y & \CatUniv
  \arrow["\CatElSub", from=1-1, to=1-2]
  \arrow["{\CatElM(\CatElSub \cdot b)}"', from=1-1, to=2-1]
  \arrow["{\CatElM(b)}"', from=1-2, to=2-2]
  \arrow["\CatElSub", from=2-1, to=2-2]
  \arrow["{\CatElM(s \cdot a)}"', from=2-1, to=3-1]
  \arrow["\lrcorner"{anchor=center, pos=0.125}, draw=none, from=1-1, to=2-2]
  \arrow["\lrcorner"{anchor=center, pos=0.125}, draw=none, from=2-1, to=3-2]
  \arrow["{\CatElM(a)}"', from=2-2, to=3-2]
  \arrow["b", shorten >=1ex, bend left=30, from=2-2, to=3-3]
  \arrow["s"', from=3-1, to=3-2]
  \arrow["{\PiCode(s \cdot a, \CatElSub \cdot b)}"', shorten >=0.5ex, bend right=40, from=3-1, to=3-3]
  \arrow["a", shorten >=0.3ex, shift left, bend left=10, from=3-2, to=3-3]
  \arrow["{\PiCode(a, b)}"', shorten >=0.3ex, shift right, bend right=10, from=3-2, to=3-3]
\end{tikzcd}
\]
We required $\CatEl(s \cdot \PiCode(a, b))$ to be a pullback of $s$ and $\CatElM(\PiCode(a, b))$,
so we have an isomorphism $h : \iso{\substTy{s}{\CatEl(\PiCode(a, b))}}{\CatEl(s \cdot \PiCode(a, b))}$.
We also require the following diagram to commute.
\[
\begin{tikzcd}[column sep = 60pt]
  {\substTy{s}{\CatEl(\PiCode(a, b))}} && {\substTy{s}{\prod_{\CatEl(a)} \CatEl(b)}} \\
  {\CatEl(s \cdot \PiCode(a, b))} \\
  {\CatEl(\PiCode(s \cdot a, \CatElSub \cdot b))} & {\prod_{\CatEl(s \cdot a)} (\CatEl(\CatElSub \cdot b))} & {\prod_{a'} b'}
  \arrow["{\substTy{s}{\PiIso(a, b)}}", from=1-1, to=1-3]
  \arrow["{h}"', from=1-1, to=2-1]
  \arrow["{\prodStable{s}{\CatEl(a)}{\CatEl(b)}}", from=1-3, to=3-3]
  \arrow["\CatElEq"', from=2-1, to=3-1]
  \arrow["{\PiIso(s \cdot a, \CatElSub \cdot b)}"', from=3-1, to=3-2]
  \arrow["{\prod_{f} g}"', from=3-2, to=3-3]
\end{tikzcd}
\]
Here the isomorphisms $f$ and $g$
and the morphisms $a'$ and $b'$ are defined using the following pullback squares.
\[
\begin{tikzcd}
  {\CatEl(\CatElSub \cdot b)} && {\substTy{\substTyMor{s}}{\CatEl(b)}} & {\CatEl(b)} \\
  {\CatEl(s \cdot a)} && {\substTy{s}{\CatEl(a)}} & {\CatEl(a)} \\
  & x && y
  \arrow["g", from=1-1, to=1-3]
  \arrow["{\CatElM(\CatElSub \cdot b)}"', from=1-1, to=2-1]
  \arrow["{\substTyMor{\substTyMor{s}}}", from=1-3, to=1-4]
  \arrow["{b'}"{description}, from=1-3, to=2-3]
  \arrow["{\CatElM(b)}", from=1-4, to=2-4]
  \arrow["f"{description}, from=2-1, to=2-3]
  \arrow["{\CatElM(s \cdot a)}"', from=2-1, to=3-2]
  \arrow[""{name=0, anchor=center, inner sep=0}, "{\substTyMor{s}}", from=2-3, to=2-4]
  \arrow["{a'}"{description}, from=2-3, to=3-2]
  \arrow["\lrcorner"{anchor=center, pos=0.125}, draw=none, from=2-3, to=3-4]
  \arrow["{\CatElM(a)}", from=2-4, to=3-4]
  \arrow["s"', from=3-2, to=3-4]
  \arrow["\lrcorner"{anchor=center, pos=0.125}, draw=none, from=1-3, to=0]
\end{tikzcd}
\]
\end{defi}

Again the analogous closure condition for small maps follows from a family of $\Pi$-codes.

\begin{propL}[\coqdocurl{Bicategories.ComprehensionCat.Universes.SmallMaps}{pi_types_small}]
\label[propL]{prop:small-pi}
If the morphisms $f : y \rightarrow z$ and $g : x \rightarrow y$ are small,
then $\prodMor{f}{g} : \Pi_f g \rightarrow z$ is small as well. 
\end{propL}

\begin{proof}
Proven the same way as \Cref{prop:small-sigma}
using $\PiCode$ and $\PiIso$ instead of $\SigmaCode$ and $\SigmaIso$.
\end{proof}

To define the displayed bicategory of $\Pi$-codes,
we define when a functor preserves such codes.
Whenever we have a functor $\F$ between locally Cartesian closed categories $\C_1$ and $\C_2$,
then,
by definition,
we have an isomorphism $\functorProd{\F}(f, g) : \iso{\F(\Pi_f g)}{\Pi_{\F(f)} \F(g)}$
for each $f : y \rightarrow z$ and $g : x \rightarrow y$.

\begin{defi}
\label[defi]{def:pi-code-disp-bicat}
Suppose that we have elements maps $\CatEl_1$ and $\CatEl_2$ for $\CatUniv_1 : \C_1$ and $\CatUniv_2 : \C_2$ respectively,
where $\C_1$ and $\C_2$ are univalent locally Cartesian closed categories.
Let $\F : \C_1 \rightarrow \C_2$ be a functor that preserves exponentials
and let $\FunctorEl{\F}$ be an elements preservation family for $\F$.
Let $\PiCode_1$ and $\PiCode_2$ be $\Pi$-codes for $\CatEl_1$ and $\CatEl_2$ respectively.
We say that $\F$ \conceptDef{preserves $\Pi$-codes}{Bicategories.ComprehensionCat.Universes.CatTypes.PiTypesBasics}{functor_preserves_stable_codes_pi}
whenever the following diagrams commute for all $a : y \rightarrow \CatUniv_1$ and $b : \CatEl(a) \rightarrow \CatUniv_1$.
\[
\begin{tikzcd}
  {\F(y)} && {\F(\CatUniv_1)} \\
  & {\CatUniv_2}
  \arrow["{\F(\PiCode_1(a,b))}", from=1-1, to=1-3]
  \arrow["{\PiCode_2(a',b')}"', from=1-1, to=2-2]
  \arrow["{\FunctorUniv{\F}}", from=1-3, to=2-2]
\end{tikzcd}
\]
\[
\begin{tikzcd}[column sep = huge]
  {\F(\CatEl_1(\PiCode_1(a, b)))} && {\F(\prod_{\CatEl_1(a)} \CatEl_1(b))} \\
  {\CatEl_2(\F(\PiCode_1(a, b) \cdot \FunctorUniv{\F}))} \\
  {\CatEl_2(\PiCode_2(a', b'))} & {\prod_{\CatEl_2(a')} \CatEl_2(b')} & {\prod_{\F(\CatEl_1(a))} \F(\CatEl_1(b))}
  \arrow["{\F(\PiIso_1(a, b))}", from=1-1, to=1-3]
  \arrow["{\FunctorEl{\F}}"', from=1-1, to=2-1]
  \arrow["{\functorProd{\F}(\CatEl_1(a), \CatEl_1(b))}", from=1-3, to=3-3]
  \arrow["\CatElEq"', from=2-1, to=3-1]
  \arrow["{\PiIso_2(a',b')}"', from=3-1, to=3-2]
  \arrow["\prod_{\FunctorEl{\F}(a)} g"', from=3-2, to=3-3]
\end{tikzcd}
\]
Here we define $a'$ to be $\F(a) \cdot \FunctorUniv{\F}$
and $b'$ to be $(\FunctorEl{\F}^{-1} \cdot \F(b)) \cdot \FunctorUniv{\F}$
as displayed below.
\[
\begin{tikzcd}[column sep = huge]
  {\CatEl(\F(a) \cdot \FunctorUniv{\F})} && {\F(\CatEl(a))} \\
  & {\F(y)} && {\F(\CatUniv_1)} & {\CatUniv_2}
  \arrow["{\CatElM(\F(a) \cdot \FunctorUniv{\F})}"', from=1-1, to=2-2]
  \arrow["{\FunctorEl{\F}^{-1}}", from=1-1, to=1-3]
  \arrow["{\F(\CatElM(a))}"{description}, from=1-3, to=2-2]
  \arrow["{\F(b)}", from=1-3, to=2-4]
  \arrow["{\F(a)}"', from=2-2, to=2-4]
  \arrow["{\PiCode_2(a',b')}"', bend right=20, from=2-2, to=2-5]
  \arrow["{\FunctorUniv{\F}}"', from=2-4, to=2-5]
\end{tikzcd}
\]
The morphism $g : \substTy{\FunctorEl{\F}}{\CatEl_2(b')} \rightarrow \F(\CatEl_1(b))$ is defined as the following composition.
\[
\begin{tikzcd}
  {\substTy{\FunctorEl{\F}}{\CatEl_2(b')}} & {\CatEl_2(b')} & {\CatEl_2(\FunctorEl{\F}^{-1} \cdot (\F(b) \cdot \FunctorUniv{\F}))} & {\CatEl_2(\F(b) \cdot \FunctorUniv{\F})} & {\F(\CatEl_1(b))} \\
  {\F(\CatEl_1(a))} & {\CatEl_2(a')}
  \arrow["{{\substTyMor{\FunctorEl{\F}}}}", from=1-1, to=1-2]
  \arrow[from=1-1, to=2-1]
  \arrow["\lrcorner"{anchor=center, pos=0.125}, draw=none, from=1-1, to=2-2]
  \arrow["\CatElEq", from=1-2, to=1-3]
  \arrow["{\CatElM_2(b')}", from=1-2, to=2-2]
  \arrow["\CatElSub", from=1-3, to=1-4]
  \arrow["{{\FunctorEl{\F}(b)^{-1}}}", from=1-4, to=1-5]
  \arrow["{\FunctorEl{\F}}"', from=2-1, to=2-2]
\end{tikzcd}
\]

We define the displayed bicategory \conceptDef{$\EltopNNOUnivProd$}{Bicategories.ComprehensionCat.InternalLanguageTopos.ToposNatUniv}{disp_bicat_topos_with_NNO_univ_pi}
over $\EltopNNOUniv$ whose
\begin{itemize}
  \item objects over a univalent elementary topos $\C$ with an elements map $\CatEl$
    are $\Pi$-codes for $\CatEl$;
  \item morphisms over $(\F, \FunctorEl{\F})$ are proofs that $\F$ preserves $\Pi$-codes.
\end{itemize}
The type of 2-cells over an arbitrary natural transformation is the unit type.
\end{defi}

\subsubsection{Topos Universes}
Now finally everything is in place for the biequivalence that we aimed to construct in this section.
We start by collecting the displayed bicategories that we constructed so far
to obtain the displayed bicategories of topos universes.

\begin{defi}
\label{def:topos-universe-bicat}
We define the displayed bicategory
\conceptDef{$\EltopNNOTopUnivD$}{Bicategories.ComprehensionCat.InternalLanguageTopos.ToposNatUniv}{disp_bicat_topos_with_NNO_univ_topos_univ}
over $\EltopNNOUniv$
to be
\[
\EltopNNOUnivNat \times \EltopNNOUnivOmega \times \EltopNNOUnivResize \times \EltopNNOUnivSigma \times \EltopNNOUnivProd
\]
The bicategory
\conceptDef{$\EltopNNOTopUniv$}{Bicategories.ComprehensionCat.InternalLanguageTopos.ToposNatUniv}{bicat_topos_with_NNO_univ_topos_univ}
is defined to be $\total{\EltopNNOTopUnivD}$.
In addition, we define
the displayed bicategory
\conceptDef{$\EltopNNOCCTopUnivD$}{Bicategories.ComprehensionCat.InternalLanguageTopos.ToposNatUniv}{disp_bicat_topos_with_NNO_univ_topos_univ}
over $\EltopNNOCCUniv$ to be
$\reindex{\CompCatToFinLimUnivTop}{\EltopNNOTopUnivD}$,
and we define $\EltopNNOCCTopUniv$ to be $\total{\EltopNNOCCTopUnivD}$.
\end{defi}

\begin{propL}[{\coqdocurl{Bicategories.ComprehensionCat.InternalLanguageTopos.ToposNatUniv}{disp_univalent_2_disp_bicat_topos_with_NNO_univ_topos_univ}}]
The displayed bicategory $\EltopNNOTopUnivD$ is univalent.
\end{propL}

\begin{proof}
It suffices to show that each of the displayed bicategories
$\EltopNNOUnivNat$, $\EltopNNOUnivOmega$, $\EltopNNOUnivResize$, $\EltopNNOUnivSigma$, and $\EltopNNOUnivProd$ is univalent.
For $\EltopNNOUnivNat$,
it suffices to show that each proof that the identity preserves codes for natural numbers
from $(\NatCode_1, \NatIso_1)$ to $(\NatCode_2, \NatIso_2)$ gives rise to an equality
between $(\NatCode_1, \NatIso_1)$ and $(\NatCode_2, \NatIso_2)$,
and details are given in the formalization.
The univalence of the other displayed bicategories is proven in the same way.
\end{proof}

By \Cref{prop:reindex-disp-bicat-univalent},
it follows that $\EltopNNOCCTopUnivD$ is univalent as well.
Let us state what the objects of the bicategory $\EltopNNOCCTopUniv$ are.
An object is given by a DFL comprehension category $\compfunctor : \D \rightarrow \ArrD{\C}$
that satisfies the local property $\EltopNNOP$ fiberwise
and with a universe object $\CatUniv : \dobP{\D}{\ECtx}$
and elements map $\CatEl$.
This universe object is required to satisfy various closure conditions,
and we explicate the closure condition for $\Nat$.
Since $\compfunctor$ is a fiberwise equivalence by \Cref{prop:comprehension-eso},
each slice category $\slice{\C}{\Gamma}$ has a natural numbers object,
and we write $\Nat$ for the natural numbers object of $\slice{\C}{\ECtx}$.
By definition,
we have a morphism $\NatCode : \ECtx \rightarrow \CatUniv$
and an isomorphism $\NatIso : \iso{\CatEl(\NatCode)}{\Nat}$.
These morphisms give us a term $\NatCodeTm : \Tms{\ECtx}{\CatUniv}$
together with an isomorphism $\NatIsoTm$ between $\CatEl(\NatCodeTm)$
and the fiberwise natural numbers object in $\D$,
which proves the expected derivation rules.
For the other type formers,
we can give a similar argument.

We finish this section with the desired biequivalence.

\begin{problem}
\label{prob:topos-universe-biequiv}
To construct a biequivalence $\biequiv{\EltopNNOTopUniv}{\EltopNNOCCTopUniv}$.
\end{problem}

\begin{construction}{\coqdocurl{Bicategories.ComprehensionCat.InternalLanguageTopos.ToposNatUniv}{internal_language_univ_topos_with_NNO_topos_univ}}{prob:topos-universe-biequiv}
\label{constr:topos-universe-biequiv}
By \Cref{constr:total-biequiv,constr:reindex-biequiv}.
\end{construction}

\section{Related Work}
\label{sec:related-work}
In this section, we discuss related work,
and we start by recalling how our work compares to that of Clairambault and Dybjer~\cite{clairambault:2014}.
Clairambault and Dybjer~\cite[Theorem 6.1]{clairambault:2014} constructed a biequivalence between
the bicategory of categories with finite limits
and the bicategory of democratic CwFs that support extensional identity types and $\Sigma$-types,
and a biequivalence between
the bicategory of locally Cartesian closed categories
and the bicategory of democratic CwFs that support extensional identity types, $\Sigma$-types and $\Pi$-types.
This paper is very much in the spirit of Clairambault and Dybjer,
and we give analogous statements in \Cref{constr:biequiv,constr:lccc-biequiv}.
The main differences arise from the fact that we work in univalent foundations and with univalent categories,
while Clairambault and Dybjer work in set-theoretic foundations.
Because of the difference in foundations,
we use comprehension categories, whereas they use categories with families,
and we do not use split fibrations in our development.
This changes how we define the pseudofunctor from categories to comprehension categories (\Cref{constr:finlim-to-compcat}),
because we do not need to construct an equivalent split fibration~\cite{benabou:1985}.

The methods and results in \Cref{sec:topos} are based on the work by Maietti~\cite{maietti:2005}.
However,
Maietti only proved soundness and completeness for various dependent type theories~\cite[Theorems 5.8 and 5.28]{maietti:2005},
while our goal is to construct biequivalences.
For this reason,
we strengthened our notion of local property (\Cref{def:local-property})
to guarantee pseudofunctoriality.

In \Cref{sec:universes,sec:universe-types},
we considered universes in categories,
and our notions were inspired by the development of universes
in the context of algebraic set theory~\cite{awodey:2014,joyal:1995}.
Specifically, we followed Streicher's notion of universe in toposes~\cite{streicher:2005},
but there is one fundamental difference between our work and Streicher's.
We express closure conditions for universes as operations satisfying suitable substitution laws
instead of properties on some class of small maps.
Similar ideas are used by Angiuli and Gratzer to define universes in categories with families~\cite[Structure 6.4.17]{angiuli2024principles}.

Otten and Spadetto~\cite{otten:2025} constructed a biequivalence between path categories~\cite{vandenberg:2018}
and comprehension categories that support $\Sigma$-types and axiomatic identity types.
Path categories give a categorical structure in which one can interpret homotopy theory
and intensional versions of type theory.
Specifically,
they consider axiomatic type theory,
which is intensional type theory where computation rules only hold up to propositional equality,
while we only considered extensional type theory.

Ahrens et al.~\cite{ahrens:2018} compared various categorical structures to model dependent type theory in univalent foundations,
and analyzed how these structures transfer along weak and strong equivalences of categories.
More specifically, they studied categories with families, split type-categories, and representable maps of presheaves.
Note that these structures do not satisfy the requirements in \Cref{sec:models},
because the types in each context are required to be a set.
Moreover, they proved that the types of these structures are equivalent as types,
whereas we construct biequivalences between univalent bicategories.
The latter is more general,
because such biequivalences give rise to equivalences between the types of objects~\cite[Proposition 3.10]{ahrens:2021}.

In \Cref{sec:models}, we argued that categories with families are unable to capture the set model in univalent foundations,
and we explained why we use comprehension categories instead.
Another approach was used by Gratzer et al.~\cite{gratzer:2024}.
Instead of looking at all sets,
they constructed the set model of type theory by restricting themselves to \textbf{iterative sets}.
Since the type of iterative sets is a set itself~\cite[Theorem 12]{gratzer:2024},
there is a CwF whose types are given by sets.
However, the category of iterative sets is only a setcategory and not a univalent category,
and this is the main difference between our work and that of Gratzer et al.~\cite{gratzer:2024}.
While we are interested in models in the realm of univalent categories,
they gave a model in the realm of setcategories.
Note that Altenkirch and Kaposi used inductive-recursive types for similar purposes~\cite{altenkirch:2016},
and that both constructions can be generalized by restricting a comprehension category to a universe~\cite{najmaei:2026a,voevodsky:2015}.

\section{Conclusion}
\label{sec:conclusion}
In this paper, we set out to prove internal language theorems for various classes of univalent categories.
We started by arguing that comprehension categories give a suitable categorical structure for interpreting dependent type theory in univalent foundations (\Cref{sec:models})
while categories with families do not,
and we constructed the univalent bicategories of full comprehension categories (\Cref{def:disp-bicat-full-comp-cat})
and DFL comprehension categories (\Cref{def:dfl-comp-cat}).
After that we proved internal language theorems for univalent categories (\Cref{constr:biequiv,constr:lccc-biequiv})
and extensions to various classes of toposes (\Cref{exa:internal-language-topos}),
including elementary toposes with a universe (\Cref{constr:topos-universe-biequiv}).

Univalence simplified our constructions and proofs,
because univalent categories are identified up to adjoint equivalence.
As a consequence,
we did not use split fibrations,
which simplified \Cref{constr:finlim-to-compcat}.
In addition,
we frequently had to transfer structure and properties along equivalences in our constructions,
which we can do for free for univalent categories.
Univalence also simplified the proofs of \Cref{prop:pointwise-adj-equiv,prop:dfl-compcat-adjequiv}.

There are numerous ways to extend this work.
While we considered some predicative versions of toposes,
such as pretoposes and $\Pi$-pretoposes,
our primary focus has been on elementary toposes.
In particular,
we only considered universes for elementary toposes in \Cref{sec:universe-types}.
One interesting direction would be to extend our development to $\Pi$W-pretoposes
and to Martin-L\"of categories
with universes~\cite{abbott:2005,berg:2012,moerdijk:2002}.
For such an extension,
W-types have to be considered as well,
and for those,
one could use similar ideas as in \Cref{sec:topos}.
Specifically,
Moerdijk and Palmgren~\cite[Proposition 3.8]{moerdijk:2000} showed that W-types are preserved under slicing,
and thus one has a local property of W-types in locally Cartesian closed categories.
In addition,
one has to consider additional closure conditions for universes in a predicative setting.
For instance,
the universe must be closed under W-types,
and rather than all monomorphisms, it is required to contain all equalizers and isomorphisms~\cite{vandenberg:2009a}.
Note that such a biequivalence is present in the work by Abbott et al.~\cite{abbott:2005},
who used Martin-L\"of categories as a semantical framework to study containers.
Their development rests on equivalence between type theories and categories as originally established by Seely~\cite{seely:1984},
so that they can use type theory to reason about categories.

Even though there are various instances of univalent comprehension categories,
such as sets and presheaves,
they do not capture every example of interest.
As an example,
we consider the groupoid model~\cite{hofmann:1998}.
The groupoid model can be incarnated in two ways in univalent foundations:
one could either use setgroupoids (i.e., groupoids whose type of objects is a set),
or univalent groupoids (i.e., $1$-types).
While one can construct a univalent comprehension category of setgroupoids,
there is no univalent comprehension category of 1-types.
Since the type of 1-types is only 2-truncated,
one cannot have a univalent category of 1-types by \Cref{prop:hlevel-obj},
and thus it becomes necessary to use univalent bicategories instead.
Hence,
to construct models of type theory based on, for instance, 1-types or types without any truncation requirement,
one needs to consider higher dimensional models~\cite{ahrens:2023,kraus:2021}.
Such models have been considered in both type theory and set-theoretic foundations~\cite{ahrens:2023,garner:2009,kraus:2021,licata:2012},
but suitable internal language theorems for them remain to be developed.

Finally,
we did not discuss syntax and the initial model in this paper.
If we have a comprehension category $\compfunctor : \D \rightarrow \ArrD{\C}$,
then we have two methods to interpret the syntax of type theory in it.
Assuming that $\compfunctor$ comes with a universe $\CatUniv$,
then we can restrict $\compfunctor$ to $\CatUniv$ to obtain a CwF~\cite{najmaei:2026a,voevodsky:2015}.
This construction generalizes the iterative sets model by Gratzer et al.~\cite{gratzer:2024}
and the model by Altenkirch and Kaposi~\cite{altenkirch:2016} using an inductive-recursive universe.
It also is similar to Damato's construction of the container model as a CwF~\cite{damato:2026}.
However,
the CwF obtained from this strictification construction is not equivalent to the comprehension category that we started with,
since we restrict ourselves to the types in a fixed universe.

The other method to interpret type theory in a comprehension category
is via the groupoid syntax of type theory introduced by Altenkirch, Kaposi, and Xie~\cite{altenkirch:2025}.
The key feature of the groupoid syntax is that the types are assumed to form a groupoid rather than a set,
and it is constructed as a higher inductive-inductive type~\cite{kaposi:2020a}
instead of a quotient inductive-inductive type~\cite{altenkirch:2016}.
However, such a definition presents us with a coherence issue.
Since the syntax is only required to form a 1-type,
we need to prove ourselves that every identity between types is equal.
Altenkirch, Kaposi, and Xie proved this property for a version of Martin-L\"of type theory
with $\Pi$-types and a universe~\cite{altenkirch:2025}.
Our interest is in a much wider class of type formers,
as we also considered type constructors such as
$\Sigma$-types, extensional identity types, natural numbers, and subobject classifiers.
We can extend the groupoid syntax of type theory to include these type formers as well.
Specifically, Maietti gave syntactical rules for each of these type formers~\cite[Section 3]{maietti:2005},
and they can soundly be interpreted in comprehension categories with the corresponding structure.
By adding the appropriate constructors to the higher inductive-inductive type that represents the groupoid syntax,
we can thus include each of these type formers.
A coherence theorem stating that the resulting syntax is still set-truncated remains to be proven.

\section*{Acknowledgements}
  The author greatly benefited from discussions and suggestions from Peter Dybjer.
  In addition, the author thanks Benedikt Ahrens for reviewing the code related to this paper and for providing valuable feedback.
  The author also thanks Yun Liu, Nima Rasekh,
  and the anonymous reviewers for LICS and LMCS for their careful reading and valuable comments on this paper.
  The author also thanks Fernando Chu for providing a list of comments to improve the writing of the paper.
  The author also thanks Daniel Gratzer for comments and discussion to clarify the coherence issue.
  The author gratefully acknowledges the work by the Rocq development team in providing the Rocq proof assistant and surrounding infrastructure, as well as their support in keeping UniMath compatible with Rocq.
  This research was supported by the NWO project ``The Power of Equality'' OCENW.M20.380, which is financed by the Dutch Research Council (NWO).
  This research was funded by the European Union (ERC, HOTT, 101170308).
  Views and opinions expressed are however those of the author(s) only and do not
  necessarily reflect those of the European Union or the European Research Council.
  Neither the European Union nor the granting authority can be held responsible for them.

\bibliographystyle{alphaurl}
\bibliography{literature}
  
\end{document}